\documentclass[notitlepage,leqno]{report}

\usepackage{latexsym}
\usepackage[english]{babel}
\usepackage{t1enc}
\usepackage{mathrsfs}
\usepackage{ifthen}
\usepackage{url}
\usepackage{epsfig}
\usepackage{amsbsy}
\usepackage{extarrows}
\usepackage{amsfonts}
\usepackage{amsmath}
\usepackage{amssymb}
\usepackage{amsthm}
\usepackage{amsxtra}
\usepackage{bbm}
\usepackage{color}
\usepackage{relsize} 
\usepackage{enumitem}
\usepackage{graphicx}
\usepackage{caption}
\usepackage{subcaption}
\usepackage{placeins}
\usepackage{wrapfig}
\usepackage{float}
\usepackage[vmargin=1.2in,hmargin=1.75in]{geometry}
\usepackage{verbatim}

\usepackage{hyperref}

\newtheorem{theorem}{Theorem}[chapter]
\newtheorem{proposition}[theorem]{Proposition} 
\newtheorem{corollary}[theorem]{Corollary}     
\newtheorem{lemma}[theorem]{Lemma}

\theoremstyle{definition}
\newtheorem{definition}[theorem]{Definition}

\newtheorem{convention}[theorem]{Convention} 

\theoremstyle{remark}
\newtheorem{remark}[theorem]{Remark}

\numberwithin{section}{chapter}
\numberwithin{equation}{chapter}
\numberwithin{figure}{chapter}

\newenvironment{Proof}[1][Proof]{\begin{proof}[\sc{#1}]}{\end{proof}} 

\newcommand{\NLemma}[2] {
        \begin{lemma}[#1] \label{lmm:#1}
                #2
        \end{lemma}
        }

\newcommand{\NCorollary}[2] {
        \begin{corollary}[#1] \label{crl:#1}
                #2
        \end{corollary}
        }

\newcommand{\NDefinition}[2] {
        \begin{definition}[#1] \label{def:#1}
                #2
        \end{definition}
        }

\newcommand{\bels}[2] {
        \begin{equation} \label{#1} \begin{split} 
                #2 
        \end{split} \end{equation}
        }
\newcommand{\bea}[1]{
	\begin{align*}
		#1
	\end{align*}
	}

\newcommand{\chapterl}[1]    {\chapter{#1} \label{chp:#1}}
\newcommand{\sectionl}[1]    {\section{#1} \label{sec:#1}}
\newcommand{\subsectionl}[1] {\subsection{#1} \label{ssec:#1}}

\definecolor{olivegreen}{rgb}{0,0.6,0.1}

\newcommand{\Lp}[1]{\mathrm{L}^{\!#1}}					
\newcommand{\BB}{\mathscr{B}}						

\newcommand{\Pob}{Q}

\newcommand{\Ind} {\mathbbm{1}}

\newcommand{\vect}[1]{\mathbf{#1}}					

\newcommand{\1} {\mspace{1 mu}}
\newcommand{\2} {\mspace{2 mu}}
\newcommand{\msp}[1] {\mspace{#1 mu}}

\newcommand{\la} {\langle}
\newcommand{\ra} {\rangle}
\newcommand{\avg}[1] {\la #1 \ra}
\newcommand{\avgb}[1] {\bigl\la #1 \bigr\ra}
\newcommand{\avgB}[1] {\Bigl\la #1 \Bigr\ra}
\newcommand{\avgbb}[1] {\biggl\la #1 \biggr\ra}

\newcommand{\eps}{\varepsilon}

\newcommand{\dist} {\mrm{dist}}                   
	
\DeclareMathOperator*{\sign}{sign}						

\newcommand{\mrm}[1] {\mathrm{#1}}
\newcommand{\mcl}[1] {\mathcal{#1}}

\newcommand{\brm}[1] {\boldsymbol{\mathrm{#1}}}

\newcommand{\wti}[1] {\widetilde{#1}}
\newcommand{\wht}[1] {\widehat{#1}}
\newcommand{\ul}[1] {\underline{#1}}

\newcommand{\EE} {\mathbbm{E}}
\newcommand{\PP}  {\mathbbm{P}}

\newcommand{\mat}[1]{\begin{bmatrix} #1 \end{bmatrix}} 

\newcommand{\diag} {\mrm{diag}}

\DeclareMathOperator{\supp} {supp}

\newcommand{\trans}{\mathrm{T}}
\newcommand{\ins} {\msp{1}\in\msp{1}}
\newcommand{\Spec}{\mrm{Spec}}

\newcommand{\abs}[1]{\lvert #1 \rvert}
\newcommand{\absb}[1]{\big\lvert #1 \big\rvert}
\newcommand{\absB}[1]{\Bigl\lvert #1 \Bigr\rvert}
\newcommand{\absbb}[1]{\biggl\lvert #1 \biggr\rvert}

\newcommand{\norm}[1]{\lVert #1 \rVert}
\newcommand{\normb}[1]{\big\lVert #1 \big\rVert}
\newcommand{\normB}[1]{\Bigl\lVert #1 \Bigr\rVert}

\newcommand{\R} {\mathbb{R}}
\newcommand{\C} {{\mathbb{C}}}

\newcommand{\N} {\mathbb{N}}

\newcommand{\Cp} {\mathbb{H}}
\newcommand{\eCp} {\overline{\mathbb{H}}}

\newcommand{\D} {\mathbb{D}}

\newcommand{\Borel}{\mathcal{B}}

\newcommand{\sett}[1] { \{ {#1} \} }

\newcommand{\setb}[1] { \bigl\{ {#1} \bigl\} }
\newcommand{\setB}[1] { \Bigl\{ {#1} \Bigr\} }
\newcommand{\setbb}[1] { \biggl\{\, {#1} \,\biggr\} }

\newcommand{\cmpl} {\mathrm{c}}

\newcommand{\titem}[1] {\item[\emph{(#1)}]} 

\newcommand{\genarg} {{\,\cdot\,}}  

\newcommand{\dif} {\mathrm{d}}
\newcommand{\cI} {\mathrm{i}}
\newcommand{\nE} {\mathrm{e}}
\newcommand{\Ord} {\mathcal{O}}

\renewcommand{\Im}{\mathrm{Im}}
\renewcommand{\Re}{\mathrm{Re}}

\newcommand{\MM} {\mathbb{M}}
\newcommand{\MMe} {\mathbb{M}_{\1\eps}}
\newcommand{\DD} {\mathbb{D}}
\newcommand{\DDe} {\mathbb{D}_{\eps}}

\newcommand{\KK} {\mathbb{K}}
\newcommand{\LL} {\mathbb{L}}

\newcommand{\Sx} {\mathfrak{X}} 
\newcommand{\Px} {\pi}          

\newcommand{\am} {q} 

\newcommand{\tsfrac}[2] {{\textstyle \frac{#1}{#2}}}

\newcommand{\nnorms}[1]{{\left\vert\kern-0.25ex\left\vert\kern-0.25ex\left\vert #1 
    \right\vert\kern-0.25ex\right\vert\kern-0.25ex\right\vert}}

\newcommand{\nnorm}[1]{{\vert\kern-0.25ex\vert\kern-0.25ex\vert #1 
    \vert\kern-0.25ex\vert\kern-0.25ex\vert}}

\usepackage{mathtools}
\makeatletter
\DeclareRobustCommand\widecheck[1]{{\mathpalette\@widecheck{#1}}}
\def\@widecheck#1#2{%
    \setbox\z@\hbox{\m@th$#1#2$}%
    \setbox\tw@\hbox{\m@th$#1%
       \widehat{%
          \vrule\@width\z@\@height\ht\z@
          \vrule\@height\z@\@width\wd\z@}$}%
    \dp\tw@-\ht\z@
    \@tempdima\ht\z@ \advance\@tempdima2\ht\tw@ \divide\@tempdima\thr@@
    \setbox\tw@\hbox{%
       \raise\@tempdima\hbox{\scalebox{1}[-1]{\lower\@tempdima\box
\tw@}}}%
    {\ooalign{\box\tw@ \cr \box\z@}}}
\makeatother

\makeindex

\begin{document}

\renewcommand{\thefootnote}{\fnsymbol{footnote}}
\title{\huge \bf Quadratic vector equations on complex upper half-plane 
}

\author{
\begin{minipage}{0.3\textwidth}
 \begin{center}
Oskari H. Ajanki\footnotemark[1]\\
\footnotesize 
{IST Austria}\\
{\url{oskari.ajanki@iki.fi}}
\end{center}
\end{minipage}
\begin{minipage}{0.3\textwidth}
\begin{center}
L\'aszl\'o Erd{\H o}s\footnotemark[2]  \\
\footnotesize {IST Austria}\\
{\url{lerdos@ist.ac.at}}
\end{center}
\end{minipage}
\begin{minipage}{0.3\textwidth}
 \begin{center}
Torben Kr\"uger\footnotemark[3]\\
\footnotesize 
{IST Austria}\\
{\url{torben.krueger@ist.ac.at}}
\end{center}
\end{minipage}
}

\date{} 

\maketitle
\thispagestyle{empty} 
	
\footnotetext[1]{Partially supported by ERC Advanced Grant RANMAT No.\ 338804, and SFB-TR 12 Grant of the German Research Council.}
\footnotetext[2]{Partially supported by ERC Advanced Grant RANMAT No.\ 338804.}
\footnotetext[3]{Partially supported by ERC Advanced Grant RANMAT No.\ 338804, and SFB-TR 12 Grant of the German Research Council}

\vspace{-0.5cm}
\begin{abstract}
We consider the  nonlinear equation $-\frac{1}{m}=z+Sm$ with a parameter $z$ in the complex upper half plane $\mathbb{H} $, where $S$ is a positivity preserving symmetric linear operator acting on bounded functions. The solution  with values in $ \mathbb{H}$ is unique and its $z$-dependence is conveniently described as the Stieltjes transforms of a family of measures $v$ on $\mathbb{R}$. In~\cite{AEK1cpam} we   qualitatively identified  the possible singular behaviors of $v$: under suitable conditions on $S$ we showed  that in the density of $v$  only algebraic singularities of degree two or three may occur. In this paper we give a comprehensive analysis of these singularities with uniform quantitative controls. We also find a universal shape describing the transition regime between the square root and cubic root singularities. Finally, motivated by random matrix applications in the companion paper~\cite{AEK2}, we  present a complete stability analysis of the equation for any $z\in \mathbb{H}$, including the vicinity of the singularities.
\end{abstract}
\vspace{0.8cm}
{\bf Keywords:} Stieltjes-transform,
Algebraic singularity,
Density of states, \\
Cubic cusp, Wigner-type random matrix.
\\
{\bf AMS Subject Classification (2010):} \texttt{45Gxx}, \texttt{46Txx}, \texttt{60B20}, \texttt{15B52}.

\pagenumbering{roman}
\setcounter{page}{1}
\tableofcontents

\chapterl{Introduction}
\pagenumbering{arabic}

One of the basic problems in the theory of large random matrices is to compute the asymptotic density of eigenvalues as the dimension of the matrices goes to infinity. For several prominent ensembles this question is ultimately related to the solution of a system of nonlinear equations of the form
\bels{NSCE}{
-\,\frac{1}{m_i\!}\;=\;z\,+a_i+\,\sum_{j=1}^N s_{ij} \1 m_j
\,, 
\qquad  i =1,\dots, N
\,,
}
where  the complex parameter $z$ and the unknowns $ m_1,\dots,m_N$ lie in the complex upper half-plane $\Cp := \sett{z\in \C :\, \Im \,z\,>\,0} $.  
The given vector $ \brm{a} = (a_i)_{i=1}^N $ has real components, and the matrix $\brm{S}=(s_{ij})_{i,j=1}^N$ is  symmetric with non-negative entries  and it is determined by  the second moments of the matrix ensemble. 

The simplest example for the emergence of \eqref{NSCE}  are the  \emph{Wigner-type matrices}, defined as follows. 
Let $ \brm{H}=(h_{ij})$ be an $N\times N$ real symmetric or complex hermitian matrix with expectations $\EE\,h_{ij}= -\1a_i\delta_{ij}$  and variances $\EE\2\abs{h_{ij}}^2=s_{ij}$.
We assume that the matrix elements are independent up to the symmetry constraint, $ h_{ji}= \overline{h}_{ij} $. 
Let $ \brm{G}(z)=(\brm{H}-z)^{-1}$ be the resolvent of $\brm{H}$ with a spectral parameter $ z \in \Cp $.
Second order perturbation theory indicates that for the diagonal matrix elements $ G_{ii}=G_{ii}(z) $ of the resolvent we have
\bels{Gii}{
- \frac{1}{G_{ii}\!} \,\approx\,  z +a_i+ \sum_{j=1}^N s_{ij} G_{jj}
\,,
}
where the error is due to fluctuations that vanish in the large $N$ limit.
In particular, if the system of equations \eqref{NSCE} is stable, then $ G_{ii} $ is close to $ m_i $ and the average $N^{-1}\sum_i m_i$ approximates the normalized trace of the resolvent, $ N^{-1}\mrm{Tr}\, \brm{G}$.
Being determined by $N^{-1} \Im\,\mrm{Tr}\, \brm{G}$, as $ \Im\,z \to 0 $, the empirical spectral measure of  $\brm{H}$
approaches the non-random measure with density
\bels{DOS}{
\rho(\tau) \,:=\, 
\lim_{\eta\1\downarrow\1 0}
\frac{1}{\pi N}
\sum_{j=1}^N
\Im\,m_j\msp{-1}(\tau+\cI\1\eta)
\,, \qquad \tau\in \R
\,,
}
as $ N $ goes to infinity, see \cite{ShlyakhtenkoGBM,Guionnet-GaussBand,AZind}.
Apart from a few specific cases, this procedure via \eqref{NSCE} is the only known method to determine
the limiting density of eigenvalues for large Wigner-type random matrices.

When $ \brm{S} $ is doubly stochastic, i.e., $ \sum_j s_{ij} = 1 $ for each row $ i $, then it is easy to see that the only solution to  \eqref{NSCE} is the constant vector, $m_i =m_{\mrm{sc}} $ for each $i$, where $ m_{\mrm{sc}} =m_{\mrm{sc}}(z) $  is the Stieltjes transform of Wigner's semicircle law, 
\bels{SC}{
m_{\mrm{sc}}(z)  =  
\int_\R \frac{\rho_{\mrm{sc}}(\tau)\1\dif \tau\!}{\tau-z}
\,, 
\qquad\text{with}\quad 
\rho_{\mrm{sc}}(\tau)\2:=\frac{1}{2\1\pi} \sqrt{\max\sett{\20\2,4-\tau^{\12}}}
\,.
}
The system of equations \eqref{NSCE} thus reduces to the simple scalar equation
\bels{sceq}{
-\frac{1}{\1m_{\mrm{sc}}\msp{-8}} \,=\1 z + m_{\mrm{sc}}
\,.
}
Comparing \eqref{Gii} and \eqref{NSCE}, we see from \eqref{DOS} that the density of the eigenvalues in the large $N$ limit is given by the semicircle law.
The corresponding random matrix ensemble was called \emph{generalized Wigner ensemble} in \cite{EYY}.

Besides Wigner-type matrices and certain random matrices with translation invariant dependence structure \cite{AEK3}, the equation \eqref{NSCE} has previously appeared in at least two different contexts. First, in \cite{AZdep} the limiting density of eigenvalues for a certain class of random matrix models with dependent entries was determined by the so-called \emph{color equations} (cf. equation (3.9)  in \cite{AZdep}), which can be rewritten in the form \eqref{NSCE}. For more details on this connection we refer to Subsection 3.4 of \cite{AEK1cpam}. 
The second application of \eqref{NSCE} concerns the Laplace-like operator, 
\[
(Hf)(x) \,=\, \sum_{y\2\sim\2 x} t_{xy}\, (\2f(x) - f(y)\1)\,,   \qquad f: V \to \C
\]
on rooted tree graphs $\Gamma$ with vertex set $V$  (see \cite{KLW} for a review article and references therein).
Set $ m_x = (H_x-z)^{-1}(x,x)$, where $H_x$ is the operator $H$ restricted to the forward subtree with root $ x $. A simple resolvent formula then shows that \eqref{NSCE} holds with $ a =0 $ and $s_{xy} = \abs{\1t_{xy}}^2 \Ind\sett{x<y} $, where $ x<y $ indicates that $ x $ is closer to the root of $\Gamma$ than $y$. In this example  $ (s_{xy}) $ is not a symmetric matrix, but in a related model it may be chosen symmetric (rooted trees of finite cone types associated with a substitution matrix $S$, see \cite{Sadel-Tree}). In particular, real analyticity of the density of states (away from the spectral edges) in this model follows from our analysis (We thank C. Sadel for pointing out this connection).

The central role of \eqref{NSCE} in  the context of random matrices has been recognized by many authors, see, e.g. \cite{Ber, WegnerNorb, Girko-book, KhorunzhyPastur94, ShlyakhtenkoGBM, AZdep, Guionnet-GaussBand} and some basic properties of the solution, such as existence, uniqueness and regularity in $ z $ away  from  the real axis have been established, see e.g. \cite{Girko-book, Helton2007-OSE, PasturShcerbinaAMSbook} and further  references therein. 
The existence of the limit in \eqref{DOS} has been shown but no description of the limiting density $ \rho $ was given.

Motivated by this problem, in \cite{AEK1cpam} we initiated a comprehensive study of a general class of nonlinear equations of the form 
\[
-\frac{1}{m} = z + a + Sm
\,, 
\] 
in a possibly infinite dimensional setup. 
Under suitable conditions on  the linear operator  $ S $, we gave a \emph{qualitative}  description of the possible singularities of $m$ as $z$ approaches the real axis. We showed that singularities can occur at most at finitely many points and  that  they are algebraic of order two or three. 
The solution $m$ is conveniently represented as the Stieltjes transforms of a family of  probability measures.
The singularities of $ m $ occur at points where the densities of these measures approach to zero
and the type of singularity depends on how the densities vanish.
We found that  the  densities behave  like  a square root near the edges of  their  support and, additionally, they may exhibit a cubic root cusp singularity inside the  interior of the  support; no other singularity type occurs. 

All these results translate into statements  about the spectral densities of large random matrices on the macroscopic scale.
Recent developments in the theory of random matrices, however, focus on \emph{local laws}, i.e. precise description of the eigenvalue density down to  very small scales almost comparable  with the eigenvalue spacing. This requires understanding the solution $m $ and the stability of \eqref{NSCE} with  an  effective \emph{quantitative} control as $ z $ approaches the real line. 
In particular, a detailed description of the singular behavior of the solution close to the spectral edges is necessary.

The current paper is an extensive generalization of the qualitative singularity analysis of \cite{AEK1cpam}. Here we give a precise description (cf. Theorem~\ref{thr:Shape of generating density near its small values} below)  of the density around the singularities in a neighborhood of order one  with effective error bounds, while in  \cite{AEK1cpam} we only proved the limiting behavior as $z$ approached the singularities  without uniform control. We analyze the density  around  all local minima inside the interior of the support, even when the value of the density is small but non-zero. We demonstrate that a \emph{universal} density shape emerges  also at these points, which are far away from any singularity. In Subsection~\ref{sec:Relationship between theorems} we demonstrate the strength of the current bounds over the qualitative results in \cite{AEK1cpam} by considering a one-parameter family of operators $ S $. 
By varying the parameter, this family exhibits all possible shapes of the density in the regime where 
the density is close to zero.
This example illustrates how these density shapes are realized as rescalings of two universal shape functions. 
Furthermore, in the current work we impose weaker conditions on  $ a $ and $ S $ than in  \cite{AEK1cpam}.
Especially, when $ a = 0 $ our assumptions on $ S $ are essentially optimal.
Finally, we also give a detailed stability analysis against small perturbations; the structure of our stability bounds  is directly motivated by their application in random matrices.

The uniform control of the solution near the  singularities, as well as the  quantitative  stability estimates are not only natural mathematical questions on their own. 
They are also indispensable for establishing local laws and universality of local spectral statistics of Wigner type random matrices. 
We heavily use them in the companion papers to prove such results for  Wigner-type matrices with independent entries \cite{AEK2}, as well as for matrices with correlated Gaussian entries \cite{AEK3}. 
Random matrices, however, will not appear in the main body of this work. In Chapter~\ref{chp:Local laws for large random matrices} we only illustrate how our analysis of \eqref{NSCE} is  used to prove  a simple version of the local law.

While the current work is a generalization of \cite{AEK1cpam}, it is essentially self-contained; only very few auxiliary results will be taken over from \cite{AEK1cpam}. To achieve the required uniform control, we need to restart the analysis from its beginning. 
After establishing a priori bounds on the solution $ m $ and on the stability of the linearization of \eqref{NSCE} in Chapters~\ref{chp:Existence and uniqueness}--\ref{chp:Uniform bounds}, there are two main steps.
First, in Chapter~\ref{chp:Perturbations when generating density is small} we derive an approximate cubic equation to determine the leading behavior of the density in the regime where it is very small and, second, we analyze this cubic equation. The first step is much more involved in this paper than in \cite{AEK1cpam} since we also need to analyze points where the density is small but nonzero and we require all bounds to be effective in terms of a small number of model parameters. The second step   follows a completely new argument. In \cite{AEK1cpam} the correct roots of the cubic equation have been selected locally and by using a proof by contradiction which cannot give any effective control. In the current paper we select the roots by matching the solutions  at neighboring singularities to ensure the effective control  in an order one neighborhood. 
This procedure takes up Chapter~\ref{chp:Behavior of generating density where it is small}, the most technical part of our work. 
The  nonlinear stability analysis is presented in Chapter~\ref{chp:Stability around small minima of generating density}; this is the strongest version needed in the random matrix analysis in \cite{AEK2}. 
Finally, in Chapter~\ref{chp:Examples} we present  examples illustrating various aspects of the main results and the necessity of the assumptions on $ a $ and $ S $.

\medskip
\emph{Acknowledgement.}  We are grateful to Zhigang Bao and Christian Sadel for several comments and 
suggestions along this work.   A special thanks goes to Johannes Alt for carefully proofreading the entire manuscript.

\medskip
\emph{Recent developments.} After completing this manuscript, 
the systematic study of the quadratic vector equation \eqref{NSCE} initiated
in the current work has been substantially extended.
Wigner type matrices $\brm{H}$
can  be further generalised to allow for (i) correlations among the matrix entries as well as  (ii) non-zero
expectation for any matrix elements. Both extensions require to solve the \emph{Matrix Dyson equation (MDE)},
a matrix version of \eqref{NSCE}  of the form 
\[- M^{-1}  = z-A + {\mathcal S}[M]\,.\] 
Here
$A$ is an arbitrary deterministic Hermitian matrix
and ${\mathcal S}$ is a positivity preserving linear map on the space of matrices; in applications to random 
matrix theory we set $A=\EE {\brm{H}}$ and ${\mathcal S} [R]= \EE ({\brm{H}}-A)R (\brm{H}-A)$.

The stability analysis of the MDE was performed in \cite{Ajanki2019} where the local law
and spectral universality in the bulk spectrum for correlated random matrices with fast correlation decay have also
been proven. The case of slow correlation decay \cite{erdos_kruger_schroder_2019} required a very different
probabilistic technique, but the deterministic component of the proof  relied on the same MDE analysis.
The effective shape analysis of the density \eqref{DOS} for the MDE, including the precise description of the singularities,
has been completed in \cite{arXiv:1804.07752}. The main features of the density
in the more general MDE setup are the same as for the vector equation \eqref{NSCE}; only  square root
singularities at the regular edges and cusp singularities with cubic root behaviour in the interior of the spectrum may occur. Compared with  \eqref{NSCE}, however,
the non-commutative setup of the MDE posed major difficulties in the proofs.
The shape analysis for the MDE was one of the main ingredients in proving local laws and spectral universality 
for general correlated matrices at the regular edges \cite{2018arXiv180407744A}, as well as in the proof 
of the cusp universality  for Wigner type matrices \cite{arXiv:1809.03971,cipolloni2019}.

The general theory of the MDE has also been successfully  applied to various concrete random matrix models.
Local and global laws for \emph{Kronecker matrices} with block correlation structure have been proven in \cite{alt2019}. This was
further specialised to linearizations of polynomials in random matrices \cite{2018arXiv:1804.11340}. Via standard
hermitization, even non-Hermitian random matrices can be studied by the MDE; this has led to the
local law for Gram matrices \cite{alt2017,Gramalt2017} and to the local version of the inhomogeneous circular law
in the bulk \cite{alt2018} as well as at the edge  \cite{arXiv:1907.13631}. 
The decay rate of the solution to a large system of linear differential equations with random coefficients,
a standard model in the dynamics of neural networks, have also been studied via the analysis of the corresponding MDE
\cite{doi:10.1137/17M1143125,arXiv:1908.05178}. All these developments have been inspired by the key ideas of the current book.

\chapterl{Set-up and main results}

In this chapter we formulate a generalized  version of the equation \eqref{NSCE} which allows us to treat all dimensions $ N $, including the limit $ N \to \infty $, in a unified manner. After introducing three assumptions {\bf A1-3} on $ a $ and $ S $ we state our main results. 

Let $ \Sx $ be an abstract set of labels. 
We introduce the Banach space, 
\bels{def of BB}{
\BB
\,:=\,
\setB{
w: \Sx \to \C:\;  \sup_{x\ins \Sx}\,|\1w_x|\,<\,\infty
}
\,,
}
of bounded complex valued functions on  $\Sx$, equipped with the norm 
\bels{def of BB-norm}{
\norm{w} := \sup_{x\ins \Sx}\,\abs{\1w_x}
\,.
}
We also define the subset
\bels{def of BB_+}{
\BB_+\,:=\,\setB{ w\in\BB : \;\Im\, w_x\1>\10 \;\text{ for all } x \in \Sx  \2}
\,,
}
of functions with values in the complex upper half-plane $ \Cp $. 
 
Let $ S : \BB \to \BB $ be a non-zero bounded linear operator,  and $ a \in \BB $ a real valued bounded function.
The main object of study in this paper is the equation,
\bels{QVE}{
-\2\frac{1}{m(z)\msp{-1}}
\,=\,
z\2+\1a\1+\2S\1m(z)
\,, 
\qquad \forall \; z \in \Cp
\,, 
} 
and its solution $ m: \Cp \to \BB_+$. Here we view $ m: \Sx\times\Cp \to \Cp, \,(x,z) \mapsto m_x(z)$ as a function of the two variables $ x $ and $ z $, but we will often suppress the $ x $ and/or $ z $ dependence of $m$ and other related functions.
The symbol $ m $ will always refer to a solution of \eqref{QVE}. We will refer to  \eqref{QVE} as the {\bf Quadratic Vector Equation} ({\bf QVE}).

We assume that $ \Sx $ is equipped with a probability measure $ \Px $ and a $ \sigma$-algebra $ \mcl{S} $  such that $ (\Sx,\mcl{S},\pi) $ constitutes a probability space. 
We will denote the space of measurable functions $ u : \Sx \to \C $, satisfying $ \norm{u}_p := (\1\int_\Sx \abs{u_x}^p\Px(\dif x)\1)^{1/p} < \infty $, as $ \Lp{p}=\Lp{p}(\Sx;\C) $, $ p \ge 1 $. The usual $ \Lp{2}$-inner product, and the averaging are denoted by
\bels{def of L2-inner product and avg}{
\avg{\1u,w\1} \2:=\! \int_\Sx \overline{\1 u_x\msp{-7}}\msp{8}w_x \Px(\dif x)
\,,\qquad\text{and}\qquad
\avg{\1w\1} := \avg{\11,w\1}
\,,\qquad
u,w \in \Lp{2}
\,,
}
respectively.
For a linear operator $ A $, mapping a Banach space $ X $ to another Banach space $ Y $, we denote the corresponding operator norm by $ \norm{A}_{X \to Y} $. 
However, when $ X = Y = \BB $ we use the shorthand $ \norm{A} = \norm{A}_{\BB\to\BB} $.
Finally, if $ w $ is a function on $\Sx$ and $ T $ is a linear operator acting on such functions then $ w + T $ denotes the linear operator $ u \mapsto w\1 u + Tu $, i.e., we interpret $ w $ as a multiplication operator when appropriate.

In the entire paper we assume that the bounded linear operator $ S : \BB \to \BB $ in \eqref{QVE} is:
\begin{enumerate}
\item[{\bf A1}]
\label{item-A1} 
\emph{Symmetric and positivity preserving}, i.e.,
for every $ u,w \in \BB $ and every real valued and non-negative $ p \in \BB $: 
\bea{
\avg{\1u,Sw\1} \,=\, \avg{Su,w\1}
\,,\qquad\text{and}\qquad
\inf_x\, (Sp)_x \,\ge\, 0
\,.
}
\end{enumerate}

For the existence and uniqueness no other assumptions on $ a $ and $ S $ are needed.
\begin{theorem}[Existence and uniqueness]
\label{thr:Existence and uniqueness}
Assume {\bf A1}. Then for each $ z \in \Cp $, 
\bels{QVE for fixed z}{
-\2\frac{1}{m} \,=\, z +a + S \1 m \,,
}
has a unique solution $ m = m(z) \in \BB_+ $.
The solutions for different values of $ z$ constitute an analytic function $ z \mapsto m(z) $ from $\Cp$ to $\BB_+$. 
Moreover, for each $ x \in \Sx $ there exists a positive measure $ v_x $ on $ \R $, with 
\bels{bound on supp v}{
\supp v_x \1\subset\2 
[-\Sigma\1,\2\Sigma\,]
\,,\quad\text{where}\quad
\Sigma \1:=\1 \norm{a} + 2\1\norm{S}^{1/2}
\,,
}
and $v_x(\R)=\pi$, such that
\bels{m as stieltjes transform}{
m_x(z) \,=\,\frac{1}{\pi}\int_\R \frac{\1v_x(\dif \tau)\!}{\tau\1-\1z}
\,, 
\qquad \forall \; z \in \Cp 
\,.
}
The measures $ v_x $ constitute a measurable function $ v := (x \mapsto v_x) : \Sx \to \mcl{M}(\R) $, where $\mcl{M}(\R) $ denotes the space of finite Borel measures on $ \R $ equipped with the weak topology. 

Furthermore, if $ a = 0 $, then the solution $m(z)$ is in $\Lp{2} $ whenever $ z \neq 0 $,
\bels{L2 bound on m when a=0}{
\norm{m(z)}_2 \,\leq\, \frac{2}{|z|}\,, \qquad \forall \; z \in \Cp
\,,
}
and the measures $ v_x $ are symmetric, in the sense that $ v_x(-A)=v_x(A) $ for any measurable set $ A \subset \R $.
\end{theorem}

The existence and uniqueness part of Theorem~\ref{thr:Existence and uniqueness} is considered standard knowledge in the literature \cite{AZind,Girko-book,Helton2007-OSE,KLW2,KhorunzhyPastur94}. 
A proof of existence and uniqueness that is tailored to the current setup, including the Stieltjes transform representation \eqref{m as stieltjes transform}, was presented in \cite{AEK1cpam}. The novelty in the statement of Theorem~\ref{thr:Existence and uniqueness} in this paper is the $\Lp{2}$-bound \eqref{L2 bound on m when a=0} for the case $ a = 0 $, which is proven in Chapter~\ref{chp:Properties of solution}.

We remark that if the solution space $ \BB_+ $ is replaced by $ \BB $, then the  equation \eqref{QVE for fixed z} in general may have multiple, even infinitely many, solutions. 
Since $ v $ generates the solution $ m $ through \eqref{m as stieltjes transform} we call the $x$-dependent family of measures $ v=(v_x)_{x \in \Sx} $ the {\bf generating measure}.

In order to  prove results beyond the existence and uniqueness we need to additionally assume that $ S $ is:
\begin{itemize}
\item[{\bf A2}] 
\emph{Smoothing}, in the sense that it extends to a bounded operator from $ \Lp{2}$ to $ \BB $ that is represented by a symmetric non-negative measurable kernel function 
$ (x,y) \mapsto S_{xy}:\Sx^2 \to [\10\1,\infty) $, i.e., $ \norm{S}_{\Lp{2}\to\BB} < \infty $, and  
\bels{}{
(Sw)_x = \int_{\Sx} \!S_{xy}w_y\Px(\dif y)
\,.
} 
\item[{\bf A3}] 
\emph{Uniformly primitive}, i.e., 
there exist an integer $ L\in \N $, and a constant $ \rho > 0 $, such that
\bels{uniform primitivity}{
u \in \BB\,,\;u \ge 0 \quad\implies\quad 
(S^{L}u)_x \,\ge\, \rho\2\avg{\1u\1}
\qquad\forall\,x\in \Sx
\,.
}
\end{itemize}
The finiteness of the norm $\norm{S}_{\Lp{2} \to \BB}$ in  condition {\bf A2} means that the integral kernel $S_{xy}$ representing the operator $S$ satisfies
\bels{2 to infty bound on S}{
\!
\norm{S}_{\Lp{2}\to\BB} \;=\,\sup_{x\ins \Sx}\,\Bigl(\,\int_\Sx (S_{x y})^2\Px(\dif y)\,\Bigr)^{1/2}
<\;
\infty
\,.
}
In particular, $ S $ is a Hilbert-Schmidt operator on $ \Lp{2} $. 
The condition {\bf A3} is an effective lower bound on the coupling between the components $ m_x $ in  the QVE.
In the context of matrices with non-negative entries this property is known as \emph{primitivity} - hence our terminology. 

\begin{remark}[Scaling and translation]
\label{rmk:Scaling and translation}
If we replace the pair $ (a,S) $ in the QVE with $ (a',S') := (\lambda^{1/2}a+\1\tau\1,\lambda\1S) $, for some constants $ \lambda > 0 $ and $ \tau \in \R $, then the modified QVE is solved by $ m' : \Cp \to \BB$, where $ m'_x(z) := \lambda^{-1/2}m_x(\lambda^{-1/2}(z-\tau\1)\1) $.
By this basic observation, we may assume, without loss of generality, that $ S $ is normalized and $ a $ is  centered, i.e., $ \norm{S} = 1 $ and $ \avg{\1a\1} = 0 $. 
\end{remark}

All important estimates in this paper are quantitative in the sense that they depend on $ a $ and $ S $ only through a few special  parameters (see also Section~\ref{sec:Relationship between theorems}). The following convention makes keeping track of this dependence easier.

\begin{convention}[Comparison relations, model parameters and constants] 
\label{conv:Comparison relations, model parameters and constants}
For brevity we introduce the concept of {\bf comparison relations}: If $\varphi=\varphi(u) $ and  $ \psi=\psi(u) $  are non-negative functions on some set $ U $, then the notation $ \varphi \lesssim \psi $, or equivalently, $ \psi \gtrsim \varphi $, means that there exists a constant $ 0 < C < \infty $ such that $ \varphi(u) \leq C\1\psi(u) $ for all $ u \in U $. If $ \psi \lesssim \varphi \lesssim \psi $ then we write $ \varphi \sim \psi $, and say that $\varphi$ and $\psi$ are {\bf comparable}.
Furthermore, we use $ \psi = \phi + \Ord_X(\xi) $ as a shorthand for $ \norm{\psi-\phi}_X \lesssim \norm{\1\xi\1}_X $, where $ \xi, \psi, \varphi  \in X $ and $X$ is a normed vector space. For $X=\C$ we simply write $\Ord$ instead of $\Ord_\C$.
When the  implicit constants $ C $ in the comparison relations depend on some parameters $ \Lambda $ we say that the {\bf comparison relations depend on $ \Lambda $}. 
Typically, $ \Lambda $ contains the parameters appearing in the hypotheses, and we refer to them as {\bf model parameters}.

We denote by $ C,C',C_1,C_2,\dots $ and $ c,c',c_1,c_2,\dots $, etc., generic constants that depend only on the model parameters.
The constants $ C,C',c,c' $ may change their values from one line to another, while the enumerated constants, such as  $ c_1,C_2 $, have constant values within an argument or a proof. 
\end{convention}

We usually express the dependence on the variable $ z $ explicitly in the statements of theorems, etc. However, in order to avoid excess clutter we often suppress the variable $ z $ within the proofs e.g., we write $ m $ instead of $ m(z) $,  when $ z $ is considered fixed.

\sectionl{Generating density}

This section contains our main results, Theorem~\ref{thr:Regularity of generating density} and Theorem~\ref{thr:Shape of generating density near its small values}, concerning the generating measure, when $ S $ satisfies {\bf A1-3}, and the solution of the QVE is uniformly bounded.
Sufficient conditions on $ S $ and $ a $ that guarantee the uniformly boundedness of $ m $ are also given (cf. Theorem~\ref{thr:Qualitative uniform bounds}).     

For any $ I \subseteq \R $ we introduce the seminorm on functions $ w : \Cp \to \BB $:
\bels{def of nnorm of m}{
\nnorm{w}_{I} \,:=\, \sup \setb{\2\norm{w(z)} : \Re\,z \in I\,,\; \Im\,z \in (0,\infty)}
\,.
}

\begin{theorem}[Regularity of generating density]
\label{thr:Regularity of generating density}
Suppose $ S $ satisfies {\bf A1-3}, and the solution $ m $ of \eqref{QVE} is uniformly bounded everywhere, i.e., 
\[ 
\nnorm{m}_\R \,\leq\, \Phi
\,,
\]
for some constant $ \Phi < \infty $. 
Then the following hold true: 
\begin{itemize}
\titem{i}
The generating measure has a Lebesgue density (also denoted by $ v $), i.e., $ v_x(\dif \tau) = v_x(\tau)\1\dif \tau $.
The components of the {\bf generating density} are comparable, i.e.,
\[ 
\qquad
v_x(\tau) \,\sim\, v_y(\tau)\,,
\quad\forall\,\tau \in \R\,,\;\forall\2x,y \in \Sx
\,.
\]
In particular, the support of $ v_x $ is independent of $ x $, and hence we write $ \supp v  $ for this common support. 
\titem{ii}
$ v(\tau) $ is real analytic in $ \tau $, everywhere except at points $ \tau \in \supp v $ where $ v(\tau) = 0 $. More precisely, there exists $ C_0 \sim 1 $, such that the derivatives satisfy the bound
\[
\norm{\1\partial_\tau^k v(\tau)} 
\,\leq\,k!\, 
\Bigl(\frac{C_0}{\avg{v(\tau)}^3}\Bigr)^k
\,,\qquad\forall\,k\in\N\,,
\]
whenever $ \avg{v(\tau)} > 0 $.
\titem{iii} The density is uniformly $ 1/3$-H\"older-continuous everywhere, i.e.,
\[
\norm{v(\tau_2)-v(\tau_1)} \,\lesssim\, \abs{\1\tau_2-\tau_1\1}^{1/3}\,,
\qquad \forall\; \tau_1,\, \tau_2 \in \R
\,.
\]
\end{itemize}
The comparison relations in these statements depend on the model parameters $\rho$, $L$, $\norm{S}_{\Lp{2}\to\BB}$, $ \norm{a} $ and $\Phi $. 
\end{theorem}

Here we assumed an a priori uniform bound on $ \nnorm{m}_\R$. 
We remark that without such a bound a regularity result weaker than Theorem~\ref{thr:Regularity of generating density} can still be proven (cf. Corollary~\ref{crl:Regularity of mean generating density}).

For simplicity we assume here that $ \nnorm{m}_\R $ is bounded. In fact, all the results in this paper can be localized on any real interval $ [\alpha,\beta]$, i.e., the statements apply for $ \tau \in [\alpha,\beta] $ provided $ \nnorm{m}_{[\alpha-\eps,\beta+\eps]} $ is bounded for some $ \eps > 0 $.  The straightforward details are left to the reader.

The next theorem describes the behavior of the generating density in the regime where the average generating density $ \avg{v} $ is small.
We start with defining two universal shape functions.

\begin{definition}[Shape functions] 
\label{def:Shape functions}
Define $ \Psi_{\msp{-2}\mrm{edge}} : [\10,\infty) \to [\10,\infty) $, and $ \Psi_{\msp{-2}\mrm{min}} : \R \to [\10,\infty) $, by
\begin{subequations}
\label{defs of shape functions}
\begin{align}
\label{def of Psi_edge}
&\Psi_{\msp{-2}\mrm{edge}}(\lambda) 
\\
\notag
&\;:=\;
\frac{\sqrt{(1+\lambda\1)\2\lambda\,}}
{\bigl(\21+2\1\lambda \,+\,2\sqrt{(1+\lambda)\2\lambda\,}\,\bigr)^{2/3}\!+\,\bigl(\21+2\1\lambda \,-\,2\sqrt{(1+\lambda)\2\lambda\,}\,\bigr)^{2/3}+\,1\,}\,,
\\
\label{def of Psi_min}
&\Psi_{\msp{-2}\mrm{min}}\msp{-1}(\lambda) 
\;:=\; 
\frac{\sqrt{\21\2+\2\lambda^{\12}\,}}{(\sqrt{1+\lambda^2\,}+\lambda\1)^{2/3}+(\sqrt{1+\lambda^2\,}-\lambda\1)^{2/3}-1\,}
\,-\,1
\,.
\end{align}
\end{subequations}   
\end{definition}

\hspace{-0.03\textwidth}
\begin{minipage}{0.32\textwidth}
\hspace{15pt}As the names suggest, the appropriately rescaled versions of the shape functions $\Psi_{\msp{-2}\mrm{edge}}$ and $\Psi_{\msp{-2}\mrm{min}}$
will describe how $ v_x(\tau_0+\omega) $ behaves when $ \tau_0 $ is an {\bf edge} of $ \supp v $, i.e., $\tau_0 \in \partial\supp v$,
and when $ \tau_0 $ is a local minimum of $ \avg{v} $ with $ \avg{v(\tau_0)}>0 $ sufficiently small, respectively.
\end{minipage}
\hspace{2pt}
\begin{minipage}{0.66\textwidth}
\begin{figure}[H]
\centering
\includegraphics[width=1\textwidth]{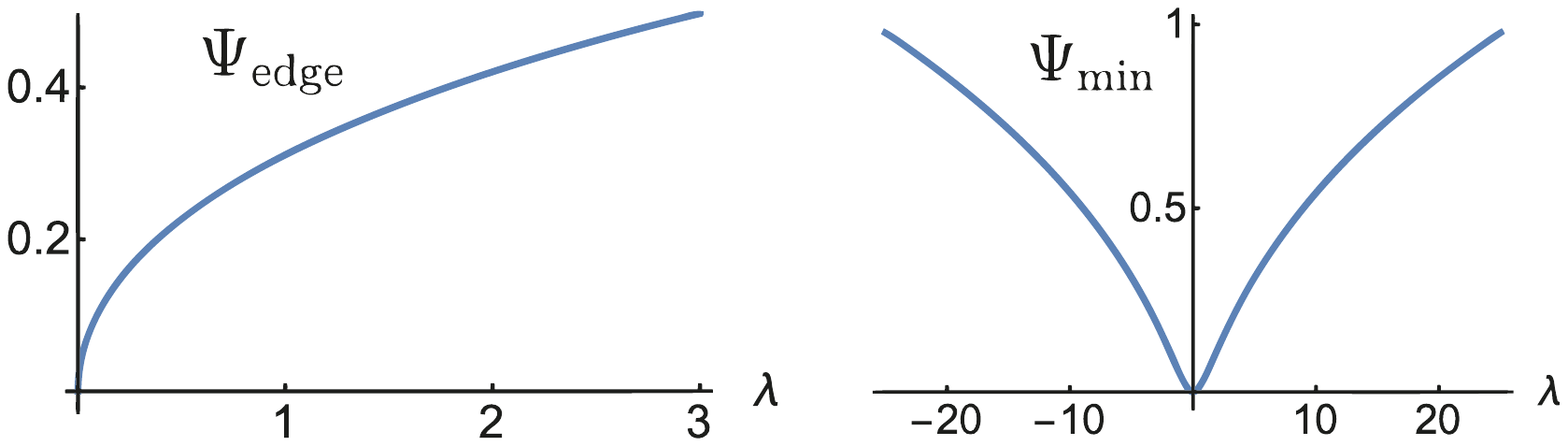}
\caption{The two shape functions $\Psi_{\rm edge}$ and $\Psi_{\rm min}$.}
\label{Fig:ShapeFunctions}
\end{figure}
\end{minipage}

The next theorem is our main result. Together with Theorem~\ref{thr:Regularity of generating density} it classifies the behavior of the generating density of a general bounded solution of the QVE.
The theorem generalizes Theorem 2.6 from \cite{AEK1cpam}. For more details on how these two results compare, we refer to Section~\ref{sec:Relationship between theorems}.

\begin{theorem}[Shape of generating density near its small values]
\label{thr:Shape of generating density near its small values}
Assume  {\bf A1-3}, and 
\[ 
\nnorm{m}_\R\, \leq\,  \Phi 
\,,
\]
for some $ \Phi < \infty $. 
Then the support of the generating measure consists of $K' \sim 1 $ disjoint intervals, i.e., 
\bels{defining property of of alpha_i and beta_i}{
\supp v \,=\, \bigcup_{i=1}^{K'}\, [\1\alpha_i,\beta_i]
\,,\quad\text{where}\quad
\beta_i -\alpha_i \sim 1
\,,\quad
\text{and}\quad 
\alpha_i < \beta_i< \alpha_{i+1}
\,.
}
Moreover, for all $ \eps > 0 $ there exist $ K'' = K''(\eps) \sim 1 $ points $ \gamma_1,\dots,\gamma_{K''} \in \supp v $ such that $ \tau \mapsto \avg{v(\tau)} $ has a local minimum at $ \tau = \gamma_k $ with $ \avg{v(\gamma_k)} \leq \eps $, $ 1\leq k \leq K'' $.  
These minima are well separated from  each other and from the edges, i.e.,
\bels{eps-disjointness of DD}{
\abs{\gamma_i-\gamma_j} \,\sim\, 1\,, 
\quad\forall\,i\neq j\,,
\msp{10}\text{and}\quad
\abs{\gamma_i-\alpha_j}
\,\sim\, 1\,,
\quad
\,\abs{\gamma_i-\beta_j}
\,\sim\, 1 
\,,\quad\forall\,i,j
\,.
}
Let $\MM$ denote the set of edges and these internal local minima,
\bels{def of MM}{
\MM \,:=\, \sett{\alpha_i} \cup \sett{\beta_j} \cup \sett{\gamma_k}
\,,
}
then small neighborhoods of $\MM$
cover the entire domain where $ 0 < \avg{v} \leq \eps $, i.e., there exists $ C \sim 1 $ such that 
\bels{avg-v < eps contained in nn}{
&\setb{\2\tau \in \supp\,v : \avg{\1v(\tau)\1}\leq \eps\2}  
\\
&\quad\subseteq\;
\bigcup_i \,[\2\alpha_i,\alpha_i+C\1\eps^2\2]  
\;\cup\; 
\bigcup_j\, [\2\beta_j-C\1\eps^2,\beta_j\1] 
\;\cup\; 
\bigcup_k \,[\2\gamma_k-C\1\eps^3,\gamma_k+C\1\eps^3\2]
\,.
}
The generating density is described by expansions around the points of $\MM$, i.e. for any $ \tau_0 \in \MM $ we have
\bels{expansion of v around tau_0}{
v_x(\tau_0+\omega\1) \,=\, v_x(\tau_0) \,+\, h_x\1 \Psi(\omega) \,+\, \Ord\Bigl(\,v_x(\tau_0)^2+\Psi(\omega)^2\Bigr)
\,,
\qquad
\omega \in I
\,,
}
where $ h_x \sim 1 $ depends on $ \tau_0 $. 
The interval $ I = I(\tau_0) $ and the function $ \Psi : I \to [\10,\infty) $ depend only on the type of $ \tau_0 $ according to the following list:
\begin{subequations}
\label{cases for shapes}
\begin{itemize}
\item Left edge: If $ \tau_0 =\alpha_i $, then \eqref{expansion of v around tau_0} holds with $  v_x(\tau_0) = 0\2 $, $ I = [\10,\infty\1) $, and 
\bels{left edge}{
\qquad\Psi(\omega) \,=\, 
(\alpha_i-\beta_{i-1})^{1/3}\,
\Psi_{\msp{-2}\mrm{edge}}\biggl(\frac{\omega}{\1\alpha_i-\msp{-1}\beta_{i-1}\!}\biggr)
\,,
}
with the convention $ \beta_0-\alpha_1 = 1 $.
\item 
Right edge: If $ \tau_0 =\beta_j $, then \eqref{expansion of v around tau_0} holds with $ v_x(\tau_0) = 0\2 $, $ I = (-\infty\1,0\1]$ , and 
\bels{right edge}{
\qquad\Psi(\omega) \,=\, 
(\1\alpha_{j+1}\msp{-2}-\beta_j)^{1/3}\,
\Psi_{\msp{-2}\mrm{edge}}
\biggl(\frac{\msp{-6}-\,\omega\,}{\1\alpha_{j+1}\msp{-2}-\beta_j}\biggr)
\,,
}
with the convention $ \alpha_{K'+1}-\beta_{K'} = 1 $.
\item 
Minimum: If $ \tau_0 = \gamma_k $, then \eqref{expansion of v around tau_0} holds with $ I = \R $, and 
\bels{minimum}{
\qquad\Psi(\omega) \,=\, \rho_k\,
\Psi_{\msp{-2}\mrm{min}}\msp{-2}\biggl(\frac{\,\omega\,}{\msp{-1}\rho_k^{\13}\!}\biggr)
\,,\quad\text{where}\quad
\rho_k \,\sim\, \avg{\1v(\gamma_k)}
\,.
} 
In case $\avg{v(\gamma_k)}=0$ we interpret  \eqref{minimum} as its $\rho_k\to 0$ limit, i.e.,  $\Psi(\omega)=2^{-2/3}\abs{\omega}^{1/3}$.
\end{itemize}
\end{subequations}
All  comparison relations depend only on the model parameters $\rho$, $L$, $\norm{S}_{\Lp{2}\to\BB}, \norm{a}$ and $  \Phi $. 
\end{theorem}

Figure \ref{Fig:AllShapes} shows an average generating measure which exhibits each of the possible singularities described by \eqref{expansion of v around tau_0} and \eqref{cases for shapes}. 
Note that the expansions \eqref{expansion of v around tau_0} become useful for the non-zero minima $ \tau_0 = \gamma_k $  only when $ \eps > 0 $ is chosen to be so small that the term $h_x\Psi(\omega)$ dominates $v_x(\tau_0)^2$ which itself is smaller than
$\eps^2$.

\begin{remark}[Universality of shapes]
The function $ \Delta^{1/3} \Psi_{\msp{-2}\mrm{edge}}(\omega/\Delta) $ describing the edge shape interpolates between a square root and a cubic root growth with the switch in the growth rate taking place when its argument becomes of the size $ \Delta $.
Similarly, the function $ \rho\,\Psi_{\msp{-2}\mrm{min}}(\omega/\rho^3)$ can be seen as a cubic root cusp $ \omega \mapsto \abs{\omega}^{1/3} $ regularized at  scale $ \rho^3 $.  

Suppose $ \tau_0 $ is an internal edge with a gap of size $ \Delta>0  $ to the left. 
As $ \Delta $ becomes small, the function $ \lambda \mapsto \Delta^{\!-1/3} v(\tau_0+\Delta\lambda\2) $ approaches the universal shape function $ \Psi_{\msp{-2}\mrm{edge}} $ up to a $ \lambda$-independent scaling factor.
More precisely, consider a family of data $ (a^{(\Delta)},S^{(\Delta)}) $, $ \Delta \in (0,c) $ parameterized by $ \Delta \in (0,c) $, such that the supports of the corresponding generating densities $ v = v^{(\Delta)} $ have gaps of size $ \Delta $ between opposing internal edges $ \tau_0 = \tau^{(\Delta)}_0 $ and $ \tau_0 -\Delta $.
If the hypotheses of Theorem~\ref{thr:Shape of generating density near its small values} hold uniformly in $ \Delta $, then 
\[
\qquad
\lim_{\Delta \downarrow 0}\,
\frac{\,v(\tau_0+\Delta\lambda_1)}{v(\tau_0+\Delta\lambda_2)} 
\,=\,
\frac{\Psi_{\msp{-2}\mrm{edge}}(\lambda_1)}{\Psi_{\msp{-2}\mrm{edge}}(\lambda_2)}
\,,
\qquad\forall\2\lambda_1,\lambda_2 > 0
\,.
\]
An analogous statement holds for non-zero local minima and the associated universal shape function $ \Psi_{\!\mrm{min}}$. 
%
A simple example of a family of QVEs where the gap closes and then becomes a small minima is given in Section~\ref{sec:Simple example that exhibits all universal shapes}. 
\end{remark}

\begin{figure}[h]
\centering
\hspace{-0.025\textwidth}
\includegraphics[width=1.02\textwidth]{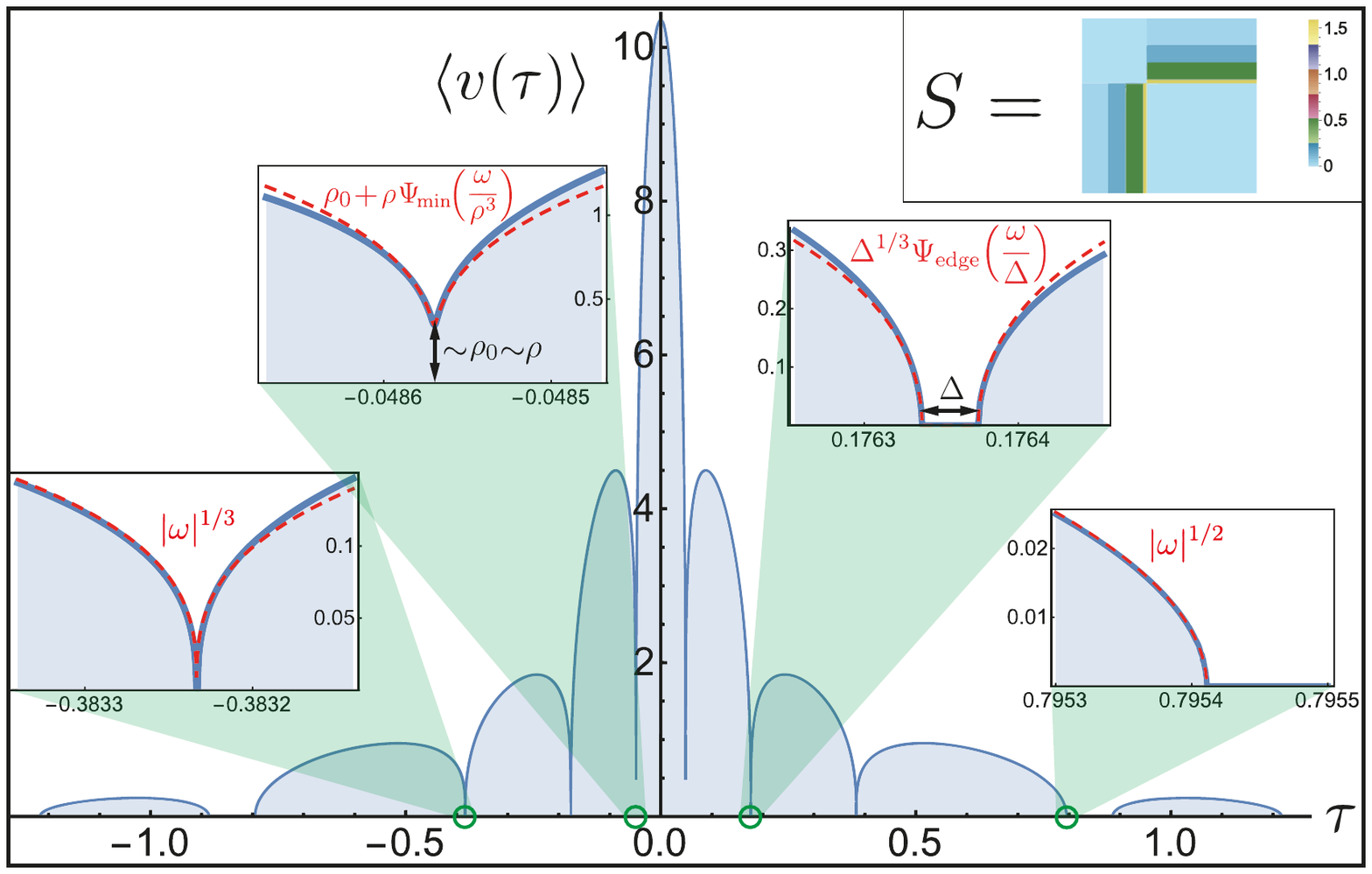}
\caption{
Average generating density $ \avg{v} $ when $ a = 0 $ and the kernel $ S_{xy} $ is a block-constant function as specified with greyscale encoding in the upper right corner. All the possible shapes appear in this example.
At the qualitative level each component $ v_x $ looks similar.
If $ a $ is non-zero the $ v(\tau) $ is not necessarily a symmetric function of $ \tau $ any more.
}
\label{Fig:AllShapes}
\end{figure}

\begin{remark}[Choice of non-zero minima]
\label{rmk:Choice of non-zero minima}
We formulated Theorem~\ref{thr:Shape of generating density near its small values} for an arbitrary threshold parameter $ \eps $, but it is easy to see that only  small values of $\eps$ are relevant. In fact, without loss of generality one may assume that $ \eps \sim 1 $ is so small that the intervals on the right hand side of \eqref{avg-v < eps contained in nn} are disjoint. In this case the internal minima where $\langle v\rangle$ vanishes, i.e., the edges $\alpha_i$, $\beta_j$ and those $\gamma_k$'s that correspond to cusps, turn out to be the unique minima within the corresponding intervals. However,  the local minima of $ \avg{v} $ where  $\avg{v} \neq 0 $, i.e., the non-cusp elements of $\{\gamma_1, \ldots, \gamma_{K''} \}$ might not be unique even for small $\eps$. 
In fact, along the proof of Theorem~\ref{thr:Shape of generating density near its small values} we also show (Corollary~\ref{crl:Location of non-zero minima}) that these nonzero local minima are either tightly clustered or well separated from each other in the following sense: If $ \gamma,\gamma' \in \supp v \backslash \partial \supp v $ are two local minima of $ \avg{v} $, then either 
\[
\abs{\1\gamma-\gamma'} \,\lesssim\, \min\setb{\avg{v(\gamma)},\avg{v(\gamma')}}^4\,,
\qquad\text{or}\quad\qquad
\abs{\1\gamma-\gamma'} \,\sim\, 1\,.
\] 
In particular, for small $ \eps \sim 1 $, each interval in \eqref{avg-v < eps contained in nn} contains at most one such cluster of local minima.
Within each cluster we may choose an arbitrary representative $ \gamma_k $; Theorem~\ref{thr:Shape of generating density near its small values} will hold for any such choice.
\end{remark}

We will now discuss two sufficient and checkable conditions that together with {\bf A1-3} imply $ \nnorm{m}_\R < \infty $, a key input of Theorems \ref{thr:Regularity of generating density} and \ref{thr:Shape of generating density near its small values}.
The first one involves a regularity assumption on $ a $ and the family of row functions, or simply {\bf rows}, of $S $, 
\bels{def of S_x}{
S_x: \Sx \to [\20,\infty),\, y \mapsto S_{xy}
\,,\qquad
x \in \Sx\,,
}
as elements of $ \Lp{2} $.
It expresses that the set of pairs $ \sett{(a_x,S_x):x \in \Sx} $ should not have  outliers in the sense that,
\bels{qualitative B1}{
\lim_{\eps\1\downarrow\1 0}\inf_{x \in \Sx}\int_\Sx \frac{\Px(\dif y)}{\,\eps  +(a_x-a_y)^2 +\norm{S_x-S_y}_2^2}\;=\; \infty
\,,
}
holds.
In other words, this means that no $(a_x,S_x) $ is too different from all the other pairs $ (a_y,S_y) $, $ y \neq x $.
We will see that in case $ a = 0 $, the property \eqref{qualitative B1} alone implies a bound for $ m(z) $ when $ z $ is away from zero. 
When $ a = 0 $ the point $ z = 0 $ is special, and an extra structural condition is needed to ensure that $ m(0) $ is also bounded.  
In order to state this additional condition we need the following definitions.

\NDefinition{Full indecomposability}{
A $ K\times K $ matrix $ \brm{T} $ with non-negative elements $ T_{ij} \ge 0 $, is called {\bf fully indecomposable} (FID) provided that for any subsets $ I,J \subset \sett{1,\dots,K} $, with $ \abs{\1I\1} + \abs{J} \ge K $, the submatrix $(T_{ij})_{i \in I, j \in J}$ contains a non-zero entry.

The integral operator $ S : \BB \to \BB $ is {\bf block fully indecomposable} if there exist
an integer $K$, a fully indecomposable matrix $ \brm{T} = (T_{ij})_{i,j=1}^K $ and a measurable partition $ \mcl{I} := \sett{I_j}_{j=1}^K $ of $ \Sx$,  such that 
\bels{B2:S bound}{
\Px(I_i) = \frac{1}{K}\,,\qquad\text{and}\qquad
S_{xy} \,\ge\, T_{ij}\,,\quad\text{ whenever}\quad(x,y) \in I_i \times I_j
\,,
}
for every $ 1\leq i,j\leq K $.
}
The FID property is standard for matrices with non-negative entries \cite{BR97book}. 
The most useful properties of FID matrices are listed in  Proposition~\ref{prp:Properties of FID matrices} and Appendix~\ref{sec:Scalability of matrices with non-negative entries} below.
With these definitions we have the following qualitative result on the boundedness of $ m $.

\begin{theorem}[Qualitative uniform bounds]
\label{thr:Qualitative uniform bounds}
Suppose that in addition to {\bf A1}, {\bf A2} and \eqref{qualitative B1}, either of the following holds:
\begin{itemize}
\titem{i} 
$ a = 0 $ and $ S $ is block fully indecomposable;
\titem{ii} 
$ S$  satisfies {\bf A3}, and
\bels{qualitative strong diagonal condition}{
\inf\setbb{\frac{\avg{w,Sw}}{\avg{w}^2}: w \in \BB\,,\; w_x\ge 0}  \;>\;0\,.
}
\end{itemize}
Then the solution of the QVE is uniformly bounded,  $ \nnorm{m}_{\R} < \infty $, and in the case \emph{(i)} $ S $ has the property {\bf A3}. In particular, the conclusions of both Theorem~\ref{thr:Regularity of generating density} and Theorem~\ref{thr:Shape of generating density near its small values} hold.
\end{theorem}

When $ a = 0 $ and $\Sx $ is discrete the full indecomposability of $ S $ is not only a sufficient but also a necessary condition for the boundedness of $ m $ in $ \BB $.  More precisely, in Theorem~\ref{thr:Scalability and full indecomposability} we will show that in the discrete setup the QVE is stable and has a bounded solution if and only if $ S $ is a fully indecomposable matrix.

We also remark that {\bf A3} and the condition \eqref{qualitative strong diagonal condition} imply that $ S $ is block fully indecomposable in the discrete setup. In general, neither implies the other however. 
In Chapter \ref{chp:Uniform bounds} we present quantitative versions of Theorem~\ref{thr:Qualitative uniform bounds}: Theorem~\ref{thr:Quantitative uniform bounds when a = 0} and Theorem~\ref{thr:Quantitative uniform bound for general a} correspond to the parts (i) and (ii) of Theorem~\ref{thr:Qualitative uniform bounds}, respectively.

In the prominent example $ (\Sx,\Px(\dif x)) = ([0,1],\dif x) $ the condition \eqref{qualitative B1} is satisfied if the  map $ x \mapsto (a_x,S_x) : \Sx \mapsto \R \times \Lp{2} $ is piecewise $1/2$-H\"older continuous, in the sense that for some finite partition $ \sett{I_k} $ of $ [0,1] $ into non-trivial intervals, the bound 
\bels{def of PW-1/2-Holder}{
\abs{\1a_x-a_y}
+\norm{S_x-S_y}_2 \,\leq\, C_1\2\abs{x-y}^{1/2}\,,\qquad\forall\,x,y \in I_k  
\,,
} 
holds for every $ k $. 
Furthermore, if $ S $ has a positive diagonal, such that 
\bels{fulldiag}{
S_{xy} \,\ge\, \eps\2\Ind\sett{\2\abs{x-y}\leq \delta\1}
\,,\qquad\forall\,x,y\in [\10,1\1]
\,,
}
for some $ \eps,\delta > 0 $,  then it is easy to see that $ S $ is block fully indecomposable and also satisfies \eqref{qualitative strong diagonal condition}, as well as its quantitative version \eqref{quantitative strong diagonal condition} (cf. Chapter \ref{chp:Uniform bounds}).

Next we discuss the special situation in which the generating measure is supported on a single interval.
A sufficient condition for this to hold is that the pairs $(a_x,S_x) $, $ x\in \Sx $, can not be  split into two well separated subsets in a sense specified by the inequality \eqref{xi_S(0) leq xi_ast} below.
The following result is a quantitative version of Theorem 2.8 in \cite{AEK1cpam}. 

\begin{theorem}[Generating density supported on single interval]
\label{thr:Generating density supported on single interval}
Assume $ S $ satisfies {\bf A1-3}, and  $ \nnorm{m}_\R \leq \Phi $. 
Then there exists a threshold $ \xi_\ast \sim 1 $ such that under the assumption
\bels{xi_S(0) leq xi_ast}{ 
\sup_{A \1\subset\1 \Sx}\, \inf_{\substack{x \,\in\, A\\y \,\notin\, A}}
\Bigl(\, \abs{\1a_x-a_y}+ \norm{S_x-S_y}_1\Bigr)
\,\leq\, \xi_\ast
\,
}
the generating density is supported on a single interval, i.e.
$ \supp v =[\alpha, \beta]$, with $ \beta-\alpha \sim 1 $, and $ \abs{\alpha},\abs{\beta} \leq \Sigma $.
Moreover, for every $ 0< \delta<(\beta-\alpha)/2 $, we have
\begin{subequations}
\label{single interval}
\begin{align}
\label{single interval: bulk}
\qquad v_x(\tau) \;&\gtrsim\; \delta^{\11/2}\,,\qquad\qquad \tau \in [\1\alpha+\delta\1,\2\beta-\delta\2]
\\
\label{single interval: edges1}
v_x(\alpha + \omega)
\,&=\, 
h_x\2\omega^{1/2} + \Ord(\omega)
\qquad \omega \in [\10\1,\1\delta\2],
\\
\label{single interval: edges2}
v_x(\beta - \omega)
\,&=\, 
h_x'\2\omega^{1/2} + \Ord(\omega)
\qquad \omega \in [\10\1,\1\delta\2],
\end{align}
\end{subequations}
where $ h,h' \in \BB $ with $ h_x,h'_x \sim 1 $. Furthermore, $ v(\tau) $ is uniformly $ 1/2$-H\"older continuous in $ \tau $.
Here  $\rho$, $L$, $\norm{S}_{\Lp{2}\to\BB}$, $ \norm{a} $ and $\Phi $ are considered the model parameters.
\end{theorem}

\hspace{-0.07\textwidth}
\begin{minipage}{0.51\textwidth}
\begin{figure}[H]
\centering
\includegraphics[width=1\textwidth]{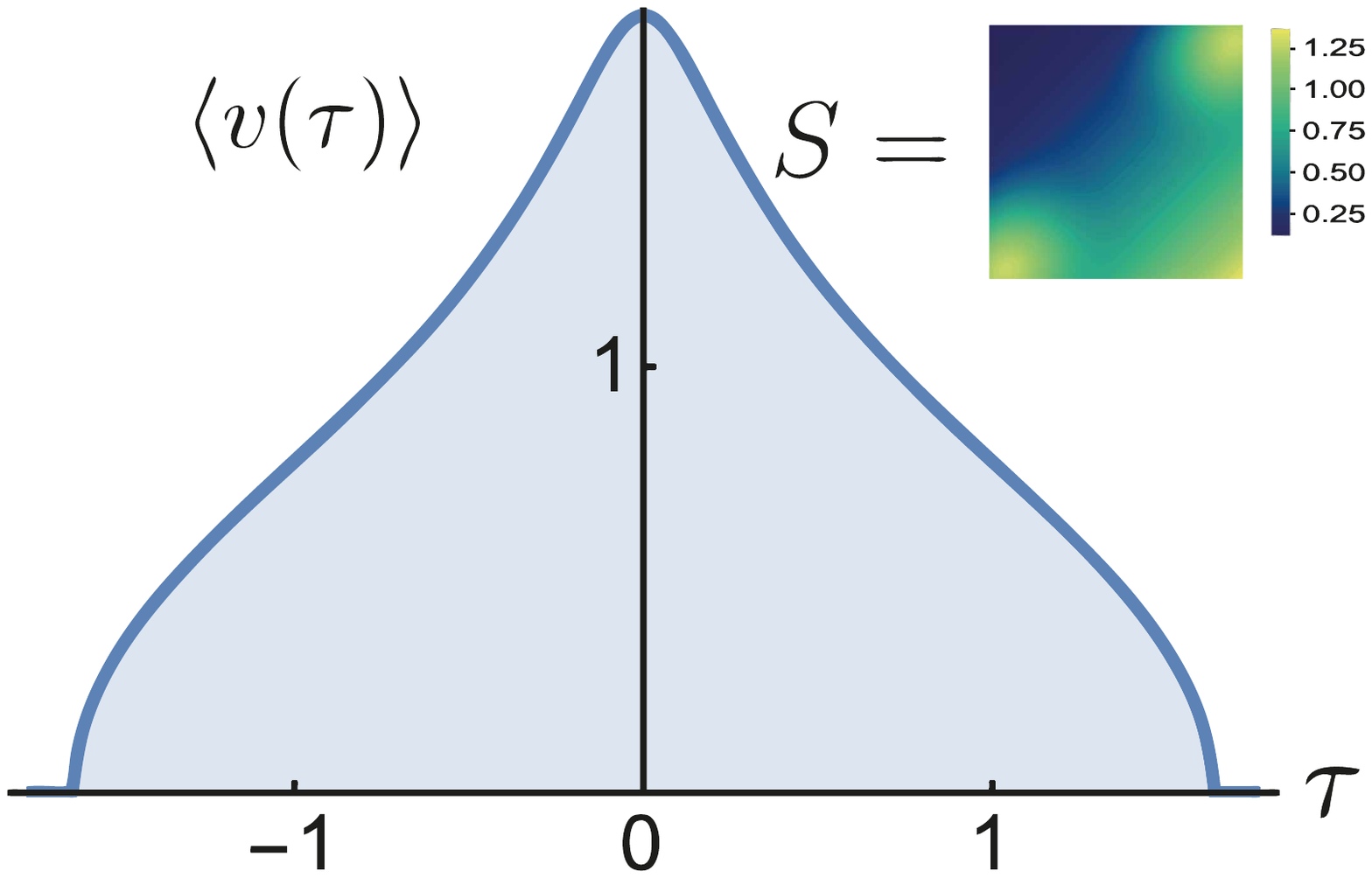}
\caption{The smooth profile of $S$ leads to a generating density that is supported on a single interval.}
\label{Fig:SingleIntervalSupport}
\end{figure}
\end{minipage}
\hspace{0.01\textwidth}
\begin{minipage}{0.48\textwidth}
Combining the last two theorems we proved that under the conditions of Theorem~\ref{thr:Qualitative uniform bounds} on $ (a,S) $ in addition to \eqref{xi_S(0) leq xi_ast} all conclusions of Theorem~\ref{thr:Generating density supported on single interval} hold.
For example, if $\Sx=[0,1]$, $ a = 0 $,  and $ S $ satisfies {\bf A1} and {\bf A2}, it is block fully indecomposable, and the row functions $ S_x $ are $ 1/2$-H\"older continuous on the whole set $ [0,1] $, then the conclusions \eqref{single interval} of Theorem~\ref{thr:Generating density supported on single interval} hold true. Figure \ref{Fig:SingleIntervalSupport} shows an average generating density corresponding to an integral operator $S $ with a smooth kernel when $ a = 0 $. 
\end{minipage}

\sectionl{Stability}

Now we discuss the stability properties of the QVE \eqref{QVE}. 
These results are the cornerstone of the proof of the local law for \emph{Wigner-type} random matrices proven in \cite{AEK2},  
 see Chapter~\ref{chp:Local laws for large random matrices} for more details.
Fix $ z\in \eCp $, and suppose $ g \in \BB $ satisfies
\bels{perturbed QVE - 1st time}{ 
-\frac{1}{g} \;=\; z +a+ Sg +d
\,.
}
This equation is viewed as a perturbation of the QVE \eqref{QVE} by a "small" function $ d\in \BB $. Our final result provides a bound on the difference between $ g $ and the unperturbed solution $ m(z) $. 
The difference will be measured both in strong sense (in $ \BB $-norm) and in weak sense (integrated against a fixed bounded function). 
\begin{theorem}[Stability]
\label{thr:Stability}
Assume $ S $ satisfies  {\bf A1-3} and $ \nnorm{m}_\R \leq \Phi $, for some $ \Phi < \infty $. 
Then there exists $ \lambda \sim 1 $ such that if $ g,d \in \BB $ satisfy the perturbed QVE \eqref{perturbed QVE - 1st time} for some fixed $ z \in \eCp $, then the following holds: 
\begin{itemize}
\titem{i} 
{\bf Rough stability:} 
Suppose  that for some $ \eps \in (0,1) $,
\bels{condition for being away from critical points}{ 
\avg{\1v(\Re\,z)}
\,\ge\,\eps\,,
\qquad\text{or}\qquad
\dist(z,\supp v\1) 
\,\ge\, \eps
\,,
}
and $ g $ is sufficiently close to $ m(z) $, 
\bels{bulk maximum deviation}{
\norm{\1g-m(z)} \,\leq\, \lambda\2\eps\,. 
}  
Then their distance is bounded in terms of $ d $ as
\begin{subequations}
\label{rough perturbation-bounds}
\begin{align}
\label{rough perturbation-bound: sup}
\norm{g-m(z)} 
\;&\lesssim\; 
\eps^{-2}
\norm{d}
\\
\label{rough perturbation-bound: w-average}
\qquad\abs{\avg{w,g-m(z)}} \,&\lesssim\,
\eps^{-6}
\norm{w}\norm{d}^2
\,+\,
\eps^{-2}
\abs{\1\avg{\1J(z)w,d\2}}
\,,\qquad\forall\,w \in \BB\,,
\end{align}
\end{subequations}
for some $ z $-dependent family of linear operators $ J(z) :  \BB \to \BB $, that depends only on $ S $ and $ a $, and satisfies $ \norm{J(z)} \lesssim 1 $.
\titem{ii} {\bf Refined stability:} There exist $ z $-dependent families $ t^{(k)}(z) \in \BB $, $ k=1,2 $, depending only on $ S$, and satisfying $ \norm{\1t^{(k)}(z)} \lesssim 1 $, such that the following holds. Defining
\begin{subequations}
\label{defs of varpi, rho, and delta}
\begin{align}
\label{def of varpi}
\varpi(z) \,&:=\, \dist(z,\supp\,v|_\R)
\\
\label{def of rho}
\rho(z) \,&:=\, \avg{\1v(\Re\,z)\1}
\\
\label{def of delta}
\delta(z,d) \,&:=\, \norm{d\1}^2\,+\,\abs{\avg{\1t^{(1)}(z),d\2}}\,+\, \abs{\avg{\1t^{(2)}(z),d\2}}
\,,
\end{align}
\end{subequations} 
assume $ g $ is close to $ m(z) $, in the sense that
\bels{refined perturbation-bounds: a priori estimate for g-m}{
\norm{\1g-m(z)} 
\,\leq\, 
\lambda\2
\varpi(z)^{2/3}+
\lambda\2\rho(z)
\,.
}
Then their distance is bounded in terms of the perturbation as
\begin{subequations}
\label{refined perturbation-bounds}
\begin{align}
\label{refined perturbation-bound: sup}
\norm{g-m(z)} 
\;&\lesssim\; 
\Upsilon(z,d)
\,+\,
\norm{d} 
\\
\label{refined perturbation-bound: w-average}
\abs{\avg{\1w,g-m(z)}} \,&\lesssim\,
\Upsilon(z,d)\2
\norm{w} 
\,+\,\abs{\avg{\1T(z)w,d\2}}
\,,\qquad\forall\,w \in \BB\,,
\end{align}
\end{subequations}
for some $ z $-dependent family of linear operators $ T(z) :  \BB \to \BB $, that depends only on $ S $ and $ a $, and satisfies $ \norm{T(z)} \lesssim 1 $.
Here the key control parameter is
\bels{refined stability: def of Upsilon}{
\Upsilon(z,d) 
\,:=\, 
\min\setbb{\!
\frac{\delta(z,d)}{\rho(z)^2\!}
\,,
\frac{\delta(z,d)}{\,\varpi(z)^{2/3}\!}
\,,
\delta(z,d)^{1/3}\!
}
\,.
} 
\end{itemize}
The comparison relations depend on $\rho$, $L$, $\norm{S}_{\Lp{2}\to\BB}$, $ \norm{a} $  and $\Phi $. 
\end{theorem}

Note that the  existence of $ g $ solving \eqref{perturbed QVE - 1st time} for a given $ d $ is part of the assumptions of Theorem~\ref{thr:Stability}.
In Proposition~\ref{prp:Analyticity} we will actually prove the existence and uniqueness of $ g $ close to $ m $ provided $ d $ is sufficiently small.
An important aspect of the estimates \eqref{rough perturbation-bounds} and \eqref{refined perturbation-bounds} is that the upper bounds depend only on the unperturbed problem, i.e., on $ z $, $ a $ and $ S $, possibly through $ m(z) $,  apart from the explicit dependence of $ d $. 
They do not depend on $g$.

The condition \eqref{refined perturbation-bounds: a priori estimate for g-m} of (ii) in the preceding theorem becomes increasingly restrictive  when $ z $ approaches points in $\supp v$ where $v$ takes small values.
A stronger but less transparent perturbation estimate is given as Proposition~\ref{prp:Cubic perturbation bound around critical points} below.

The guiding principle behind these estimates is that the norm bounds \eqref{rough perturbation-bound: sup} and \eqref{refined perturbation-bound: sup} are linear in $ \norm{d} $, while the bounds \eqref{rough perturbation-bound: w-average} and \eqref{refined perturbation-bound: w-average} for the average of $ g-m $  are quadratic in $ \norm{d} $ and linear in a specific average in $ d $. 
The motivation behind  the average bounds  is that in the random matrix theory (cf. Chapter~\ref{chp:Local laws for large random matrices}) the perturbation $ d $ will be random.
In fact,  $ d $ will be subject to the \emph{fluctuation averaging} mechanism, i.e., its (weighted) average is typically comparable to  $ \norm{d}^2 $ in size. 
In the part (ii) of the theorem we see how the stability estimates deteriorate as $ z $ approaches the part of the real line where $ \avg{v} $ becomes small, in particular near the edges of $ \supp v $.

Another trivial application of our general stability result is to show that the QVE \eqref{QVE} is stable under perturbations of $ a $ and $ S $.

\begin{remark}[Perturbations of $ a$ and $S$]
\label{rmk:Perturbations on a and S}
Suppose $ S $ and $ T $ are two integral operators satisfying {\bf A1-3} and $ a,b\in \BB $ are real valued. Let $ m $ and $ g $ be the unique solutions of the two QVE's 
\[
-\frac{1}{m} = z  + a  + Sm
\quad\text{and}\quad
-\frac{1}{g} = z  + b  + Tg 
\,.
\]
Then $ g $ can be considered as a solution of the perturbed QVE \eqref{perturbed QVE - 1st time}, with 
\[ 
d \,:=\,  (b-a) +  (T-S)\1g 
\,.
\] 
Thus if $ \nnorm{m}_\R < \infty $, then Theorem~\ref{thr:Stability}  may be used to control ${g-m}$  in terms of $ b-a $ and $T-S$.  
\end{remark}

\newpage


\section{Relationship between Theorem~2.6 and Theorem 2.6 of [{\bf AEK16b}]}
\label{sec:Relationship between theorems}

\begin{minipage}{0.65\textwidth} 

Theorem~\ref{thr:Shape of generating density near its small values} is a quantitative generalization of 
Theorem 2.6 of \cite{AEK1cpam}.  We comment on the differences between the two results. The main novelty in Theorem~\ref{thr:Shape of generating density near its small values} is that it provides a precise description of the generating density around the expansion points $ \tau_0 $
in an environment whose size 
is comparable to $ 1 $. 
Moreover, its statement is uniform in the operator $S$ and the function $ a $, given the model parameters. 
In \cite{AEK1cpam}, on the other hand, the operator $S$ is fixed and only asymptotically small expansion environments are considered.  Theorem~\ref{thr:Shape of generating density near its small values} also provides explicit quantitative error bounds in terms of the model parameters.  

\hspace{15pt}
To illustrate the distinction between the two results we consider a fixed $ a $, and a continuous one-parameter family of operators $S=S^{(\delta)}$ with the following properties: 
\begin{enumerate}
\item The family $S^{(\delta)}$ satisfies {\bf A1-3} uniformly in $\delta$.
\item The corresponding solutions $m^{(\delta)}$ are uniformly bounded, $\sup_\delta\nnorm{m^{(\delta)}}_\R \leq \Phi $.
\item There is an expansion point $\tau_0(\delta)$, depending continuously on $\delta$, such that  (cf. Figure \ref{Fig:ShapeFamily}) 
\begin{enumerate}
\item At a critical value $\delta=\delta_c$ the generating density corresponding to $S^{(\delta_c)}$ has a cubic root cusp at $\tau_0=\tau_0(\delta_c)$, i.e.,  the expansion point  $\tau_0$ is a minimum in the sense of \eqref{minimum} and $\avg{v(\tau_0)}=0$.
\item For $\delta>\delta_c$ the expansion point  $\tau_0=\tau_0(\delta)$ is a minimum in the sense of \eqref{minimum} of the generating density corresponding to $S^{(\delta)}$ with $\avg{v(\tau_0)}>0$.
\item For $\delta<\delta_c$ the expansion point  $\tau_0=\tau_0(\delta)$ is a left edge in the sense of \eqref{left edge} of the generating density corresponding to $S^{(\delta)}$.
\end{enumerate}
\end{enumerate}
\end{minipage}
\hspace{0.01\textwidth}
\begin{minipage}{0.34\textwidth}
\begin{figure}[H]
\vspace{-0.5cm}
\includegraphics[width=\textwidth]{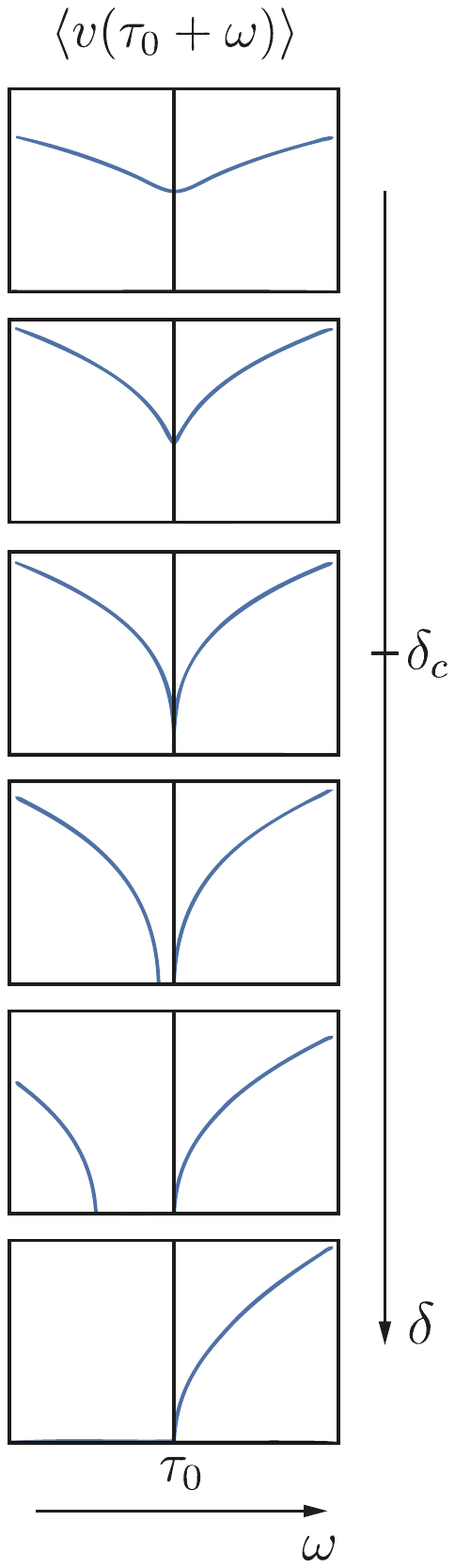}
\caption{\label{Fig:ShapeFamily} Different singularity shapes emerge at $\tau_0$ by varying $\delta$.}
\end{figure}
\end{minipage}
\vspace{0.1cm}

We refer to Section \ref{sec:Simple example that exhibits all universal shapes} for an explicit example of such a family of operators $S=S^{(\delta)}$. The results of \cite{AEK1cpam} analyze the situation only for a fixed value of the parameter $\delta$ and they are restricted to a description of the generating density in asymptotically small expansion environments. In other words, in each case \eqref{left edge}, \eqref{right edge} and \eqref{minimum} only the limiting behavior as $\omega \to 0$ of the function $\Psi$ is tracked. 
Indeed, Theorem~\ref{thr:Shape of generating density near its small values} reduces to Theorem~2.6 in \cite{AEK1cpam}  containing the following statements:  
\begin{itemize}
\item[(a)]  At the critical value $\delta=\delta_c$ we have $v_x(\tau_0+\omega)=2^{-2/3}\1h_x\1 \abs{\omega}^{1/3}(1+o(1))$ as $\omega \to 0$.
\item[(b)]  For  any fixed  $\delta>\delta_c$ we have $v_x(\tau_0+\omega)=v_x(\tau_0)(1+o(1))$ as $\omega \to 0$.
\item[(c)]  For   any fixed  $\delta<\delta_c$ we have $v_x(\tau_0+\omega)=\frac{1}{\,3 \Delta^{1/6}\msp{-6}}\;h_x\2\omega^{1/2}(1+o(1))$ as $\omega \downarrow 0$. Here, $\Delta>0$ is the length of the gap in the support of the generating density whose right boundary is $\tau_0$.
\end{itemize}
In particular, the statement (b) does not contain any interesting information and thus expansion points $\tau_0$ of the minimum type with $\avg{v(\tau_0)}>0$  were even  not considered in Theorem 2.6 of \cite{AEK1cpam}.
In Theorem~\ref{thr:Shape of generating density near its small values} of the current paper, however, the description is uniform in $\delta$ and covers an expansion neighborhood around $\tau_0$ whose size is comparable to $1$, i.e., it describes the shape of $\Psi(\omega)$ for all $\abs{\omega}\leq c$ for some constant $c\sim 1$. Thus, the new  result resolves the two universal shape functions from \eqref{defs of shape functions} and reveals how these functions give rise to a continuous one-parameter family of shapes interpolating between them.
In particular, it shows how the cusp singularity emerges  when a gap closes or when the value of $v$ at a local minimum drops down to zero. 
Indeed, as the length of the gap $\Delta$ in the support of the generating density at $\tau_0$ shrinks (as $\delta\uparrow \delta_c$ for the example family $S^{(\delta)}$)  the shape function $\Psi(\omega)=\Delta^{1/3}\Psi_{\mathrm{edge}}(\Delta^{-1}\omega)$ approaches the cusp shape $\Psi(\omega)=2^{-4/3}\abs{\omega}^{1/3}$. 
On the other hand, as the value $\rho\sim \avg{v(\tau_0)}$ in the shape function $\Psi(\omega)=\rho\1\Psi_{\mathrm{min}}(\rho^{-3}\omega)$ at  a local minimum  approaches zero  ($\delta\downarrow \delta_c$ for the family $S^{(\delta)}$), the cubic root cusp emerges as well. The following table summarizes the differences between the current Theorem~\ref{thr:Shape of generating density near its small values} and Theorem 2.6 of \cite{AEK1cpam}.
\begin{center}
\begin{tabular}{l|c|c}
&  
Theorem 2.6 in \cite{AEK1cpam} 
& 
Current Theorem~\ref{thr:Shape of generating density near its small values}
\\
\hline
Input parameters:	
&	
$a,S$ fixed		
&	 
model parameters
\\
\hline
Expansion points $\tau_0$:		& $\sett{\alpha_i} \cup \sett{\beta_j} $	 & $\sett{\alpha_i} \cup \sett{\beta_j} \cup \sett{\gamma_k}$
\\
\hline
Expansion environment: 
&	
$\abs{\omega} \ll 1$	
& 
$\abs{\omega} \lesssim 1$
\end{tabular}
\end{center}

\section{Outline of proofs}

In this section we will explain and motivate the basic steps leading to our main results. 
%

\medskip
\noindent{\scshape Stieltjes transform representation, $\Lp{2}$- and uniform bounds:}
It is a structural property of the QVE that its solution admits a representation as the Stieltjes transform of some generating measure  on the real line (cf. \eqref{m as stieltjes transform}). This representation implies that $m$ can be fully reconstructed from its own imaginary part near the real line. 

From the Stieltjes transform representation of $m$ a trivial bound, $ \abs{m_x(z)}\leq (\Im\2z)^{-1}$, directly follows. 
A detailed analysis of the QVE near the real axis, however, requires bounds that are independent of $ \Im\,z $ as its starting point. 
When $ a = 0 $ and $ z $ is bounded away from zero the $ \Lp{2} $-bound \eqref{L2 bound on m when a=0} meets this criterion. The estimate \eqref{L2 bound on m when a=0} is a structural property of the QVE as well in the sense that it follows from positivity and symmetry of $S$ alone, and therefore quantitative assumptions such as {\bf A2} and {\bf A3} are not needed. 
This $ \Lp{2}$-bound is derived from spectral information about a specific operator $F=F(z)$, constructed from the solution $m=m(z)$, that appears naturally when taking the imaginary part on both sides of the QVE. Indeed, \eqref{QVE} implies 
\bels{Im of QVE and F}{
\frac{\Im\2m}{|m|}\,=\, |m|\2\Im\2z +F \2\frac{\Im\2m}{|m|} 
\,,\qquad
Fu\,:=\, |m|S(|m|u)\,.
}
As $ \Im\,z $ approaches zero we may view this as an eigenvalue equation for the positive symmetric linear operator $F$. In the limit this eigenvalue equals $ 1 $ and $f=\Im\, m/|m|$ is the corresponding eigenfunction, provided $ \Im\,m $ does not vanish. 
The Perron-Frobenius theorem, or more precisely, its generalization to compact operators,  the Krein-Rutman theorem, implies that this eigenvalue coincides with the spectral radius of $F$. This, in turn, implies the $\Lp{2}$-bound on $m$, when $ a = 0 $. These steps are carried out in detail at the end of Chapter~\ref{chp:Existence and uniqueness}. In fact, the norm of $F(z)$, as an operator on $\Lp{2}$, approaches $ 1 $ if and only if $ z $ approaches the support of the generating measure. Otherwise it stays below $1$.
When $ a \neq 0 $ this spectral bound still holds for $ F $, however, it does not automatically yield useful $ \Lp{2}$-estimates on $ m(z) $ when $ \abs{z} \leq \norm{a} $. In order, to obtain an $ \Lp{2}$-bound in this case as well, we need to assume more about $ S $. In Chapter \ref{chp:Uniform bounds} it is shown that the condition \eqref{qualitative strong diagonal condition}, or its quantitative version {\bf B2} on p. \pageref{def of B2}, together with the spectral bound on $ F $, yield an $ \Lp{2} $-bound on $ m $.    

Requiring the additional regularity condition \eqref{qualitative B1} on $ S $ enables us to improve the $\Lp{2}$-bound on $m$ to a uniform bound (Proposition~\ref{prp:Converting L2-estimates to uniform bounds}).
When $ a = 0 $ the point $z=0$ requires a special treatment, because the structural bound $ \norm{m(z)}_2 \leq 2/\abs{z} $  becomes ineffective. 
The block fully indecomposability condition is an essentially optimal condition (Theorem~\ref{thr:Scalability and full indecomposability})  to ensure the uniform boundedness of $ m $ in a vicinity of $z=0$ when $ a = 0 $.
The uniform bounds are a prerequisite for most of our results concerning regularity and stability of the solution of the QVE. 
We consider finding quantitative uniform bounds on $ m $ as an independent problem, that is addressed in Chapter~\ref{chp:Uniform bounds}.

\medskip
\noindent{\scshape Stability in the region where $ \Im\,m $ is large:}
%
Stability properties of the QVE under small perturbations are essential, not just for applications in random matrix theory (cf. Chapter \ref{chp:Local laws for large random matrices}), but also as tools to analyze the regularity of the solution $ m(z)\in \BB_+$ as a function of $ z $.
Indeed, the stability of the QVE translates directly to regularity properties of the generating measure as described by Theorem~\ref{thr:Regularity of generating density}. 
The stability of the solution deteriorates as $ \Im\,m $ becomes small. This happens around the expansion points in $ \MM $ from Theorem~\ref{thr:Shape of generating density near its small values}

In order to see this deterioration of the stability, let us suppose that for a small perturbation $d \in \BB$, the perturbed QVE has a solution $ g(d) $ which depends smoothly on $ d $,
\bels{eq for g(d)}{
-\2\frac{1}{g(d)\!}\,=\, z + a + Sg(d) + d
\,.
}
Indeed, the existence and uniqueness of such a function $ d \mapsto g(d) $ is shown in Proposition \ref{prp:Analyticity} as long as both $ d $ and $ g-m $ are sufficiently small.
For $d=0$ we get back our original solution $ g(0) = m $, with $ m = m(z) $.
We take the functional derivative with respect to $ d $ on both sides of the equation. 
In this way we derive a formula for the (Fr\'echet-)derivative $ Dg(0) $, evaluated on some $w \in \BB$:
\bels{delta m equation}{
(\11-m^2S)Dg(0)w\,=\, m^2w
\,.
}
This equation shows that the invertibility of the linear operator $1-m^2S$ is relevant to the stability of the QVE. Assuming uniform lower and upper bounds on $ \abs{m} $, the invertibility of $1-m^2S$ is equivalent to the invertibility of the following related operator:
\[
B\,:=\, U-F\,=\, \frac{|m|}{\2m^2\!}\,(\11-m^2 S\1)|m|\,,\qquad Uw\2:=\2\frac{|m|^2}{m^2} \2w\2.
\]
Here, $ \abs{m} $ on the right of $ S $ is interpreted as a multiplication operator by $ \abs{m} $. Similarly, $U$ is a unitary multiplication operator and $F$ was introduced in \eqref{Im of QVE and F}. 
Away from the support of the generating measure the spectral radius of $F$ stays below $1$ and the invertibility of $B$ is immediate. On the support of the generating measure the spectral radius of $F$ equals $1$. Here, the fundamental bound on the inverse of $B$ is
\bels{inverse B generic bound}{
\norm{B^{-1}} \,\lesssim\, \avg{\2\Im\, m\1}^{-1}\,,
}
apart from some special situations (cf. Lemma~\ref{lmm:Bounds on B-inverse}). 

Let us understand the mechanism that leads to this bound in the simplest case, namely when $x\mapsto m_x(z)$ is a constant function, e.g., when $ a =0 $ and $ \avg{S_x} =1 $, so that each component equals $ m_{\mrm{sc}}(z) $ from \eqref{SC}. In this situation, the operator $U$ is simply multiplication by a complex phase, $U=\nE^{\cI\2\varphi}$ with $\varphi \in (-\pi,\pi]$.  The uniform bounds on $m$ ensure that the operator $F$ inherits certain properties from $S$. Among these are the conditions {\bf A2} and {\bf A3}. From these two properties we infer a spectral gap $ \eps > 0 $,
\[
\Spec(F) \,\subseteq \, [-1+\eps,1-\eps\1]\cup\{\11\1\}\,,
\]
on the support of the generating measure.
We readily verify the following bound on the norm of the inverse of $B$:
\[
\norm{B^{-1}}_{\Lp{2} \to \Lp{2}}\,\leq\, 
\begin{cases}
|\1\nE^{\1\cI\1\varphi}-1|^{-1}\2\sim \2 \avg{\2\Im\,m\1}^{-1}	\quad &\text{if }\; \varphi \in [-\varphi_*,\varphi_*]\,;
\\
\;\eps^{-1} 				&\text{ otherwise.}
\end{cases}
\]
Here, $\varphi_* \in [0,\pi/2]$ is the threshold defined through $\cos \varphi_* =1-\eps/2$, where the spectral radius $ \norm{F}_{\Lp{2}\to\Lp{2}} $ becomes more relevant for the bound than the spectral gap (cf. Lemma~\ref{lmm:Spectral gap for positive bounded operators}).
Similar bounds for the special case, when  $ m_x = m_{\mrm{sc}} $ is constant in $x$ and equals the Stieltjes transform $m_{\mrm{sc}}$ of the semicircle law in every component first appeared in \cite{EYY}.

The bound \eqref{inverse B generic bound} on the inverse of $B$ implies a bound on the derivative $ Dg(0) $ from \eqref{delta m equation}. For a general perturbation $d$ this means that the QVE is stable wherever the average generating measure is not too small. 
If $ d $ is chosen to be a constant function $ d_x = z'-z $ then this argument yields the bound for the difference $ m(z')-m(z) $, as $ g(z'-z) = m(z') $.  This can be used to estimate the derivative of $m(z)$ with respect to $z$ and to prove existence and H\"older-regularity of the Lebesgue-density of the generating measure.
In particular, the regularity is uniform in $ \Im\,z $ and hence we can  extend the solution of the QVE to the real axis. 
This analysis is carried out in Chapters \ref{chp:Properties of solution} and \ref{chp:Regularity of solution}.

\medskip
\noindent{\scshape Stability in the regime where $ \Im\,m $ is small:}
%
The bound \eqref{inverse B generic bound} becomes ineffective when $ \avg{\2\Im\,m\1} $ approaches zero.
In fact, the norm of $ B^{-1} $ diverges owing to  a single isolated eigenvalue, $\beta \in \C$, close to zero. 
This point is associated to the spectral radius of $F$, and the corresponding eigenvector, $B\1b = \beta\2 b$, is close to the Perron-Frobenius eigenvector of $F$, i.e., $b= f +\Ord(\1\avg{\1\Im\,m})$, with $Ff=f$. 
The special direction $b$, in which $B^{-1}$ becomes unbounded, is treated separately in Chapter \ref{chp:Behavior of generating density where it is small}. It is split off from the derivative $Dg(0)$ in the stability analysis. 
The coefficient of the component $ g(d)-m $ in the bad direction $b$, is given by the formula
\[
\Theta(d) \,:=\,\frac{\avgb{\2\overline{b},g(d)-m\1}}{\avg{\2b^{\12}}}
\,. 
\]
Chapter \ref{chp:Perturbations when generating density is small} is concerned with deriving a cubic equation for $ \Theta(d) $  and expanding its coefficients in terms of $\avg{\1\Im\,m}\ll 1$ at the edge.

\medskip
\noindent{\scshape Universal shape of $ v $ near its small values:}
In this regime understanding the dependence of the solution $g(d)$ of \eqref{eq for g(d)}, is essentially reduced to understanding the  scalar quantity $\Theta(d)$. This quantity satisfies a cubic equation (cf. Proposition~\ref{prp:General cubic equation}), in which the coefficients of the non-constant terms depend only on the unperturbed solution $ m $.
In particular, we can follow the dependence of $m_x(z)$ on $ z \in \R $ by analyzing the solution of this equation by choosing $ z := \tau_0 \in \R $ and $d_x :=\tau-\tau_0 $, a real constant function. 
The special structure of the coefficients of the cubic equation, in combination with specific selection principles, based on the properties of the solution of the QVE, allows only for a few possible shapes that the solution $ \tau \mapsto \Theta(\tau-\tau_0)$ of the cubic equation may have. This is reflected in the universal shapes that describe the growth behavior of the generating density at the boundary of its support. In Chapter \ref{chp:Behavior of generating density where it is small} we will analyze the three branches of solutions for the cubic equation in detail and select the one that coincides with $\Theta$. 
This will complete the proof of Theorem~\ref{thr:Shape of generating density near its small values}. 

\medskip
\noindent{\scshape Optimal Stability around small minima of $ \avg{v}$:}
For the random matrix theory we need  optimal stability properties of the perturbation $ g(d) $ around $ g(0) = m $ for a random perturbation (cf. Chapter~\ref{chp:Local laws for large random matrices}).
This is achieved in Chapter~\ref{chp:Stability around small minima of generating density} by describing the coefficients of the cubic more explicitly based on the shape analysis. 
All the necessary results are collected in Proposition~\ref{prp:Cubic perturbation bound around critical points}. These technical results generalize Theorem~\ref{thr:Stability}.

\chapterl{Local laws for large random matrices}
 
The QVE plays a fundamental role in the theory of large random matrices.
First, it provides the only known effective way to determine the asymptotic eigenvalue density for prominent matrix ensembles 
as described in the introduction (cf. Section 3 of \cite{AEK1cpam} for details).
Second, the QVE theory is essential when establishing \emph{local laws} for the distribution of the eigenvalues at the scale comparable to the individual eigenvalue spacings for so-called \emph{Wigner-type} matrices.
Here we explain how our results can be utilized for this purpose. 
Since all  technical details are already carried out in \cite{AEK2} we highlight the structure of the proofs in the simplest possible setup by showing how the probabilistic estimates and the stability properties of the QVE can be turned into very precise probabilistic bounds on the resolvent elements of the random matrix.

Let us recall from \cite{AEK2} the following definition.
\NDefinition{Wigner-type random matrix}{
A real symmetric or complex hermitian $ N\times N$ random matrix $ \brm{H} = (h_{ij})_{i,j=1}^N $ is called {\bf Wigner-type}, if it has 
\begin{itemize}
\item[(i)]
Centred entries: $ \EE\,h_{ij} = 0 $;
\item[(ii)] 
Independent entries: $ (\1h_{ij}:1\leq i\leq j\leq N) $ are independent;
\item[(iii)] 
Mean-field property: The {\bf variance matrix}
$ \brm{S} = (s_{ij})_{i,j=1}^N $, $ s_{ij} := \EE\1\abs{h_{ij}}^2 $, 
satisfies
\bels{mean-field property}{
\qquad(\brm{S}^L)_{ij} \,\ge\,\frac{\rho}{N}
\qquad
\text{and}\qquad
s_{ij} \leq \frac{\2S_\ast\!}{N}
\,,
\qquad 
1 \leq i,j \leq N\,,
}
for some parameters $ \rho,L,S_\ast < \infty $.
\end{itemize}
}

If in addition to (i)-(iii) the variance matrix is doubly stochastic, i.e., $ \sum_j s_{ij} = 1 $ for each $ i$, and  \eqref{mean-field property} holds with $ L =1 $, then $ \brm{H} $ is called a \emph{generalized Wigner matrix} (first introduced in \cite{EYY}).

A given variance matrix $ \brm{S} $ defines a QVE through 
\bels{RM setup}{
\Sx := \sett{1,2,\dots,N}\,,
\quad
\Px(A) := \frac{\abs{A}}{N\msp{-1}}\,,
\quad
a = 0\,,
\quad
(Sw)_i := \sum_{j=1}^N s_{ij}
w_j   
\,,
}
where the subset $ A \subset \sett{1,\dots,N} $ and the function $ w : \Sx \to \C $ are arbitrary. 
The kernel of the operator $ S :\BB\to\BB $ is related to the variances by $ S_{ij} :=  Ns_{ij}$.
In particular, if \eqref{mean-field property} is assumed, then the operator $ S $ satisfies {\bf A1}, as well as {\bf A2} and {\bf A3} with parameters $ \rho,L$ and $ \norm{S}_{\Lp{2}\to\BB} \leq S_\ast $, respectively.

A local law for $ \brm{H} $ roughly states that the density of the eigenvalues $ \lambda_1 \leq \ldots \leq \lambda_N $ of $ \brm{H} $ is predicted by the associated QVE through
\bels{density of states for z from Im m}{
\rho(z) \,:=\, \frac{1}{\pi}\avg{\2\Im\,m(z)\1}
\,,
}
all the way down to the optimal scale $ \Im\,z \gg N^{-1} $, just above the typical eigenvalue spacing. 
Moreover, the local law implies that the eigenvectors are completely \emph{delocalized}, i.e., no component of an $ \ell^2$-normalized eigenvector of $ \brm{H} $ is much larger than $ N^{-1/2} $ with very high probability (cf. Corollary 1.14 of \cite{AEK2}).
A local law is most generally stated in term of the entries of the resolvent 
\bels{def of brm-G}{
\brm{G}(z) := (\1\brm{H} -z\1)^{-1}
\,,\quad
z \in \Cp
\,.
}
The following is a simplified version of the main local law theorem of \cite{AEK2}. It states that $ \brm{G}(z) $ approaches the diagonal matrix  determined by the solution $ m(z) $ of the QVE, provided the imaginary part of the  spectral parameter $ z $ is slightly larger than the eigenvalue spacing, $ N^{-1}$, inside the bulk of the spectrum. 
Indeed, denoting
\bels{}{
\DD^{(N)}_\gamma
:= \setb{z \in \C : N^{\gamma-1}< \Im\,z \leq \Sigma\1}
\,.
}
where  $ \gamma > 0 $ and $ \Sigma > 0 $ is from \eqref{bound on supp v}, the theorem reads:

\begin{theorem}[Entrywise local law from \cite{AEK2}]
\label{thr:Entrywise local law from AEK2}
Let $ \brm{H} $ be a Wigner-type random matrix, and suppose the associated QVE \eqref{RM setup} has a bounded solution $ m $, with $ \nnorm{m}_\R \leq \Phi $. 
If additionally, the moments of $\brm{H} $ are bounded by the variances,
\bels{RM moment conditions}{ 
\qquad
\EE\,\abs{h_{ij}}^{2p} \leq \mu_p\1
s_{ij}^{\2p}
\,,
\qquad \forall\2p \ge 2
\,,
}
then the entries $ G_{ij}(z) $ of the resolvent \eqref{def of brm-G} satisfy for every $ \phi, \gamma,p > 0 $, 
\bea{
\PP\Biggl\{\, 
\exists\1 
z \in \DD^{(N)}_\gamma
\text{ s.t. }
\absb{\1G_{ij}(z)-\delta_{ij}\1m_i(z)} 
\,>
\frac{N^{\phi}\!}{\sqrt{N\1\Im\, z}}\,
\Biggr\}
\,\leq\, \frac{C(\phi,\gamma,p\1;\brm{\xi}_S,\ul{\mu})}{\,N^p\!}
\,,
}
where the function $ C(\genarg,\genarg,\genarg;\brm{\xi}_S,\ul{\mu}\1)< \infty $ depends on $ \brm{H} $ only through the parameters $ \brm{\xi}_S := (\rho,L,S_\ast,\Phi) $ and $ \ul{\mu} = (\mu_p:p\ge 1) $.
\end{theorem}

We stress that the error bound in the local law does not depend on the variance matrix through anything else than the parameters $ \rho $, $ L $, $ S_\ast $, and $ \Phi $.
If the operator $ S $ also satisfies the quantitative versions of the assumptions (i) of Theorem~\ref{thr:Qualitative uniform bounds}, then the implicit constant $ \Phi $ can also be effectively bounded in terms of the variance matrix using a few additional model parameters appearing in the hypotheses of Theorem~\ref{thr:Quantitative uniform bounds when a = 0} below.

It is also shown in \cite{AEK2} that under the conditions of the previous theorem an  \emph{averaged local law} holds with an improved error bound. More precisely, for any non-random weights $ w_k $, and $ \phi > 0 $, we have
\bels{averaged local law}{
\frac{1}{N}\absbb{\sum_k w_k\,(G_{kk}(z)-m_k(z)\1)}
\;\leq\; 
N^\phi \max_i \abs{w_i}\,
\frac{\mcl{E}_N(z)}{\2N\1\Im\, z}
\,,
}
with very high probability for sufficiently large $ N $. Here the error term $ \mcl{E}_N(z) $ is $ \Ord(1)$, except when $ z $ approaches an asymptotically small non-zero minimum or an asymptotically small gap in $ \supp v $ (cf. formulas (1.21) and (1.23)-(1.25) in \cite{AEK2} for details). 
In particular, choosing $ w_k = 1 $ in \eqref{averaged local law} and considering the spectral parameters $ z $ in the bulk of the spectrum, so that $ \avg{v(\Re\,z)} > 0 $, we find for every $ \phi > 0 $,
\bels{LL for empirical DOS}{
\absB{\frac{1}{N}\mrm{Tr}\,\brm{G}(z)-\avg{\1m(z)}} \,\leq\,  \frac{N^\phi}{N\1\Im z}
\,,
}
with very high probability.
This estimate is the starting point for proving the \emph{local bulk universality} for eigenvalues of $ \brm{H} $. 
For more details see Theorem 1.7 of \cite{AEK2}.

All these result have been originally obtained for generalized Wigner matrices in a sequence of papers \cite{EYY,EYYber,EYYrigi}, see \cite{EKYY} for a summary. The main difference is that for generalized Wigner matrices the limiting density is given by the explicit Wigner semicircle law \eqref{SC}, while  Wigner-type matrices have a quite general density profile that is known only implicitly from the solution of the QVE using \eqref{density of states for z from Im m}. 
In particular, the density may have cubic root singularities (cf. Theorem~\ref{thr:Shape of generating density near its small values}), as opposed to two square root singularities of the semicircle law, and these new kind of singularities require a new proof for the local law.

\sectionl{Proof of local law inside bulk of the spectrum}

In order to see why Theorem~\ref{thr:Entrywise local law from AEK2} should hold we first apply the Schur complement formula for the diagonal entries of the resolvent \eqref{def of brm-G} to get 
\bels{Schur 0}{
-\frac{1}{G_{kk}(z)} 
\,=\, 
z \1- h_{kk} + \sum_{i,j}^{(k)} h_{ki}\1G^{(k)}_{ij}\msp{-2}(z)\2h_{jk} 
\,,
}
where $ \sum^{(k)}_{i,j} $ denotes the sum over all indices $ i,j $ not equal to $ k $, and $ G^{(k)}_{ij}(z) $ are the entries of the resolvent of the matrix obtained by setting the $ k$-th row and the $ k$-th column of $ \brm{H} $ equal to zero.
Replacing the terms on the right hand side of \eqref{Schur 0} by their partial averages w.r.t. the $ k $-th row and column, and regarding the rest as perturbations, we arrive at a perturbed QVE,
\bels{Schur}{
-\frac{1}{G_{kk}(z)} 
\,&=\, 
z \2- \EE\,h_{kk} 
+ \frac{1}{N}\sum_{i=1}^N S_{ki}
G_{ii}(z) + d_k(z)
\,,
}
for  the diagonal entries of the resolvent $ G_{kk}(z) $.
Here the random error is given by 
\begin{equation}
\label{form of perturbation d}
\begin{split}
d_k(z) \;=\; 
&\sum_{i\neq j}^{(k)} h_{ki}\1G^{(k)}_{ij}\msp{-2}(z)\2h_{jk} 
\,+\, \sum_i^{(k)} 
\bigl(\2\abs{h_{ki}}^2-\EE\1\abs{h_{ki}}^2\2\bigr)
\2 G^{(k)}_{ii}\msp{-2}(z) 
\\
&+\,\frac{1}{N}\sum_i^{(k)} S_{ki}\1(\1G^{(k)}_{ii}\msp{-2}(z)-G_{ii}(z)\1)
\,-\, 
(\1h_{kk}-\EE\,h_{kk}\1) 
\,-\, \frac{S_{kk}\!}{N}\2
G_{kk}(z)
\,.
\end{split}
\end{equation}
Setting $ a_k = -\1\EE\,h_{kk} $ we identify \eqref{Schur} with the perturbed QVE \eqref{perturbed QVE - 1st time}.
For the sake of simplicity, we consider only the case $ \EE\,h_{kk} = 0 $ here. 

Since $ \brm{G}^{\msp{-1}(k)}\msp{-1}(z) $, by definition does not depend on the $ k$-th row/column of $ \brm{H} $, the centered terms $ h_{ki}$, $h_{jk} $ and $ (\abs{h_{ki}}^2-\EE\2\abs{h_{ki}}^2\1) $  are independent of $  \brm{G}^{\msp{-1}(k)}\msp{-1}(z) $ in \eqref{form of perturbation d}.
Therefore the first term on the right hand side of \eqref{form of perturbation d} can be controlled by the standard large deviation estimate (cf. Appendix B of \cite{EKYY13}) of the form
\bels{RM:LD}{
\qquad \PP\Biggl\{ \,\absB{\,\sum_{i\neq j} a_{ij}X_iX_j}^2 \ge N^\kappa \sum_{i\neq j}\, \abs{a_{ij}}^2
\Biggr\}
\,\leq\, \frac{C(\kappa,q)}{N^q\!}
\,.
}
Here $ X_i$'s are independent and centered random variables with finite moments, and the exponents $ \kappa,q > 0 $ are arbitrary.
A similar bound holds for the second term on the right hand side of \eqref{form of perturbation d}.

Lemma 2.1 in \cite{AEK2}, states that if
\bels{RM: def of Lambda}{
\Lambda(z) \,:=\, 
\max_{i,j=1}^N \absb{G_{ij}(z)-\delta_{ij}m_i(z)}
}
satisfies a rough a priori estimate, then the perturbation $ d_k(z) $ can be shown to be very small using standard large deviation estimates, such as \eqref{RM:LD}, and standard resolvent identities. A simplified version of this lemma is formulated as follows:

\begin{lemma}[Probabilistic part for simplified local law]
\label{lmm:Probabilistic part for simplified local law}
Under the assumptions of Theorem~\ref{thr:Entrywise local law from AEK2}, there exist a threshold $ \lambda_1 > 0 $ and constants $ C_1(\kappa,q) < \infty $, for any $ \kappa, q > 0 $, such that for any $ z \in \Cp $
\bels{RM: prob bound for fixed z}{
\!\PP\Biggl\{\,
\biggl(
\norm{d(z)}
+ \max_{\substack{i,j=1\\ i\1\neq\1 j}}^N\abs{G_{ij}(z)}\biggr)\1
\Ind\setb{\1\Lambda(z) \leq \lambda_1}
\;>
N^\kappa \delta_N(z)\,
\Biggr\}
\,\leq\, \frac{C_1(\kappa,p)\!}{\,N^q\!}
\,,
}
where $ \norm{\genarg} $ denotes the supremum norm, and 
\bels{def of error delta_N}{
\delta_N(z) := \frac{1}{\sqrt{N\1\Im\,z\1}\1} + \frac{1}{\sqrt{N}}
\,.
}
Here $ \lambda_1 $ and the constants $C_1(\kappa,p) $ are independent of $ z $. They depend on the random matrix $ \brm{H} $ only through the parameters $ \brm{\xi}_S,\ul{\mu} $ defined in Theorem~\ref{thr:Entrywise local law from AEK2}.
\end{lemma}

We we will now show how to prove the entrywise local law, Theorem~\ref{thr:Entrywise local law from AEK2}, in the special case where the spectral parameter $ z $ satisfies the \emph{bulk assumption} \eqref{condition for being away from critical points} for some $ \eps > 0 $. 
The proof demonstrates the general philosophy of how the non-random stability results for the QVE, such as Theorem~\ref{thr:Stability}, are used together with probabilistic estimates, such as Lemma~\ref{lmm:Probabilistic part for simplified local law} above.
Our estimates will deteriorate as the lower bound $ \eps $ in the bulk assumption approaches zero.
In order to get the local law uniformly in $ \eps $ a different and much more complicated  argument (cf. Section~4 of \cite{AEK2}) is needed.
In particular, Theorem \ref{thr:Stability} must be replaced by its more involved version, Proposition~\ref{prp:Cubic perturbation bound around critical points}.

In order to  obtain the averaged local law \eqref{averaged local law}, under the bulk assumption, the componentwise estimate \eqref{rough perturbation-bound: sup} must be replaced by the averaged estimate \eqref{rough perturbation-bound: w-average}, which bounds $ G_{kk}(z)-m_k(z) $, in  terms of a weighted average of $ d_k(z) $.
The improvement comes from the \emph{fluctuation averaging} mechanism introduced in~\cite{EKYY,EYYber}.
In fact, Theorem 3.5 of \cite{AEK2} shows that $ \avg{w,d(z)} $, for any non-random $ w \in \BB $, is typically of size $ \norm{w}\2\norm{d(z)}^2 $, and hence much smaller than the trivial bound $ \norm{w}\norm{d} $ used in the entrywise local law. %
For the averaged bounds, the bulk assumption  \eqref{condition for being away from critical points} can be removed as well by using Proposition~\ref{prp:Cubic perturbation bound around critical points} in place of Theorem \ref{thr:Stability}. 

\begin{Proof}[Proof of Theorem~\ref{thr:Entrywise local law from AEK2} in the bulk]
Let us fix $ \tau_0 \in \R $ such that \eqref{condition for being away from critical points} holds for all $ z $ on the line
\bels{RM: def of LL}{
\LL \,:=\, \tau_0 \2+\2 \cI\,[\1N^{\gamma-1}\msp{-6},N\2]
\,.
} 
We will also fix an arbitrary $ \gamma > 0 $.
Clearly, it suffices to prove the local law only when $ N $ is larger than some threshold $ N_0 = N_0(\phi,\gamma,p) < \infty $, depending only on $ \brm{\xi}_S,\ul{\mu} $, in addition to the arbitrary exponents $ \phi,\gamma,p > 0 $. 

Combining \eqref{RM: prob bound for fixed z} with the stability of the QVE  under the perturbation $ d(z) $, Theorem~\ref{thr:Stability}, we obtain 
\bels{RM:QVE stability}{
\biggl(
\max_{k=1}^N\2 \abs{G_{kk}(z)-m_k(z)}
\biggr)\1
\Ind\setb{\Lambda(z) \leq \lambda\1\eps}
\,&\leq\,
\frac{C_2}{\eps^{2}\!}\1
\norm{d(z)}\,
\,.
}
Here the indicator function guarantees that the part (i) of Theorem~\ref{thr:Stability} is applicable. 
The constant $ \lambda  \sim 1 $ is taken from that theorem, while $ C_2 \sim 1 $ is the hidden constant in \eqref{rough perturbation-bound: sup}. 

Combining \eqref{RM:QVE stability} with \eqref{RM: prob bound for fixed z} we see that for every $ \kappa,q > 0 $, and every fixed $ z \in \Cp $, there exists an event $ \Omega_{\1\kappa,q}(z)$, of very high probability
\bels{RM:def of good event at z}{
\PP(\2\Omega_{\1\kappa,q}(z)\1) \,\ge\, 1-C_3(\kappa,q)N^{-q},
}
such that for a sufficiently large threshold $ N_0 $ and every $ N \ge N_0 $ we get 
\bels{RM:z-wise Lambda bound for a single z}{
\Lambda(z;\omega)
\Ind\sett{\1\Lambda(z;\omega) \leq 2\lambda_\ast}
\,&\leq\,
N^{2\1\kappa}\delta_N(z)
\,,\qquad \forall\1\omega \in \Omega_{\kappa,q}(z)
\,,
}
where $ 2\lambda_\ast := \min\sett{\lambda_1,\lambda\1\eps} $.

The event $ \Omega_{\kappa,q}(z) $ depends on the spectral point $ z \in \LL $.
As a next step we replace the uncountable family of events $  \Omega_{\kappa,q}(z) $, $ z \in \LL $, in \eqref{RM:z-wise Lambda bound for a single z} by a single event, that covers all $ z \in \LL $.
To this end, we use the regularity of the resolvent elements and of the solution to the QVE in the spectral variable $ z $. 
Indeed, they are both Stieltjes transforms of probability measures (cf. \eqref{m as stieltjes transform}), and thus their derivatives are uniformly bounded by $ (\Im\,z)^{-2} \leq N^2 $ when $ z \in \LL $. 
In particular, it follows that
\bels{Lipschitz for Lambda}{
\abs{\1\Lambda(z')-\Lambda(z)} \,\leq\, 2\1N^2\abs{z'-z}
\,,\qquad 
z,z' \in \LL
\,.
} 
Let $ \LL_N $ consist of $ N^5 $ evenly spaced points on $ \LL $, such that the $ N^{-4} $-neighborhood of $ \LL_N $ covers $ \LL $.
Combining \eqref{Lipschitz for Lambda} and \eqref{RM:z-wise Lambda bound for a single z} we see that for any $ \phi, p > 0$, the intersection event,
\bels{def of good event for all z}{
\Omega_{\phi,p} \,:= \bigcap_{z \ins \LL_N} \Omega_{\1\phi/3\1,\1p\1+5}(z)
\,,
}
has the properties 
\begin{subequations}
\label{RM:good event}
\begin{align}
\label{RM:good event:HP}
\PP(\2\Omega_{\phi,p}) \,&\ge\, \11\2-\1
C_1(\phi,p)\1N^{-p}
\\
\label{RM:good event:LL}
\qquad\Lambda(z;\omega)\1\Ind\sett{\2\Lambda(z;\omega) \leq \lambda_\ast}
\,&\leq\,
N^{\phi}\delta_N(z)
\,,\qquad\quad \forall\1(z,\omega) \in \LL \times \Omega_{\phi,p}
\,.
\end{align}
\end{subequations}
Here $ C_1(\phi,p) := C_3(\phi/3,p+5) $, with $ C_3(\genarg,\genarg) $ taken from \eqref{RM:def of good event at z}.
In order to prove \eqref{RM:good event:LL} pick an arbitrary pair $ (z,\omega) \in \LL \times \Omega_{\phi,p} $, and set $ \kappa := \phi/3 $ and $ q := p + 5 $.
If $ z \in \LL_N $, then the claim follows directly from \eqref{RM:z-wise Lambda bound for a single z} and \eqref{def of good event for all z}.
In the case $ z \notin \LL_N $, let $ z' \in \LL_N $ be such that $ \abs{z'-z} \leq N^{-4} $.
Suppose now that $ \Lambda(z;\omega) \leq \lambda_\ast $. By the continuity \eqref{Lipschitz for Lambda} we see that $ \Lambda(z';\omega) \leq 2\lambda_\ast $, and thus \eqref{RM:z-wise Lambda bound for a single z} yields $ \Lambda(z';\omega) \leq N^{2\kappa}\delta_N(z') $.
Using  \eqref{Lipschitz for Lambda} together with $ \delta_N(z) \ge N^{-1/2}$ and $ \abs{\delta_N(z)-\delta_N(z')} \leq N^{1/2}\abs{z-z'} $ we get  $ \Lambda(z;\omega) \leq N^{3\kappa}\delta_N(z) $. 
 This proves \eqref{RM:good event:LL}.

The proof of the local law is now completed by showing that the indicator function is identically equal to one in \eqref{RM:good event:LL} for $ (z,\omega) \in \LL \times \Omega_{\phi,p} $, provided $ \phi < \gamma/2 $.
Indeed, if $ N_0 $ is so large that $ N_0^{\phi-\gamma/2} < \lambda_\ast/2 $, then $ N^\phi \delta_N(z) < \lambda_\ast/2 $, for $ N \ge N_0 $, and thus the bound \eqref{RM:good event:LL} implies
\[
\qquad
\Lambda(z;\omega) \notin \biggl[\frac{\2\lambda_\ast\!}{2}\2,\lambda_\ast\biggr]
\,,\qquad \forall\1(z,\omega) \in \LL \times \Omega_{\1\phi,p}
\,.
\]
Fix $ \omega \in \Omega_{\phi,p} $.
Since $ z \mapsto \Lambda(z;\omega) $ is continuous, the set $ \Lambda(\1\LL;\omega) $ is simply connected. 
Therefore it is contained either in $ [\10\1,\lambda_\ast/2] $, or in $ [\1\lambda_\ast\1,\1\infty\1) $. 
The latter possibility is excluded by considering the point $ z_0 := \tau_0 + \cI\1N \in \LL $. Indeed, from \eqref{def of brm-G} and the Stieltjes transform representations it follows that
\[
\abs{\1G_{ij}(z_0;\omega)-\delta_{ij}\1m_i(z_0)} \leq \frac{2}{\2\Im\,z_0} = \frac{2}{N\msp{-1}}
\,,
\qquad i,j =1,\dots,N
\,.
\]
Assuming that $ N_0  $ is so large that $ 2/N_0 <\lambda_\ast/2 $, we see $ \Lambda(z_0;\omega) <\lambda_\ast/2 $. 
This completes the proof of Theorem \ref{thr:Entrywise local law from AEK2} for spectral parameters $z$ satisfying the bulk condition \eqref{condition for being away from critical points}.
\end{Proof}

\chapter{Existence, uniqueness and $\Lp{2}$-bound}
\label{chp:Existence and uniqueness}

This chapter contains the proof of Theorem~\ref{thr:Existence and uniqueness}. Namely assuming,
\begin{itemize}
\item \emph{$ S $ satisfies} {\bf A1},
\end{itemize}
we show that the QVE \eqref{QVE} has a unique solution, whose components $ m_x $ are Stieltjes transforms (cf. \eqref{m as stieltjes transform}) of $ x-$dependent probability measures, supported on the interval  $[-\Sigma,\1\Sigma\2] $. 
We also show that if $ a = 0 $, then $ m(z) \in \Lp{2} $, whenever $ z \neq 0 $ (cf. \eqref{L2 bound on m when a=0}). 
The existence and uniqueness part of Theorem~\ref{thr:Existence and uniqueness} is proven by considering the QVE as a fixed point problem in the space $\BB_+$.
The choice of an appropriate metric on $\BB_+$ is suggested by the general theory of Earle and Hamilton \cite{EarleHamilton70}.  
A similar line of reasoning for the proof of existence and uniqueness results that are close to the one presented here has appeared before (see e.g. \cite{AZind,Helton2007-OSE,KLW2,FHS2006}).
The structural $ \Lp{2}$-estimate in Section~\ref{sec:Operator F and structural L2-bound} is the main novelty of this chapter. 

For the purpose of defining the correct metric on  $\BB_+$ we use the standard hyperbolic metric $d_\Cp$ on the complex upper half plane $\Cp$.
This metric has the additional benefit of being invariant under $z\mapsto -z^{-1}$, which enables us to exchange the numerator and denominator on the left hand side of the QVE. 

We start by summarizing a few basic properties of $d_\Cp$. These will be expressed through the function 
\bels{cDefinition}{
D(\zeta,\omega)\,:=\,\frac{\abs{\1\zeta-\omega\1}^2}{(\1\Im\,\zeta\1)\1(\1\Im\,\omega)}\,,
\qquad \forall\; \zeta, \,\omega \in \Cp
\,,
}
which is related to the hyperbolic metric through the formula
\bels{dAndc}{
D(\zeta,\omega)\,=\,2\2(\1\cosh d_\Cp(\zeta,\omega)-1\,)
\,.
}

\begin{lemma}[Properties of hyperbolic metric]
\label{lmm:Properties of hyperbolic metric} 
The following three properties hold for $D$:
\begin{enumerate}
\item \label{cProp1}
Isometries: If $ \psi : \Cp \to \Cp $, is a linear fractional transformation, of the form 
\[
\psi(\zeta) \,=\,\frac{\alpha \1 \zeta+\beta}{\gamma \1 \zeta+\mu}
\,,\qquad
\mat{\alpha & \beta \\ \gamma &\mu} \in \mrm{SL}_2(\R)\,, 
\]
then 
\[
D\big(\psi(\zeta),\psi(\omega)\big) \,=\, D(\zeta,\omega) 
\,.
\]
\item \label{cProp2}
Contraction: If $ \zeta $, $\omega \in \Cp $ are shifted in the positive imaginary direction by $ \lambda> 0$ then  
\bels{}{    
D(\1\zeta + \cI\1\lambda\1, \omega + \cI\1\lambda)
\,=\,
\Bigl(1+\frac{\lambda}{\Im\,\zeta}\Bigr)^{\!-1}
\Bigl(1+\frac{\lambda}{\Im\,\omega}\Bigr)^{\!-1}D(\zeta,\omega)
\,.}
\item \label{cProp3}
Convexity:
Suppose $ 0 \neq \phi \in \BB^* $ is a bounded non-negative linear functional on $\BB$, i.e., $\phi(u)\geq 0$ for all $u\in \BB$ with $u\geq0$. Let $u,w \in \BB_+$ with imaginary parts bounded away from zero, $ \inf_x \Im \,u_x$, $\inf_x\Im\, w_x>0 $. 
Then
\bels{cConvexity}{
D\big(\1\phi(u),\2 \phi(w) \big)
\,\leq\,
\sup_{x \in \Sx} D(\1u_x,w_x)
\,.
}
\end{enumerate}
\end{lemma}
\begin{Proof}
Properties \ref{cProp1} and \ref{cProp2} follow immediately from \eqref{dAndc} and \eqref{cDefinition}. It remains to prove Property~ \ref{cProp3}. 
The functional $\phi$ is non-negative. Thus, $\abs{\phi(w)}\leq \phi(\abs{w})$ for all $w \in \BB$. Therefore,
\bels{D first bound}{
D\bigl(\1\phi(u),\2 \phi(w) \bigr)
\,\leq\, 
\frac{\phi(\abs{u-w})^2}{\phi(\1\Im\2 u)\2\phi(\1\Im \2w)}
\,=\, 
\frac{\phi\bigl(\1
(\Im \2 u)^{1/2} (\Im \2 w)^{1/2} D(u,w)^{1/2}\2
\bigr)^2}{\phi(\1\Im\2 u)\2\phi(\1\Im \2w)}\,,
}
where we used the definition of $D$ from \eqref{cDefinition} two times. 
We apply a version of Jensen's inequality for bounded linear, non-negative and normalized functionals on $\BB$ to estimate further,
\bels{D second bound}{
\frac{
\2\phi\bigl(\1
(\Im \2 u)^{1/2} (\Im \2 w)^{1/2} D(u,w)^{1/2}\1
\bigr)^2}{\phi\bigl(\1(\Im \2 u)^{1/2} (\Im \2 w)^{1/2}\bigr)}
\,\leq\, 
\phi\bigl(\1(\Im \2 u)^{1/2} (\Im \2 w)^{1/2} D(u,w)\1\bigr)
\,.
}
We combine \eqref{D first bound} with \eqref{D second bound} and use the non-negativity of $\phi$ to estimate $ D(u,w)\leq \sup_{x \in \Sx}D(u_x,w_x)$ inside its argument,
\bels{D final inequality}{
D\big(\1\phi(u),\2 \phi(w) \big)
\,\leq\, 
\frac{\1\phi\bigl(\1(\Im \2 u)^{1/2} (\Im \2 w)^{1/2}\1\bigr)^2}{\phi(\1\Im\2 u)\2\phi(\1\Im \2w)}\,
\sup_{x \in \Sx}D(\1u_x,w_x)
\,.
}
Finally we use $2\1\phi(\1g^{1/2}h^{1/2})\leq \phi(g) + \phi(h)$ for the choice $ g:=\Im\1 u/\phi(\Im\1 u)$ and $ h := \Im\1 w/\phi(\Im\1 w)$ to show that the fraction on the right hand side of \eqref{D final inequality} is not larger than $1$.
This finishes the proof of \eqref{cConvexity}.
\end{Proof}

In order to show existence and uniqueness of the solution of the QVE for given $ S $ and $ a $, we see that for any fixed $z  \in \Cp$, a solution $m=m(z)\in \BB_+$ of \eqref{QVE for fixed z} is a fixed point of the map
\bels{}{
\Phi(\genarg;z) : \BB_+\to\BB_+\,,\qquad\Phi(u\1;z) \,:=\, -\,\frac{1}{z+a+Su}
\,.
}
Let us fix a constant $ \eta_0 \in (\10\1, \min\sett{1,1/\norm{a}} \1) $ such that $z$ lies in the domain 
\bels{}{
\Cp_{\eta_0} \,:=\, \setb{ z \in \Cp:\, |z| \1<\1 \eta_0^{-1} \, , \; \Im \2 z \1 > \1 \eta_0 }
\,.
}
We will now see that $\Phi(\genarg;z)$ is a contraction on the subset
\bels{}{
\BB_{\eta_0} \,:=\,
\setbb{ u \in \BB_+: \, \norm{u} \1\leq\1 \frac{1}{\eta_0\!}\,, \; \inf_{x \in \Sx} \Im \2 u_x \1\geq\1 \frac{\eta_0^{\13}\!}{(\12+\norm{S})^2\msp{-6}}
\2}\,,
}
equipped with the metric 
\bels{}{
d(u,\1w) \,:=
\sup_{x \in \Sx} d_{\Cp}(u_x ,\1w_x)
\,,
\qquad u,w \in \BB_{\eta_0}
\,.
}
On $ \BB_{\eta_0} $ the metric $ d $ is equivalent to the metric induced by the uniform norm \eqref{def of BB-norm} of $ \BB $. 
Since  $\BB_{\eta_0}$ is closed in the uniform norm metric it is a complete metric space with respect to $ d $.  

\begin{lemma}[$ \Phi $ is contraction]
\label{lmm:Phi as contraction} 
For any $z \in \Cp_{\eta_0}$, the function $\Phi(\genarg;z)$ maps $\BB_{\eta_0}$ into itself and satisfies
\bels{contraction property}{
\sup_{x\in \Sx}
D\Bigl( \bigl(\Phi(u\1;z)\bigr)_x\2,\2 \bigl(\Phi(w\1;z)\bigr)_x \Bigr)
\,\leq\, 
\Big(\2 1+ \frac{\eta_{\10}^{\12}\!}{\norm{S}} \2\Big)^{\!-2}
\sup_{x \in \Sx} D( u_x, w_x)
\,,
}
for any $ u,w \in \BB_{\eta_0} $, and $D$ defined in \eqref{cDefinition}. 
\end{lemma}
\begin{Proof}
First we show that $\BB_{\eta_0}$ is mapped to itself. For this let $u \in \BB_{\eta_0}$
be arbitrary. We start with the upper bound 
\[
\abs{\Phi(u\1;z)} \,\leq\, \frac{1}{\Im (z + a+ Su)} \,\leq\, \frac{1}{\Im \2 z} \,\leq\, \frac{1}{\eta_0}
\,,
\]
where in the second inequality we employed the non-negativity property of $S$ and that $\Im \1 u \geq 0$. S
Since $ \abs{z} \leq \eta_0^{-1}$ and $ \eta_0 \leq 1/\norm{a} $, we also find a lower bound, 
\[
|\Phi(u;z)| 
\,\ge\, \frac{1}{\abs{z}+\abs{a}+\abs{Su}} 
\,\ge\, \frac{1}{\eta_0^{-1}+\norm{a}+\norm{S}\eta_0^{-1}}
\,\ge\,  
\frac{\eta_0\!}{\,2+\norm{S}}
\,.
\]
Now we use this as an input to establish the lower bound on the imaginary part,
\[
\Im \, \Phi(u;z) \,=\, \frac{\2\Im (z+a+Su)\1}{\abs{z+a+Su}^2} \,\geq\, |\Phi(u;z)|^2 \2\Im \2 z \,\geq\, \frac{\eta_0^{\13}\!}{(\12+\norm{S})^2\msp{-6}}
\,.
\]

We are left with establishing the inequality in \eqref{contraction property}. For that we use the three properties of $D$ in Lemma~\ref{lmm:Properties of hyperbolic metric}. By Property \ref{cProp1}, the function $D$ is invariant under the isometries $\zeta \mapsto -1/\zeta$ and $\zeta \mapsto \zeta- a_x-\Re\1 z$ of $\Cp$.
Therefore for any $u,w\in \BB_{\eta_0} $ and $ x \in \Sx$:
\bels{c calculation 1}{
D\Bigl(\2 \big(\Phi(u\1;z)\big)_x\2,\2 \big(\Phi(w\1;z)\big)_x \Bigr)
\,&=\, 
D\big(\2z+a_x+(Su)_x\1,\2 z+a_x+(Sw)_x \big)
\\
&=\,
D\big(\2\cI\2 \Im\2 z+(Su)_x\1,\2 \cI\2\Im\2 z+(Sw)_x \big)
\,.
}
In case the non-negative functional $S_x\in \BB^*$, defined through $S_x(u):=(Su)_x$, vanishes identically, the expression in \eqref{c calculation 1} vanishes as well. Thus we may assume that 
$S_x\neq 0$. 
In view of Property \ref{cProp2} we estimate
\bea{
&D\big(\2\cI\2 \Im\2 z+(Su)_x\1,\2 \cI\2\Im\2 z+(Sw)_x \big)
\\
&\leq\,
\Big(\2 1+ \frac{\Im \2 z}{\Im \1(S u)_x} \2\Big)^{\!-1} 
\Big(\2 1+ \frac{\Im \2 z}{\Im \1(S w)_x\!} \2\Big)^{\!-1} 
D\big((Su)_x\1,(Sw)_x \big)
\,.
}
Plugging this back into \eqref{c calculation 1} and recalling $ \Im\1 z \ge \eta_0$ and $ \norm{Sw} \leq \norm{S}\eta_0^{-1} $, for $z \in \Cp_{\eta_0}$ and $ w \in \BB_{\eta_0} $, respectively, we obtain
\[
D\Big( \big(\Phi(u\1;z)\big)_x\2,\2 \big(\Phi(w\1;z)\big)_x \Big)
\,\leq\, 
\Big(\2 1+ \frac{\eta_0^{\12}\!}{\norm{S}} \2\Big)^{\!-2}
\; D \big((Su)_x\1,\2(Sw)_x \big)
\,.
\]
Using Property \ref{cProp3} in Lemma~\ref{lmm:Properties of hyperbolic metric} we find 
\[
D \big((Su)_x\1,\2(Sw)_x \big) \,\leq \sup_{x \in\Sx}D(u_x,\1w_x)
\,.
\]
This finishes the proof of the lemma.
\end{Proof}

Lemma~\ref{lmm:Phi as contraction} shows that the sequence of iterates $ (u^{(n)})_{n=0}^\infty $, with $ u^{(n+1)}:=\Phi(u^{(n)};z) $, is Cauchy for any initial function $u^{(0)} \in \BB_{\eta_0}$ and any $z \in \Cp_{\eta_0}$.
Therefore, $(u^{(n)})_{n \in \N}$ converges to the unique fixed point $ m = m(z) \in \BB_{\eta_0} $ of $ \Phi(\genarg;z)$.  We have therefore shown existence and uniqueness of \eqref{QVE for fixed z} for any given $z \in \Cp_{\eta_0}$ and thus, since $\eta_0$ was arbitrary, even for all $z \in \Cp$.

\section{Stieltjes transform representation}

In order to show that $m_x$ can be represented as a Stieltjes transform 
(cf. \eqref{m as stieltjes transform}), we will first prove that $m_x$ is a holomorphic function on $\Cp$. 
We can use the same argument as above on a space of function which are also $ z $ dependent. 
Namely, we consider the complete metric space, obtained by equipping the set
\bels{holomorphic function space}{
\mathfrak{B}_{\msp{-2}\eta_0} 
\,:=\,
\setb{\2 \mathfrak{u}:\Cp_{\eta_0} \to \BB_{\eta_0} :\, \mathfrak{u} \text{ is holomorphic}}
\,,
}
of $ \BB_{\eta_0}$-valued functions $\mathfrak{u}$ on $\Cp_{\eta_0}$, with the metric
\bels{}{
d_{\eta_0}(\mathfrak{u},\mathfrak{w}) 
\,:= 
\sup_{z \in \Cp_{\eta_0}} 
d(\mathfrak{u}(z), \mathfrak{w}(z))
\,,
\qquad \mathfrak{u},\mathfrak{w} \in \mathfrak{B}_{\msp{-2}\eta_0}\,.
}
Here the holomorphicity of $ \mathfrak{u} $ means that the map $z \mapsto \phi(\mathfrak{u}(z))$ is holomorphic on $\Cp_{\eta_0}$ for any element $ \phi $ in the dual space of $\BB$.
%
Since the constant $(1+\eta_0^2/\norm{S})^{-2}$ in \eqref{contraction property} only depends on $\eta_0$, but not on $z$, we see that the function $ \mathfrak{u} \mapsto \Phi(\mathfrak{u}) $, defined by
\bels{}{
(\Phi(\mathfrak{u}))(z) \,:=\, \Phi(\1\mathfrak{u}(z)\1;z)\,, \qquad \forall \; \mathfrak{u} \in  \mathfrak{B}_{\msp{-2}\eta_0} ,
}
inherits the contraction property from $\Phi(\genarg;z)$. Thus the iterates $ \mathfrak{u}^{(n)}:=\Phi^n(\mathfrak{u}^{(0)}) $ for any initial function $\mathfrak{u}^{(0)} \in \mathfrak{B}_{\msp{-2}\eta_0}$ converge to the unique holomorphic function $ m: \Cp_{\eta_0}\to \BB_{\eta_0}$, which satisfies $ m(z)=(\Phi(m))(z)$ for all $ z \in \Cp_{\eta_0}$. 
Since $ \eta_0>0 $ was arbitrary and by the uniqueness of the solution on $ \Cp_{\eta_0} $, we see that there is a holomorphic function $ m:\Cp \to \BB_+$ which satisfies $ m(z) = (\Phi(m))(z)=\Phi(m(z);z) $, for all $z \in \Cp$. This function $ z \mapsto m(z)  $ is the unique holomorphic solution of the QVE.

Now we show the representation \eqref{m as stieltjes transform} for $ m(z) $.
We use that  a holomorphic function $ \phi: \Cp \to \Cp $ on the complex upper half plane $ \Cp $ is a Stieltjes transform of a probability measure on the real line if and only if $\abs{\1\cI\1\eta \2 \phi(\cI \1 \eta)+1\1} \to 0$ as $ \eta \to \infty$ (cf. Theorem 3.5 in \cite{Garnett-BA2007}).
In order to see that
\bels{condition for S-transform representation}{
\lim_{\eta \to \infty}\sup_x\,\absb{\2\cI \1 \eta \,m_x(\cI\1\eta)+1\1} \,=\,0\,, 
}
we write the QVE in the form
\[
z\1m_x(z)+1 \,=\, 
-\2m_x(z)\,(a+Sm(z))_x
\,.
\]
We bound the right hand side by taking the uniform norms,
\[
\abs{\1z\1m_x(z)+1\1} \,\leq\, \norm{a}\norm{m(z)}+\norm{S}\norm{m(z)}^2
\,.
\]
We continue by using $ \Im\,m(z) \ge 0 $ and the fact that $ S $ preserves positivity:
\bels{m trivial bound}{
\abs{m(z)} \,=\, \frac{1}{\abs{z+a+Sm(z)}}\,\leq\, \frac{1}{\Im(z + a + Sm(z))} \,\leq \, \frac{1}{\2\Im \2z\2}\,, \qquad 
\forall\; z \in \Cp\,.
}
Choosing $ z = \cI\1\eta $, we get 
\[ 
\abs{\2\cI \1 \eta \,m(\cI\1\eta)+1\1} \,\leq\,  \norm{a}\1\eta^{-1}\!+\norm{S}\1\eta^{-2} 
\,,
\] 
and hence \eqref{condition for S-transform representation} holds true. 
This completes the proof of the Stieltjes transform representation \eqref{m as stieltjes transform}.

As the next step we show that the measures $ v_x $, $ x \in \Sx $, in \eqref{m as stieltjes transform} are supported on an interval $ [-\Sigma,\Sigma\1] $, where $ \Sigma = \norm{a} +2\1\norm{S}^{1/2} $. 
We start by extending these measures to functions on the complex  upper-half plane.

\NDefinition{Extended generating density}{ 
Let $ m $ be the solution of the QVE. Then we define 
\bels{def of v_x(z)}{
v_x(z) \,:=\, \Im \, m_x(z)\,, \qquad \forall\; x\in\Sx, \, z \in \Cp
\,.
}
The union of the supports of the generating measures \eqref{m as stieltjes transform} on the real line is denoted by:
\bels{def of supp v}{
\supp v \,:=\, \bigcup_{x\in\Sx} \supp v_x|_\R 
\,.
}
}
This extension is consistent with the generating measure $ v_x $ appearing in \eqref{m as stieltjes transform} since $ v_x(z) $, $ z \in \Cp$, is obtained by regularizing the generating measure with the Cauchy-density at the scale $ \eta > 0 $. Indeed, \eqref{def of v_x(z)} is equivalent to
\bels{v_x(z) as eta-regularization of measure v_x}{
\quad
v_x(\tau +\cI\1\eta) \,=\, 
\int_{-\infty}^\infty \frac{1}{\eta}\2\Pi\Bigl(\frac{\tau-\omega}{\eta}\Bigr)\,v_x(\dif\omega)
\,,\qquad
\Pi(\lambda) := \frac{1}{\pi}\frac{1}{1+\lambda^2}
\,,
} 
for any $ \tau \in \R $ and $ \eta > 0 $.

We will now show that the support of the generating measure $ v $ lies inside an interval with endpoints $ \pm\1\Sigma $, with $ \Sigma = \norm{a} +2\1\norm{S}^{1/2} $. 
To this end, suppose that 
\bels{assumption on m bound}{
\norm{m(z)} <\frac{\abs{z}-\norm{a}}{2\1\norm{S}}
\;,\quad
\text{for some}\quad
\abs{z} \,>\, \Sigma
\,,
}
where we have used $ \norm{S} > 0 $.
Feeding \eqref{assumption on m bound} into the QVE we obtain a slightly better bound:
\[
\norm{m(z)} 
\,\leq\, 
\frac{1}{\2\abs{z}-\norm{a}-\norm{S}\norm{m(z)}}
\,\leq\, 
\frac{2}{\abs{z}-\norm{a}\!} 
\,.
\]
Denoting 
\[ 
\D_\eps \,:=\, \setB{z\in\Cp: \abs{z} \ge \norm{a} + 2\1\norm{S}^{1/2}(\11\msp{-1}+\msp{-1}\eps\1)} 
\,,
\] 
for an arbitrary $ \eps \in (0,1/4) $, we have shown that the range of the restriction of the norm function $ \norm{m} $ to ${\D_\eps} $ is a union of two disjoint sets, i.e., 
\bels{norm-function on Deps}{
z \,\mapsto\, \norm{S}^{1/2}\norm{m(z)} \,:\,\DD_\eps \,\to\, \bigl[\20\1,(\11+\eps\1)^{-1}\bigr] \2\cup\2 \bigl[\21\1+\2\eps\1,\infty\2\bigr) 
\,.
}  
From the Stieltjes transform representation \eqref{m as stieltjes transform} we see that  \eqref{norm-function on Deps} is a continuous function.
The bound \eqref{condition for S-transform representation} implies $ \norm{m(\cI\1\eta)} \leq (1+\eps)^{-1} $ for sufficiently large $ \eta > 0 $. 
For large $ \eta $ we also have $ \cI\1\eta \in \D_\eps $. 
As  $\D_\eps $ is  a connected set, the continuity of \eqref{norm-function on Deps} implies that for any $ \eps > 0 $
\bels{upper bound on m outside half-ball}{
\norm{S}^{1/2}\norm{m(z)} \leq (1+\eps)^{-1} \,,
\quad\text{when}\quad\abs{z} \ge \Sigma +2\1\norm{S}^{1/2}\,\eps
\,.
}

Now we take the imaginary part of the QVE to get
\bels{imaginary part of QVE 1}{
\frac{v(z)}{|m(z)|^2} \,=\, -\,\Im\,\frac{1}{m(z)} \,=\, \Im\, z + Sv(z)
\,,
}
where $ v(z) $ is from \eqref{def of v_x(z)}.
Taking the norms in this formula and rearranging it, we obtain 
\bels{final support bound}{
\Bigl(\21-\bigl(\2\norm{S}^{1/2}\norm{m(z)}\bigr)^2\Bigr)\,\norm{v(z)} \,\leq\, \norm{m(z)}^2\,\Im\, z
\,.
}
Consider $z := \tau + \cI\1\eta $, with $ \abs{\tau}> \Sigma $ and $ \eta > 0 $.
Then the coefficient in front of $ \norm{v(z)} $ is larger than $ (1-(1+\eps)^{-1}) > 0 $, with $ \eps := (\abs{\tau}-\Sigma)/(2\1\norm{S}^{1/2}) > 0 $. In particular, this bound is uniform in $ \eta $. 
We estimate $ \norm{m}$ on the right hand side of \eqref{final support bound} by \eqref{upper bound on m outside half-ball}. 
Thus we see that $ v(\tau+\cI\1\eta) \to 0 $ by taking the limit $ \eta \to 0 $ locally uniformly for $ \abs{\tau} > \Sigma $.

\section{Operator $ F $ and structural $\Lp{2}$-bound}
\label{sec:Operator F and structural L2-bound}

In this section we finish the proof of Theorem \ref{thr:Existence and uniqueness} by considering the remaining the special case $ a = 0 $.
First we note that the real and imaginary parts of the solution $ m $ of the QVE are odd and even functions of $ \Re\,z $ with fixed $ \Im\,z$, respectively when $ a = 0 $, i.e., 
\bels{m symmetry when a=0}{
m(-\2\overline{z\1}\1) \,=\, -\overline{\,m(z)} \,, \qquad \forall\; z \in \Cp
\,.
}
Combining this with \eqref{def of supp v} we obtain the symmetry of the generating measure. 

The proof of the upper bound \eqref{L2 bound on m when a=0} on the $\Lp{2} $-norm of $ m(z) $ relies on the analysis of the following symmetric positivity preserving operator $F(z)$, generated by $ m(z) $. 

\begin{definition}[Operator $ F $]
\label{def:Operator F in the general setting}
The operator $ F(z) : \BB \to \BB $ for $ z \in\Cp$, is  defined by
\bels{F operator}{
F(z)\1w := \abs{m(z)}\2S(\1\abs{m(z)}\2w\1)
\,,
\qquad w \in \BB\,,
}
where $ m(z) $ is the solution of the QVE at $ z $.
\end{definition}
The operator $ F(z) $ will play a central role in the upcoming analysis. 
In particular, using $ F(z) $ we prove the structural $ \Lp{2}$-bound for the solution. 

\begin{lemma}[Structural $ \Lp{2}$-bound]
\label{lmm:Structural L2-bound}
Assuming {\bf A1}, we have 
\bels{L2-bound for a}{
\norm{m(z)}_2 \,\leq\, \frac{2}{\dist(\1z,\sett{a_x:x\in\Sx})}
\,,\qquad \forall\,z\in \Cp
\,.
}
\end{lemma}

\begin{Proof}
We start by writing the QVE in the form
\bels{}{
-\1(z+a\1)\1 m(z) \,=\, 1 + m(z) \1 Sm(z)\,.
}
Taking the $\Lp{2}$-norm on both sides yields 
\bels{L2 bound on m step 1}{
\norm{m(z)}_2 
\,&\leq\, 
\bigl(\21 + \norm{m(z)\2 Sm(z)}_2\2\bigr)\,\norm{\1(z+a)^{-1}}
\,\leq\,
\frac{\11+\norm{F(z)}_{2\to2}\!}{\dist(\1z,\sett{a_x:x\in\Sx})}
\,.
}
Here the last bound follows by writing $ \abs{\1m(z)\1Sm(z)} \leq \abs{m(z)}S\abs{m(z)} = F(z)\1e $, where $ e \in \BB $ stands for the constant function equal to one, and then estimating:   
\bels{mSm to L2-norm of F}{ 
\norm{m(z)\1Sm(z)}_2 \,=\, \norm{F(z)\1e\1}_2 \,\leq\, \norm{F(z)}_{\Lp{2}\to\Lp{2}}
\,. 
}
The bound \eqref{L2-bound for a} now follows by bounding $ F(z) $ as an operator on $ \Lp{2} $. In fact, we now show that 
\bels{F general uniform bound}{
\norm{F(z)}_{\Lp{2} \to \Lp{2}} \,<\, 1
\,, 
\qquad \forall\; z \in \Cp
\,.
}

The operator $ S $ is bounded on $ \BB $. Therefore $ \norm{S}_{\Lp{\infty}\to\Lp{\infty}} \leq \norm{S} < \infty $. Since $ S $ is symmetric we have $ \norm{S}_{\Lp{1}\to\Lp{1}} = \norm{S}_{\Lp{\infty}\to\Lp{\infty}} $. 
Using the Riesz-Thorin interpolation theorem we hence see that $ \norm{S}_{\Lp{p}\to\Lp{p}} \leq \norm{S}$, for every $ p \in [1,\infty] $.

For each $ z \in \Cp$ the operator $F(z)$ is also bounded on $\Lp{q} $, as $ \abs{m(z)} $ is trivially bounded by $(\Im\,z)^{-1} $ (cf. \eqref{m trivial bound}). 
Furthermore, from the Stieltjes transform representation \eqref{m as stieltjes transform} it follows that $ m_x(z) $ is also  bounded away from zero:
\bels{trivial lower bound on v}{
\Im \,m_x(z) \,\geq\, \frac{1}{\pi}\frac{\Im\, z}{(\1\Sigma +\abs{z})^2}\,,
\qquad \forall\; x \in \Sx
\,.
}

The  estimate \eqref{F general uniform bound} is obtained by considering the imaginary part \eqref{imaginary part of QVE 1} of the QVE. 
Rewriting this equation in terms of $ F = F(z) $ we get
\bels{v equation}{
\frac{v}{\abs{m}} \,=\,  \abs{m}\,\Im \2 z  \,+\, F \frac{v}{\abs{m}}
\,.
}
In order to avoid excess clutter we have suppressed the dependence of $ z $ in our notation.
The trivial lower bound \eqref{trivial lower bound on v} on $v(z)$ and the trivial upper bound $|m(z)|\leq (\Im\2z)^{-1}$ imply that there is a scalar function $\eps:\Cp \to (0,1)$, such that
\bels{}{
\qquad F \frac{v}{\abs{m}}  \,\leq\, (\11-\eps\1)\2\frac{v}{\abs{m}}
\,,\qquad
\eps \,:=\, (\1\Im \2z)\2\inf_x \frac{\abs{\1m_x}^2}{\!v_x} \in (\10,1\1] \,. 
}
The fact that $\eps \in (\10,1\1]$ follows from \eqref{v equation}, the strict pointwise positivity of $v$ and the positivity preserving property of $F$.
If $ \eps = 1 $ we have nothing to show since $ F = 0 $ in this case. 
If $ \eps <1 $, then we apply Lemma~\ref{lmm:Subcontraction} below with the choices,
\[ 
T \,:=\, 
\frac{F}{1-\eps}\,,
\qquad\text{and}\qquad
h \,:=\,\frac{v}{\abs{m}} \,\gtrsim\,\frac{(\Im\,z)^2}{1+\abs{z}^2} 
\;,
\] 
to conclude $\norm{F}_{\Lp{2} \to \Lp{2}} < 1$. 
\end{Proof}

\NLemma{Subcontraction}{
Let $ T $ be a bounded symmetric operator on $\Lp{2} $ that preserves non-negative functions, i.e., if $ u \ge 0 $ almost everywhere, then also $ Tu \ge 0 $ almost everywhere.
If there exists an almost everywhere positive function $ h \in \Lp{2} $, such that almost everywhere $ Th \leq h $, then $\norm{T}_{\Lp{2} \to \Lp{2}} \leq 1$.
}

The proof of Lemma~\ref{lmm:Subcontraction} is postponed to Appendix \ref{sec:Proofs of auxiliary results in Chapter:Existence and uniqueness}

\chapterl{Properties of solution}

In this chapter we prove various technical estimates for the solution $m$ of the QVE and the associated operator $ F $ (cf. \eqref{def:Operator F in the general setting}). In the second half of the chapter we start analyzing the stability of the QVE under small perturbations.
For the stability analysis, we introduce the concept of the (spectral) gap of an operator.

\NDefinition{Spectral gap}{
Let $ T:\Lp{2} \to \Lp{2} $ be a compact self-adjoint operator. The spectral gap $ \mrm{Gap}(T) $ is the difference between the two largest eigenvalues of $ \abs{T} $. If $ \norm{T}_{\Lp{2}\to\Lp{2}}$ is a degenerate eigenvalue of $ \abs{T} $ then $ \mrm{Gap}(T) = 0 $. 
}

We will frequently use comparison relations $ \sim $, $ \lesssim $ in the sequel that depend on a certain set of model parameters (c.f. Convention~\ref{conv:Comparison relations, model parameters and constants}).
This set may be different in various lemmas and propositions. In order to avoid constantly listing them we extend Convention~\ref{conv:Comparison relations, model parameters and constants} as follows: 

\begin{convention}[Standard model parameters]
\label{conv:Standard model parameters}
The norm $ \norm{a} $ is always considered a model parameter. 
If the property {\bf A2} of $ S $ is assumed in some statement, then the associated constant $ \norm{S}_{\Lp{2}\to\BB} $ is automatically a model parameter.
Similarly, if {\bf A3} is assumed, then $ \rho $ and $ L $ are considered  model parameters.
Any additional model parameters will be declared explicitly.
Naturally, inside a proof of a statement the comparison relations depend on the model parameters of that  statement.
\end{convention}

\section{Relations between components of $ m $ and $ F $}

The following proposition collects the most important estimates in the special case when the solution is uniformly bounded. 

\begin{proposition}[Estimates when solution is bounded]
\label{prp:Estimates when solution is bounded}
Suppose $ S $ satisfies \emph{\bf A1-3}. 
Additionally, assume that for some $ I \subseteq \R $, and $ \Phi < \infty $ the uniform bound 
\[ 
\nnorm{m}_I \leq \Phi 
\,,
\] 
applies.
Then, considering $ \Phi $ an additional model parameter, the following estimates apply for every $ z \in \Cp $, with $ \Re\,z \in I $:
\begin{itemize}
\titem{i} The solution $m$ of the QVE satisfies the bounds
\bels{m scaling for regular S}{
|m_x(z)| \,\sim\, \frac{1}{1+|z|} \,, \qquad \forall \; x \in \Sx\,.
}
\titem{ii} The imaginary part is comparable to its average, i.e.
\bels{v_x sim avg-v}{
v_x(z) \,\sim\, \la v(z) \ra \,, \qquad \forall \; x \in \Sx\,.
}
\titem{iii} The largest eigenvalue $ \lambda(z) $ of $ F(z) $ is single, and satisfies $ \lambda(z) \leq 1 $, and
\bels{scaling of L2-norm of F in abs-z}{
\lambda(z) \,=\, \norm{F(z)}_{\Lp{2} \to \Lp{2}} 
\,\sim\, 
\frac{1}{1+|z|^2}
\,.
} 
\titem{iv} The operator $F(z)$ has a uniform spectral gap, i.e., 
\bels{scaling of the gap for bounded m}{
\mrm{Gap}(F(z)) \,\sim\, \norm{F(z)}_{\Lp{2}\to\Lp{2}}
\,.
}
\titem{v} The unique eigenvector $ f(z) \in \BB $, satisfying 
\bels{defining properties of vector f(z)}{
F(z)f(z) = \lambda(z)f(z)\,,\qquad
f_x(z) \ge 0\,,
\quad\text{and}\quad
\norm{f(z)}_2=1
\,,
}
is comparable to $1$, i.e.
\bels{regular S: f comparable to 1 comp.wise}{
f_x(z) \,\sim\, 1 \,, \qquad \forall \; x \in \Sx\,.
}
\end{itemize}
\end{proposition}

For a complete proof of Proposition~\ref{prp:Estimates when solution is bounded} (cf. p. \pageref{proof of prp:Estimates when solution is bounded}) we first prove various auxiliary results, under the standing assumption in this chapter:
\begin{itemize}
\item \emph{$ S $ satisfies} {\bf A1-3}.
\end{itemize}

We start by pointing out a few simple properties of $ S $ that we need in the following.
The smoothing condition {\bf A2} implies that for every $x \in \Sx$ the linear functional $S_x: \Lp{2} \to \R,\2 w \mapsto (Sw)_x$ is bounded. Hence, the row-function $ y \mapsto S_{x y} $ is in $ \Lp{2}$. The family of functions satisfies  $\sup_x \norm{S_x}_2=\norm{S}_{\Lp{2} \to \BB}$.  The bound \eqref{2 to infty bound on S} implies that $S$ is a Hilbert-Schmidt operator.

The uniform primitivity condition {\bf A3} guarantees that norms of the row functions $ S_x $ as well as various operator norms of $ S $ are comparable to one. Indeed, letting $ x \in \Sx $ be fixed and choosing the constant function $ u =1 $ in \eqref{uniform primitivity}, we obtain 
\[
\rho 
\;\leq
\int (S^L)_{xy}\Px(\dif y) 
\;\leq\,
\biggl(\int S_{xu}\2\Px(\dif u)\biggr)
\biggl(\sup_t \int (S^{L-1})_{ty}\2\Px(\dif y)\biggr)
\,\leq\,
\norm{S^{L-1}}\1\avg{\1S_x}
\,.
\]
Since $ \norm{S^{L-1}} \leq \norm{S}^{L-1}$ this  yields the first inequality of
\bels{averaged row is comparable to BB-norm of S}{ 
\rho\2\norm{S}^{-(L-1)}
\leq\,
\avg{\1S_x} 
\,\leq\,\norm{S}
\,,
\qquad x \in \Sx
\,.
}
The last bound is trivial since $ \norm{S} = \sup_x \avg{\1S_x} $.
By Riesz-Thorin interpolation theorem (cf. proof of Lemma \ref{lmm:Structural L2-bound}) we have $ \norm{S}_{\Lp{p}\to\Lp{p}} \leq \norm{S} $. On the other hand, letting $ S $ act on the constant $ 1 $ function, we have  
\[
\norm{S}_{\Lp{p}\to\Lp{p}} \,\ge\,\inf_x\, \avg{\1S_x}
\,.
\]
Combining this with \eqref{averaged row is comparable to BB-norm of S}, the trivial bound  $ \norm{S} \leq \norm{S}_{\Lp{2}\to\BB} $, and the fact that $ \norm{S}_{\Lp{2}\to\BB} $ is a model parameter (cf. Convention~\ref{conv:Standard model parameters}), we thus conclude 
\bels{operator norms of S are comparable}{
\avg{\1S_x} \2\sim\2 1\,,\qquad
\norm{S} \2\sim\2 1
\,,\qquad\text{and}\qquad
\norm{S}_{\Lp{p}\to\Lp{p}} \sim\, 1
\,,\quad
p \in [\21\1,\infty\1]
\,.
}

The  following lemma shows that a component $ \abs{m_x(z)} $ may diverge only when $ \norm{m(z)}_2 = \infty $ or $ \avg{v(z)} = 0 $.
Furthermore, the lemma implies that if a component $ \abs{m_x(z)} $, for some  $ x \in \Sx $, approaches zero while $ z $ stays bounded, then another component $ \abs{m_y(z)} $ will always diverge at the same time.

\begin{lemma}[Constraints on solution] 
\label{lmm:Constraints on solution}
If $ S $ satisfies \emph{\bf A1-3}, then:
\begin{itemize}
\titem{i} 
The solution $ m $ of the QVE satisfies for every $ x \in \Sx $ and $ z \in \Cp $: 
\bels{unif bound of m}{
&\min\setbb{\!\frac{1}{1+\abs{z\1}}\,,\,
\inf_y \1\abs{\1z-a_y}
+ \frac{1}{\1\norm{m(z)}_2\!}} 
\\
&\msp{80}\lesssim\;
\absb{\1m_x(z)} 
\;\lesssim\;
\min\setbb{\!\frac{1}{\inf_y \abs{\1m_y(z)}^{2L-2}\2\avg{\1v(z)}}\,,\frac{1}{\dist(z, \supp v \1)}\!}
\,.
}
\titem{ii} 
The imaginary part, $v_x(z) $ is comparable to its average, such that for every $ x \in \Sx $ and $ z \in \Cp $ with $ |z|\leq 2\1\Sigma $:
\bels{v compared to avg-v}{
\inf_y \absb{m_y(z)}^{2L}
\;\lesssim\; 
\frac{v_x(z)}{\avg{v(z)}} 
\;\lesssim\; 
\biggl(1+\frac{1}{\inf_y \abs{m_y(z)}}\biggr)^{\!2}\norm{m(z)}^4
\,.
}	
For $\abs{z}\ge 2\1\Sigma $ the function $v$ satisfies $v_x(z)\sim \avg{v(z)}$. 
\end{itemize}
\end{lemma}

These bounds simplify considerably when $ m = m(z) $ is uniformly bounded for every $ z $ (cf. Proposition~\ref{prp:Estimates when solution is bounded}). 

\begin{Proof} 
We start by proving the lower bound on $ \abs{m} $.  
This is done by establishing an upper bound on $ 1/\abs{m} $. 
Using the QVE we find
\bels{lower bound on m_x}{
\frac{1}{\abs{m}\!} 
\,=\,
\abs{\1z+a+Sm} 
\,\leq\, 
\abs{z}+\norm{a}+ \norm{S}_{\Lp{2}\to\BB}\norm{m}_2 
\,\lesssim\,
1+\abs{z} + \norm{m}_2
\,.
}
Taking the reciprocal on both sides yields $ \abs{m} \gtrsim \min\setb{\2(1+\abs{z})^{-1}\!, \norm{m_2}^{-1}} $. 
Combining the $ \Lp{2}$-norm with \eqref{L2-bound for a} yields the lower bound in \eqref{unif bound of m}.

Now we will prove the upper bound on $\abs{m}$.
To this end, recall that
\[
m_x(z) = \frac{1}{\pi}\int_\R \frac{v_x(\dif\tau)\!}{\tau-z}
\,,
\]
where $ v_x/\pi $ is a probability measure. 
Bounding the denominator from below by $ \dist(z,\supp v) $, with $\supp v = \cup_x \supp v_x$,  we obtain  one of the upper bounds of \eqref{unif bound of m}:
\[
\abs{m_x(z)} \,\leq\, \frac{1}{\dist(z,\supp v\1) }\,.
\]

For the derivation of the second upper bound we rely on the positivity of the imaginary part of $ m $:
\bels{m upper bound: start}{
\abs{m} \,=\, \frac{1}{\abs{\2\Im\,(z+a+ Sm)\1}} \,\leq\, \frac{1}{Sv}
\,.
}
In order to continue we will now bound $ Sv $ from below. 
This is achieved by estimating $ v $ from below by $ \avg{v} $. Indeed, writing the imaginary part of the QVE, as  
\[
\frac{v}{\abs{m}^2\msp{-6}} \;=\, -\2\Im\,\frac{1}{m} \,=\, \Im\,z + Sv\,,
\]
and ignoring $ \Im\,z > 0 $, yields
\bels{v lower bounded by Sv}{
v\,\ge\, \abs{m}^2Sv \,\ge\, \phi^2Sv\,,
}
where we introduced the abbreviation
\[
\phi\,:=\,\inf_x\2 \abs{\1m_x}\,.
\]
Now we make use of the uniform primitivity {\bf A3} of $ S $ and of \eqref{v lower bounded by Sv}. In this way we get the lower bound on $Sv$,
\[
Sv\,\ge\,\phi^2S^2v\,\,\ge\, \dots \,\ge\, \phi^{2L-2}\,S^L v\,\ge\, \phi^{2L-2}\rho\,\avg{v},
\]
Plugging this back into \eqref{m upper bound: start} finishes the proof of the upper bound on $\abs{m}$.

We continue by showing the claim concerning $ v/\avg{v} $. We start with the lower bound. 
We use \eqref{v lower bounded by Sv} in an iterative fashion and employ assumption {\bf A3},
\bels{v lower bounded by avg-v after iteration}{
v\,\ge\, \phi^2Sv \,\ge\, \dots\,\ge\, \phi^{2L}\,S^L v \,\ge\, \phi^{2L}\rho\,\avg{v} 
\,.
}
This proves the lower bound  $ v/\avg{v} \gtrsim \phi^{2L} $. 

In order to derive upper bounds for the ratio $ v/\avg{v} $, we first write
\bels{v from QVE}{
v \,=\, \abs{m}^2\1(\1\Im\,z + Sv\1) \,\leq\, \norm{m}^2\,(\1\Im\,z+Sv\1)
\,.
}
We will now bound $ \Im\,z $ and $ Sv $ in terms of $ \avg{v} $. We start with $\Im\,z $. By dropping the term $ Sv$ from \eqref{v from QVE}, and estimating $ \abs{m} \ge \phi$, we get $ v \ge \phi^2\,\Im\,z $.
Averaging this yields 
\bels{Im z bounded by avg-v}{
\Im\,z \,\leq\, \frac{\avg{v}}{\phi^{\12}\msp{-5}}
\;.
}
In order to bound $ Sv $, we apply $ S$ on both sides of \eqref{v from QVE}, and use the bound on $ \Im\,z $, to get
\bels{bound for Sv}{
Sv \,\leq\,
\biggl(
\frac{\avg{v}}{\2\phi^{\12}\!}+ 
S^2v
\biggr) \,\norm{m}^2
\,.
}
The expression involving $ S^2 $ is useful, as we may now estimate the kernel $ (S^2)_{xy} $ uniformly:
\bels{uniform bound on S^2 kernel}{ 
(S^2)_{xy} 
\,\leq\, \avg{S_x,S_y} 
\,\leq\, \norm{S_x}_2\norm{S_y}_2 
\,\leq\, \sup_x \norm{S_x}_2^2 
\,=\, \norm{S}_{\Lp{2}\to \BB}^2 \,\sim\,1
\,.
}
In particular, $ S^2v \leq \norm{S}_{\Lp{2}\to\BB}^2 \avg{v} \sim \avg{v} $, and thus
\[ 
Sv \,\lesssim\,\Bigl(1+\frac{1}{\2\phi^2\!}\,\Bigr)\norm{m}^2\,\avg{v}
\,.
\]
With this and \eqref{Im z bounded by avg-v} plugged back into \eqref{v from QVE} we get the upper bound of \eqref{v compared to avg-v}: 
\[
v 
\,\lesssim\, 
\Bigl(1+\frac{1}{\phi}\,\Bigr)^{\!2}\norm{m}^4\,\avg{v}\,.
\]
Here we have also used the lower bound $ \norm{m(z)} \gtrsim 1 $ to replace $ \norm{m}^2 $ by $ \norm{m}^4 $ in the regime $ \abs{z} \leq 2\1\Sigma $, where $ \Sigma = \norm{a} + 2\1\norm{S}^{1/2} \sim 1  $ by \eqref{operator norms of S are comparable}.
The lower bound on $ \norm{m} $ follows directly from the QVE and $ \norm{S} \sim 1 $:
\bea{
1 \,&=\, \abs{\1(z + a + Sm)\,m\1} 
\,\lesssim\,
\bigl(\2\abs{z}+\norm{a} + \norm{S}\1\norm{m}\1\bigr)\,\norm{m}
\,.
}
On the other hand, if $|z|\ge 2\2\Sigma $, then $ v(z) \sim \avg{v(z)} $ holds because  $v_x(z)$ is the harmonic extension \eqref{m as stieltjes transform} of the measure $v_x(\dif \tau)$ which is supported inside the interval with endpoints $ \pm\1\Sigma $.
\end{Proof}

Since the solution $ m(z) $ for $ z \in \Cp$ of the QVE is bounded by the trivial bound (cf. \eqref{m trivial bound}), the operator $ F(z) $ introduced in Definition~\ref{def:Operator F in the general setting} is  a Hilbert-Schmidt operator. Consistent with the notation for $ S $ we write $ F_{xy}(z)$ for the symmetric non-negative measurable kernel representing this operator. The largest eigenvalue and the corresponding eigenvector of $ F(z) $ will play a key role when we analyze the sensitivity of $ m(z) $ to changes in $ z $, or more generally, to any perturbations of the QVE. 
The following lemma provides an exact formula for this eigenvalue.

\begin{lemma}[Operator $ F $]
\label{lmm:Operator F}
Assume that $ S $ satisfies {\bf A1-3}. Then for every $ z \in \Cp $ the operator $ F(z) $, defined in \eqref{F operator}, is a Hilbert-Schmidt integral operator on $\Lp{2} $, with the integral kernel 
\bels{integral kernel F_xy(z)}{
F_{xy}(z) \,=\, \abs{m_x(z)}\2S_{xy}\2\abs{m_y(z)}
\,.
}
The norm $ \lambda(z) := \norm{F(z)}_{\Lp{2}\to\Lp{2}} $ is a single eigenvalue of $ F(z) $, and it satisfies:
\bels{F and alpha}{
\norm{F(z)}_{\Lp{2}\to\Lp{2}} 
\;=\; 
1\,-\, \frac{\Im\, z}{\alpha(z)}\,\big\la f(z) \2|m(z)| \big\ra
\;<\;1
\,,
\qquad z \in \Cp\,.
}
Here the positive eigenvector $ f: \Cp \to \BB $ is defined by \eqref{defining properties of vector f(z)}, while $ \alpha: \Cp \to (0,\infty)$ is the size of the projection of $ v/\abs{m} $ onto the direction $ f $:
\bels{def of alpha(z)}{
\alpha(z)\,:=\,\avgB{ f(z)\1, \frac{v(z)}{\abs{m(z)}}}
\,.
}
\end{lemma}

\begin{Proof}
The existence and uniqueness of $ \norm{F(z)}_{\Lp{2}\to\Lp{2}} $ as a non-degenerate eigenvalue and $ f(z) $ as the corresponding eigenvector satisfying \eqref{defining properties of vector f(z)} follow from Lemma~\ref{lmm:Maximal eigenvalue of scaled S} below by choosing $ r := \abs{m(z)} $, using the trivial bound $ \norm{m(z)} \lesssim (\Im\, z)^{-1} $ to argue (using \eqref{unif bound of m}) that also $ r_{\!-}:=\inf_x\abs{m_x} > 0 $.

In order to obtain \eqref{F and alpha} we take the inner product of \eqref{v equation} with $f=f(z) $. Since  $F(z)$ is symmetric, we find
\bels{}{
\Big\la \1 \frac{f\2 v}{|m|} \1 \Big\ra
\,=\,
\big\la f \2 |m| \big\ra\,\Im \2 z
+
\norm{F}_{\Lp{2}\to\Lp{2}} \2\Big\la \1 \frac{f\1 v}{|m|} \1 \Big\ra\,.
}
Rearranging the terms yields the identity \eqref{F and alpha}. 
\end{Proof}

The following lemma demonstrates how the spectral gap, $\mrm{Gap}(F(z))$, the norm and the associated eigenvector of $ F(z) $ depend on the component wise estimates of $ \abs{m_x(z)} $. Since we will later need this result for a general positive function $ r : \Sx \to (0,\infty) $ in the role of $ \abs{m(z)} $ we state the result for a  general operator $ \wht{F}(r) $ below.     
        
\begin{lemma}[Maximal eigenvalue of scaled $ S $] 
\label{lmm:Maximal eigenvalue of scaled S}
Assume $ S $ satisfies {\bf A1-3}.
Consider an integral operator $ \wht{F}(r) : \Lp{2} \to \Lp{2} $, parametrized by $ r \in \BB $, with $ r_x \ge 0 $ for each $ x $, and defined through the integral kernel 
\bels{def of wht-F}{
\wht{F}_{xy}(r) \,:=\, r_rS_{xy}\1r_y
\,.
}
If there exist upper and lower bounds, $ 0<r_{\!-}\leq r_{\!+}< \infty $, such that  
\[
 r_{\!-}\leq r_x \leq r_{\!+}\,,\qquad \forall\1x\in \Sx
\,,
\]
then  $ \wht{F}(r) $ is Hilbert-Schmidt, and $ \wht{\lambda}(r) \,:=\, \norm{\wht{F}(r)}_{\Lp{2}\to \Lp{2}} $ is a single eigenvalue satisfying the upper and lower bounds
\bels{bounds for wht-lambda}{
r_{\!-}^2
\,\lesssim\,
\wht{\lambda}(r) 
\,\lesssim\,
r_{\!+}^2
\,.
} 
Furthermore, there is a spectral gap,
\bels{Gap for wht-F}{
\mrm{Gap}(\wht{F}(r)) 
\,\gtrsim\, 
r_{\!-}^{\12L}\,r_{\!+}^{-8}\;\wht{\lambda}(r)^{-L+5}
\,,
}
and the unique eigenvector, $ \wht{f}(r) \in \Lp{2} $, satisfying 
\bels{defining properties of vector wht-f}{
\wht{F}(r)\wht{f}(r) = \wht{\lambda}(r)\wht{f}(r)\,,\qquad
\wht{f}_x(r) \ge 0\,,
\quad\text{and}\quad
\norm{\wht{f}(r)}_2=1
\,,
}
is comparable to its average in the sense that
\bels{wht-f(r) bounds}{
\biggl(\!\frac{r_{\!-}^{\12}}{\,\wht{\lambda}(r)}\biggr)^{\msp{-7}L}
\;\lesssim\;
\frac{\wht{f}_x(r)}{\avg{\wht{f}(r)}}
\;\lesssim\; 
\frac{r_{\!+}^4}{\,\wht{\lambda}(r)^{\12}\!} 
\;.
}
If $\wht{F}$ is interpreted as a bounded operator on $\BB$, then the following relationship between the norm of the $\Lp{2}$-resolvent and the $\BB$-resolvent holds
\bels{resolvent of F on L2 and BB}{
\norm{\1(\wht{F}(r)-\zeta\1)^{-1}}
\,\lesssim \, \frac{1}{\abs{\zeta}}\biggl(\,1\,+\,r_+^2\norm{(\wht{F}(r)-\zeta)^{-1}}_{\Lp{2} \to\Lp{2}}\biggr)
\,,
}
for every $ \zeta \not \in \Spec(\wht{F}(r))\cup\{0\} $. 

\end{lemma}

Feeding \eqref{bounds for wht-lambda} into \eqref{wht-f(r) bounds} yields $ \Phi^{-2L}\avg{\wht{f}(r)}\lesssim \wht{f}(r) \lesssim \Phi^{\14}\avg{\wht{f}(r)} $, where $ \Phi := r_{\!+}/r_{\!-} $. 
For the proof of Lemma~\ref{lmm:Maximal eigenvalue of scaled S} we need a simple on the spectral gap that is well known in various forms. 
For the convenience of the reader we include a proof in Appendix \ref{sec:Proofs of auxiliary results in Chapter:Properties of solution}.

\begin{lemma}[Spectral gap for positive bounded operators]
\label{lmm:Spectral gap for positive bounded operators}
Let $T$ be a symmetric compact integral operator on $\Lp{2}(\Sx) $ with a non-negative integral kernel $ T_{xy} = T_{yx} \ge 0 $. 
Then
\[
\mathrm{Gap}(T) \,\ge\, 
\biggl(\!\frac{\norm{h}_{\Lp{2}}\!}{\!\norm{h}}\!\biggr)^{\!2}\!\inf_{x,y \ins \Sx}T_{x y}
\,,
\]
where $ h $ is an eigenfunction with $ Th = \norm{T}_{\Lp{2}\to \Lp{2}}\1h$. 
\end{lemma}

\begin{Proof}[Proof of Lemma~\ref{lmm:Maximal eigenvalue of scaled S}]
Since $ S $ is compact, and $ r \leq r_{\!+} $ also $ \wht{F} = \wht{F}(r) $ is compact. 
The operator $ \wht{F} $ preserves the cone of non-negative functions $ u \ge 0 $. Hence by the Krein-Rutman theorem $ \wht{\lambda} = \norm{\wht{F}}_{\Lp{2} \to \Lp{2}}$ is an eigenvalue, and there exists a non-negative normalized eigenfunction $ \wht{f} \in \Lp{2}(\Sx)$ corresponding to $ \wht{\lambda} $.
The smoothing property  {\bf A2} and the uniform primitivity assumption {\bf A3} combine to 
\[
\inf_{x,y\ins\Sx}(S^{\2 L})_{xy} \,\geq\, \rho
\,.
\]
Since $ r_{\!-} > 0  $, it follows that  the integral kernel of $ \wht{F}^{\2L}$ is also strictly positive everywhere. 
In particular, $ \wht{F}$ is irreducible, and thus the eigenfunction $ \wht{f} $ is unique. 

Now we derive the upper bound for $ \wht{\lambda} $.
Since $ \norm{w}_p \leq \norm{w}_q $, for $ p \leq q $, we obtain
\[
\wht{\lambda}^2 =\, 
\norm{\wht{F}}^2_{\Lp{2}\to \Lp{2}} 
= 
\norm{\wht{F}^2}_{\Lp{2}\to \Lp{2}} 
\leq\, 
\norm{\wht{F}^2}_{\Lp{1}\to\BB} 
\,=\, 
\sup_{x,y}\, 
(\wht{F}^2)_{xy}
\leq\, 
r_+^4\norm{S}_{\Lp{2}\to\BB}^2
\,,
\]
which implies $ \wht{\lambda} \lesssim r_+^2 $. Here we have used $ (S^2)_{xy} = \avg{S_x,S_y} \leq \norm{S_x}_2\norm{S_y}_2 $, and $ \sup_x\norm{S_x}_2 = \norm{S}_{\Lp{2}\to\BB} $ to estimate:
\bels{kernel upper bound for why-f}{
(\wht{F}^2)_{xy} \,\leq\, r_+^4(S^2)_{xy} \,\leq\, r_+^4 \norm{S}_{\Lp{2}\to\BB}^2
\,.
}
For the lower bound on $ \wht{\lambda} $, we use first \eqref{operator norms of S are comparable} and \eqref{operator norms of S are comparable} to get $ \iint\Px(\dif x)\2\Px(\dif y)\1 S_{xy} \sim 1 $. Therefore
\bels{lower bound for lambda}{
\wht{\lambda} 
\,=\, 
\norm{\wht{F}}_{\Lp{2}\to \Lp{2}} 
\,\ge\, 
\avg{\1e,\wht{F}\1e\1} 
\,\ge\, r_{\!-}^2 \iint \Px(\dif x)\1\Px(\dif y)\1S_{xy} 
\,\sim\, 
r_{\!-}^2
\,,
}
where $ e \in \BB $ is a function equal to one $ e_x = 1 $.

Now we show the upper bound for the eigenvector. Applying \eqref{kernel upper bound for why-f}, and $ \avg{\1\wht{f}\2} = \norm{\wht{f}}_1 \leq  \norm{\wht{f}}_2 = 1 $, yields
\[
\wht{\lambda}^2 \wht{f}_x 
\,=\, (\wht{F}^2\wht{f})_x 
\,\lesssim\, 
r_{\!+}^4 \avg{\1\wht{f}\,} 
\,\leq\, 
r_{\!+}^4 
\,.
\]
This shows the upper bound on $\wht{f}_x /\avg{\1\wht{f}\2}$ and, in addition, $\wht{f}_x\lesssim r_{\!+}^4/\wht{\lambda}^2$.

In order to estimate the ratios $ \wht{f}_x/\avg{\wht{f}} $, $ x \in \Sx $, from below, we consider the operator
\bels{auxiliary operator T}{
T := \Bigl(\frac{\wht{F}}{\wht{\lambda}}\Bigr)^L
\,.
}
Using $ \inf_{x,y}(S^L)_{xy} \ge \rho $, we get
\[
\inf_{x,y} T_{xy} \,\ge\, \frac{r_-^{2L}}{\wht{\lambda}^L}\,(S^L)_{xy} 
\,\gtrsim\,  
\Bigl(\frac{r_{\!-}^2}{\wht{\lambda}}\Bigr)^{\!L}
\,.
\]
Hence, we find a lower bound on $\wht{f}$ through
\bels{wht-f_x bounded by avg-wht-f}{
\wht{f}_x = (T\wht{f})_x 
\,\gtrsim\,  
\Bigl(\frac{r_{\!-}^2}{\wht{\lambda}}\Bigr)^{\!L}\avg{\1\wht{f}\,}
\,.
}

In order to prove \eqref{Gap for wht-F}, we apply Lemma~\ref{lmm:Spectral gap for positive bounded operators} to the operator $ T $, to get
\[
\mrm{Gap}(T) \ge \frac{\inf_{x,y}T_{xy}}{\norm{\wht{f}}^2}
\,\gtrsim\, 
\frac{(r_{\!-}^2/\wht{\lambda})^L}{(r_{\!+}^4/\wht{\lambda}^2)^2} 
\,=\, 
r_{\!-}^{\12L}r_{\!+}^{-8}\wht{\lambda}^{-(L-4)}
\,.
\]
Since $ L \sim 1 $, this implies,
\[
\frac{\mrm{Gap}(\wht{F})}{\wht{\lambda}} \,=\, 1-\bigl(1-\mrm{Gap}(T)\bigr)^{1/L} \,\ge\, 
\frac{\mrm{Gap}(T)}{L} 
\,\sim\, 
r_{\!-}^{\12L}r_{\!+}^{-8}\wht{\lambda}^{-(L-4)}
\,.
\]

Finally, we show the bound \eqref{resolvent of F on L2 and BB}.
Here the smoothing condition {\bf A2} on $ S $ is crucial.
Let $d,w \in \BB$ satisfy $(\wht{F}-\zeta)^{-1} w = d $. For $\zeta \notin \Spec(\wht{F})\cup\{0\}$, we have
\bels{BB-resolvent bound from L2-resolvent bound}{
\norm{d}_2\,\leq\, \norm{(\wht{F}-\zeta)^{-1}}_{\Lp{2} \to \Lp{2}}\1\norm{w}_2\,\leq\, \norm{(\wht{F}-\zeta)^{-1}}_{\Lp{2} \to \Lp{2}}\1\norm{w}\,.
}
Now, using $\norm{S}_{\Lp{2} \to \BB}\lesssim 1$, we bound the uniform norm of $ d $ from above by the corresponding $ \Lp{2}$-norm:
\[
\abs{\zeta}\2\norm{d} 
\,=\, 
\norm{\wht{F} d -w}
\,\leq 
\norm{\wht{F}}_{\Lp{2} \to \BB}\norm{d}_2+\norm{w}
\,\leq\, 
r_+^2 \norm{S}_{\Lp{2}\to\BB}\norm{d}_2+\norm{w}
\,.
\] 
The estimate \eqref{resolvent of F on L2 and BB} now follows by using the operator norm on $\Lp{2}$ for the resolvent, i.e., the inequality \eqref{BB-resolvent bound from L2-resolvent bound} to estimate $ \norm{d}_2 $ by $ \norm{w} $.  
\end{Proof}

\begin{Proof}[Proof of Proposition~\ref{prp:Estimates when solution is bounded}]
\label{proof of prp:Estimates when solution is bounded}
All the claims follow by combining Lemma~\ref{lmm:Constraints on solution}, Lemma~\ref{lmm:Operator F} and Lemma~\ref{lmm:Maximal eigenvalue of scaled S}.
Indeed, let $ z \in I +\cI\2(0,\infty) $, so that $  \norm{m(z)} \leq \Phi \sim 1 $.  
Since $ \supp\,v \subset [-\Sigma,\Sigma] $, with $ \Sigma = \norm{a} + 2\1\norm{S}^{1/2} \sim 1 $ (cf. \eqref{operator norms of S are comparable}), the upper bound of \eqref{unif bound of m} yields
$\norm{m(z)} \leq 2\1\abs{z}^{-1} $ for $ \abs{z} \ge 2\1\Sigma $. Thus $ \norm{m(z)} \lesssim (1+\abs{z})^{-1} $ for all $ \Re\,z \in I $. 
Using this upper bound in the first estimate of \eqref{unif bound of m} yields the part (i) of the proposition: 
\bels{abs-m_x size}{
\quad \abs{m_x(z)} \,\sim\, (\11+\abs{z}\1)^{-1}
\,,\qquad x \in \Sx\,,\;\Re\,z \in I
\,.
}

When $ \abs{z} \leq 2\1\Sigma $ the comparison relation $ v_x(z) \sim \avg{v(z)} $ follows by plugging \eqref{abs-m_x size} into \eqref{v compared to avg-v}. If $ \abs{z} > 2\1\Sigma $, then $ v $ and its average are comparable due to the Stieltjes transform representation \eqref{m as stieltjes transform} and the bound \eqref{bound on supp v} for the support of $ v|_\R $. This completes the proof of the part (ii).

For the claims concerning the operator $ F(z) $ we  use the formula \eqref{integral kernel F_xy(z)} to identify $ F(z) = \wht{F}(\abs{m(z)}) $, where $ \wht{F}(r) $ for $ r \in \BB$  satisfying $ r \ge  0 $, is the operator from Lemma~\ref{lmm:Maximal eigenvalue of scaled S}. 

The parts (iii-v) follow from Lemma~\ref{lmm:Maximal eigenvalue of scaled S} with the choice $  r_{\!-} := \inf_x \abs{m_x} $ and $ r_{\!+} := \sup_x \abs{m_x} $, since  $ r_{\!\pm} \sim (1+\abs{z})^{-1} $ by \eqref{abs-m_x size}.
\end{Proof}

\section{Stability and operator $ B $}
\label{sec:Stability and operator B}

The next lemma introduces the operator $ B $ that plays a central role in the stability analysis of the QVE. At the end of this section (Lemma~\ref{lmm:Stability when m and inv-B bounded}) we present the first stability result for the QVE which is effective when $ m $ is uniformly bounded and $ B^{-1} $ is bounded as operator on $ \BB $. 
Subtracting the QVE from \eqref{perturbed QVE - 2nd time} an elementary algebra yields the following lemma.

\NLemma{Perturbations}{
Suppose $ g,d \in \BB $, with $ \inf_x \abs{g_x} > 0 $, satisfy the perturbed QVE,
\bels{perturbed QVE - 2nd time}{ 
-\frac{1}{g} \;=\; z + a  + Sg +d
\,,
}
at some fixed $ z \in \Cp $ and suppose $ m = m(z) $ solves the unperturbed QVE. 
Then 
\bels{}{
u \,:=\, \frac{\2g\1-\1m(z)}{\abs{\1m(z)}}
\,,
}
satisfies the equation
\bels{eq. for u}{
B\1u \;=\; \nE^{-\cI\1\am}\1u\2Fu \,+\, \abs{m}\1d \,+\, \abs{m}\2\nE^{-\cI \1\am}u\1d
\,,
} 
where the operator $ B = B(z) $, and the function  $ \am = \am(z) : \Sx \to [\10,2\1\pi) $ are given by 
\bels{def of B and a}{
B \,:=\, \nE^{-\cI\12\1\am} \1-\1 F
\,,\qquad\text{and}\qquad
\nE^{\1\cI\1q} :=\, \frac{m}{\abs{m}}
\,.
} 
}
\qed

Lemma~\ref{lmm:Perturbations} shows that the inverse of the non-selfadjoint operator $ B(z) $ plays an important role in the stability of the QVE against perturbations. In the next lemma we estimate the size of this operator in terms of the solution of the QVE.

\begin{lemma}[Bounds on $ B^{-1} $] 
\label{lmm:Bounds on B-inverse}
Assume {\bf A1-3}, and consider $ z \in \Cp $ such that $ \abs{z} \leq 2\1\Sigma $. Then the following estimates hold:
\begin{itemize}
\titem{i} 
If $ \norm{m(z)}_2 \leq \Lambda $, for some $ \Lambda < \infty $, then
\bels{m L2-bounded: B-inv norm bound on L2 and BB}{
\norm{B(z)^{-1}}_{\Lp{2}\to\Lp{2}}
\,\lesssim\,  \2\avg{v(z)}^{-12}
\,,\quad\text{and }\quad
\norm{B(z)^{-1}}
\,\lesssim\,  \2\avg{v(z)}^{-14}
\,,
}
with $ \Lambda $ considered an additional model parameter.
\titem{ii}
If $ \norm{m(z)} \leq \Phi $, for some $ \Phi < \infty $, then  
\begin{subequations}
\label{m BB-bounded: B-inv norm bound on L2 and BB}
\begin{align}
\label{m BB-bounded: inv-B from L2 to BB}
\norm{B(z)^{-1}}  
\,&\lesssim\,
1+\norm{B(z)^{-1}}_{\Lp{2}\to \Lp{2}} 
\\
\label{B-inv norm bound on BB at E-line}
&\lesssim\, 
(\2\abs{\sigma(z)}+\avg{v(z)})^{-1}\avg{v(z)}^{-1}
\,,
\end{align}
\end{subequations}
with the function $\sigma :\Cp \to \R$, defined by
\bels{def of sigma(z)}{
\sigma(z) \,:=\, \avgb{ \1 f(z)^3 \sign \Re \2 m(z) \1}
\,,
}
and $ \Phi $ considered  an additional model parameter.
\end{itemize}
\end{lemma}

We remark that \eqref{B-inv norm bound on BB at E-line} improves on the analogous bound $\norm{B^{-1}}\lesssim \avg{v}^{-2}$ that was proven in \cite{AEK1cpam}.
We will see below that \eqref{B-inv norm bound on BB at E-line} is sharp in terms of powers of $ \avg{v} $. On the other hand, the exponents in \eqref{m L2-bounded: B-inv norm bound on L2 and BB} may be improved.
For the proof of Lemma~\ref{lmm:Bounds on B-inverse} we need the following auxiliary result which was provided as Lemma~5.8 in \cite{AEK1cpam}. Since it plays a fundamental role in the analysis its proof is reproduced in Appendix~\ref{sec:Proofs of auxiliary results in Chapter:Properties of solution}.

\begin{lemma}[Norm of $ B^{-1} $-type operators on $ \Lp{2}$] 
\label{lmm:Norm of B^-1-type operators on L2}
Let $ T $ be a compact self-adjoint and $U$ a unitary operator on $ \Lp{2}(\Sx) $. 
Suppose that $ \mrm{Gap}(T)> 0 $ and $ \norm{T}_{\Lp{2}\to\Lp{2}} \leq 1 $.
Then there exists a universal positive constant $C$ such that
\bels{Norm of B^-1-type operators on L2}{
\norm{\1(\1U-T\1)^{-1}\1}_{\Lp{2}\to \Lp{2}} 
\,\leq\, 
\frac{C}{\1\mrm{Gap}(T)\,\abs{\21 - \norm{T}_{\Lp{2}\to \Lp{2}}\2\avg{\1 h\1, Uh\1}\1}}
\,,
}
where $ h $ is the $ \Lp{2}$-normalized eigenvector of $T$, corresponding to the non-degenerate eigenvalue 
$\norm{T}_{\Lp{2}\to \Lp{2}}$.
\end{lemma}

\begin{Proof}[Proof of Lemma~\ref{lmm:Bounds on B-inverse}]
We will prove the estimates \eqref{m L2-bounded: B-inv norm bound on L2 and BB} and \eqref{m BB-bounded: B-inv norm bound on L2 and BB} partly in parallel. 
Depending on the case, $ z $ is always assumed to lie inside the appropriate domain, i.e., either $ z $ is fixed such that $ \norm{m(z)}_2 \leq \Lambda $, or $ \nnorm{m}_{\sett{\tau}} \leq \Phi $, with $ \Re\,z = \tau $.
Besides this, we consider $ z $ to be fixed.
Correspondingly, the comparison relations in this proof depend on either $ (\rho,L, \norm{a}, \norm{S}_{\Lp{2}\to\BB}, \Lambda) $ or $ (\rho,L,\norm{a}, \norm{S}_{\Lp{2}\to\BB},\Phi) $ (cf. Convention~\ref{conv:Standard model parameters}).  
We will also drop the explicit $ z $-arguments in order to make the following formulas more transparent. In both  cases the lower bound $\abs{m_x(z)} \gtrsim 1$ follows from \eqref{unif bound of m}.

We start the analysis by noting that it suffices to consider only the norm of $ B^{-1} $ on $ \Lp{2} $, since
\bels{inv-B-norm: from L2toL2 to BBtoBB}{
\norm{B^{-1}} \;\lesssim\, 1 \,+ \norm{m}^2\norm{B^{-1}}_{\Lp{2}\to\Lp{2}}
\,.
}
In order to see this, we use the smoothing property {\bf A2} of $S$ as in the proof of \eqref{resolvent of F on L2 and BB} before. In fact, besides replacing the complex number $\zeta $ with the function $ \nE^{2\1\cI\1\am} $, the proof of \eqref{BB-resolvent bound from L2-resolvent bound} carries over without further changes.

By the general property \eqref{F and alpha} of $ F $ we know that $ \norm{F}_{\Lp{2}\to\Lp{2}} \leq 1 $.
Furthermore, it is immanent from the definition of $F$ and \eqref{averaged row is comparable to BB-norm of S} that $\norm{F}_{\Lp{2}\to \Lp{2}}\gtrsim \inf_x|m_x|^2\gtrsim 1$ in both of the considered cases. 
This shows that the hypotheses of Lemma~\ref{lmm:Norm of B^-1-type operators on L2} are met, and hence
\bels{the norm term}{
\norm{B^{-1}}_{\Lp{2}\to\Lp{2}} 
\;\lesssim\; 
\mrm{Gap}(F)^{-1}\,\absb{\,1 \,- \norm{F}_{\Lp{2}\to\Lp{2}}\2\avg{\1\nE^{\1\cI\12\am}f^{\12}\1}\1}^{-1}
\,,
}
where we have also used $ \norm{F}_{\Lp{2}\to\Lp{2}} \sim 1 $. Now, by basic trigonometry,
\[
\avg{\1\nE^{\1\cI\12\am}f^{\12}\1} \;=\; \avgb{\1(1-2\sin^2 \am\1)\1f^{\12}} \,+\, \cI\12\2\avgb{f^{\12} \sin \am \cos \am }
\,,
\]
and therefore we get
\bels{v^2+uv bound}{
&\absb{\21 \,-\, \norm{F}_{\Lp{2}\to\Lp{2}}\2\avgb{\1\nE^{\1\cI\12\am}f^{\12}\1}\1}
\\
&\gtrsim\,
1-\norm{F}_{\Lp{2}\to\Lp{2}}
+
\norm{f \sin \am}^2_2
\,+\; 
\absb{\avgb{f^{\12}\sin \am \cos \am}}
\,.
}
Here,  we have again used $1 \lesssim \norm{F}_{\Lp{2}\to\Lp{2}} \leq 1 $. Substituting this back into \eqref{the norm term}  yields
\bels{final general bound for Gamma_2}{
\norm{B^{-1}}_{\Lp{2}\to\Lp{2}}
\,\leq\, \frac{1}{\mrm{Gap}(F)}\, \frac{1}{\,1\,-\norm{F}_{\Lp{2}\to\Lp{2}}\!+
\norm{f \sin \am}^2_2
+\abs{\1\avg{f^{\12}\sin \am \cos \am\1}}\,}
\,.
}

\medskip
\noindent{\scshape Case 1 }($ m $ with $ \Lp{2}$-bound): 
In this case we drop the $ \avg{f^{\12}\sin \am \cos \am} $ term and estimate
\bels{bound for sin^2-term}{
\norm{f \sin \am}_2
\,\geq\, 
\norm{f}_2\,
\inf_x \sin  \am_x
\,=\,
\inf_x\frac{v_x}{|m_x|}
\,\gtrsim\,
\avg{\1v\1}^2
\,,
}
where the bounds  $ \norm{m} \lesssim \Lambda^C\avg{v}^{-1} \sim \avg{v}^{-1} $ and $ v \gtrsim \Lambda^{-C}\avg{v} \sim \avg{v} $ from Lemma~\ref{lmm:Constraints on solution} were used in the last inequality.  
Plugging \eqref{bound for sin^2-term} back into \eqref{final general bound for Gamma_2}, and using \eqref{Gap for wht-F} to estimate $ \mrm{Gap}(F) = \mrm{Gap}(\wht{F}(\abs{m})) \gtrsim \Lambda^{-C}\norm{m}^{-8} \gtrsim \avg{v}^{8} $ yields the desired bound:
\bels{final estimate for Gamma_2 when S not regular}{ 
\norm{B^{-1}}_{\Lp{2}\to\Lp{2}}
\,\lesssim\,
\mrm{Gap}(F)^{-1}\,
\norm{f \sin \am}_2^{-2}
\;\lesssim\;
\avg{v}^{-8}\avg{v}^{-4}
\sim\,
\avg{\2v\2}^{-12}\,.
}
The operator norm bound on $ \BB $ follows by combining this estimate with \eqref{inv-B-norm: from L2toL2 to BBtoBB}, and then using \eqref{unif bound of m} to estimate $ \norm{m} \lesssim \Lambda^{-2L+2}\avg{v}^{-1} \sim \avg{v}^{-1} $.
 
\medskip
\noindent{\scshape Case 2 }($ m $ uniformly bounded):
Now we assume $ \norm{m} \leq \Phi \sim 1 $, and thus all the bounds of Proposition~\ref{prp:Estimates when solution is bounded} are at our disposal.
This will allow us to extract useful information from the term $ \abs{\avg{f^{\12}\sin \am \,\cos \am\1}} $ in \eqref{final general bound for Gamma_2} that was neglected in the derivation of \eqref{final estimate for Gamma_2 when S not regular}.
Clearly, $ \abs{\avg{f^{\12}\sin \am \,\cos \am\1}} $ can have an important effect on \eqref{final general bound for Gamma_2} only when the term $ \norm{f \sin \am\1}_2 $ is small. Moreover, using $ |m_x| \sim 1 $ we see that this is equivalent to $ \sin \am_x  = v_x/|m_x| \sim \avg{v} $ being small.
Since $ \avg{v} \gtrsim \Im\,z $, for $ \abs{z} \leq 2\1\Sigma \sim 1 $, 
the imaginary part of $z$ will also be small in the relevant regime.

Writing the imaginary part of the QVE in terms of $ \sin \am = v/\abs{m} $, we get
\bels{QVE for sin a}{
\sin \am = \abs{m}\,\Im\,z \,+\,F \sin \am 
\,.
}
Since we are interested in a regime where $ \Im\,z $ is  small, this implies, recalling $ Ff = f $, that $ \sin \am $ will then almost lie in the span of $ f $. To make this explicit, we decompose
\bels{sin a = alpha f + eta t}{
\sin \am = \alpha\2f + (\Im\,z)\2t
\,,
\quad\text{with}\quad
\alpha = \avg{\1f,\sin \am\1}
\,,
}
for some $ t \in \BB $ satisfying  $ \avg{f,t\1} = 0 $. 
Let $ Q^{(0)} $ denote the orthogonal projection $ Q^{(0)}w := w-\avg{f,w}\1f $. Solving for $ t $ in \eqref{QVE for sin a} yields: 
\bels{expression for t}{
t \,=\, (\Im\,z)^{-1}Q^{(0)} \sin \am \,=\, (\11-F\1)^{-1}Q^{(0)}\abs{m}
\,.
}
Proposition~\ref{prp:Estimates when solution is bounded} implies $ \mrm{Gap}(F) \sim 1 $. Therefore we have 
\[
\norm{Q^{(0)}(1-F)^{-1}Q^{(0)}}_{\Lp{2}\to\Lp{2}} \,\lesssim \,\mrm{Gap}(F)^{-1}\, \sim \,1\,.
\]
In fact, since $f_x \sim 1$, a formula analogous to \eqref{inv-B-norm: from L2toL2 to BBtoBB} applies, and thus we find
\[ 
\norm{Q^{(0)}(\11-F\1)^{-1}Q^{(0)}}\, \lesssim \,1\,  .
\] 
Applying this in \eqref{expression for t} yields $ \norm{\1t\1} \lesssim 1 $, and therefore
\bels{final expansion for sin a}{
\sin \am \,=\, \alpha\1 f \,+\, \Ord_\BB(\1\Im\,z)\,.
}
Moreover, since we will later use the smallness of $ \avg{v} \sim \sin \am_x \sim \alpha $, we may expand
\bels{expansion of cos a}{
\cos \am 
\,=\, 
(\mrm{sign} \cos \am\2)\,(\21-\sin^2 \am\,)^{1/2} 
\,=\,
\sign \Re\,m \,+\, \Ord_\BB(\alpha^2)
\,.
}
Combining this with \eqref{final expansion for sin a} yields  
\bels{final expansion for sin a cos a -term}{
\avgb{\1f^{\12} \sin \am \cos \am\2} 
\,&=\,
\avgB{\1f^{\12}\2\bigl(\alpha\1f+\Ord_\BB(\Im\,z)\2\bigr)\bigl(\2\sign \Re\,m \,+\, \Ord_\BB(\alpha^2)\1\bigr)}
\\
&=\,  
\sigma\,\alpha \,+\, \Ord\big(\1\avg{\1v\1}^3+\Im\2z\big)
\,,
}
where we have again used $ \alpha \sim \avg{\1v\1} $, and used the definition,  $ \sigma = \avg{\1f^{\13}\2\sign(\Re\,m\1)} $, from the statement of the lemma. 

For the term
$1-\norm{F}_{\Lp{2}\to\Lp{2}}$ in the denominator of the r.h.s.
of the main estimate \eqref{final general bound for Gamma_2} we make use of the explicit formula \eqref{F and alpha} for the spectral radius of $F$,
\bels{F and alpha 2}{
1-\norm{F}_{\Lp{2}\to\Lp{2}}
\,=\, 
\frac{\Im\2 z}{\alpha}\,\la\1 f \2|m| \1\ra  
\,.
}

By Proposition~\ref{prp:Estimates when solution is bounded} we have $f_x \sim 1$, $|m_x|\sim 1 $ and $\mrm{Gap}(F) \sim 1$. Using this knowledge in combination with \eqref{final expansion for sin a cos a -term}, \eqref{F and alpha 2} and $\alpha \sim \la \1 v \1\ra$ we estimate the r.h.s. of 
\eqref{final general bound for Gamma_2} further:
\bels{bound for Gamma_2 end}{
\norm{B^{-1}}_{\Lp{2}\to\Lp{2}}
\;&\lesssim\;
\frac{\avg{\1v\1}}{\,\avg{\1v\1}^3 +\2\avg{\1f\1\abs{m}\1}\,\Im\, z\,+\,\absb{\1\sigma\avg{\1v\1}^2 \,+\,\Ord\big(\1\avg{\1v\1}^4+\avg{v}\,\Im\2z\big)}\,}
\,.
}

Let us now see how from this and \eqref{inv-B-norm: from L2toL2 to BBtoBB} the claim \eqref{B-inv norm bound on BB at E-line} follows.
Clearly, it suffices to consider only the case where $ \avg{v} \leq \eps $ for some $ \eps \sim 1 $. 
If $\avg{\1v\1} \ge \abs{\sigma}$, then the $\avg{\1v\1}^3$-term in the denominator is alone suffices for the final result. We may therefore assume that $\la \1v\1\ra \leq \abs{\sigma} $.
We are also done if $\Im \2z \ge |\sigma|\la \1v\1\ra^2$ since then we may use the second summand on the r.h.s. of \eqref{bound for Gamma_2 end} to get the $|\sigma|\la \1v\1\ra$-term we need for \eqref{B-inv norm bound on BB at E-line}. In particular, we can assume that the error term in \eqref{bound for Gamma_2 end} is $\Ord\bigl(|\sigma|\la \1v\1\ra^3\bigr)$. 
The bound \eqref{B-inv norm bound on BB at E-line} thus follows by choosing $ \eps \sim 1 $ small enough.
\end{Proof}

We will now show that the perturbed QVE \eqref{perturbed QVE - 2nd time} is stable as long as a priori bound on $ m $ and $ B^{-1} $ is available.

\begin{lemma}[Stability when $ m $ and $ B^{-1}	$ bounded] 
\label{lmm:Stability when m and inv-B bounded}
Assume {\bf A1}. Suppose  $ g,d \in \BB $, with $ \inf_x \abs{g_x} > 0 $,  satisfy the perturbed QVE \eqref{perturbed QVE - 2nd time} at some point $ z \in \Cp $. 
Assume 
\bels{quantitative bulk stability assumptions}{
\norm{m(z)} \,\leq\, \Phi\,,
\qquad\text{and}\qquad
\norm{B(z)^{-1}} \,\leq\, \Psi 
\,,
} 
for some constants $\Phi,\Psi \geq 1$.
There exists  a linear operator $ J(z) $ acting on $ \BB $,  and depending only on $ S $ and $ a $ in addition to $ z $, with  
$ \norm{J(z)}\leq 1 $,  such that if
\bels{small bulk perturbation}{
\norm{g-m(z)}
\,\leq\, ¨
\frac{1}{2\2 \max\sett{\11\1,\norm{S}}\2\Phi\1\Psi\1}
\,,
}
then the correction $ g-m(z) $ satisfies 
\begin{subequations}
\label{bulk perturbations}
\begin{align}
\label{bulk perturbations: BB-bound}
\qquad
\norm{g-m(z)} 
\;&\leq\; 
3\2\Psi\2\Phi^2\norm{\1d\1}
\\
\label{bulk perturbations: w-avg-bound}
\abs{\avg{w,g-m(z)}} \;&\leq\;
12\2 \max\sett{\11\1,\norm{S}}\2\Psi^3\Phi^5
\norm{w}_1\norm{\1d\1}^2
\\
\notag
&\;\quad+\,
\Psi\2\Phi^2\abs{\1\avg{J(z)\1w,d\2}}
\,,
\end{align}
\end{subequations}
for any $w \in \BB $. 
\end{lemma}

\begin{Proof}
Expressing  \eqref{eq. for u} in terms of $ h = g-m = \abs{m}\1u $, and re-arranging we obtain
\bels{exact formula for h}{
h \,=\, \abs{m}\2B^{-1}\bigl[\1\nE^{-\cI\1\am}\1h\2Sh \,+\, (\1\abs{m} + \2\nE^{-\cI \1\am}h\1)\1d\,\bigr]
\,.
}
Taking the $ \BB$-norm of \eqref{exact formula for h} yields
\[
\norm{h} 
\;\leq\; 
\Phi\2\Psi\1\norm{S}\norm{h}^2 \,+\,(\1\Phi^2\Psi+\Phi\1\Psi\1\norm{h})\2\norm{d}
\,.
\]
Under the hypothesis \eqref{small bulk perturbation} the two summands on the right hand side are less than $ (1/2)\norm{h} $ and $ (3/2)\2\Phi^2\Psi\1\norm{d} $, respectively. 
Rearranging thus yields \eqref{bulk perturbations: BB-bound}.

In order to prove \eqref{bulk perturbations: w-avg-bound} we apply the linear functional $ u \mapsto \avg{w,u} $ on \eqref{exact formula for h}, and get
\bels{bulk main bound for w-average of d}{
\abs{\1\avg{w,h\1}} 
\;&\leq\; 
\absb{\avgb{\1w,\abs{m}B^{-1}(\nE^{\cI\1\am}h\2Sh)\1}} 
\,+\,
\absb{\avgb{\1w,\abs{m}B^{-1}(\1\nE^{\cI\1\am}h\1d)\1}} 
\\
&\quad+\,
\Psi\1\Phi^2
\absb{\avg{\1J\1w,d\1}} 
\,,
}
where we have identified the operator $ J \,:=\, (\Psi\Phi^2)^{-1}\abs{m}\,(B^{-1})^\ast(\1\abs{m}\,\genarg) $ from the statement.
Clearly, $ B^\ast $ is like $ B $ except the angle function $ \am $ is replaced by $ - \am $ in the definition \eqref{def of B and a}. In particular, $ \norm{(B^\ast)^{-1}} \leq \Psi $, and thus $ \norm{J} \leq 1 $.
The estimate \eqref{bulk perturbations: w-avg-bound} now follows by bounding the first two term on the right hand side of \eqref{bulk main bound for w-average of d} separately:
\bels{bulk w-avg terms}{
\absb{\avgb{\1w,\abs{m}B^{-1}(\nE^{\cI\1\am}h\2Sh)\1}} 
\,&\leq\,
\norm{w}_1\,\normb{\1\abs{m}\2B^{-1}(\nE^{\cI\1\am}h\2Sh)} 
\\
&\leq\,
9\2\norm{S}\1\Phi^5\Psi^3\norm{w}_1\norm{d}^2
\\ 
\absb{\avgb{\1w,\abs{m}B^{-1}(\1\nE^{\cI\1\am}h\1d)\1}} 
\,&\leq\,
\norm{w}_1\,\normb{\1\abs{m}\2B^{-1}(\1\nE^{\cI\1\am}h\1d)\1} 
\,\leq\,
3\1\Phi^3\Psi^2\1\norm{w}_1\norm{d}^2
\,.
}
For the rightmost estimates we have used \eqref{bulk perturbations: BB-bound} to get $ \norm{h\2Sh} \leq \norm{S}\norm{h}^2 \leq 9\2\norm{S}\Phi^4\Psi^2\norm{d}^2 $, and $ \norm{h\1d} \leq 3\1\Phi^2\Psi\1\norm{d}^2 $, respectively. Now plugging \eqref{bulk w-avg terms} into \eqref{bulk main bound for w-average of d} and recalling $ \Phi,\Psi \ge 1 $ yields \eqref{bulk perturbations: w-avg-bound}.
\end{Proof}

\chapter{Uniform bounds}
\label{chp:Uniform bounds}

Our main results, such as Theorem~\ref{thr:Shape of generating density near its small values} rely on the assumption that the solution $ m $ of the QVE is  uniformly bounded. In other words, we assume that there is an upper bound $ \Phi < \infty $, such that 
\bels{nnorm-m on R bounded by Phi}{
\nnorm{m}_\R \leq \Phi
\,, 
}
and our results deteriorate as $ \Phi $ becomes larger.
In this chapter we introduce two sufficient quantitative conditions, {\bf B1} and {\bf B2} on $ a $ and $ S $ that make it possible to to construct a constant $ \Phi < \infty $ in \eqref{nnorm-m on R bounded by Phi} that depend on $ S $ and $ a $ only through a few  model parameters.
These extra conditions will always be assumed in conjunction with the properties {\bf A1} and {\bf A2}. 

To this end, we introduce a strictly increasing auxiliary function $ \Gamma : [\10,\infty) \to  [\10,\infty) $, determined by $ a $ and $ S $: 
\bels{def of Gamma}{
\Gamma(\tau) \,:=\, \inf_{x \2\in\2 \Sx}
\sqrt{\int_\Sx\,
\Bigl(\,\frac{1}{\tau} + \abs{a_y-a_x} +\norm{S_y-S_x}_2\Bigr)^{\!-2}\msp{-8}\Px(\dif y)\;}
\;.
}
We also define the upper limit on the range of $ \Gamma $,
\bels{def of Gamma(infty)}{
\Gamma(\infty) \,:=\, \lim_{\tau\to\infty} \Gamma(\tau)
\,.
}
As a strictly increasing function $ \Gamma $ has an inverse $ \Gamma^{-1} $ defined on $ (0,\Gamma(\infty)) $. This inverse satisfies $ \Gamma^{-1}(\lambda) > \lambda $, for $ 0< \lambda < \infty $, and we extend it to $ (\20\1,\infty\1) $ by setting $ \Gamma^{-1}(\lambda) := \infty $, when $ \lambda \ge \Gamma(\infty) $. 

The function $ \Gamma(\tau) $ will be used to convert $ \Lp{2} $ bounds on $ m(z) $ into  uniform bounds.
We will consider the cases $ a = 0 $ and $ a \neq 0 $ separately. 

When $ a = 0 $ Lemma~\ref{lmm:Structural L2-bound} implies $ \norm{m(z)}_2 \leq 2\1\abs{z}^{-1} $, and hence we only need to obtain an additional $ \Lp{2}$-estimate for $ m(z) $ around $ z = 0 $.
To this end, we introduce the following condition:
\begin{itemize}
\item[{\bf B1}]
\label{def of B1}
\emph{Quantitative block fully indecomposability:} 
There exist two constants $ \varphi > 0 $, $ K \in \N $, a fully indecomposable matrix $ \brm{Z} = (Z_{ij})_{i,j=1}^K $, with $Z_{ij} \in \sett{0,1} $, and a measurable partition $ \mcl{I} := \sett{I_j}_{j=1}^K $ of $ \Sx $, such that for every $ 1 \leq i,j \leq K $ the following holds:
\bels{B2: Quantitative block FID condition}{
\Px(I_j) \,=\,\frac{1}{K}\,,
\qquad\text{and}\qquad
S_{xy} \,\ge\,\varphi\1 Z_{ij}\,,\quad\text{whenever}\quad(x,y) \in I_i \times I_j
\,.
}
\end{itemize}
Here the constants $ \varphi, K $ are the model parameters associated to {\bf B1}.
The property {\bf B1} amounts to a quantitative way of requiring $ S $ to be a block fully indecomposable operator (cf. Definition~\ref{def:Full indecomposability}).
We also remark that {\bf B1} implies {\bf A3} by the part (iii) of Proposition~\ref{prp:Properties of FID matrices} and the estimate \eqref{Z FID implies S satisfies A3} below.

Our main result concerning the uniform boundedness in the case $ a = 0 $ is the following:

\begin{theorem}[Quantitative uniform bounds when $a =0$]
\label{thr:Quantitative uniform bounds when a = 0}
Suppose $ a = 0 $, and assume $ S $ satisfies {\bf A1} and {\bf A2}. 
Then the following uniform bounds hold:
\begin{itemize}
\titem{i}
{\bf Neighborhood of zero:} If additionally {\bf B1} holds, then there are constants $ \delta > 0  $ and $  \Phi < \infty $, both depending only on $ S $ only through the parameters $ \varphi,K $, s.t.,   
\bels{bound around tau=0}{
\norm{\1m(z)} \,\leq\,\Phi
\,,\qquad\text{for}\qquad
\abs{z} \leq \delta
\,.
}
\titem{ii}
{\bf Away from zero:} 
\bels{uniform bound outside z=0}{
\norm{m(z)}
\,\leq\,
\frac{\abs{z}}{2}\,\Gamma^{-1}\msp{-2}\Bigl(\frac{4}{\2\abs{z}^2\!}\Bigr)
\,,\qquad\text{for}\quad
\abs{z} >\2 \frac{\,2}{\!\sqrt{\2\Gamma(\infty)}}
\,.
}
\end{itemize}

In particular, if $ S $ satisfies {\bf B1} and $ \Gamma(\infty) > 4\2\delta^{\1-2} $, then
\bels{uniform bound everywhere when a=0}{
\nnorm{m}_\R
\,\leq\,
\max\setbb{\Phi\2,\2\frac{\1\delta\1}{2}\2\Gamma^{-1}\msp{-1}\Bigl( \frac{4}{\2\delta^{\12}\!}\Bigr)\!}
\,,
}
where $ \delta $ and $ \Phi $ are from \eqref{bound around tau=0}.  
\end{theorem}
The condition in (i) for the bound around $ z = 0 $ is optimal for block operators by  Theorem~\ref{thr:Scalability and full indecomposability} below. 
In Section~\ref{sec:Blow-up at z=0 when a=0 and assumption B1} we have collected simple examples that demonstrate how the solution can become unbounded around $ z = 0 $ when the condition {\bf B1} does not hold. 
In order to demonstrate the role of $ \Gamma $ in the part (ii) of the theorem we demonstrate in Section~\ref{sec:Divergences for special x-values: Outlier rows} that some components of the solution of the QVE may blow up even when {\bf A1-3} hold uniformly.

\begin{remark}[Piecewise $ 1/2$-H\"older continuous rows when $a=0$]
\label{rmk:Piecewise 1/2-Holder continuous rows when a=0}
Consider the setup $ (\Sx,\Px) = (\1[\10\1,1],\dif x) $ with $ a = 0 $.  
Assume $ S $ satisfies {\bf A1-2}, and that its rows $ x \mapsto S_x  \in \Lp{2} $ are piecewise $1/2$-H\"older continuous, such that \eqref{def of PW-1/2-Holder} holds for some finite partition $ \sett{I_k} $ of $[0,1]$ with $ \min_k \abs{I_k}>0$. 
Since the function $ \tau \mapsto \abs{\tau}^{-1} $ is not integrable around $ \tau = 0 $ the range of $ \Gamma $ is unbounded, i.e., $ \Gamma(\infty) = \infty $. 
Therefore applying the part (ii) of Theorem~\ref{thr:Quantitative uniform bounds when a = 0} we obtain for any $ \delta > 0 $ the uniform bound
\[
\norm{m}_{\R\backslash[-\delta,\delta\1]} 
\,\leq\, 
\frac{\delta \exp(\12\1C_1^2\delta^{-4})}{C_1\sqrt{ \min_k\abs{I_k} \,}}
\,,
\]
where the constant $ C_1 $ is from \eqref{def of PW-1/2-Holder}.
\end{remark}

The next remark gives a simple example of a block fully indecomposable $ S $. 
  
\begin{remark}[Positive diagonal when $ a = 0$]
\label{rmk:Positive diagonal when a=0}
The part (i) of Theorem~\ref{thr:Quantitative uniform bounds when a = 0} implies that for any $ S $ with a positive diagonal the solution of the QVE is bounded around $ z = 0 $, e.g., if $ (\Sx,\Px) = (\1[\10\1,1],\dif x) $, and there are constants $ \eps,\lambda > 0 $ such that
\bels{example:S has pos.diagonal}{
S_{xy} \ge \eps\,\Ind\sett{\2\abs{x-y}\leq \lambda\2}
\,,
}
then $ m(z) $ is bounded on a neighborhood of $ z = 0 $, because  $ S $  satisfies {\bf B1}, with $ K $ and $ \varphi $ depending only on $ \eps $ and $ \lambda $.
\end{remark} 

Now we consider the uniform boundedness in the case $ a \neq 0 $. 
In this case the structural $ \Lp{2}$-estimate from Lemma~\ref{lmm:Structural L2-bound} covers only the regime $ \abs{z} > \norm{a} $. 
In order to get $ \Lp{2}$-bounds also in the remaining regime $ \abs{z} \leq \norm{a}$, we introduce a weaker version of the assumption (2.4) used in \cite{AEK1cpam}:
\begin{itemize}
\item[{\bf B2}]
\label{def of B2}
\emph{Strong diagonal:} There is a constant $ \psi > 0 $, such that 
\bels{quantitative strong diagonal condition}{
\avg{w,Sw} \,\ge\,  \psi
\1 \avg{w}^2
\,,
\qquad \forall\,w \in \BB, \text{ s.t. }w_x \ge 0\,.
}
\end{itemize}
Here $ \psi $ is considered  a model parameter.
Since \eqref{qualitative strong diagonal condition} implies {\bf B2} for some $ \psi > 0 $, the property {\bf B2} constitutes a quantitative version of \eqref{qualitative strong diagonal condition}.

The following result is a quantitative version of the part (ii) of Theorem~\ref{thr:Qualitative uniform bounds}.

\begin{theorem}[Quantitative uniform bound for general $a$]
\label{thr:Quantitative uniform bound for general a}
Assume {\bf A1-3} and {\bf B2}. 
Then there exists a constant $ \Omega_\ast \ge 1 $, depending only on the model parameters $ \norm{S}_{\Lp{2}\to\BB}, \rho,L,\psi $, such that if 
\bels{Range of Gamma is larger than L2-bound from B2}{
\Gamma(\infty)
\,>\, \Omega_\ast
\,,
}
then
\bels{uniform bound for general a}{
\nnorm{m}_\R
\leq 
\frac{\,\Gamma^{-1}\msp{-1}(\1\Omega_\ast)}{\Omega_\ast^{1/2}\!}
\,.
}
\end{theorem}

The threshold $ \Omega_\ast $ is determined explicitly in \eqref{def of Omega_ast} below. The following remark provides a simple example in which this theorem is applicable.

\begin{remark}[Positive diagonal and $1/2$-H\"older regularity]
\label{rmk:Positive diagonal and 1/2-Holder regularity}
Consider the QVE in the setup $ (\Sx,\Px) = (\1[\10\1,1],\dif x) $. 
Assume {\bf A1-2}. 
If the map $ x \mapsto (a_x,S_x) : [0,1] \to \R \times \Lp{2} $ is piecewise $1/2$-H\"older continuous in the sense of \eqref{def of PW-1/2-Holder}, then similarly as in Remark~\ref{rmk:Piecewise 1/2-Holder continuous rows when a=0} we see that $ \Gamma(\infty) = \infty $. 
If $ S $ also has a positive diagonal \eqref{example:S has pos.diagonal}, then {\bf A3} and {\bf B2} hold with $ L $, $ \rho $, and $ \psi $ depending only on $ \eps $ and $ \lambda $. 
Hence an application of Theorem~\ref{thr:Quantitative uniform bound for general a} yields a bound $ \nnorm{m}_\R \leq \Phi $, where $ \Phi $ depends  only on the constants $ C_1 $ and $ \min_k \abs{I_k} $ from \eqref{def of PW-1/2-Holder} and the constants $ \lambda $ and $ \eps $ from \eqref{example:S has pos.diagonal}, in addition to the model parameters $ \norm{S}_{\Lp{2}\to\BB} $, $\norm{a} $ from {\bf A2}.
\end{remark}

\section{Uniform bounds from $ \Lp{2}$-estimates}

The next result shows that for a fixed $ x $ the corresponding component $ m_x $ of an $ \Lp{2}$-solution $ m $ of the QVE may diverge only if the pair $ (a_x,S_x) \in \R \times \Lp{2} $ is sufficiently far away from most of the other pairs $(a_y,S_y) $, $ y \neq x$.
In order to state this result we introduce the refined versions of the auxiliary function \eqref{def of Gamma},
\bels{def of Gamma_Lambda,x}{
\Gamma_{\!\Lambda,x}(\tau) \,:=\, 
\sqrt{\int_\Sx\,
\Bigl(\,\frac{1}{\tau} + \abs{a_y-a_x} +\norm{S_y-S_x}_2\2\Lambda\2\Bigr)^{\!-2}\msp{-8}\Px(\dif y)\;}
\;,
}
where $ \Lambda \in (0,\infty) $ and $ x \in \Sx $ are considered parameters. 
We remark that \eqref{def of Gamma} is related to this operator by $ \Gamma(\tau) := \inf_x \Gamma_{\msp{-2}1\1,x}(\tau) $.

\begin{proposition}[Converting $ \Lp{2}$-estimates to uniform bounds]
\label{prp:Converting L2-estimates to uniform bounds}
Assume {\bf A1} and {\bf A2}. Suppose the solution of the QVE satisfies an $ \Lp{2}$-bound,
\[ 
\norm{m(z)}_2 \,\leq\, \Lambda 
\,,
\]
for some $ \Lambda < \infty $ and $ z \in \Cp $. Then  
\bels{abs-m_x bounded using Gamma_x and L2-bound}{
\abs{\1m_x(z)} 
\;\leq\;
(\2\Gamma_{\!\Lambda,\1x})^{-1}\msp{-1}(\1\Lambda)
\,,
\qquad x\in \Sx
\,,
}
with the convention that the right hand side if $ \infty $ if $ \Lambda $ is out of the range of $ \Gamma_{\!\Lambda,x} $. 

In particular, if $ a = 0 $ or $ \Lambda \ge 1 $, then the simplified estimate holds:
\bels{simplified L2 to BB conversion}{
\norm{m(z)} \,\leq\, \frac{\,\Gamma^{-1}\msp{-1}(\1\Lambda^{\msp{-1}2})}{\Lambda}
\,.
}
\end{proposition} 

\begin{Proof}
Since $ m $ solves the QVE we have
\bea{
\absbb{\frac{1}{\1m_y\!}} 
\,&=\, 
\absbb{\frac{1}{\1m_x\!} - \frac{1}{\1m_x\!}+\frac{1}{\1m_y\!}\,}
\,=\, 
\absbb{\frac{1}{m_x\!} + a_x-a_y +\avg{\1S_x-S_y,\1m\2}}
\\
&\leq\, 
\absbb{\frac{1}{\1m_x\!}} + \abs{\1a_y-a_x} + \norm{S_y-S_x}_2\1\norm{m}_2
\,,
}
for any $ x,y \in \Sx $. Using $ \norm{m}_2 \leq \Lambda $, we obtain  
\bels{the main line of abs-m_x from L2-bound}{
\Lambda^{\msp{-1}2} 
\,&\ge\, 
\int_\Sx \abs{\1m_y}^2\Px(\dif y)
\,\ge\, \int_\Sx\, \biggl(\frac{1}{\abs{\1m_x}} + \abs{a_y-a_x} + \norm{S_y-S_x}_2\,\Lambda\2\biggr)^{\!-2}\!\Px(\dif y)
\\
&=\; \Gamma_{\!\Lambda,x}(\1\abs{\1m_x}\1)^2
\,.
}
As $ \Gamma_{\!\Lambda,x}(\tau) $ is strictly increasing in $ \tau $ we see from the definition \eqref{def of Gamma_Lambda,x} that this is equivalent to \eqref{abs-m_x bounded using Gamma_x and L2-bound}.

If $ a=0$ or $ \Lambda \ge 1 $, then we can take the factor $ \Lambda^{\!-2} $ outside from last integral on the first line of \eqref{the main line of abs-m_x from L2-bound}. This yields the estimate 
\[
\Lambda^{\msp{-2}-1}\2 \Gamma_{\!1,x}(\1\Lambda\2\abs{\1m_x}) 
\,\leq\, \Gamma_{\!\Lambda,x}(\abs{\1m_x}) 
\,\leq\, 
\Lambda
\,.
\] 
Multiplying by $ \Lambda $ and taking the infimum over $ x $ as a parameter of $ \Gamma_{\!\Lambda,x} $, the left most expression reduces to $ \Gamma(\1\Lambda\abs{m_x}) \leq \Lambda^2 $. 
This is equivalent to \eqref{simplified L2 to BB conversion}.
\end{Proof} 

\begin{Proof}[Proof of the part (ii) and \eqref{uniform bound everywhere when a=0} of Theorem~\ref{thr:Quantitative uniform bounds when a = 0}]
Since $ a = 0 $ the structural $ \Lp{2} $-bound \eqref{L2 bound on m when a=0} reads $ \norm{m(z)}_2 \leq 2/\abs{z} $.  
If $ \Gamma(\infty) > (2/\abs{z})^2 $ then we may use the estimate \eqref{simplified L2 to BB conversion} of Proposition~\ref{prp:Converting L2-estimates to uniform bounds} to convert this $ \Lp{2} $-estimate into an uniform bound, and we obtain \eqref{uniform bound outside z=0}.
The bound \eqref{uniform bound everywhere when a=0} follows by combining this estimate with the part (i) of the theorem.
\end{Proof}

In order to prove Theorem~\ref{thr:Quantitative uniform bound for general a} we need an $ \Lp{2} $-bound also when $ \abs{z} \leq \norm{a} $. 
For this purpose we introduce the following estimate that relies on the property {\bf B2}. 

\begin{lemma}[Quantitative $ \Lp{2} $-bound]
\label{lmm:Quantitative L2-bound}
If {\bf A1-3} and {\bf B2} hold, then 
\bels{L2-bound from B2}{
\;
\sup_{z \2\in\2 \Cp} \;\norm{\1m(z)}_2 
\;\leq\; 
\frac{\norm{S}^{2L-2}\norm{S}_{\Lp{2}\to\BB}}{\rho^{\12}}
\biggl(
\psi^{-1/2}\norm{S}_{\Lp{2}\to\BB} +\2 2\2\norm{a} \1+\sqrt{2\1\norm{S}\1}\,
\biggr)
\,.
}
\end{lemma}
\begin{Proof}[Proof of Theorem~\ref{thr:Quantitative uniform bound for general a}]
Using Lemma \ref{lmm:Quantitative L2-bound} we obtain an $ \Lp{2} $-bound \eqref{L2-bound from B2}. We define the threshold,  
\bels{def of Omega_ast}{
\Omega_\ast := \max\setB{\21\1,\text{RHS\eqref{L2-bound from B2}}^2}
}
Applying the simplified estimate \eqref{simplified L2 to BB conversion} of Proposition~\ref{prp:Converting L2-estimates to uniform bounds} yields \eqref{uniform bound for general a}.
\end{Proof}

\begin{Proof}[Proof of Lemma~\ref{lmm:Quantitative L2-bound}]
Let $ \kappa > 0 $ be a parameter to be fixed later. 
We will consider the two regimes $ \abs{z} \leq \norm{a} + \kappa $ and $ \abs{z} \ge \norm{a} + \kappa $, separately. Using the structural $ \Lp{2}$-estimate from Lemma~\ref{lmm:Structural L2-bound}, we see that 
\bels{L2-bound when z is large}{
\norm{m(z)}_2 \,\leq\, \frac{2}{\kappa}\,,\qquad \abs{z} \ge \norm{a}+\kappa 
\,.
}

Let us now consider the regime $ \abs{z} \leq \norm{a} + \kappa $.
Similarly as in \eqref{mSm to L2-norm of F} we estimate the $ \Lp{2} $-norm of $ \abs{m}\1S\abs{m} $ by the spectral norm of the operator $ F = F(z) $, 
\bels{L2-bound:1}{
\inf_x \2(S\abs{m})_x\, \norm{m}_2  \,\leq\, \norm{\1\abs{m}\1S\abs{m}\1}_2 \,\leq\,\norm{F\1e\1}_2 \,\leq\, \norm{F}_{\Lp{2} \to \Lp{2}}
\,,
}
where $ e\in\BB $ with $ e_x = 1 $ for every $ x $. From \eqref{F and alpha} we know that $ \norm{F}_{\Lp{2} \to \Lp{2}} \leq 1 $. 
Let us write $ S_{xy} = \avg{\1S_x\1}\,P_{xy} $, so that $ P_{xy}\pi(\dif y) $, is a probability measure for every fixed $ x $. 
By using  \eqref{L2-bound:1} and Jensen's inequality we get
\bels{L2-bound:2}{
\norm{m}_2 
\,&\leq\,
\sup_x  \frac{1}{\avg{\1S_x}\2\avg{\1P_x,\abs{m}\1}} 
\,\leq\,
\sup_x 
\frac{1}{\avg{\1S_x}} \avgB{P_x,\frac{1}{\abs{m}}}
\,\leq\,
\sup_x 
\frac{1}{\avg{\1S_x}^2}\avgB{S_x,\frac{1}{\abs{m}}}
\\
&\leq\,
\sup_x 
\frac{\norm{\1S_x}_2}{\avg{S_x}^2} \normB{\frac{1}{m}}_2
\,.
}
By writing the last term in terms of the QVE, and using \eqref{averaged row is comparable to BB-norm of S} to estimate $ \avg{\1S_x} \ge \norm{S}^{-L+1}\rho $, we obtain
\bels{L2-bound:3}{
\norm{m}_2 \,\leq\, 
\frac{\norm{S}^{2L-2}\norm{S}_{\Lp{2}\to\BB}}{\rho^{\12}}
\biggl(
\abs{z} + \norm{a} + \norm{Sm}_2
\biggr)
\,.
}
The last term inside the parenthesis can be bounded using the $ \Lp{1}$-norm of $ m $,
\bels{L2-bound:3c}{
\norm{Sm}_2^2 &= \avg{\2m\1,S^2m\1} 
\,\leq\, 
\sup_{x,y} (S^2)_{xy} \abs{\avg{\1m\1}}^2 
\,\leq\, 
\norm{S}_{\Lp{2}\to\BB}^2\avg{\1\abs{m}\1}^2 
\,.
}
Here we have used \eqref{uniform bound on S^2 kernel} for the last inequality.
In order to bound the $ \Lp{1}$-norm, we use the property {\bf B2} to obtain 
\bels{L1-bound for m using B3}{
\avg{\1\abs{m}\1}^2 \,\leq\, \frac{\avg{\1\abs{m},S\abs{m}\1}}{\psi} \leq \frac{\norm{F}_{\Lp{2}\to\Lp{2}}\!}{\psi} 
\,\leq\,
\psi^{-1}
\,.
}
Here we have again expressed the norm of $ m $ in terms of $ F $ and used $ \norm{F}_{\Lp{2}\to\Lp{2}} \leq 1 $.
Using \eqref{L1-bound for m using B3} in \eqref{L2-bound:3c}, and plugging the resulting bound into \eqref{L2-bound:3}, we see that
\bels{final L2 bound for small z}{
\norm{m(z)}_2 
\,\leq\,
\frac{\norm{S}^{2L-2}\norm{S}_{\Lp{2}\to\BB}\!}{\rho^{\12}}\,
\biggl(
\psi^{-1/2}\norm{S}_{\Lp{2}\to\BB} +\2 2\2\norm{a} \1+\2\kappa\,
\biggr)
\,,
}
for every $ \abs{z} \leq \norm{a} +\kappa $.
Choosing $ \kappa := \sqrt{2\2\norm{S}} $ and using \eqref{averaged row is comparable to BB-norm of S} we see that \eqref{L2-bound when z is large} and \eqref{final L2 bound for small z} yield \eqref{L2-bound from B2}.
\end{Proof}

\section{Uniform bound around $z=0 $ when $ a=0$}
\label{sec:Uniform bound around z=0 when a=0}

In this section we prove the part (i) of Theorem~\ref{thr:Quantitative uniform bounds when a = 0}.
It is clear from Lemma~\ref{lmm:Constraints on solution} and \eqref{L2 bound on m when a=0} that $ \Re\,z = 0 $ is a special point for the QVE when $ a = 0 $. 
From \eqref{m symmetry when a=0} we read that in this case the real and imaginary parts of the solution $ m $ of the QVE are odd and even functions of $ \Re\,z $ with fixed $ \Im\,z$, respectively. 
In particular, $\Re\2m(\cI\1\eta) =0 $ for $\eta>0 $, and therefore the QVE becomes an equation for $ v = \Im\,m $ alone,  
\bels{QVE on imaginary axis}{
\frac{1}{v(\cI \2\eta)}\,=\, \eta \,+\,Sv(\cI \2\eta)\,,
\qquad \forall \; \eta \,>\,0
\,.
}
It is therefore not surprising that there is a connection between the well posedness of the QVE at $ z = 0 $ and the question of whether $S$ is {\bf scalable}. 
We call $ S $  scalable if there exists a positive measurable function $h$ on $\Sx$, such that 
\bels{DAD problem for S}{
h_x \2(Sh\1)_x \,=\, 1 \,, \qquad\forall \; x \in \Sx
\,.
}
In other words, there exists a positive diagonal operator $ H $ such that $ HSH $ is doubly stochastic.
In the discrete setup this scalability has been widely studied, see for example Theorem~\ref{thr:General scalability} borrowed from \cite{sinkhorn1967}.
The continuous setup has been considered in \cite{Borwein1994}. 
Here we will show that $ \norm{h} \lesssim 1 $, where the comparison relation is defined w.r.t. the model parameters $ (\norm{S}_{\Lp{2}\to\BB},\varphi,K) $, with $ \varphi $ and $ K $ given in {\bf B1}. 
In order to prove the assertion (i) of Theorem~\ref{thr:Quantitative uniform bounds when a = 0} we use the fact that the solution of the QVE at $ \Re\,z = 0 $ is a minimizer of a functional on positive integrable functions $ \Lp{1}_+$, where
\bels{def of Lp_+}{ 
\Lp{p}_+ := \setb{w \in \Lp{p} : \text{ $w_x > 0 $, for $ \Px$-a.e. $x \in \Sx$}}
\,,\qquad p \in [1,\infty]
\,.
}

\NLemma{Characterization as minimizer}{
Suppose $ S $ satisfies {\bf A1-2} and $ \eta > 0 $. Then the imaginary part $ v(\cI\1\eta) = \Im\,m(\cI\1\eta)$ of the solution of the QVE is $ \Px $-almost everywhere on $\Sx$ equal to the unique minimizer of the functional $ J_\eta : \Lp{1}_+ \to \R $,  
\bels{def of functional J_eta}{
J_\eta(w) \,:=\, \avg{w,Sw} \2-\22\2\avg{\2\log w} \2+\2 2\2\eta \2\avg{w}\,,
}
i.e.,
\[ 
J_\eta(v(\cI\1\eta)) 
\,=\, 
\inf_{w \ins \Lp{1}_+} J_\eta(w)
\,. 
\]
}
The characterization of the solution of the continuous scalability problem as a minimizer has been used with $ \eta = 0 $ in \cite{Borwein1994}. 

We will use the following well known properties of FID matrices. 

\begin{proposition}[Properties of FID matrices \cite{BR97book}]
\label{prp:Properties of FID matrices}
Let $ \brm{T} = (T_{ij})_{i,j=1}^K $ be a symmetric FID matrix. Then the following holds:
\begin{enumerate}
\titem{i} 
If $ \brm{P} $ is a permutation matrix, then $ \brm{P}\brm{T} $ and $ \brm{T}\brm{P} $ are FID; 
\titem{ii} 
There exists a permutation matrix $ \brm{P} $ such that $ (\brm{T}\brm{P})_{ii} > 0 $ for every $ i =1,\dots, K$;
\titem{iii} 
$ (\1\brm{T}^{K-1})_{ij} > 0 $, for every $ 1 \leq i,j \leq K$.
\end{enumerate}
\end{proposition}

The first two properties are trivial. The property (iii) is equivalent to Theorem~2.2.1 in \cite{BR97book}. For more information on FID matrices and their relationship to some other classes of matrices see Appendix~\ref{sec:Scalability of matrices with non-negative entries}.

\begin{Proof}[Proof of the part (i) of Theorem~\ref{thr:Quantitative uniform bounds when a = 0}]
Since $ \brm{Z}$ is a $ K$-dimensional FID matrix with $\sett{0,1}$-entries it follows from the part (iii) of Proposition~\ref{prp:Properties of FID matrices} that $ \min_{i,j}(\brm{Z}^{K-1})_{ij} \ge 1 $.
This implies that $ S $ is uniformly primitive, 
\bels{Z FID implies S satisfies A3}{
(S^{\2K-1})_{xy} \,\geq\, 
\varphi^{\1K-1}
\sum_{i,j=1}^K (\brm{Z}^{\1K-1})_{i j} \,\Ind\sett{x \in I_i, \,y \in I_j}
\,.
}

Showing the uniform bound \eqref{bound around tau=0} on $m$ is somewhat involved and hence we split the proof into two parts. 
First we consider the case $ \Re \2z = 0 $ and show that the solution of the QVE, $ m(\cI\1\eta) = \cI\1v(\cI\1\eta) $, is uniformly bounded.
Afterwards we use a perturbative argument, which allows us to extend the uniform bound on $ m $ to a neighborhood of the imaginary axis. 

Because of the trivial bound $v(\cI \1\eta) \leq \norm{m(\cI\1\eta)} \leq \eta^{-1} $, we restrict ourselves to the case $\eta \leq 1$.

\medskip
\noindent{\scshape Step 1 (Uniform bound at $ \Re\2 z = 0$): } 
Here we will prove
\bels{uniform bound at E=0}{
\sup_{\eta > 0}\,\norm{v(\cI \1\eta)} \,\lesssim\, 1
\,,
} 
where by Convention~\ref{conv:Standard model parameters} the constants $ \varphi $ and $ K $ are considered as additional model parameters. 
As the first step we show that it suffices to bound the average of $ v $ only, since
\bels{BB-bound of v is bounded by avg-v}{
\norm{v(\cI \1\eta)} \,\lesssim\,\avg{\1v(\cI \1\eta)\1}\,, \qquad \forall \2 \eta \in (\10,1\1]\,.
} 
In order to obtain \eqref{BB-bound of v is bounded by avg-v} we recall \eqref{operator norms of S are comparable} and use Jensen's inequality similarly as in \eqref{L2-bound:2},
to get
\[
\frac{1}{\int_\Sx S_{xy}v_y \Px(\dif y)} \,\lesssim\, \int_\Sx \frac{S_{x y}}{v_y}\2 \Px(\dif y)
\,.
\]
This is used for $v=v(\cI\1\eta)$ together with the QVE on the imaginary axis (cf. \eqref{QVE on imaginary axis}) in the chain of inequalities,
\bels{v at zero: L1 bound implies uniform bound}{
v \,=\, \frac{1}{\eta+ Sv} 
\,\leq\,
\frac{1}{Sv} 
\,\lesssim\, 
S\Big(\2\frac{1}{v} \2\Big)
\,=\, S(\eta + Sv)
\,\leq\, \eta+ S^2v
\,\lesssim\, \eta +  \avg{v}
\,.
}
In the last inequality we used the uniform upper bound \eqref{uniform bound on S^2 kernel} on the integral kernel of $S^{\12}$. This establishes \eqref{BB-bound of v is bounded by avg-v}.

In order to bound $ \avg{v} $ we argue as follows:
First we note that
\bels{avg-v bounded by local averages}{
\avg{\1v\1} \,\leq\, \max_{i=1}^K\,\avg{\1v\1}_i 
\,.
}
Here we defined local averages,
\bels{def of avg-w_i}{
\avg{w}_i \,:=\, K\int_{I_i} \!w_x\1\Px(\dif x) \,, \qquad \forall\; i=1, \dots, K
\,,
}
for any $ w \in \Lp{1} $, noting $ \pi(I_i) = K^{-1} $.
Let us also introduce a discretized version $ \wti{J} : (0,\infty)^K \to \R $ of the functional $ J_\eta $ by
\bels{def of discretized minimizer}{
\wti{J}(\brm{w}) \,:=\,\frac{\varphi}{K}\sum_{i,j=1}^K  w_i Z_{ij}w_j -2\sum_{i=1}^K \log w_i
\,,
\qquad \vect{w} \1=\1 (w_i)_{i=1}^K \in (0,\infty)^K\,,
}
where the matrix $ \brm{Z} $ and the model parameter $ \varphi > 0 $ are from {\bf B1}.
The discretized functional is smaller than $ J_\eta $, in the following sense:
\bels{wti-J bounded by J}{
\wti{J}(\2\avg{w}_1,\dots,\avg{w}_K)
\,\lesssim\, J_\eta(w)
\,,\qquad
\forall\,w \in \BB\,,\; w > 0 
\,.
}
To see this we use {\bf B1} to estimate $ S_{xy} \ge \varphi\1Z_{ij}$, $ (x,y) \in I_i \times I_j $, for the quadratic term in the definition \eqref{def of functional J_eta} of $ J_\eta $.
Moreover, we use Jensen's inequality to move the local average inside the logarithm. In other words, \eqref{wti-J bounded by J} follows, since 
\bels{approximation of J_eta by wti-J}{
J_\eta(w) \,&\ge\, \varphi\sum_{i,j=1}^K \pi(I_i)\avg{w}_i Z_{ij}\,\pi(I_j) \avg{w}_j
-2 \sum_{i=1}^K \pi(I_i)\, \avg{ \1\log w}_i
\\
&\ge\,
\frac{1}{K}
\Biggl\{
\frac{\varphi}{K}\!\sum_{i,j=1}^K \avg{w}_iZ_{ij} \avg{w}_j
-2 \sum_{i=1}^K  \,\log \avg{w}_i
\Biggr\}
\\
&=\, 
\frac{1}{K}\,
\wti{J}(\2\avg{w}_1,\dots,\avg{w}_K)
\,,
}
for an arbitrary $ w \in \Lp{1}_+ $.
Since $ K \in \N $ is considered  a model parameter in the statement (ii) of Theorem~\ref{thr:Quantitative uniform bounds when a = 0} the estimate \eqref{wti-J bounded by J} follows. 

Now, by Lemma~\ref{lmm:Characterization as minimizer} the solution  $ v = v(\cI\1\eta) $ of the QVE at $ z = \cI\1\eta $ is the (unique) minimizer of the functional $ J_\eta : \Lp{1}_+ \to \R $. 
In particular, it yields a smaller value of the functional than the constants function, and thus
\[
J_\eta(v) \,\leq\, J_\eta(1) \,=\, 1+2\1\eta \,\leq\, 3
\,.
\]
Combining this with \eqref{wti-J bounded by J} we see that 
\bels{upper bound on wti-J on local v-avgs}{
\wti{J}(\2\avg{v}_1,\dots,\avg{v}_K)
\,\leq\, 
3\1K\,\sim\, 1 \,. 
}
Now we apply the following lemma which relies on $ \brm{Z} $ being FID. The lemma is proven in Appendix \ref{ssec:Variational bounds when Re z = 0}.

\begin{lemma}[Uniform bound on discrete minimizer]
\label{lmm:Uniform bound on discrete minimizer}
Assume $ \brm{w}:=(w_i)_{i=1}^K \in (0,\infty)^K $ satisfies  
\[
\wti{J}(\brm{w}) \,\leq\, \Psi\,,
\]
for some $ \Psi < \infty $, where $ \wti{J} : (0,\infty)^K \to \R $ is defined in \eqref{def of discretized minimizer}.
Then there is a constant $ \Phi < \infty $ depending only on $ (\Psi,\varphi,K) $, such that
\bels{}{
\max_{k=1}^K w_k \leq \Phi
\,.
}
\end{lemma}

From \eqref{upper bound on wti-J on local v-avgs} we see that we can apply Lemma~\ref{lmm:Uniform bound on discrete minimizer} to the discretized vector $ \brm{v} := (\2\avg{v}_1,\dots,\avg{v}_K\1) $, with $ \Psi := 3\1K \sim 1 $, and obtain:
\[ 
\max_{i=1}^K\,\avg{v}_i
\,\lesssim\, 1 \,.
\]
Plugging this into \eqref{avg-v bounded by local averages} and the resulting inequality for $ \avg{v} $ into \eqref{BB-bound of v is bounded by avg-v} yields the chain of bounds, $ \norm{v} \lesssim \avg{v} \leq \max_k \avg{v}_k  \lesssim 1  $. This completes the proof of \eqref{uniform bound at E=0}

\medskip
\noindent{\scshape Step 2 (Extension to a neighborhood):}
It remains to show that there exists $ \delta \sim 1 $, such that
\bels{Stability at zero}{
\norm{m(\tau+\cI\1\eta\1)-m(\cI\1\eta)} \;\lesssim\; \abs{\tau}\;,
\qquad
\text{when}\quad
\abs{\tau} \leq \delta
\,.
}
Here $ \Phi := \sup_\eta \norm{m(\cI\1\eta)} < \infty $ is considered a model parameter. In particular, the bound \eqref{unif bound of m} on $|m(\cI\1\eta)|=v(\cI\1\eta)$ implies $v(\cI\1\eta)\sim 1$. 
By \eqref{B-inv norm bound on BB at E-line} of Lemma~\ref{lmm:Bounds on B-inverse} we find $\norm{B(\cI\1\eta)^{-1}}\lesssim 1$.
The bound \eqref{Stability at zero} follows now from Lemma~\ref{lmm:Stability when m and inv-B bounded} by choosing $ z = \cI\1\eta $ and $ d_x = \tau $. 
Indeed, the lemma states that with the abbreviation 
\[ 
h(\tau) \,:=\, m(\1\cI\1\eta + \tau\1)-\2\cI\1v(\1\cI\1\eta)\,, 
\] 
the following holds true.
If $ \norm{h(\tau)} \leq c_0 $ for a sufficiently small constant $c_0\sim 1$, then actually $ \norm{h(\tau)} \leq C_1\1 \abs{\tau} $ for some large constant $C_1 $ depending only on $\Phi$ and the other model parameters. 

The Stieltjes transform representation \eqref{m as stieltjes transform} implies that 
$h(\tau)$
is a continuous function in $ \tau $. 
As $ h(0) = 0 $, by definition, the bound $ \norm{h(\tau)} \leq C_1\1\abs{\tau} $ applies as long as $ C_1\1\abs{\tau}\leq c_0 $ remains true.
With the choice $\delta := c_0/C_1$ we finish the proof of \eqref{bound around tau=0}.
\end{Proof}

\chapter{Regularity of solution}
\label{chp:Regularity of solution}

We will now estimate the complex derivative $\partial_z m$ on the upper half plane $ \Cp $. When $ \nnorm{m}_\R < \infty $ these bounds turn out to be uniform in $ z $. This makes it possible to extend the domain of the map $ z \mapsto m(z) $ to the closure $\eCp=\Cp\cup\R $.
Additionally, we prove that the solution and its generating density are $ 1/3$-H\"older continuous (Proposition~\ref{prp:Holder regularity in z and extension to real line}), and analytic (Corollary~\ref{crl:Real analyticity of generating density}) away from the special points $ \tau \in \supp\,v $ where $ v(\tau) = 0 $. 
Combining these two results we prove Theorem~\ref{thr:Regularity of generating density} at the end of this chapter. 
Even if the uniform bound, $ \nnorm{m}_\R < \infty $, is not available we still obtain weaker regularity for the averaged solution $ \avg{m} $.
The analyticity of the solution of the QVE is not restricted to the variable $ z $ alone. In Proposition~\ref{prp:Analyticity} we show that the QVE perturbed by a small element $ d \in \BB$ still has a unique solution $ g = g(z,d) $ close to $ m(z) $ that depends analytically on $ d $ provided $ z $ is not close to a point $ \tau \in \supp\,v $ with $ v(\tau) = 0 $.

At the technical level, the proofs of both the H\"older-continuity and the analyticity of $ m $ boil down to considering 
small, in fact infinitesimally small, perturbations of the QVE and then applying the estimates from Section~\ref{sec:Stability and operator B}.

\begin{proposition}[H\"older regularity in $ z $ and extension to real line]
\label{prp:Holder regularity in z and extension to real line}
Assume \\ \emph{\bf A1-3}. For an interval $ I \subset \R $ and a constant $ \eps > 0 $, set
\bels{holder-domain}{
\DD \,:= \setb{z \in \Cp : \dist(\2z\1,[-2\1\Sigma,\22\1\Sigma\2]\1\backslash\1I\,) \ge \eps \2}
\,.
}
Then the following hold:
\begin{itemize}
\titem{i} 
If there is $ \Lambda < \infty $, such that 
\bels{L2-bound for Re z in I}{
\qquad
\norm{m(z)}_2 \leq \Lambda\,,
\qquad
\Re\,z \in  I
\,,
}
then the averaged solution of the QVE is uniformly H\"older-continuous, 
\bels{generic m: 1/13-Holder continuity of avg-m}{
\abs{\1\avg{\1m(z_1)} - \avg{\1m(z_2)}} 
\,\lesssim\,
\abs{\1z_1-z_2}^{1/13},
\qquad
z_1,z_2 \in \DD
\,,
}
where $ \eps $ and $ \Lambda $ are considered additional model parameters.
\titem{ii}
If \eqref{L2-bound for Re z in I} is replaced by the uniform bound,  $ \nnorm{m}_I \leq \Phi < \infty $, then the H\"older continuity is improved to
\bels{bounded m: 1/3-Holder continuity of m}{
\norm{m(z_1)-m(z_2)} 
\,\lesssim\,
\abs{z_1-z_2}^{1/3},
\qquad
z_1,z_2 
\in \DD
\,,	
}
where $\eps$ and $\Phi $ are considered additional model parameters.
\end{itemize}
\end{proposition}

We remark that if $ S $ satisfies {\bf B2} (cf. Chapter~\ref{chp:Uniform bounds}), in addition to {\bf A1-3}, then 
Lemma~\ref{lmm:Quantitative L2-bound} provides an effective upper bound $ \Lambda $ for the $ \Lp{2}$-norm of $ m(z) $, with $ I = \R $.
Similarly, quantitative uniform bounds can be obtained using Theorem~\ref{thr:Quantitative uniform bounds when a = 0} and Theorem~\ref{thr:Quantitative uniform bound for general a}. 
In a slightly different setup a qualitative version of the $1/3$-H\"older continuity \eqref{bounded m: 1/3-Holder continuity of m} was established in \cite{AEK1cpam} as Proposition~5.1.

\begin{convention}[Extension to real axis] 
\label{conv:Extension to real axis}
When $ m $ is uniformly bounded everywhere, i.e., $ \nnorm{m}_\R < \infty $, then \eqref{bounded m: 1/3-Holder continuity of m} guarantees that $ m $ can be extended to the real axis. 
We will then automatically consider $ m $, and all the related quantities as being defined on the extended upper half plane $ \eCp = \Cp \cup \R $. 
\end{convention}

In the proof of the part (ii) of Proposition~\ref{prp:Holder regularity in z and extension to real line} we actually show the following estimate on the derivative of $ m(z) $.

\NCorollary{Bound on derivative}{
In Proposition~\ref{prp:Holder regularity in z and extension to real line} the inequality \eqref{bounded m: 1/3-Holder continuity of m} can be replaced by a stronger bound,
\[
\qquad
\bigl(\2\abs{\sigma}\1\avg{\2\Im\,m\1}+ \avg{\2\Im\,m\1}^2\2\bigr)\1\norm{\1\partial_z m} \,\leq\, C_0
\,,\quad
\text{on}\quad\DD
\,.
\]
Here the function $ \sigma $ is from \eqref{def of sigma(z)} and $ C_0 $  depends on the model parameters from the part (ii) of Proposition~\ref{prp:Holder regularity in z and extension to real line}.
}

The proof of Proposition~\ref{prp:Holder regularity in z and extension to real line} also yields a regularity result for the mean generating measure when $ a = 0 $.

\NCorollary{Regularity of mean generating density}{
Assume {\bf A1-3}, and suppose $ a = 0 $.
Then the  normalized mean generating measure 
\bels{normalized mean generating measure}{  
\nu(\dif \tau) \,:=\, \frac{1}{\pi}\avg{\1v(\dif \tau)} 
\,,
}
has the representation 
\bels{nu-representation}{
\nu(\dif \tau) \;=\; \wti{\nu}(\tau) \2 \dif \tau \2+\2 \nu_0\2\delta_0(\dif \tau)
\,,
}
where $ 0\leq \nu_0 \leq 1 $, and the Lebesgue-absolutely continuous part $\wti{\nu}(\tau)$ is symmetric in $ \tau $, and locally H\"older-continuous on $\R\backslash\sett{0} $. 
More precisely, for every $\eps>0$, 
\bels{nu is 1/13-Holder}{
\abs{\2\wti{\nu}(\tau_2)-\1\wti{\nu}(\tau_1)\1} 
\,\lesssim\,
\abs{\tau_2-\tau_1}^{1/13}\,,
\qquad \forall\;\tau_1,\tau_2 \in \R \backslash (-\1\eps\1,\2\eps)
\,,
}
where $ \eps $ is an additional model parameter. 

If additionally, {\bf B1} holds then $ \nu_0=0 $ in \eqref{nu-representation} and \eqref{nu is 1/13-Holder} holds for all $ \tau_1,\tau_2 \in \R $ with $ C_3 $ depending only on the model parameters from {\bf A1-2} and {\bf B1}.
}

\begin{Proof}
As an intermediate step of the proof of Proposition~\ref{prp:Holder regularity in z and extension to real line} below, we identify $ \wti{\nu}|_I $ as the uniformly $ 1/13$-H\"older continuous extension of $ \avg{v} $ to any real interval $ I $ such that \eqref{L2-bound for Re z in I} holds.

Let us now assume {\bf A1-3}, and fix some $ \eps > 0 $. By \eqref{L2 bound on m when a=0} we have the uniform $ \Lp{2}$-estimate $ \norm{m(z)}_2 \leq 2\2\eps^{-1} $, for $ z \in \Cp $ satisfying $ \abs{\1\Re\,z} \ge \eps $. In other words, the hypothesis \eqref{L2-bound for Re z in I} of Proposition~\ref{prp:Holder regularity in z and extension to real line} holds with $ \Lambda = 2\2\eps^{-1} $ and $ I := \R\backslash (-\eps,\,\eps\1) $, and thus both \eqref{nu-representation} and \eqref{nu is 1/13-Holder} follow.

If {\bf B1} is assumed in addition to {\bf A1-3}, then the part (i) of Theorem~\ref{thr:Quantitative uniform bounds when a = 0} implies that $ \norm{m(z)} \leq \Phi $ when $ \abs{\Re\,z} \leq \delta $, for some $ \Phi,\delta \sim 1 $. 
Combining this with the $ \Lp{2}$-estimate $ \norm{m(z)}_2 \leq 2/\delta $ valid for $ \abs{\Re\,z} \ge \delta $, we see that Proposition~\ref{prp:Holder regularity in z and extension to real line} is applicable with $ I = \R $ and $ \Lambda := \max\sett{\2\Phi,2\2\delta^{-1}} $.
\end{Proof}

\begin{Proof}[Proof of Proposition~\ref{prp:Holder regularity in z and extension to real line}]
The solution $ m $ is a holomorphic function from $\Cp$ to $\BB$ by Theorem~\ref{thr:Existence and uniqueness}. 
In particular, if $ \abs{z} > 2\1\Sigma $, then the claims of the proposition follow trivially from \eqref{m as stieltjes transform} and \eqref{bound on supp v}. 
Thus we will assume  $ \abs{z} \leq 2\1\Sigma $ here.

Taking the derivative with respect to $z$ on both  sides of \eqref{QVE} yields
\[
(\1 1-m(z)^2\msp{-1}S\1)\1 \partial_z m(z) \,=\, m(z)^2\,,
\qquad \forall \; z \in \Cp
\,.
\]
Expressing this in terms of the operator $ B = B(z)$ from \eqref{def of B and a}, and suppressing the explicit $ z $-dependence, we obtain 
\bels{equation for derivative of v}{
\cI\12 \2 \partial_z v \,=\, \partial_zm\,=\,\abs{m}\2B^{\1-1}\msp{-1}\abs{m}
\,.
}
Here we have also used the general property $\partial_z \phi = \cI\12\1\partial_z \Im\,\phi $, valid for all  analytic functions $ \phi : \KK \to \C $, $ \KK \subset \C $, to replace $ m $ by $ v = \Im\,m$.

\medskip
\noindent{\scshape Case 1 }(No uniform bound on $ m $): 
Consider $ z \in \Cp $ satisfying $ \abs{z} \leq 2\1\Sigma $ and $ \Re\,z \in I $. 
Taking the average of \eqref{equation for derivative of v} yields 
\[ 
2\1\cI\1\partial_z\avg{v} =  \avg{\2\abs{m}\1,\1B^{-1}\abs{m}\2} 
\,,
\]
where $ v = \Im\,m $ by \eqref{v_x(z) as eta-regularization of measure v_x}, and thus
\bels{Holder continuity for avg-v emerges}{
\absb{\partial_z\avg{v}} 
\,\leq\,
2^{-1} 
\norm{m}_2\,\norm{B^{-1}}_{\Lp{2}\to \Lp{2}}\norm{m}_2
\,\lesssim\,
\avg{\1v\1}^{-12}
\;,  
\qquad 
\Re\,z \in I
\,. 
} 
In the last step we used \eqref{m L2-bounded: B-inv norm bound on L2 and BB} to get $ \norm{B^{-1}}_{\Lp{2}\to \Lp{2}} \lesssim \avg{v}^{-12} $. This is where the assumption \eqref{L2-bound for Re z in I} was utilized. 
The bound \eqref{Holder continuity for avg-v emerges} implies that $z \mapsto \avg{v(z)}$ is uniformly $ 1/13$-H\"older-continuous when $ \Re\,z \in I $. 
Consequently, the probability measure $ \nu $  has a Lebesgue-density on $ I $, 
\bels{Leb-density for nu on I}{
\wti{\nu}(\tau) \,=\,
\frac{\nu(\dif \tau)}{\dif \tau} \,=\, \frac{1}{\pi}\lim_{\eta \downarrow 0}\,\avg{\1v(\tau \msp{-2}+\msp{-1} \cI\1\eta\1)\1}\,,
\qquad\tau \in I\,,
}
and this density inherits the uniform H\"older continuity from \eqref{Holder continuity for avg-v emerges}. 

It remains to extend this regularity from the mean generating measure $ \nu|_I $ to its Stieltjes transform $ \avg{m}|_\DD $.
To this end, let us denote the left and right end points of the real interval $ I $ by $ \tau_- $ and $ \tau_+ $, respectively.  
Let us split $ \nu $, into two non-negative measures, 
\[
\nu\,=\,\nu_1+\nu_2
\,.
\]
Here the first measure is defined by  $ \nu_1(\dif \tau) := \varphi(\tau)\2\nu(\dif \tau)$, with the function $ \varphi : \R \to [0,1] $, being a piecewise linear such that, $ \varphi(\tau) = 0 $ for $ \tau\in  \R\1\backslash \1[\2\tau_-\!+\eps/3\1,\tau_+\!-\eps/3\2] $, $ \varphi(\tau) = 1 $ when $ \tau_- +(2/3)\1\eps  \leq \tau \leq \tau_+-(2/3)\1\eps $, and linearly interpolating in between.
It follows, that $\nu_1$ has a Lebesgue-density $\wti{\nu}_1$ and is supported in $[-\Sigma,\2\Sigma\2]$, since 
$\supp v \subseteq [-\Sigma,\2\Sigma\2]$ by Theorem~\ref{thr:Existence and uniqueness}. 
Furthermore,
\bels{}{
\abs{\2\wti{\nu}_1(\tau_1)-\wti{\nu}_1(\tau_2)} 
\,\lesssim\,
\abs{\tau_1-\tau_2}^{1/13}\,,
\qquad \forall\; \tau_1,\tau_2 \in \R\,.
}
For the measure $ \nu_2$ we know that $\nu_2(\R) \leq \nu(\R) = 1 $, and 
\[
\supp\,\nu_2 
\,\subseteq\, 
\bigl[-\Sigma,\tau_- \!+\tsfrac{2}{3}\1\eps\2\bigr]
\,\cup\,
\bigl[\2\tau_+ \!-\tsfrac{2}{3}\1\eps\2,\Sigma\2\bigr]
\,,
\]
where one of the intervals may be empty, i.e., $ [\1\tau',\tau''] := \emptyset $, for $ \tau' > \tau'' $.
The Stieltjes transform 
\[
\avg{\1m(z)} \,=\, \int_\R \frac{\,\nu(\dif \tau)}{\tau-z}
\,,
\]
is a sum of the Stieltjes transforms of $ \nu_1 $ and $  \nu_2 $.
The Stieltjes transform of $\nu_1$ is H\"older-continuous with H\"older-exponent $1/13$ since this regularity is preserved under the Stieltjes transformation. For the convenience of the reader, we state this simple fact as Lemma~\ref{lmm:Stieltjes transform inherits regularity} in the appendix. 
On the other hand, since $ \Re\,z $ is away from the support of $ \nu_2 $, the Stieltjes transform of $\nu_2$ satisfies
\[
\absB{\partial_z \int_\Sx \frac{\nu_2(\dif \tau)}{\tau-z}}
\,\leq\,\frac{9}{\eps^{\12}\msp{-5}} \;\lesssim\, 1\;,
\qquad\text{when}\quad
z \in \DD
\,,
\]
and hence \eqref{generic m: 1/13-Holder continuity of avg-m} follows.

\medskip
\noindent{\scshape Case 2 }(solution uniformly bounded): 
Now we make the extra assumption $ \nnorm{m}_I \leq \Phi \sim 1 $, $ I := [\tau_-,\tau_+] \subseteq \R$.
Taking the $ \BB$-norm of \eqref{equation for derivative of v}  immediately yields
\bels{1/3-Holder formula for derivative of v}{
\abs{\1\partial_zv_x(z)}
\,\leq\, 
\norm{m(z)}^2\norm{B(z)^{-1}}
\,\lesssim\,
\avg{v(z)}^{-2} \,\sim\, v_x(z)^{-2} 
\,.
}
Here we used \eqref{B-inv norm bound on BB at E-line} to estimate the norm of $ B^{-1} $, and the part (ii) of Proposition~\ref{prp:Estimates when solution is bounded} to argue that $ v(z) \sim \avg{\1v(z)} $.
We see that $ z \mapsto v_x(z) $ is $ 1/3$-H\"older continuous uniformly in $ z \in I + \cI\2(0,\infty) $ and $ x \in \Sx $.
Repeating the localization argument used to extend the regularity of $ \wti{\nu} = \pi^{-1}\avg{\1v\1} $ to the corresponding Stieltjes transform yields \eqref{bounded m: 1/3-Holder continuity of m}.
\end{Proof}

\begin{Proof}[Proof of Corollary~\ref{crl:Bound on derivative}]
Using all the terms of \eqref{B-inv norm bound on BB at E-line} for the second bound of \eqref{1/3-Holder formula for derivative of v} and using \eqref{equation for derivative of v} to estimate $ \abs{\partial_zm} \sim \abs{\partial_zv} $ yields the derivative bound of the corollary.
\end{Proof}

Next we show that the perturbed QVE \eqref{perturbed QVE - 1st time} has a unique solution.
For the statements of this result we introduce a shorthand 
\[ 
\DD_\BB(h,\rho) := \setb{g \in \BB: \norm{g-h} < \rho} 
\,,
\] 
for the open $ \BB $-ball centred at $ h $ with radius $ \rho > 0 $. 
We also recall that for complex Banach spaces $ X $ and $ Y $, a map $ \phi : U  \to Y $ is called \emph{holomorphic} on an open set $ U \subset X $ if for every $ x_0 \in U $, every $ x_1 \in X $, 
and every bounded linear functional $ \gamma \in X' $ the map $ \zeta \mapsto \avg{\gamma,\phi(x_0+\zeta\2x_1\1)} : \C \to \C $ defines a holomorphic function in a neighborhood of  $ \zeta = 0 $. This is equivalent (cf. Section 3.17 of \cite{HillePhillips}) to the existence of a  Fr\'echet-derivative of $ \phi $ on $ U $, i.e., for every $ x \in U $ there exists a bounded complex linear operator $ D\phi(x) : X \to Y $, such that
\[
\frac{\norm{\1\phi(x+d)-\phi(x) - D\phi(x)\1d\2}_Y\!}{\,\norm{d}_X\!} \,\to\, 0
\,,
\quad\text{as}\quad\norm{d}_X \to 0
\,.
\] 

\begin{proposition}[Analyticity]
\label{prp:Analyticity}
Assume {\bf A1-3}, and consider a fixed $ z \in \overline{\Cp} $ satisfying $ \abs{z} \leq 2\2\Sigma $, where $ \Sigma := \norm{a}+ 2\1\norm{S}^{1/2} $,  such that
\bels{holo:m and Binv norm}{
\norm{m(z)} \,\leq\, \Phi\,,
\qquad\text{and}\qquad
\norm{B(z)^{-1}} \,\leq\, \Psi 
\,,
} 
for some constants $\Phi< \infty $ and $ \Psi \ge 1 $.
Let us define
\bels{condition for delta in terms of eps}{
\eps \,:=\, \frac{1}{\13\2\Sigma + 9\1\norm{S}\Phi\Psi}
\,,\qquad\text{and}\qquad
\delta \,:=\,\frac{\eps}{8\2\Phi^2\Psi\1}
\,.
}

Then there exists a holomorphic map $ d \mapsto g(z,d) : \DD_\BB(0,\delta) \to \DD_\BB(m(z),\eps) $, where $ g = g(z,d) $ is the unique solution of the perturbed QVE,
\bels{g of d and z QVE}{
-\frac{1}{g}\,=\, z + a  + Sg + d
\,,
}
in $ \DD_\BB(m(z),\eps) $. 
The Fr\'echet-derivative $ Dg(z,d) $ of $ g(z,d) $ w.r.t. $ d $ is uniformly bounded, $ \norm{Dg(z,d)} \leq  8\1\Psi\1\Phi^2/\norm{S} $. In particular, 
\bels{Lipschitz bound in d}{
\norm{g(z,d)-m(z)} \,\leq\, \frac{8\1\Psi\Phi^2}{\norm{S}}\,\norm{d}
\,,
\qquad \forall\,d \in \DD_\BB(0,\delta)
\,.
} 
\end{proposition} 

Before proving Proposition~\ref{prp:Analyticity} we consider its applications. First we show that apart from a set of special points the generating measure $ v $ has an analytic density on the real line. 

\begin{corollary}[Real analyticity of generating density]
\label{crl:Real analyticity of generating density}
Assume {\bf A1-3}, and consider a fixed $ \tau \in \R $. If additionally, either of the following three sets of conditions are assumed,
\begin{itemize}
\titem{i} $\avg{v(\tau)} > 0 $, and {\bf B2} holds;
\titem{ii} $ \abs{\tau} > \norm{a} $ and $ \avg{v(\tau)} > 0 $;
\titem{iii} $ \tau = 0 $, $ a = 0 $, and {\bf B1} holds, 
\end{itemize}
then the generating density $ v $ is real analytic around $ \tau $.
\end{corollary}

\begin{Proof} 
Since $ \partial_zm(z) = Dg(z,0\1)\1e $, where $ e_x = 1 $ for all $ x \in \Sx $, the result follows immediately from Proposition~\ref{prp:Analyticity} once we have shown that both $ \norm{m(\tau)} < \infty $ and $ \norm{B(\tau)^{-1}} < \infty $ hold.
Actually, it suffices to only prove $ \norm{m(\tau)} < \infty $ and $ \avg{v(\tau)} > 0 $ in all the three cases (i)-(iii). Indeed, with these estimates at hand, the bound \eqref{B-inv norm bound on BB at E-line} of Lemma~\ref{lmm:Bounds on B-inverse} yields $ \norm{B(\tau)^{-1}} < \infty $. 

In the case (i) we first use Lemma~\ref{lmm:Quantitative L2-bound} to obtain $ \sup_{z \in \Cp}\norm{m(z)}_2 \leq \Lambda $, for $ \Lambda \sim 1 $. We then plug this $ \Lp{2}$-bound in the lower bound of the part (i) of Lemma~\ref{lmm:Constraints on solution} to get a uniform lower bound $ \inf_x \abs{m_x(\tau)} \gtrsim \Lambda^{-1} $. Using this in the upper bound of the part (i) of Lemma~\ref{lmm:Constraints on solution} yields $ \norm{m(\tau)} \lesssim \Lambda^{-C}\avg{v}^{-1} \sim 1 $.

In the case (ii)
we first note that $ \abs{\tau} > \norm{a} $ implies $ \dist(\tau,\sett{a_y}) > 0 $, and thus the first inequality of \eqref{unif bound of m} yields $ \inf_x\abs{m_x(\tau)} > 0 $. Plugging this into the second inequality of \eqref{unif bound of m}, and using the assumption $ \avg{v(\tau)} > 0 $, we obtain an uniform bound for $ m(\tau) $. 

In the case (iii), we use the part (ii) of Theorem \ref{thr:Quantitative uniform bounds when a = 0} to get the uniform bound. 
From the symmetry \eqref{m symmetry when a=0} we see that $ m(0) = \cI\1v(0) $. Hence \eqref{unif bound of m} yields  $ \inf_x v_x(0) > 0 $. Feeding this into \eqref{v compared to avg-v} yields $ \avg{v(0)} \sim 1 $.
\end{Proof}

Combining the analyticity and the H\"older regularity we prove  Theorem~\ref{thr:Regularity of generating density}.

\begin{Proof}[Proof of Theorem~\ref{thr:Regularity of generating density}]
Here we assume $ \nnorm{m}_\R \leq \Phi $, with $\Phi < \infty $ considered as a model parameter.
The assertion (i) follows from (ii) of Proposition~\ref{prp:Estimates when solution is bounded}. 

Using the bound \eqref{bounded m: 1/3-Holder continuity of m} of Proposition~\ref{prp:Holder regularity in z and extension to real line}, with $ I = \R $, we see that $ m $ can be extended as a $ 1/3$-H\"older continuous function to the real line. 
Hence, from \eqref{v_x(z) as eta-regularization of measure v_x} we read off that the generating measure must have a Lebesgue-density equal to $ \Im\,m|_\R $. In particular, this density function inherits the H\"older regularity from  $ m|_\R $, i.e., for some  $ C_1 \sim 1 $: 
\bels{v is 1/3-Holder explicit}{
\abs{v_x(\tau')-v_x(\tau)} \,\leq\, C_1\1\abs{\1\tau'-\tau}^{1/3}\,,\qquad \forall\,\tau,\tau' \in \R\,.
}
This proves the part (iii) of the theorem.

Since $ \nnorm{m}_\R \sim 1 $ using Lemma~\ref{lmm:Constraints on solution} we see that   $ v_x(z) \sim v_y(z)$ for $ z \in \eCp $. 

Let $ \tau_0 \in \R $ be such that $v(\tau_0) > 0  $. In order to bound the derivatives of $ v $ at $ \tau_ 0 $ we use \eqref{v is 1/3-Holder explicit} to estimate
\[
\inf \setb{\,\abs{\omega} : v(\tau_0+\omega) = 0,\;\omega \in \R}
\;\ge\, 
C_1^{-3}\avg{v(\tau_0)}^3 =:\,\varrho \,>\, 0 
\,.
\]
By Corollary~\ref{crl:Real analyticity of generating density} this implies that $ v $ is analytic on the ball of radius $ \varrho $ centered at $ \tau_0  $. The Cauchy-formula implies that the $ k $-th derivative of $ v $ at $ \tau_0 $ is bounded by $ k!\,\varrho^{\1-k} $. This proves the assertion (ii) of the theorem.
\end{Proof}

\begin{Proof}[Proof of Proposition~\ref{prp:Analyticity}]
As $ z $ is fixed, we write $ m = m(z) $. 
We start with general $ \eps $ and $ \delta $, i.e., \eqref{condition for delta in terms of eps} is not assumed.
Since $ \abs{z} < 2\1\Sigma $, we see directly from the QVE that
$
1/\abs{m}
\leq  \abs{z} + \norm{a}+\norm{S}\norm{m} 
\leq 3\1\Sigma + \norm{S}\Phi\,.
$
Writing $ \abs{w/m} \leq 1 + \norm{1/m}\norm{w-m} $, we thus find that 
\bels{condition 1 for eps}{
\absB{\frac{w}{m}} \2\leq\2 2
\,,\quad \forall\2w \in \DD_\BB(m,\eps)\,,
\quad\text{provided}\quad
\eps \,\leq\, \frac{1}{3\1\Sigma+\norm{S}\1\Phi\2}
\,. 
}
We will assume below that $ \eps $ satisfies the above condition.

Consider now $ z $ and $ d \in \DD_\BB(0,\delta)$ fixed. We will first construct a function $ \lambda \mapsto g(\lambda) : [0,1] \to \DD_\BB(m,\eps) $, such that $ g(\lambda) $ solves \eqref{g of d and z QVE} with $ \lambda\1d $ in place of $ d$, i.e., $ Z(\lambda,g(\lambda)) = 0 $, where 
\bels{holo:def of F}{
Z(\lambda,w) \,:=\, w \,+\, \frac{1}{z+a+Sw+\lambda\1d\1}
\,.
}
Let us define $ R : \DD_\BB(m,\eps) \to \BB $ by 
\bels{RHS of ODE}{
R(w) \,:=\, (\11-w^2S\1)^{-1}(w^2d\1)
\,.
}
The function $ \lambda \mapsto g(\lambda) $ is obtained by solving the  Banach-space valued ODE
\bels{Banach ODE flow}{
\partial_\lambda g(\lambda) \,&=\, R(\1g(\lambda))\,,
\qquad \lambda \ins [\10\1,1]\,,
\\
g(0) \,&=\, m
\,,
}
where $ m = m(z)$.
Indeed, a short calculation shows that if $ \lambda \mapsto g(\lambda) $ solves the ODE, then 
\[
\frac{\dif}{\dif \lambda}Z(\lambda,g(\lambda)) \,=\, 0 
\,.
\]
As $ Z(0,g(0)) = 0 $ by the definition of $ m(z)$, this implies that also $ Z(1,g(1)) = 0 $, which is equivalent to $ g = g(1) $ solving  \eqref{g of d and z QVE}.

We will now find $ \eps,\delta \sim 1 $ such that $ \norm{R(w)} \leq \eps $, for $ w \in \DD_\BB(m,\eps) $ and $ \norm{d}_\BB \leq \delta $. 
Under this condition the elementary theory of ODEs (cf. Theorem 9.1 of \cite{ColemanBanachODE}) yields the unique solution $ g(\lambda) \in \DD_\BB(m,\eps) $ to \eqref{Banach ODE flow}.
We start by estimating the the norm of the following operator
\bels{holo:inverse of 1-uvS}{
(1-u\1w\1S)^{-1} \,=\, (\21+\abs{m}\1B^{-1}D\1)^{-1}\abs{m}B^{-1}\Bigl(
\Bigl(\frac{\abs{m}}{m}\Bigr)^2\frac{\cdot}{\abs{m}}\Bigr)
\,,
}
for arbitrary $u,w \in \DD_\BB(m,\eps) $.
Here, $ D := (\abs{m}/m)^2\2m^{-1}\1(m^2-uw)\1S $.
Since $ m^2-u\1w = m\2(m-u)+ u\2(m-w) $ we get $ \norm{D} \leq 3\2\norm{S}\2\eps $ using \eqref{condition 1 for eps}.
Thus requiring $ \norm{\1\abs{m}B^{-1}D} \leq 3\2\Phi\1\Psi\1\norm{S}\2\eps $ to be less than $ 1/2 $, we see that
\bels{condition 2 for eps}{
\norm{(\21+\abs{m}\1B^{-1}D\1)^{-1}} \,\leq\, 2\,,
\quad\text{provided}\quad
\eps\,\leq\,\frac{1}{6\2\Phi\1\Psi\1\norm{S}}
\,.
}
Using this bound for the first factor on the right hand side of \eqref{holo:inverse of 1-uvS} yields
\bels{bound on the inverse of 1-uvS}{
\norm{(1-u\1w\1S)^{-1}h} \,\leq\, 2\2\Phi\1\Psi\,\norm{h/m}
\,,\qquad
\forall\,h \in \BB
\,,\quad
\forall\,u,w \in \DD_\BB(m,\eps)
\,,
}
provided the condition for $ \eps $ in \eqref{condition 2 for eps} holds.
In order to estimate $ \norm{R(w)} $ for $ w \in \DD_\BB(m,\eps) $ we choose $ u = w $ and $ h = w^2d $ in \eqref{bound on the inverse of 1-uvS}, and get
\bels{norm of R(w)}{
\norm{R(w)} 
\,\leq\, 
2\2\Phi\1\Psi\,\norm{w^2\!/m}\norm{d} 
\,\leq\,
8\2\Phi^2\Psi\2\delta
\,, 
}
where $ \norm{d} \leq \delta $ and  $ \norm{w^2\!/m} = \norm{w/m}^2\norm{m} \leq 4\1\Phi $ were used for the last bound.
With the choice \eqref{condition for delta in terms of eps} for $ \delta $ we see that the rightmost expression in \eqref{norm of R(w)} is less than $ \eps $.
Moreover, if $ \eps $ is chosen according to \eqref{condition for delta in terms of eps}, then the conditions from the estimates \eqref{condition 1 for eps} and \eqref{condition 2 for eps} are both satisfied as  $ \Psi \ge 1 $.
We conclude that the ODE \eqref{Banach ODE flow} has a unique solution in $ \DD_\BB(m,\eps) $ if we choose $ \eps $ and $ \delta $ to satisfy \eqref{condition for delta in terms of eps}.

In order to show that not only the ODE but the perturbed QVE \eqref{g of d and z QVE} in general has a unique solution in $ \DD_\BB(m,\eps) $, we establish a more general stability result.
To this end, assume that $ g,g' \in \DD_\BB(m(z),\eps) $ and $ d,d' \in \DD_\BB(0,\delta) $ are such that $ g $ solves \eqref{g of d and z QVE}, while $ g' $ solves the same equation with $ d $ replaced by $ d' $. Then by definition,  
\bels{g' and g difference}{
(\11-g\1g'S\1)(g'-g)\,=\, g\1g'\1(d'-d)
\,.
}
Applying \eqref{bound on the inverse of 1-uvS} to \eqref{g' and g difference} and recalling $ \abs{\1g/m},\abs{\1g'\msp{-2}/m} \leq 2 $ we obtain 
\bels{holo: bound g'-g using d'-d}{
\norm{g'\msp{-2}-g\1} \,\leq\, 8\2 \Phi^2\Psi\1\norm{\1d'\msp{-2}-d\1} 
\,.
}

The uniqueness of the solution to \eqref{g of d and z QVE} follows now from \eqref{g' and g difference}.
In particular, this implies that the map $ d \mapsto g(z,d) : \DD_\BB(0,\delta) \to \DD_\BB(m(z),\eps) $ is uniquely defined with $ g(z,d) := g(1) $, where $ g(1)$ is the value of the solution of the ODE \eqref{Banach ODE flow} at $ \lambda =1 $.

It remains to show that $ g(z,d) $ is analytic in $ d $. 
To this end, let $ h \in \BB $ be arbitrary, and consider \eqref{g' and g difference} with $ g = g(z,d) $, $ g' = g(z,d') $, where $ d'= d + \xi\2h $ for some sufficiently small $ \xi \in \C $. Using the stability bound \eqref{holo: bound g'-g using d'-d} we argue that the differences $  g-g' $ vanish in the limit $ \xi \to 0 $. Therefore we obtain from  \eqref{g' and g difference}
\[
Dg(z,d)\1h 
\,:=\,  
\lim_{\xi \1\to\1 0} \frac{g(z,d+\xi\2h)-g(z,d)}{\xi} 
\,=\,
(\11-g(z,d)^2S\1)^{-1}(\2g(z,d)^2h\1)
\,,
\]
where $ Dg(z,d) : \BB \to \BB $ is the Fr\'echet-derivative of $ g(z,d) $ w.r.t. $ d $ at $(z,d) $.
\end{Proof}

\chapterl{Perturbations when generating density is small}

In this chapter we analyze the stability of the QVE \eqref{QVE} in the neighborhood of parameters $z$ with a small value of the average generating density $\avg{v(z)}$, against adding a perturbation $d \in \BB$ to the right hand side. In the special case when $d$ is a real constant function, i.e., when $m(z)$ is compared to $m(z+\omega)$, and when $z \in \supp \avg{v}$ with $\avg{v(z)}=0$,  this analysis has been carried out in \cite{AEK1cpam}. 
In that special case the upcoming proofs simplify considerably for the following three reasons. 
First, an expansion in $\alpha$ (cf. Lemma~\ref{lmm:Expansion of B in bad direction}) is not needed. 
Second, we do not need to show that the expansions are uniform in the model parameters. 
Third, the complicated selection process of the roots in Subsection \ref{ssec:Two nearby edges} is avoided as we do not have to consider very small gaps in the support of the generating density.

We will assume in this and the following chapters that $ S $ satisfies {\bf A1-3} and that the solution is uniformly bounded everywhere $ \nnorm{m}_\R \leq \Phi < \infty$. In particular, all the comparison relations (Convention~\ref{conv:Comparison relations, model parameters and
constants}) will depend on:
\begin{equation}
\label{Model parameters for the shape analysis}
\text{'The model parameters'} \,:=\, (\2\rho\1,L,\norm{a}, \norm{S}_{\Lp{2}\to\BB},\Phi\1)
\,.
\end{equation}
Due to the uniform boundedness, $ m $ and all the related quantities are extended to $ \eCp $ (cf. Proposition~\ref{prp:Holder regularity in z and extension to real line}).
Furthermore, these  standing assumptions also imply that Proposition~\ref{prp:Estimates when solution is bounded} is effective, i.e., 
\bels{Phi:basic bounds}{ 
\abs{m_x(z)}\1,\,f_x(z)\1,\,\mrm{Gap}(F(z)) \,\sim\, 1
\,,\quad\text{and}\quad
v_x(z) \,\sim\, \avg{v(z)} \,\sim\, \alpha(z) 
\,,
} 
for every $ \abs{z} \leq 2\1\Sigma $ and $ x \in \Sx $.
In particular, the three quantities $ v, \avg{v},\alpha = \avg{\1f,\sin \am\1}$, can be interchanged at will, as long as only their sizes up to constants depending on the model parameters matter. 

The stability of the QVE against perturbations deteriorates when the generating density becomes small. This can be seen from the explosion in the estimate 
\bels{scaling of norm of inv-B in v}{
\avg{v(\tau)}^{-1} \,\lesssim\; \norm{B(\tau)^{-1}} \lesssim\, \avg{v(\tau)}^{-2} 
\,,
\qquad
\tau \in \supp v|_\R
\,,
}
(cf. \eqref{B-inv norm bound on BB at E-line} and \eqref{beta expanded} below) for the inverse of the operator $B$, introduced in \eqref{def of B and a}. 
This norm appears in the estimates \eqref{bulk perturbations} relating the norm of the rescaled difference,
\bels{def of u - 2nd time}{
u = \frac{g-m}{\abs{m}}
\,,
}
of the two solutions $ g $ and $ m $ of the perturbed and the unperturbed QVE,
\[
-\frac{1}{g} = z + a+ Sg +d
\qquad\text{and}\qquad
-\frac{1}{m} = z + a+ Sm
\,,
\]
respectively, to the size of the perturbation $ d $.

The unboundedness of $ B^{-1} $ in \eqref{scaling of norm of inv-B in v}, as $ \avg{v} \to 0 $, is caused by the vanishing of $ B $ in a  one-dimensional subspace of $ \Lp{2}$ corresponding to the eigendirection of the smallest eigenvalue of $ B $. 
Therefore, in order to extend our analysis to the regime $ \avg{v} \approx 0 $ we decompose the perturbation \eqref{def of u - 2nd time} into two parts:
\bels{def of u again}{
u \,=\, \Theta\2b\,+\,r 
\,.
} 
Here, $ \Theta $ is a scalar, and $ b $ is the eigenfunction corresponding to the smallest eigenvalue of  $ B $.
The remaining part $ r \in \BB $ lies inside a subspace where $ B^{-1} $ is bounded due to the spectral gap of $ F $ (cf. Figure~\ref{Fig:B as perturbation of 1-F}). 
As $ B $ is not symmetric, $ r $ and $ b $ are not orthogonal w.r.t. the standard inner product \eqref{def of L2-inner product and avg} on $ \Lp{2}$. 
The main result of this chapter is Proposition~\ref{prp:General cubic equation} which shows that for sufficiently small $ \avg{v} \leq \eps_\ast $, the $ b $-component $ \Theta $ of $ u $ satisfies a cubic equation, and we identify its coefficients up to the leading order in the small parameters $ \avg{v} $ and $ d $.
We will use the symbol $ \eps_\ast \sim 1 $ as the upper threshold for $ \avg{v} $ and its value will be reduced along the proofs.

\section{Expansion of operator $ B $}
\label{sec:Expansion of operator B}

In this section we collect necessary information about the operator $ B : \BB \to \BB $ defined in \eqref{def of B and a}. 
Recall, that the {\bf spectral projector} $ P_\lambda $ corresponding to an isolated eigenvalue $ \lambda $ of a compact operator $ T $ acting on a Banach space $ X $ is obtained (cf. Theorem 6.17 in Chapter 3 of \cite{Kato-PT}) by integrating the resolvent of $ T $ around a loop $ \Gamma  $ encircling only the eigenvalue $ \lambda  $:
\bels{spectral projector from resolvent}{
P_\lambda \,:=\, \frac{\!-1}{2\1\pi\1\cI}\oint_{\Gamma} (T-\zeta)^{-1}\dif \zeta
\,.
}

\begin{lemma}[Expansion of $ B $ in bad direction]
\label{lmm:Expansion of B in bad direction}
There exists $ \eps_\ast \sim 1 $ such that, uniformly in $z \in \eCp$ with $\abs{z}\leq 2\1\Sigma $, the following holds true:  If 
\[
\alpha \,=\,\alpha(z) \:=\, \avgB{f(z)\1,\frac{\Im\,m(z)}{\abs{m(z)}}} 
\,\leq\, \eps_\ast
\,,
\]
then the operator $ B= B(z) $ has a unique single eigenvalue $ \beta = \beta(z) $ of smallest modulus, so that $ \abs{\beta'}-\abs{\beta} \gtrsim 1 $, $ \forall\,\beta' \in \Spec(B)\backslash \{\beta\} $. 
The corresponding eigenfunction $ b = b(z) $, satisfying $ Bb = \beta\1b $, has the properties
\bels{properties of b}{
\avg{\1f,\1b\2}\,=\,1\,,
\qquad\text{and}\qquad
\abs{\1b_x} \,\sim\, 1\,,\quad\forall\2x \in \Sx 
\,.
}
The spectral projector $ P = P(z) : \BB \to \mrm{Span}\sett{b(z)} $ corresponding to $ \beta $, is given by
\begin{equation}
\label{def of P}
P\1w \,=\, \frac{\avg{\2\overline{b}\1,w\1}}{\avg{\2b^{\12}}}\,b
\,.
\end{equation}
Denoting, $ Q := 1 -P $, we have  
\bels{Bounds on inverse of B}{ 
\norm{B^{-1}} \,\lesssim\, \alpha^{-2},
\qquad\text{but}\qquad
\norm{B^{-1}\Pob} + \norm{(B^{-1}\Pob)^\ast}\,\lesssim\, 1
\,, 
} 
where  $ (B^{-1}Q)^\ast $ is the  $ \Lp{2}$-adjoint of $ B^{-1}Q $.

Furthermore, the following expansions in $ \eta = \Im\,z $ and $ \alpha $ hold true:
\begin{subequations}
\label{B, beta and b expanded in alpha and eta}
\begin{align}
\label{B expanded}
B \;&=\; 1\,-\,F 
\,-\,
2\1\cI\2pf\2\alpha - 2f^2\1\alpha^2 \,+\,\Ord_{\BB\to\BB}(\1\alpha^3+ \eta\2)\,,
\\
\label{beta expanded}
\beta \;&=\; \avg{\1f\1\abs{m}}\frac{\eta}{\alpha} 
\,-\,\cI\12\1\sigma\2\alpha 
\,+\, 2\2(\1\psi-\sigma^2)\2\alpha^2
\,+\,\Ord(\1\alpha^3 + \1\eta\2)\,,
\\
\label{b expanded}
b \;&=\; f 
\,+\, \cI\22\2(1-F)^{-1}\Pob^{(0)}(pf^2)\2\alpha
\,+\,\Ord_{\msp{-2}\BB}(\1\alpha^2 +\1\eta\2) 
\,.
\end{align}
\end{subequations}
If $ z \in \R $, then the ratio $ \eta/\alpha $ is defined through its limit $ \eta \downarrow 0 $.
The real valued auxiliary functions $ \sigma = \sigma(z) $ and $ \psi = \psi(z) \ge 0 $ in \eqref{B, beta and b expanded in alpha and eta}, are defined by
\bels{defs of sigma and psi}{
\sigma \,:=\, \avg{\2pf^{\13}\1}
\qquad\text{and}\qquad
\psi \,:=\, \mcl{D}(\1pf^{\12}) 
\,,  
} 
where the sign function $ p = p(z) $, and the positive quadratic form $ \mcl{D} = \mcl{D}(\genarg;z) $, are given by
\bels{def of p}{
p \,:=\, \sign \Re\,m
}
and 
\bels{def of mcl-D}{
\mcl{D}(w) \,&:=\, 
\avgB{\Pob^{(0)}w,\Bigl[\1(1+\norm{F}_{\Lp{2} \to \Lp{2}})\1(\11-F)^{-1}-1\1\Bigr]\1\Pob^{(0)}w}
\\
&\ge\;
\frac{\mrm{Gap}(F)}{2}\2\norm{Q^{(0)}w}_2^2
\,,
}
respectively.
The orthogonal projector $ Q^{(0)} = Q^{(0)}(z) := 1 - f(z)\2\avg{\1f(z),\genarg} $ is the leading order term of $ Q $, i.e.,  $ Q = Q^{(0)} + \Ord_{\Lp{2}\to\Lp{2}}(\alpha) $. Furthermore, $ \mrm{Gap}(F) \sim 1 $. 

Finally, $ \lambda(z) = \norm{F(z)}_{\Lp{2}\to\Lp{2}} $, $ \beta(z),\sigma(z),\psi(z) $, as well as the vectors $ f(z), b(z) \in \BB  $, are all uniformly $ 1/3$-H\"older continuous functions of $ z $ on connected components of the domain 
\[ 
\setb{z \in \eCp: \alpha(z) \leq \eps_\ast\1,\2 \abs{z}\leq 2\1\Sigma } 
\,,
\] 
where $ \Sigma \sim 1 $ is from \eqref{bound on supp v}. The function $ p $ stays constant on these connected components.
\end{lemma}

Although, $ P $ is not an orthogonal projection (unless $ \overline{b} = b $), it follows from \eqref{properties of b} and \eqref{def of P} that
\bels{P is bounded as a map from L2 to BB}{
\norm{P},\,\norm{P^\ast}\;\lesssim\; 1
\,.
}
Here $ P^\ast = \overline{b}\,\avg{\2b,\genarg}/\avg{\1\overline{b}^{\22}} $ is the Hilbert space adjoint of $ P $.

\begin{Proof}  
Recall that $ \sin \am =  (\Im\,m)/\abs{m} $ (cf. \eqref{def of B and a}), and 
\bels{}{
B \,=\, \nE^{-\cI\12\am}-F 
\,=\, (1-F) + D
\,,
}
where $ D $ is the multiplication operator
\bels{D: from 1-F to B}{
D = -\1\cI\1 2 \cos \am \,\sin \am -2 \sin^2 \am
\,.
}
From the definition of $\alpha = \avg{f\2 \Im\,m/\abs{m}} $, and $ f,\abs{m} \sim 1 $, we see that $ \abs{\sin \am\1} \sim \alpha $, and thus 
\bels{D bound}{
\norm{D}_{\Lp{2}\to \Lp{2}} + \norm{D} \,\leq\, C_0\1\alpha 
\,,
}
for some $ C_0 \sim 1 $.
The formula \eqref{B expanded} for $ B $ follows by expanding $ D $ in $ \alpha $ and $ \eta $ using the representations \eqref{sin a = alpha f + eta t} and \eqref{expansion of cos a} of $ \sin \am $ and $ \cos \am $, respectively.
In particular, from \eqref{expression for t} we know that $ \norm{t} \lesssim 1 $, and thus $ \sin \am = \alpha f + \Ord_\BB(\eta) $. 

Let us first consider the operators acting on the space $ \Lp{2} $. By Proposition~\ref{prp:Estimates when solution is bounded} the operator $ 1-F $ has an isolated single eigenvalue of smallest modulus equal to
\bels{beta_0}{
1 -\norm{F}_{\Lp{2}\to\Lp{2}} \,=\, \frac{\eta}{\alpha}\avg{\1\abs{m}f\1}
\,,
}
and the $ \Lp{2}$-spectrum of $ 1-F $ lies inside the set
\bels{def of LL}{ 
\LL \,:=\, 
\setb{1 -\norm{F}_{\Lp{2}\to\Lp{2}}}\,\cup\,
\bigl[\21 -\norm{F}_{\Lp{2}\to\Lp{2}}+\mrm{Gap}(F)\1,\,2\2\bigr]
\,.
} 
Here the upper spectral gap of $ F $ satisfies $ \mrm{Gap}(F) \sim 1 $ by (iv) of Proposition~\ref{prp:Estimates when solution is bounded}.

The properties of $ \beta $ and $ b $, etc., are deduced from the resolvent of $ B $ by using the analytic perturbation theory (cf. Chapter 7 of \cite{Kato-PT}).
To this end denote $ R(\zeta) := (1-F-\zeta)^{-1}$, so that
\[
(B-\zeta\1)^{-1} =\, (\11+R(\zeta)D\1)^{-1}R(\zeta)\,. 
\]  
We will now bound  $ R(\zeta) = -(\1\wht{F}(\abs{m})-(1-\zeta)\1)^{-1} $ as an operator on $ \BB $, using the property \eqref{resolvent of F on L2 and BB} of the resolvent of the $ F $-like operators $ \wht{F}\, $ (cf. \eqref{def of wht-F})
\bels{BB bound of resolvent of 1-F}{
\norm{R(\zeta)}  
\,\leq\,
 \frac{1+\Phi^2\norm{R(\zeta)}_{\Lp{2}\to \Lp{2}}}{\abs{\1\zeta-1\2}}
\,.
}
Thus there exists a constant $ \delta \sim 1 $, 
\[
\norm{R(\zeta)} \,\lesssim\, 1
\,,\qquad
\dist(\1\zeta\1,\LL)\ge \delta
\,.
\]
Here we have used the fact that the set $ \LL $ contains both the $ \Lp{2}$-spectrum of $ 1-F $, and the point $ \zeta = 1 $. 
Thus \eqref{BB bound of resolvent of 1-F} shows that $ \LL $ contains also the $ \BB$-spectrum of $ 1-F $.
By requiring $ \eps_\ast $ to be sufficiently small it follows from \eqref{D bound} that $ \norm{(1+R(\zeta)D)^{-1}} \lesssim 1 $ provided $ \zeta $ is at least a distance $ \delta $ away from $ \LL $, and thus 
\bels{BB-bound on B-resolvent}{
\norm{\1(B-\zeta\2)^{-1}} \,\lesssim\, 1 
\,,\qquad
\dist(\1\zeta\1,\LL)\ge \delta
\,.
}
\begin{flushleft}
\begin{minipage}{\textwidth}
By (iv) in Proposition~\ref{prp:Estimates when solution is bounded} we see that $ \mrm{Gap}(F) \gtrsim 1 $.
By taking $ \eps_\ast $ sufficiently small the perturbation $ \norm{D} $ becomes so small that we may take $ \delta  \leq \mrm{Gap}(F)/3 $. It then follows that the eigenvalue $ \beta $ is separated from the rest of the $ \BB$-spectrum of $ B $ by a gap of size $ \delta \sim 1 $.
\begin{wrapfigure}{h}{0.67\textwidth}
\includegraphics[width=0.67\textwidth]{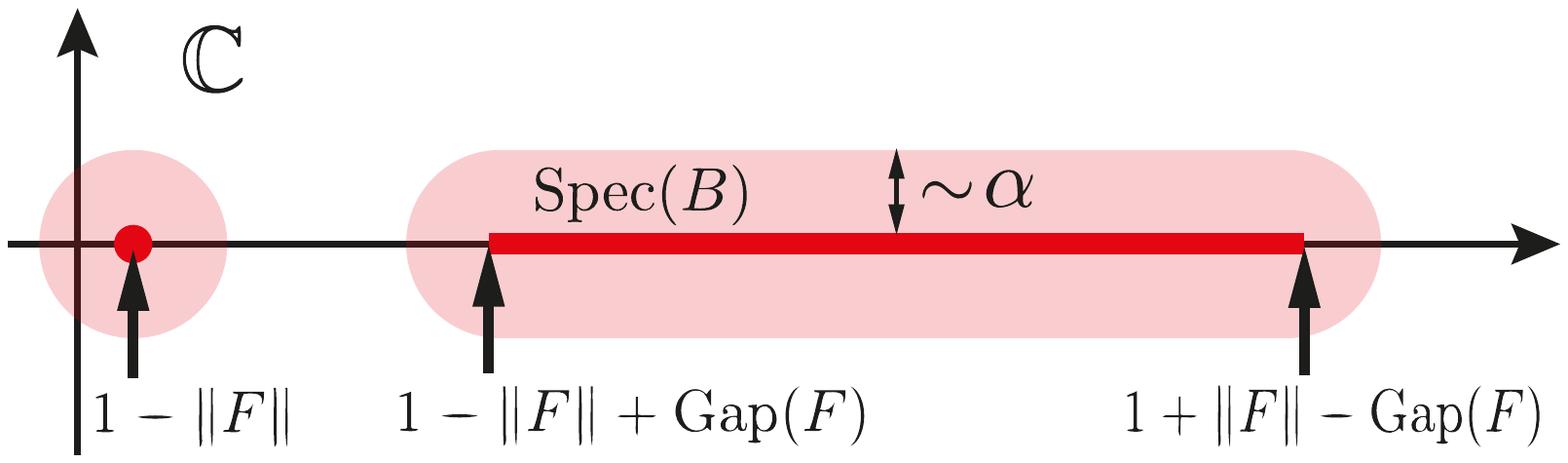}
\caption{The spectrum of $1-F$ lies inside the union of an interval with one isolated point. The perturbation $B$ of $1-F$ has spectrum in the indicated area.}
\label{Fig:B as perturbation of 1-F}
\end{wrapfigure}
Knowing this separation, the standard resolvent contour integral representation formulas (cf. \eqref{spectral projector from resolvent}) imply that $ \norm{b} \lesssim 1 $ and $ \norm{P} \lesssim 1 $, $ \norm{B^{-1}Q} \lesssim 1 $, etc.,  provided the threshold $ \eps_\ast \sim 1 $ for $ \alpha $ is sufficiently small.
Similar bounds hold for the adjoints, e.g.,  $ \norm{(B^{-1}Q)^\ast} \lesssim 1 $. For an illustration how the spectrum of the perturbation $B$ differs from the spectrum of $1-F$, see Figure~\ref{Fig:B as perturbation of 1-F}.
\end{minipage}
\end{flushleft}

Setting $ \beta^{(0)} = 1-\norm{F}_{\Lp{2}\to\Lp{2}} $ and $ b^{(0)} = f $, the formulas \eqref{beta expanded} and \eqref{b expanded} amount to determining the subleading order terms of 
\bels{}{
\beta \,&=\, 
\beta^{(0)} + \beta^{(1)}\alpha + \beta^{(2)}\alpha^2 + \Ord(\alpha^3 + \eta)
\\
b \,&=\, 
b^{(0)} + b^{(1)}\alpha + \Ord_{\!\BB}(\alpha^2 + \eta)
\,,
}
using the standard perturbation formulas.
Writing \eqref{B expanded} as
\[
B = B^{(0)} + \alpha\2B^{(1)} + \alpha^2\1B^{(2)} +\, \Ord_{\!\BB\to\BB}(\1\alpha^3+ \eta\2)
\,,
\]
with $ B^{(0)} = 1-F $, $ B^{(1)} = - 2\1\cI\2pf $, $ B^{(2)} := - 2f^2 $, we obtain 
\bels{beta^(1,2) derived}{
\beta^{(1)} 
\;&=\; \avgb{b^{(0)},B^{(1)}b^{(0)}} 
=\;
-\cI\12\avg{\1p\1f^3} \,,
\\
\beta^{(2)} 
\;&=\;
\avgb{\1b^{(0)},B^{(2)}b^{(0)}} 
\,-\, 
\avgB{\1b^{(0)},\, B^{(1)}Q^{(0)}(B^{(0)}-\beta^{(0)})^{-1}Q^{(0)}B^{(1)}b^{(0)}\1}
\\
&\msp{-50}=\; 
2\2\bigl(\11+\norm{F}_{\Lp{2} \to \Lp{2}}\bigr)\,\avgB{Q^{(0)}(\1pf^2),(1-F)^{-1}Q^{(0)}(\1pf^2)}	
\,-\,
2\,\avgb{f^4}
\,+\,\Ord\Bigl(\2\frac{\eta}{\alpha}\2\Bigr)
\,.
}
These expressions match \eqref{B, beta and b expanded in alpha and eta}.
To get the last expression of $ \beta^{(2)}$ in \eqref{beta^(1,2) derived} we have used $ \norm{Q^{(0)}\!R(\zeta)\1 Q^{(0)}}_{\Lp{2}\to \Lp{2}}  \sim 1 $, $ \zeta \in [\20,\beta^{(0)}] $, and $ \beta^{(0)} \sim \eta/\alpha $, to approximate
\[
(B^{(0)}-\beta^{(0)})^{-1}Q^{(0)}
\;=\;
(1-F)^{-1}Q^{(0)}\,+\,\Ord_{\!\BB\to\BB}\Bigl( \1\frac{\eta}{\alpha} \1\Bigr)
\,.
\]
The formula \eqref{b expanded} follows similarly 
\bea{
b^{(1)} \;&=\,-\2 (B^{(0)}-\beta^{(0)})^{-1}Q^{(0)}B^{(1)}b^{(0)} \;=\; 
\cI\22\2 (\11-F)^{-1}Q^{(0)}(\1pf^2)  \,+\,\Ord_{\!\BB}\Bigl(\frac{\eta}{\alpha}\Bigr)
\,.
}

In order to see that $ \psi \ge 0 $, we use $ \norm{F}_{\Lp{2}\to\Lp{2}} \leq 1 $ to estimate
\bea{
(\11+\norm{F}_{\Lp{2} \to \Lp{2}})\2(\11-F\1)^{-1} 
\,
\ge\, 
1+\frac{\mrm{Gap}(F)}{2} 
\,.
}
This yields the estimate in \eqref{def of mcl-D}. 

It remains to prove the $ 1/3$-H\"older continuity of the various quantities in the lemma. To this end we write
\bels{}{
B(z) \,=\, \nE^{-2\1\am(z)}-\wht{F}(\abs{m(z)})
\,,
}
where the operator $ \wht{F}(r) : \BB \to \BB $ is defined in \eqref{def of wht-F}.
Since $ \norm{S} \leq \norm{S}_{\Lp{2}\to\BB} \sim 1 $ it is easy to see from   \eqref{def of wht-F} that the map $ r \mapsto \wht{F}(r) $ is uniformly continuous when restricted on the domain of arguments $ r \in \BB $ satisfying $ c/\Phi \leq r_x \leq \Phi $.
Furthermore, the exponent $ \nE^{-\cI\12\am} = (\abs{m}/m)^2 $, has the same regularity as $m $ because  $ \abs{m} \sim 1 $. 
Since $ m(z) $ is uniformly $ 1/3 $-H\"older continuous in $ z $ (cf. \eqref{bounded m: 1/3-Holder continuity of m}) we thus have 
\bels{z,z' B-difference}{
\norm{B(z')-B(z)} \,\lesssim\, \abs{z'-z}^{1/3}
\,,
}
for any sufficiently close points $ z $ and $ z' $.
The resolvent $ (B(z)-\zeta)^{-1} $ inherits this regularity in $ z $.

The continuity of $ \beta(z),b(z),P(z) $ in $ z $ is proven by representing them as contour integrals of the resolvent $ (B(z)-\zeta)^{-1} $ around a contour enclosing the isolated eigenvalue $ \beta(z) $.
The functions $ \sigma $ and $ \psi $ inherit the $ 1/3$-H\"older regularity from their building blocks, $ 1-\norm{F}_{\Lp{2}\to\Lp{2}} $, $ f $, $ Q^{(0)} $, and the function $ p $. 
The continuity of the first three follows similarly as that of $ \beta $, $ b $ and $ Q $, using the continuity of the resolvent of $ 1-F(z) $ in $ z $. Also the continuity of the largest eigenvalue $ \lambda(z)$ of $ F(z)$ is proven this way. 
In particular, we see from \eqref{beta_0} that the limit $ \eta/\alpha(z) $ exists as $ z $ approaches the real line.

The function  $ z \mapsto p(z) = \sign \Re \2m(z) $, on the other hand, is handled differently. We show that if $ \eps_\ast > 0 $ sufficiently small, then the restriction of $ p $ to a connected component $ J $ of the set $ \sett{z: \alpha(z)\leq \eps_\ast } $ is a constant, i.e., $ p(z') = p(z) $, for any $ z,z' \in J $.
Indeed, since $ \inf_x \abs{\1m_x(z)} \ge c_0 $, and $ \sup_x \Im\,m_x(z) \leq C_1\1\eps_\ast $, for some $ c_0,C_1 \sim 1 $, we get  
\bels{lower bound on Re m on DDe}{
(\Re\,m_x)^2 \,=\, \abs{m_x}^2-(\1\Im\,m_x)^2 \,\ge\, c_0^{\12}-(C_1\eps_\ast)^2\,,\qquad\forall\1x \in \Sx
\,.
}
Clearly, for a sufficiently small $ \eps_\ast $ the real part $ \Re\,m_x(z) $ cannot vanish. Consequently, the continuity of $ m : \eCp \to \BB $ means that the components $ p_x(z) = \sign\,\Re\,m_x(z) \in \sett{-1,+1} $, may change values only when $ \alpha(z) > \eps_\ast $. 

The explicit representation \eqref{def of P} of the spectral projector $ P $ follows from an elementary property of compact integral operators: If the integral kernel $ (T^\ast)_{xy} $ of the Hilbert space adjoint of an operator $ T :\Lp{2} \to \Lp{2} $, defined by $ (Tw)_y = \int T_{xy}w_y\Px(\dif y) $, has the symmetry $ (T^\ast)_{xy} = \overline{T_{xy}\!}\, $, then the right and left eigenvectors $ v $ and $ v' $ corresponding to the right and left eigenvalues $\lambda $ and $ \overline{\lambda} $, respectively, are also related by the simple component wise complex conjugation: $ (v')_x = \overline{v_x\!}\; $. 
\end{Proof}

\sectionl{Cubic equation}

We are now ready to show that the projection of $ u $ in the $ b $-direction satisfies a cubic equation (up to the leading order) provided $  \alpha $ and $ \eta $ are sufficiently small. Recall, that $ T^\ast $ denotes the $ \Lp{2} $-adjoint of a linear operator $ T $ on $ \Lp{2}$.  

\begin{proposition}[General cubic equation]
\label{prp:General cubic equation}
Suppose $ g \in \BB $ solves the perturbed QVE \eqref{perturbed QVE - 2nd time} at $z \in \eCp$ with $\abs{z}\leq 2\1\Sigma $. 
Set
\begin{subequations}
\label{defs of u, Theta, and r} 
\bels{def of u}{
u \,:=\, \frac{g-m}{\abs{m}}
\,,
}
and define $ \Theta \in \C $ and $ r \in \BB $ by
\bels{defs of Theta and r}{
\Theta := \frac{\avg{\1\overline{b}\1,u\1}}{\avg{\2b^{\12}}}
\qquad\text{and}\qquad
r := Qu\,.
}
\end{subequations}
There exists $ \eps_\ast \sim 1 $ such that if \bels{conditions for general cubic equation}{ 
\avg{\1v\1}\leq \eps_\ast\,,\qquad\text{and}\qquad\norm{g-m} \leq \eps_\ast
\,,
}
then the following holds: 
The component $ r$  is controlled by $ d $ and $ \Theta $, 
\bels{r to leading order}{
r \,&=\,R\1d \,+\, \Ord_{\!\BB}\bigl(\2\abs{\Theta}^2+\norm{d}^2\1\bigr) 
\,,
}
where $ R = R(z) $ denotes the bounded linear operator $ w \mapsto B^{-1}Q(\1\abs{m}\1w\1) $ satisfying
\bels{bound on R and its adjoint on BB}{ 
\norm{R\1} +\, \norm{R^\ast\msp{-1}} 
\;\sim\; 1
\,.
}    

The coefficient $ \Theta $ in \eqref{defs of u, Theta, and r} is a root of the complex cubic polynomial,
\bels{general cubic}{
\mu_3\1 \Theta^{\13} \,+\, \mu_2\1\Theta^{\12} \,+\, \mu_1\1\Theta \,+\, 
\avg{\1\abs{m}\1\overline{b},d\2}
\;=\; 
\kappa(u,d)
\,,
}
perturbed by the function $ \kappa(u,d) $ of sub-leading order. This perturbation satisfies
\bels{kappa bounds}{
\abs{\1\kappa(u,d)}
\;\lesssim\; 
\abs{\Theta}^4 \,+\, \norm{d}^2 \,+\,\abs{\Theta}\2\abs{\avg{\1e,d\1}}
\,,
}
where $ e : \eCp \to \BB $ is a uniformly bounded function, $ \norm{e(z)} \lesssim 1 $, determined by $ S $ and $ a $.
The coefficient functions $ \mu_k :  \eCp \to \C $ are determined by $ S $ and $ a $ as well. They satisfy
\begin{subequations}
\label{mu_k's expanded}
\begin{align}
\label{mu_3 expanded}
\mu_3 \,&:=\; 
\Bigl(1-  \avg{\1f\1\abs{m}}\1\frac{\eta}{\alpha}\2 \Bigr)\2\psi\,+\, \Ord(\alpha)
\\
\label{mu_2 expanded}
\mu_2 \,&:=\; 
\Bigl( 1-\avg{f\1\abs{m}}\1\frac{\eta}{\alpha}\2\Bigr)\2\sigma \,+\, \cI\2(\13\1\psi \,-\,\sigma^2)\2\alpha\,+\,\Ord\bigl(\alpha^2+\eta\1\bigr)
\\
\label{mu_1 expanded}
\mu_1 \,&:=\; 
-\2\avg{f\1\abs{m}}\frac{\eta}{\alpha}\,+\,
\cI\12\1\sigma\2\alpha \,-\, 2\1(\1\psi\1-\1\sigma^2)\2\alpha^2 
\,+\,\Ord\bigl(\alpha^3 + \eta\,\bigr)
\,.
\end{align}
\end{subequations}
If $ z \in \R $, then the ratio $ \eta/\alpha $ is defined through its limit as $ \eta \to 0 $.

Finally, the cubic is {\bf stable} in the sense that
\bels{stability general cubic}{ 
\abs{\mu_3(z)} +\abs{\mu_2(z)} \,\sim\, 1 
\,.
} 
\end{proposition}

Note that from \eqref{defs of Theta and r} and \eqref{def of P} we see that $ \Theta $ is just the component of $ u $ in the one-dimensional subspace spanned by $ b $, i.e, $ Pu = \Theta\2b $.
From \eqref{defs of u, Theta, and r} and \eqref{P is bounded as a map from L2 to BB} we read that $ \abs{\Theta} \leq C_1\1\eps_\ast $ is a small parameter along with $ \alpha $ and $ \eta $. Therefore we needed to expand $ \mu_1 $ to a higher order than $ \mu_2 $, which is in turn expanded to a higher order than $ \mu_3 $ in the variables $ \alpha $ and $ \eta $ in \eqref{mu_k's expanded}.

\begin{Proof}
The proof is split into two separate parts. First, we derive formulas for the $ \mu_k$'s in terms of $ B,\beta$ and $ b $ (cf. \eqref{mu_k's half-expanded} below). 
Second, we use the formulas \eqref{B, beta and b expanded in alpha and eta} from Lemma~\ref{lmm:Expansion of B in bad direction} to expand $ \mu_k$'s further in $ \alpha $ and $ \eta $.

First, we write the equation \eqref{eq. for u} in the form 
\bels{eq. for u abstracted}{
Bu = \mcl{A}(u,u) + \abs{m}\1(1+\nE^{-\cI\1\am}u\1)\1d
\,,
}
where $ \am = \am(z) := \arg\,m(z)$, and  the symmetric bilinear map $ \mcl{A} : \BB^2 \to \BB $, is defined by
\[ 
\mcl{A}_x(h,w) := \tsfrac{1}{2}\2\nE^{-\cI \am_x}\bigl(\1 h_x\2(Fw)_x + (Fh)_x\1w_x \bigr)
\,.
\]
Clearly, $ \norm{\mcl{A}(h,w)} \lesssim \norm{h}\norm{w}$, since $ \norm{F} \leq \norm{m}^2 \lesssim 1  $.
Applying $ Q $ on \eqref{eq. for u abstracted} gives
\bels{formula for r - 1}{
r = B^{-1}Q\mcl{A}(u,u) +  B^{-1}Q\bigl[\2\abs{m}\1(1+\nE^{-\cI\1\am}u\1)\1d\,\bigr]
\,.
}
From Lemma~\ref{lmm:Expansion of B in bad direction} we know that $ \norm{QB^{-1}Q} \lesssim 1 $, and hence the boundedness of $ \mcl{A} $ implies:
\[
\norm{B^{-1}Q\mcl{A}(u,u)} 
\,\lesssim\, 
\norm{u}^2 \,\lesssim\,\abs{\Theta}^2 + \norm{r}^2\,.
\]
From the boundedness of the projections \eqref{P is bounded as a map from L2 to BB} 
\[ 
\norm{r} 
\,=\, 
\norm{Qu} 
\,\lesssim\, 
\norm{u} 
\,\leq\, \frac{\norm{g-m} }{\inf_x \abs{m_x}}
\,\lesssim\, 
\eps_\ast 
\,,
\]
where in the second to last inequality we have used $ \abs{m} \sim 1$.
Plugging this back into \eqref{formula for r - 1}, we find
\[
\norm{r}\,\leq\, C_0\1(\2\abs{\Theta}^2+\eps_\ast\2\norm{r}+\norm{d})
\,,
\]
for some $ C_0 \sim 1 $. Now we require $\eps_\ast$ to be so small that $ 2\1C_0\1\eps_\ast \leq 1 $, and get
\bels{BB-bound of r}{
\norm{r} \,\lesssim\, \abs{\Theta}^2 + \norm{d}
\,.
}
Applying this on the right hand side of $ u = \Theta\2b + r $ yields a uniform bound on $ u $,
\bels{BB-bound of u}{
\norm{u} \,\lesssim\, \abs{\Theta} + \norm{d}
\,.
}
Using the bilinearity and the symmetry of $ \mcl{A} $ we decompose $ r $ into three parts
\bels{r decomposition}{
r \;=\; 
B^{-1}Q\mcl{A}(b,b)\,\Theta^2
\,+\, R\1d \,+\, \wti{r}
\,,
}
where we have identified the operator $ R $ from \eqref{r to leading order}, and introduced the subleading order part,
\bels{}{
\wti{r} \;&:=\; 2\1B^{-1}Q\mcl{A}(b,r)\2\Theta + B^{-1}Q\mcl{A}(r,r) + B^{-1}Q(\1\abs{m}\1\nE^{-\cI \am}u\1d\2) 
\\
&\,=\; \Ord_\BB\Bigl( \2\abs{\Theta}^3 + \abs{\Theta}\norm{d}+\norm{d}^2\Bigr)
\,.
}
Applying the last estimate in \eqref{r decomposition} yields \eqref{r to leading order}. 
We know that $ B^{-1}Q $ is bounded as an operator on $ \BB $ from \eqref{Bounds on inverse of B}. 
A direct calculation using \eqref{def of P} shows that also its $ \Lp{2}$-Hilbert-space adjoint satisfies a similar bound, $ \norm{(B^{-1}Q)^\ast} \lesssim 1 $. 
From this and $ \norm{m} \lesssim 1 $ the bound \eqref{bound on R and its adjoint on BB} follows. 

From \eqref{def of P} we see that applying $\avg{\2\overline{b}\1,\genarg} $ to \eqref{eq. for u abstracted} corresponds to projecting onto the $ b $-direction
\bels{abstract cubic appears}{
&\beta\1\avg{\1b^2}\2\Theta 
\\
&\!=\; \avg{\2\overline{b}\1,\2\mcl{A}(b,b)}\2\Theta^2 \,+\,2\2 \avg{\2\overline{b},\2\mcl{A}(b,r)}\2\Theta \,+\,\avg{\2\overline{b}\1,\2\mcl{A}(r,r)} 
\,+\,
\avgb{\1\overline{b}\1,\2\abs{m}\1(1+\nE^{-\cI\1\am}u\1)\1d\2}\msp{-10} 
\\
&\!=\;
\avg{\2b\2\mcl{A}(b,b)}\2\Theta^2 
\,+\, 
2\2\avgb{\1b\2\mcl{A}(\1b\1,B^{-1}\msp{-1}Q\mcl{A}(b,b))}\2\Theta^3 
\,+\,
\avg{\2b\1\abs{m}\1d\2} 
\,+\,
\kappa(u,d\1) 
\,,
}
where the cubic term corresponds to the part $ B^{-1}Q\mcl{A}(b,b)\2\Theta^2 $ of $ r $ in \eqref{r decomposition}, while the other parts of $ \avg{\1\overline{b},\2\mcl{A}(b,r)}\2\Theta$, have been absorbed into the remainder term, alongside other small terms:
\bels{kappa(u,d) explicitly}{
\kappa(u,d) 
\,&:=\, 
2\2\avgb{\2b\2\mcl{A}(\1b,R\1d\1+\wti{r}\2)\1}\2\Theta 
\,+\, 
\avg{\1b\2\mcl{A}(r,r)}  
\,+\, 
\avgb{\1b\1\abs{m}\1\nE^{-\cI\1\am}u\1d\2} 
\\
&\,=\,
\avg{\1e,\2d\2}
\2\Theta 
\,+\,\Ord\bigl(\abs{\Theta}^4+\norm{d}^2\bigr)
\,,
}
where in the second line we have defined $ e \in \BB $ in \eqref{kappa bounds} such that
\[
\avg{\1e,w} \,:=\,2\2 \avg{\2b\2\mcl{A}(\1b,Rw)} + \avg{\1b^2\1\abs{m}\1\nE^{-\cI\1\am}w\1}
\,,
\qquad\forall\,w \in\Lp{2}\,.
\]
For the error estimate in \eqref{kappa(u,d) explicitly} we have also used \eqref{BB-bound of r}, \eqref{BB-bound of u}, and $ \norm{b} \sim 1 $. This completes the proof of \eqref{kappa bounds}.

From the definitions of $\mcl{A}$, $B$, $b$ and $\beta$, it follows
\bels{}{
\mcl{A}(b,b) \;&=\; \nE^{-\cI \am}b\2Fb 
\;=\; \nE^{-\cI \am}b\2(\nE^{-\cI 2\am}-B)\1b 
\;=\; (\1\nE^{-\cI 3\am}-\beta\1\nE^{-\cI \am})\2b^{\12}
\\
2\1\mcl{A}(\1b,w) 
\;&=\; 
\nE^{-\cI \am}\bigl(\1b\2Fw- w\2(\nE^{-\cI 2\am}-\beta)\1b\2\bigr)  
\;=\; b\,\nE^{-\cI \am}(\1\nE^{-\cI 2\am}+F-\beta)\1w\,.
}
Using these formulas in \eqref{abstract cubic appears} we see that the cubic \eqref{general cubic} holds with the  coefficients, 
\begin{subequations}
\label{mu_k's half-expanded}
\begin{align}
\label{mu_3 half-expanded}
\mu_3 \;&=\; 
\avgB{b^{\12}\2\nE^{-\cI \am}(\1\nE^{-\cI 2\am}+F-\beta)\2B^{-1}Q\bigl[\2b^{\12}\nE^{-\cI \am}(\nE^{-\cI 2\am}-\beta\1)\bigr]}
\\
\label{mu_2 half-expanded}
\mu_2 \;&=\;
\avgb{\2(\1\nE^{-\cI 3\am}-\beta\1\nE^{-\cI \am})\1b^{\13}\1}
\\
\label{mu_1 half-expanded}
\mu_1 \;&=\; - \beta\2\avg{\2b^{\12}}
\end{align}
\end{subequations}
that are determined by $ S $ and $z$ alone. 

The final expressions \eqref{mu_k's expanded} follow from these formulas by expanding $ B,\beta $ and $ b $, w.r.t. the small parameters $ \alpha $ and $ \eta $ using the expansions \eqref{B, beta and b expanded in alpha and eta}.
Let us write 
\[ 
w := (1-F)^{-1}\Pob^{(0)}(\1pf^2) 
\,,
\] 
so that $ b = f + (\cI\12\1w)\1\alpha + \Ord_\BB(\alpha^2+\eta\1)$, and $ \avg{f,w} = 0 $. 
Using \eqref{sin a = alpha f + eta t} and \eqref{expansion of cos a} we also obtain an useful representation $ \nE^{-\cI \am} = p -\cI\1f\1\alpha + \Ord_\BB(\alpha^2+\eta) $. 

First we expand the coefficient $ \mu_1 $. Using $ \avg{f^2} = 1 $ and $ \avg{f,w} = 0 $ we obtain $\avg{\1b^2} = 1 + \Ord(\alpha^2+\eta) $. Hence, only the expansion of $ \beta $ contributes at the level of desired accuracy to $ \mu_1 $,
\bea{
\mu_1 \,
&=\,  -\beta\1\avg{\1b^2} 
\;=\;
-\beta \,+\,\Ord(\1\alpha^3+\eta\1)
\\
&=\;
-\avg{f\abs{m}\1}\1\frac{\eta}{\alpha} + \cI\12\1\sigma\1\alpha - 2\2(\psi-\sigma^2)\2\alpha^2
\,+\,\Ord(\alpha^3+\eta\1)
\,.
}

Now we expand the second coefficient, $ \mu_2 $. Let us first write 
\bels{mu_2 expansion: the first step}{
\mu_2 \,&=\, 
\avgb{\2(\1\nE^{-\cI 3\am}-\beta\1\nE^{-\cI \am})\1b^{\13}\1}
\,=\, 
\avgb{(\nE^{-\cI \am}b)^3} - \beta\1\avgb{\nE^{-\cI\1\am}b^3}
\,.
}
Using the expansions we see that $  \nE^{-\cI\1\am}b = pf + \cI\1(2\1p\1w-f^2)\1\alpha + \Ord_\BB(\alpha^2+\eta) $, and thus, taking this to the third power, we find $ (\nE^{-\cI \am}b)^3 = pf^3 + \cI\13\1(2\1pf^2w-f^4) +\Ord_\BB(\alpha^2+\eta)  $. Consequently,  
\bels{mu_2 expansion: part A}{
\avgb{(\nE^{-\cI \am}b)^3} 
\;&=\;
\avg{\1pf^3} + \cI\13\1\bigl[\12\1\avg{pf^2w}-\avg{f^4}\1\bigr]\2\alpha + \Ord(\alpha^2+\eta) 
\\
&=\:\sigma + \cI\13\1(\1\psi - \sigma^2)\1\alpha + \Ord(\1\alpha^2+\eta)
\,.
}
In order to obtain expressions in terms of  $ \sigma $ and $ \psi = \mcl{D}(pf^2) $, where the bilinear positive form  $ \mcl{D} $ is defined in \eqref{def of mcl-D}, we have used 
\[
2\2\avgb{pf^2w} \;=\; 
(1+\norm{F}_{\Lp{2} \to \Lp{2}})\avgb{Q^{(0)}(pf^2),(1-F)^{-1}Q^{(0)}(\1pf^2)}+ \Ord(\eta/\alpha)\,,
\]
as well as the following consequence of $ P^{(0)}(pf^2) = \sigma\1f $ and $ \norm{f}_2 = 1 $:
\bels{}{
\avg{f^4} 
\;&=\; \norm{\1pf^2}_2^2 
\;=\; \norm{P^{(0)}(pf^2)}_2^2 + \norm{Q^{(0)}(pf^2)}_2^2 
\\
&=\; \sigma^2 + \avgb{Q^{(0)}(pf^2),Q^{(0)}(pf^2)}
\,.
}
The expansion of the last term of \eqref{mu_2 expansion: the first step} is easy since only $ \beta $ has to be expanded beyond the leading order. Indeed, directly from \eqref{beta expanded} we obtain
\bea{
\beta\2\avgb{\nE^{-\cI\1\am}b^3} \,&=\,
\Bigl(\avg{f\abs{m}}\frac{\eta}{\alpha} - \cI\12\1\sigma\1\alpha +\Ord(\alpha^2+\eta)\Bigr)
\Bigl(\avg{\1pf^3}+\Ord(\alpha+\eta)\Bigr) 
\\
&=\, -\cI\12\1\sigma^2\alpha +\avg{f\abs{m}}\frac{\eta}{\alpha}\2\sigma + \Ord\bigl(  \alpha^2+\eta\bigr)
.
}
Plugging this together with \eqref{mu_2 expansion: part A} into \eqref{mu_2 expansion: the first step} yields the desired expansion of $ \mu_2 $.

Finally, $ \mu_3$, is expanded. By the definitions and the identity \eqref{F and alpha} for $\norm{F}_{\Lp{2} \to \Lp{2}}$ we have
\[
\nE^{-\cI 2\am}+F-\beta 
\;=\;
2 - \avg{\1f\1\abs{m}}\frac{\eta}{\alpha}- B + \Ord_{\BB\to\BB}(\alpha) 
\;=\;
1+\norm{F}_{\Lp{2} \to \Lp{2}} -B + \Ord_{\BB\to\BB}(\alpha) 
\,.
\]
Recalling $ \norm{B^{-1}Q}  \lesssim 1 $ and $\eta \lesssim \alpha$, we thus obtain
\bels{int step}{
(\1\nE^{-\cI 2\am}+F-\beta)\2B^{-1}Q 
\;=\;
(1+\norm{F}_{\Lp{2} \to \Lp{2}})\2B^{-1}Q - Q + \Ord_{\BB\to\BB}(\alpha) 
\,.
}
Directly from the definition \eqref{def of P} of $ P = 1 - Q $, we see that $ Q =Q^{(0)} + \Ord_{\BB \to \BB}(\alpha) $. Thus
\[
BQ \,=\, (1-F)Q^{(0)} + \Ord_{\BB\to\BB}(\alpha)\,.
\] 
Using the general identity, $ (A+D)^{-1} = A^{-1} \!-A^{-1}D(A+D)^{-1} $, with  $ A := (1-F)Q^{(0)} $ and $ A+D := BQ $, yields
\bels{B^-1Q in terms of res of (1-F)Q^0}{
B^{-1}Q
\,=\,(1-F)^{-1}Q^{(0)} +\,  \Ord_{\BB\to\BB}(\alpha)
\,,
}
since $ B^{-1}Q $ and $ (1-F)^{-1}Q^{(0)} $ are both $ \Ord_{\!\BB\to\BB}(1)$.
By applying \eqref{B^-1Q in terms of res of (1-F)Q^0} in \eqref{int step} we get
\[
(\1\nE^{-\cI 2\am}+F-\beta)\2(QBQ)^{\msp{-1}-1} 
=\; 
Q^{(0)}\bigl[\2(1+\norm{F}_{\Lp{2} \to \Lp{2}})\2(1-F)^{-1} - 1\bigr]\1Q^{(0)} +\, \Ord_{\BB\to\BB}(\alpha) \,.
\]
Using this in the first formula of $ \mu_3 $ below yields
\bea{
&\mu_3 \,=\; 
\avgB{b^{\12}\2\nE^{-\cI \am}(\1\nE^{-\cI 2\am}+F-\beta)\2B^{-1}Q\bigl(\2b^{\12}\nE^{-\cI \am}(\nE^{-\cI 2\am}-\beta\1)\bigr)} 
\\
&
\!=\;
\Bigl(1-  \avg{\1f\1\abs{m}}\frac{\eta}{\alpha} \Bigr)
\avgB{Q^{(0)}(\1pf^2),\bigl[\2(1+\norm{F}_{\Lp{2} \to \Lp{2}})\2(1-F)^{-1} - 1\bigr]\1Q^{(0)}(pf^2)} \,+\Ord(\alpha)  
\,,
}  
which equals the second expression \eqref{mu_3 expanded} because the first term above is $ \mcl{D}(\1pf^2) $. 

Finally, we show that $ \abs{\mu_2} +\abs{\mu_3} \sim 1 $. From the expansion of $\mu_2$, we get
\[
|\mu_2|\,=\,\norm{F}_{\Lp{2} \to \Lp{2}}\2|\sigma| +\Ord(\alpha)\,\gtrsim\,|\sigma| +\Ord(\alpha) \,.
\]
Similarly, we estimate from below $|\mu_3|\gtrsim \psi +\Ord(\alpha)$. Therefore, we find that
\[
|\mu_3|+|\mu_2|^2\,\gtrsim\,\psi +\sigma^2+\Ord(\alpha) \,.
\]
We will now show that $\psi +\sigma^2 \gtrsim 1$, which implies $|\mu_2|^2 +|\mu_3|\gtrsim 1$, provided the upper bound $\eps_\ast $ of $\alpha $ is small enough. Indeed, from the lower bound \eqref{def of mcl-D} on the quadratic form $\mcl{D}$, $\mrm{Gap}(F)\sim 1$ and the identity $ \abs{\sigma} = \abs{\avg{f,pf^2}} = \norm{P^{(0)}(pf^2)}_2 $ we conclude that
\bels{lower bound on psi plus sigma2}{
\psi +\sigma^2\,\geq \, \frac{\mrm{Gap}(F)}{2}\2\norm{Q^{(0)}(\1pf^2)}_2^2+\norm{P^{(0)}(\1pf^2)}_2^2\;\gtrsim \, \norm{\1pf^2}_2^2\,.
}
Since $ \inf_x f_x \sim 1 $ and $ \abs{p} = 1 $ it follows that $ \norm{\1pf^2}_2 \sim 1 $.
\end{Proof}

\chapterl{Behavior of generating density where it is small}

In this chapter  we prove  Theorem~\ref{thr:Shape of generating density near its small values}. 
We will assume that $ S $ satisfies {\bf A1-3} and $ \nnorm{m}_\R \leq \Phi < \infty$. The model parameters are thus the same ones, \eqref{Model parameters for the shape analysis}, as in the previous chapter.  
In particular, we have $ v_x \sim \avg{v} $, and thus the support of the components of the generating densities satisfy $ \supp v_x = \supp v $ (cf. Definition~\ref{def:Extended generating density}). 
As we are interested in the generating density $ \Im\,m|_\R $ we will consider $ m $ and all the related quantities as functions on $ \R $ instead of $ \Cp $ or $ \eCp $ in this chapter.

Consider the  domain 
\bels{def of DDe}{
\DDe \,:=\, \setb{\2\tau \in \supp v: \avg{\1v(\tau)\1}\leq \eps\2} 
\,,
\qquad \eps > 0\,.
}
Theorem~\ref{thr:Shape of generating density near its small values} amounts to showing that for some sufficiently small $ \eps \sim 1 $, 
\begin{subequations}
\label{form of shape expansion}
\bels{wht-v expansion}{
v(\tau) 
\,=\; 
\wht{v}(\tau) \2+\2 \Ord_{\msp{-1}\BB}\msp{-1}\bigl(\2\wht{v}(\tau)^2\bigr)
\,,\qquad
\tau \in \DDe
\,, 
}
holds, where the leading order part factorizes,
\bels{wht-v factorizes}{
\wht{v}_x(\tau) \,=\, v_x(\tau_0) \,+\, h_x(\tau_0)\,\Psi(\tau\msp{-2}-\msp{-1}\tau_0\1;\tau_0\1)
\,,\qquad
(x,\tau) \in \Sx\times \R
\,,
}
\end{subequations}
around any expansion point $ \tau_0 $ from the set of local minima,
\bels{def of MMe}{
\MMe \,&:= 
\setb{\tau_0 \in \DDe: \tau_0 \text{ is a local minimum of }\tau \mapsto \avg{\1v(\tau)\1}}
\,,
}
and  $ h_x(\tau_0) \sim 1 $ and $ \Psi(\omega;\tau_0) \ge 0 $.
We show that the function $ \Psi(\omega;\tau_0) $ determining the shape of $\omega\to \avg{v(\tau_0+\omega)} $ is {\bf universal} in the sense that it depends on $ \tau_0 \in \MMe $ only through a single scalar parameter (cf. \eqref{cases for shapes}).

Let $\tau_0$ denote one of the minima $\tau_k$. We consider $ m(\tau_0+\omega) $ as the solution of the perturbed QVE \eqref{perturbed QVE - 2nd time} at $ z = \tau_0 $ with the scalar perturbation
\bels{d_x = omega}{
d_x(\omega) \,:=\, \omega
\,,\qquad\forall\,x \in \Sx
\,,
}  
and apply Proposition~\ref{prp:General cubic equation}.
The leading order behavior of $ m(\tau_0+\omega) $ is determined by  expressing
\bels{def of u(omega;tau)}{
u(\omega;\tau_0) \,:=\, \frac{m(\tau_0+\omega)\1-\1m(\tau_0)}{\abs{m(\tau_0)}}
\,,
}
as a sum of its projections,
\bels{defs of Theta(omega;tau) and r(omega;tau)}{
\Theta(\omega;\tau_0)\1b(\tau_0) \,:=\, P(\tau_0)\1u(\omega;\tau_0)
\qquad\text{and}\qquad 
r(\omega;\tau_0) \,:=\, Q(\tau_0)\1u(\omega;\tau_0) 
\,,
}
where $P=P(\tau_0)$ is defined in \eqref{def of P} and $Q(\tau_0)=1-P(\tau_0)$. %
 The coefficient $ \Theta(\omega;\tau_0) $ is then computed as a root of the cubic equation \eqref{general cubic} corresponding to the scalar perturbation  \eqref{d_x = omega}; its imaginary part
will give $\Psi(\omega, \tau_0)$.  Finally, the part $ r(\omega;\tau_0) $ is shown to be much smaller than $ \Theta(\omega;\tau_0) $ so that it can be considered as an error term.
The next lemma collects necessary information needed to carry out this analysis rigorously.  
This lemma has appeared as Proposition~6.2 in \cite{AEK1cpam} in the simpler case when the generating density vanishes at the expansion point, i.e.,  $v(\tau_0)=0$.

\begin{lemma}[Cubic for shape analysis]
\label{lmm:Cubic for shape analysis}
There are two constants $ \eps_\ast\1,\delta\2 \sim 1 $, such that if 
\bels{}{
\tau_0 \in \supp v
\qquad\text{and}\qquad
\avg{v(\tau_0)} \leq \eps_\ast
\,,
}
holds for some fixed {\bf base point} $ \tau_0 \in \supp v $, then
\bels{def of Theta(omega) at tau}{
\Theta(\omega) \1=\2 \Theta(\omega;\tau_0) \,=\, 
\avgbb{
\frac{b(\tau_0)}{\avg{\1b(\tau_0)^2}}\,\frac{m(\tau_0+\omega)\1-\1m(\tau_0)}{\abs{m(\tau_0)}}\2
}
\,,
}
satisfies the perturbed cubic equation
\bels{cubic for d=omega}{
\mu_3\Theta(\omega)^3 \!+ \mu_2\Theta(\omega)^2 \!+ \mu_1 \Theta(\omega) \1+\1 \Xi(\omega)\2\omega \,=\, 0 
\,,
\qquad\abs{\omega} \leq \delta\,.
}
The coefficients $ \mu_k = \mu_k(\tau_0) \in \C $ are independent of $ \omega $ and have expansions in $ \alpha  $
\begin{subequations}
\label{coefficients of cubic when z=tau_0}
\begin{align}
\label{mu_3 for shape}
\mu_3 \,&:=\, \psi\1+\1 \kappa_3\1\alpha
\\
\label{mu_2 for shape}
\mu_2 \,&:=\, \sigma + \cI\2(\13\1\psi -\1\sigma^2)\2\alpha\1+\kappa_2\1\alpha^2
\\
\label{mu_1 for shape}
\mu_1 \,&:=\, \cI\12\1\sigma\2\alpha - 2\1(\1\psi\1-\1\sigma^2)\2\alpha^2+\kappa_1\alpha^3
\,,
\end{align}
\end{subequations}
and $ \Xi(\omega) = \Xi(\omega;\tau_0) \in \C $ is close to a real constant:
\bels{def of Xi(omega)}{
\Xi(\omega) \,&:=\,\avg{f\1\abs{m}\1}\,(\21+\kappa_0\1\alpha+\nu(\omega)\1)
\,.
}
The scalars $\alpha = \avg{f,v/\abs{m}}  $, $ \sigma = \avg{f,\1p\1f^{\12}\1} $ and $ \psi = \mcl{D}(\1p\1f^{\12}) $ are defined in \eqref{def of alpha(z)}, \eqref{defs of sigma and psi} and \eqref{def of mcl-D}, respectively. They are uniformly $ 1/3$-H\"older continuous functions of $ \tau_0 $ on the connected components of the set $ \setb{\tau:\avg{v(\tau)}\leq \eps_\ast,\1|\tau|\leq 2\1\Sigma} $.
The cubic \eqref{cubic for d=omega} is {\bf stable} (cf. \eqref{stability general cubic}) in the sense that 
\bels{stability of shape cubic}{
\abs{\mu_3} \2+\2\abs{\mu_2}
\,\sim\,
\psi \,+\, \sigma^2
\,\sim\,
1
\,.
}
Both the rest term $ r(\omega) = r(\omega;\tau_0) $ (cf. \eqref{defs of Theta(omega;tau) and r(omega;tau)}) and $\Theta(\omega) $ are differentiable as functions of $ \omega $  on the domain $ \sett{\2\omega: \avg{v(\tau_0\msp{-2}+\omega)}> 0\2} $, and they satisfy:
\begin{subequations}
\label{a priori-bounds for Theta and r}
\begin{align}
\label{a priori for Theta(omega)}
\abs{\1\Theta(\omega)}\;&\lesssim\, \min\setbb{\msp{-2}\frac{\abs{\1\omega}}{\2\alpha^2\!}\,,\,\abs{\1\omega}^{1/3}\msp{-5}}  
\\
\label{a priori for r(omega)}
\norm{\1r(\omega)} \,&\lesssim\;\, \abs{\1\Theta(\omega)}^2 +\, \abs{\1\omega}  
\,.
\end{align}
\end{subequations}
The constants $ \kappa_j = \kappa_j(\tau_0) \in \C $, $ j=0,1,2,3 $, and $ \nu(\omega) = \nu(\omega;\tau_0) \in \C $ in \eqref{coefficients of cubic when z=tau_0} and \eqref{def of Xi(omega)} satisfy
\begin{subequations}
\label{bounds on perturbations of the cubic}
\begin{align}
\label{bounds for kappa_j's}
&\abs{\kappa_0}\1,\dots,\abs{\kappa_3} \,\lesssim\, 1
\\
\label{a priori for nu(omega)}
\abs{\1\nu(\omega)} 
\;\lesssim\; &\abs{\1\Theta(\omega)} + \abs{\1\omega} 
\;\lesssim\; \abs{\1\omega}^{1/3}
\,,
\end{align}
\end{subequations}
and $ \nu(\omega) $ is $ 1/3 $-H\"older continuous in $ \omega $.

The leading behavior of $ m $ on $ [\tau_0-\delta,\tau_0+\delta] $ is determined by $ \Theta(\omega;\tau_0) $:
\begin{subequations}
\label{Theta provides leading order change of m}
\begin{align}
\notag
&m_x(\tau_0+\omega) 
\\
\label{m bounded by Theta and b with Theta2-error}
&=\;m_x(\tau_0) \,+\,\abs{\1m_x(\tau_0)}\2b_x(\tau_0)\,\Theta(\omega;\tau_0) \,+\, \Ord\Bigl(\1\Theta(\omega;\tau_0)^2+\abs{\omega}\2\Bigr)
\\
\label{m bounded by Theta and f with X13-error}
&=\; m_x(\tau_0) \,+\,\abs{\1m_x(\tau_0)}f_x(\tau_0)\,\Theta(\omega;\tau_0) \,+\, \Ord\Bigl(\alpha(\tau_0)\1\abs{\omega}^{1/3}\!+\abs{\omega}^{2/3}\2\Bigr)
\,.
\end{align}
\end{subequations}
All comparison relations hold w.r.t. the model parameters \eqref{Model parameters for the shape analysis}.  
\end{lemma}

The expansion \eqref{expansion of v around tau_0} will be obtained by studying the imaginary parts of \eqref{Theta provides leading order change of m}. The factorization \eqref{wht-v factorizes} corresponds to the factorization of the second terms on the right hand side of \eqref{Theta provides leading order change of m}. In particular, $ \Psi(\omega;\tau_k) = \Im\,\Theta(\omega;\tau_k) $. The universality of the function $  \Psi(\omega;\tau_k) $ corresponds to $ \Theta(\omega) $ being close to the solution of the {\bf ideal cubic} obtained from \eqref{cubic for d=omega} and by setting $ \kappa_1=\kappa_2=\kappa_3 = 0 $ and $ \kappa_0 = \nu(\omega) = 0 $ in \eqref{coefficients of cubic when z=tau_0} and \eqref{def of Xi(omega)}, respectively.

\begin{Proof}[Proof of Lemma~\ref{lmm:Cubic for shape analysis}]
The present lemma is an application of Proposition~\ref{prp:General cubic equation} in the case where $ z = \tau_0 \in \supp v $ and the perturbation is a real number, \eqref{d_x = omega}.
Then the solution to \eqref{perturbed QVE - 2nd time} is $ g = m(\tau_0+\omega) $.
As for the assumptions of Proposition~\ref{prp:General cubic equation}, we need to verify the second inequality of \eqref{conditions for general cubic equation}, i.e., 
\[
\norm{m(\tau_0+\omega)-m(\tau_0)} \,\leq\, \eps_\ast
\,,\qquad
\abs{\omega} \leq \delta
\,.
\]
This follows from the uniform $1/3$-H\"older continuity of the solution of the QVE (cf. Theorem~\ref{thr:Regularity of generating density}) provided we choose $ \delta \sim \eps_\ast^3 $ sufficiently small. By Theorem~\ref{thr:Regularity of generating density} the solution $ m $ is also smooth on the set where $ \alpha > 0 $. 
By Lemma~\ref{lmm:Expansion of B in bad direction} and \eqref{P is bounded as a map from L2 to BB} the projectors $ P $  and $ Q $ are uniformly bounded on the connected components of the set where $ \alpha \leq \eps_\ast $. This boundedness extends to the real line. Since $ \abs{m} \sim 1$, the functions $ u(\omega) $ and $ r(\omega) $  have the same regularity in $ \omega $ as $ m(\tau) $ has in $ \tau $.
In particular, \eqref{a priori for Theta(omega)} follows this way (cf. Corollary~\ref{crl:Bound on derivative}) using $ \alpha=\alpha(\tau_0) \sim v(\tau_0) $.
Lemma~\ref{lmm:Expansion of B in bad direction} implies the H\"older regularity of $ \alpha,\sigma, \psi $. The estimate \eqref{stability of shape cubic} follows from \eqref{stability general cubic} provided $ \eps_\ast $ is sufficiently small. 
The a priori bound \eqref{a priori for r(omega)} for $ r $ follows from the analogous general estimate \eqref{r to leading order}.

The formulas \eqref{coefficients of cubic when z=tau_0} for the coefficients $ \mu_k $ follow from the general formulas \eqref{mu_k's expanded} by letting $ \eta = \Im\,z $ go to zero.  
The only non-trivial part is to establish 
\bels{eta/alpha goes to zero}{ 
\qquad
\lim_{\eta \to 0}\, \frac{\eta}{\alpha(\tau_0+\cI\1\eta)} \,=\, 0\,,\qquad
\forall\,\tau_0 \in \supp v
\,.
}
Since $ m(z) \in \BB $ is continuous in $ z $, $ F(z) $ is also continuous as an operator on $ \Lp{2} $.
Thus taking the limit $ \Im\,z \to 0 $ of the identity \eqref{v equation} yields
\[
\frac{v}{\abs{m}} \,=\, F \frac{v}{\abs{m}} 
\,,
\]
since $ \abs{m} \sim 1 $.
If $ \Re\,z = \tau_0 $, with $ v(\tau_0) \neq 0 $, then the vector $ v(\tau_0)/\abs{m(\tau_0)} \in \Lp{2}$ is non-zero, and thus an eigenvector of $ F $ corresponding to the eigenvalue $ 1 $. In particular, we get
\bels{norm of F on supp v is at least one}{
\norm{F(\tau_0)}_{\Lp{2}\to\Lp{2}} \,=\, 1\,,
\qquad\tau_0 \in \supp v
\,.
} 
If $ \tau_0 \in \supp v $ is such that $ v(\tau_0) = 0 $ then \eqref{norm of F on supp v is at least one} follows from a limiting argument $ \tau \to \tau_0 $, with $ v(\tau) \neq 0 $, and the continuity of $ F $.
Comparing \eqref{norm of F on supp v is at least one} with \eqref{F and alpha} implies \eqref{eta/alpha goes to zero}.

The cubic equation \eqref{cubic for d=omega} in $\Theta$ is a rewriting of \eqref{general cubic}.
In particular, we have
\bels{identification of kappa_0 and nu(omega)}{
1+\kappa_0\1\alpha+\nu(\omega)
\;=\;
\frac{\Xi(\omega)}{\avg{\1\abs{m}\1f\1}}
\;=\;
1 \,+\,
\frac{\1\avg{\1\abs{m}\1(\1b\1-f)\1} }{\avg{\1\abs{m}\1f\1}}
+
\frac{1}{\avg{\1\abs{m}\1f\1}}\,
\frac{\kappa(u(\omega),\omega)}{\omega}
\,,
}
where $ \kappa(u,d) $ is from \eqref{general cubic}.
We set the $ \omega$-independent term $ \kappa_0\1\alpha $ equal to the second term on the right hand side of \eqref{identification of kappa_0 and nu(omega)}. We set $ \nu(\omega) $ equal to the last term in \eqref{identification of kappa_0 and nu(omega)}.
Clearly, $ \abs{\kappa_0} \lesssim 1 $ because $ b = f + \Ord_\BB(\alpha) $ and $ \abs{m},f \sim 1 $.
The bound \eqref{kappa bounds} and the H\"older continuity of $ \Theta $ yield
\[
\absbb{\frac{\kappa(u(\omega),\omega)}{\omega}} \,\lesssim\, \frac{\abs{\Theta(\omega)}^4 + \abs{\omega}\abs{\Theta(\omega)} + \abs{\omega}^2}{\abs{\1\omega}} 
\,\lesssim\,
\abs{\Theta(\omega)} + \abs{\omega}
\,\lesssim\,\abs{\omega}^{1/3}
\,.
\]
This proves \eqref{a priori for nu(omega)}.
The expansions \eqref{Theta provides leading order change of m} follow by expressing $ m(\tau_0+\omega) $ in terms of $ \Theta(\omega;\tau_0) $ and $ r(\omega;\tau_0) $, and approximating the latter with \eqref{r to leading order}.
\end{Proof}

The following ratio,
\bels{def of Pi}{
\Pi(\tau) \,:=\, \frac{\abs{\1\sigma(\tau)}}{\avg{\1v(\tau)}^2\msp{-7}}
\;\,,
}
will play a key role in the classification of the points in $ \DDe $ when $ \eps > 0 $ is small. 
Indeed, the next result shows that if $ \Pi $ is sufficiently large, then $ v $ grows at least like a square root in the direction $ \sign \sigma $.

\NLemma{Monotonicity}{
There exist thresholds $ \eps_\ast\1,\,\Pi_\ast \sim 1 $, such that 
\bels{minimum growth for v}{
(\1\sign \sigma(\tau))\1
\partial_\tau v(\tau)
\,&\gtrsim\, 
\frac{\abs{\1\sigma(\tau)}}{\sigma(\tau)^2+v(\tau)^2}\frac{\Ind\sett{\1\Pi(\tau)\ge\Pi_\ast}}{\avg{v(\tau)}}
\,,\qquad
\tau \in \DD_{\eps_\ast}
\,.
}
}
\begin{Proof} 
By Lemma~\ref{lmm:Cubic for shape analysis} both $ \Theta(\omega;\tau) $ and $ r(\omega;\tau) $ are differentiable functions in $ \omega $, and thus,
\bels{dif of m in terms of dif of Theta and r}{
\partial_\tau m(\tau)
\,=\, \abs{m(\tau)}\2b(\tau)\,\partial_\omega\Theta(0;\tau) \,+\, \abs{m(\tau)}\2\partial_\omega r(0;\tau)
\,.
}
Let us drop the fixed argument $ \tau $ to simplify notations.  
Taking imaginary parts of \eqref{dif of m in terms of dif of Theta and r} yields
\bels{partial_tau v}{
\partial_\tau v 
\,=\, \Im\,\partial_\tau m
\;&=\, 
\abs{m}\,\Im \bigl[\2b\; \partial_\omega\msp{-1}\Theta(0)\1 \bigr] \,+\, \abs{m}\,\Im\,\partial_\omega r(0)
\,.
}
By dividing \eqref{a priori for r(omega)} by $ \omega $, and using \eqref{a priori for Theta(omega)}, we see that
\[
\absB{\frac{\1r_x(\omega)}{\omega}} 
\;\lesssim \, 
1 \,+\, \absB{\frac{\Theta(\omega)^2}{\omega}}
\,\lesssim\, 
1 + \frac{\abs{\1\omega}}{\,\alpha^{\14}\msp{-6}}
\;,
\qquad \forall\,x \in \Sx\,.
\]
Letting $ \omega \to 0 $, and recalling $ r(0) = 0 $, we see that the last term in \eqref{partial_tau v} is uniformly bounded,
\bels{partial_omega r is O(1)}{
\norm{\2\Im\,\partial_\omega r (0)\1} \,\lesssim\, 1
\,. 
}

We will now show that $ \Im[\2 b\,\partial_\omega\Theta\2] $ dominates the second term in \eqref{dif of m in terms of dif of Theta and r}, provided $ \alpha $ is sufficiently small and $ \abs{\sigma}/\alpha^2 \sim \Pi $ is sufficiently large.
To this end we first rewrite the cubic \eqref{cubic for d=omega},
\bels{derivative from the cubic}{
\biggl(1 + \frac{\mu_2}{\mu_1}\Theta(\omega) + \frac{\mu_3}{\mu_1}\Theta(\omega)^2\biggr)\2\frac{\Theta(\omega)}{\omega} 
\;=\;
-\frac{\Xi(\omega)}{\mu_1}
\,.
}
From the definition \eqref{mu_1 for shape} we obtain  
\[
\abs{\mu_1} \,\sim\, \alpha\,\absb{\2\sigma +\Ord(\alpha^2)} + \alpha^2\absb{\psi-\sigma^2 +\Ord(\alpha)}
\,,
\]
by distinguishing the cases $2\1\sigma^2 \leq \psi$ and $2\1\sigma^2 > \psi$, and using \eqref{stability of shape cubic}. 
Applying \eqref{a priori for nu(omega)} to estimate $ \Xi(\omega) $ we see that the right hand side of \eqref{derivative from the cubic} satisfies
\bels{}{
\frac{\Xi(\omega)}{\mu_1} \,=\,\frac{\avg{f\abs{m}\1}}{2}\2\frac{\,1 \,+\Ord(\2\alpha+\abs{\omega}^{1/3}\1)\!}{\,\cI\1\alpha\1\sigma-\1\alpha^2\1(\psi-\sigma^2)+\Ord(\alpha^3)}
\,.
}
From \eqref{a priori for Theta(omega)} we see that $ \Theta(\omega) \to 0 $ as $ \omega \to 0 $. 
Hence taking the limit $ \omega \to 0 $ in \eqref{derivative from the cubic} and recalling $ \abs{\mu_2},\abs{\mu_3} \lesssim 1 $, yields
\bels{dif of Theta}{
\partial_\omega\msp{-1}\Theta(0) \,=\,
\frac{\dif\Theta}{\dif \omega}\biggl|_{\omega=0}
=\;
\frac{\avg{f\abs{m}\1}}{2}\,
\frac{\alpha^2(\psi-\sigma^2)+\cI\2\alpha\2\sigma +\Ord(\alpha^3+\abs{\sigma}\alpha^2)}{\,\alpha^2\abs{\1\sigma +\Ord(\alpha^2)\1}^2+\alpha^4\1\abs{\psi-\sigma^2+\Ord(\alpha)}^2}
\,.
}
Using $b = f +\Ord_\BB(\alpha)$ and $ \avg{\1f\abs{m}\1} \sim 1 $, we conclude from \eqref{dif of Theta} that
\bels{Im b partial_omega Theta}{
(\sign \sigma)\, \Im \bigl[\2b\; \partial_\omega\msp{-1}\Theta(0)\1 \bigr] 
\;&\sim\;
\frac{
\,\abs{\sigma} \,+\Ord_\BB(\1 \alpha^2\!+\abs{\sigma}\alpha\1)\msp{-10} 
}{
\abs{\1\sigma +\Ord(\alpha^2)}^2 + \alpha^2\1\abs{\psi-\sigma^2+\Ord(\alpha)}^2}
\,\frac{1}{\alpha}
\,.
}
By definitions $ \abs{\sigma}/\alpha^2 \sim \Pi \ge \Pi_\ast $. Hence, if $ \Pi_\ast \sim 1 $ is sufficiently large, then the factor multiplying $ 1/\alpha $ on the right hand side of \eqref{Im b partial_omega Theta} scales like $ \min\setb{\abs{\sigma}^{-1},\alpha^{-2}\abs{\sigma}} $. 
Here we used again \eqref{stability of shape cubic}. 
Using \eqref{partial_tau v}, \eqref{partial_omega r is O(1)}, and $ \alpha \sim \avg{v} $ from \eqref{Im b partial_omega Theta} we obtain
\[
(\sign \sigma)\,\partial_\tau v  \,\sim \,  
 \min\setbb{\!\frac{1}{\abs{\1\sigma}}\1,\frac{\abs{\1\sigma}}{\avg{v}^2\!}\!}\frac{1}{\avg{v}}
\2+\2 
\Ord_{\!\BB}(\11\1)  
\,.
\]
By taking $ \Pi_\ast \sim 1 $ sufficiently large and $ \eps_\ast \sim 1 $ sufficiently small the term $\Ord_{\!\BB}(1) $ can be ignored and \eqref{minimum growth for v} follows.
\end{Proof}

\section{Expansion around non-zero minima of generating density}

Lemma~\ref{lmm:Monotonicity} shows that if $\tau_0 \in \DD_{\eps_\ast} $ is a non-zero minimum of $ \tau \mapsto \avg{v(\tau)} $, i.e., $ \avg{v(\tau_0)} > 0 $, then $ \partial_\tau \avg{v(\tau_0)} = 0 $, and hence $ \Pi(\tau_0) <\Pi_\ast $.
Now we show that any point $ \tau_0 $ satisfying $ \Pi(\tau_0) <\Pi_\ast $ is an approximate minimum of $ \avg{v} $, and its shape is described by the universal shape function $ \Psi_{\mrm{min}} : [\10,\infty) \to [\10,\infty) $ introduced in Definition~\ref{def:Shape functions}.

\begin{proposition}[Non-zero local minimum]
\label{prp:Non-zero local minimum}
If $ \tau_0 \in \DDe $ satisfies
\bels{}{
\Pi(\tau_0) \2\leq\, \Pi_\ast
\,,
}
where $ \Pi_\ast \sim 1 $ is from Lemma~\ref{lmm:Monotonicity} \emph{(}in particular if $ \tau_0 $ is a non-zero local minimum of $ \avg{v} $\emph{)}, then
\bels{smooth cusp:main result}{
v_x(\tau_0+\omega)\,-\,v_x(\tau_0) \;=\; 
h_x\avg{v}
\,\Psi_{\msp{-2}\mrm{min}}\msp{-2}\biggl(\msp{-1}\Gamma\frac{\omega}{\avg{v}^3\!}\biggr)\,
+\, 
\Ord\biggl(\min\setbb{\!
\frac{\abs{\omega}}{\avg{v}}\2,\abs{\omega}^{2/3}\!
}
\biggr)
}
for some $ \omega$-independent constants $ h_x = h_x(\tau_0) \sim 1 $ and $ \Gamma = \Gamma(\tau_0) \sim 1 $. 
Here $ \avg{v} = \avg{v(\tau_0)}$, $ \sigma = \sigma(\tau_0) $, etc. are evaluated at $\tau_0$.
\end{proposition}

Using \eqref{def of Psi_min} we see that the first term on the right hand side of \eqref{smooth cusp:main result} satisfies
\bels{min-expansion:scaling}{
\avg{v}
\,\Psi_{\msp{-2}\mrm{min}}\msp{-2}\biggl(\msp{-1}\Gamma\frac{\omega}{\avg{v}^3\!}\biggr)
\;\sim\; 
\min\setbb{
\frac{\abs{\omega}^2\!}{\avg{v}^5\!}\,,\abs{\omega}^{1/3}\!
}
\,,
\qquad \omega \in \R\,.
}
Comparing this with the last term of \eqref{smooth cusp:main result} we see that the first term dominates the error on the right hand side of \eqref{smooth cusp:main result} provided $ \avg{v}^4 \lesssim \abs{\omega} \lesssim 1 $. Applying the lemma at two distinct base points hence yields the following property of the non-zero minima.

\NCorollary{Location of non-zero minima}{
Suppose two points $ \tau_1,\tau_2 \in \DDe $ satisfy the hypotheses of Proposition~\ref{prp:Non-zero local minimum}. Then, either
\bels{dist between minima}{
\abs{\tau_1-\tau_2} \,\gtrsim\, 1\,,
\qquad\text{or}\qquad
\abs{\tau_1-\tau_2} \,\lesssim\, \min\setb{\avg{v(\tau_1)},\avg{v(\tau_2)}}^4
\,.
}
}
\begin{Proof}
Suppose the points $ \tau_1 $ and $ \tau_2 $ qualify as the base points for Proposition~\ref{prp:Non-zero local minimum}. Then the  corresponding expansions \eqref{smooth cusp:main result} are compatible only if the base points satisfy the dichotomy \eqref{dist between minima}. For the second bound in \eqref{dist between minima} we use \eqref{min-expansion:scaling}.
\end{Proof}

We will use the standard convention on complex powers.

\NDefinition{Complex powers}{
We define complex powers $ \zeta \mapsto \zeta^\gamma $, $ \gamma \in \C $, on $ \C\backslash (-\infty,0) $, by setting $ \zeta^{\1\gamma} := \exp(\2\gamma \log \zeta\2) $, where $ \log : \C\backslash (-\infty,0) \to \C $ is a continuous branch of the complex logarithm with $ \log 1 = 0 $. We denote by $ \arg :  \C\backslash(-\infty,0) \to (-\pi,\pi) $, the corresponding angle function.
}

\begin{Proof}[Proof of Proposition~\ref{prp:Non-zero local minimum}]
Without loss of generality it suffices to prove \eqref{smooth cusp:main result} in the case  $ \abs{\omega} \leq \delta $ for some sufficiently small constant $ \delta \sim 1 $. 
Indeed, when $ \abs{\omega} \gtrsim 1 $ the expansion \eqref{smooth cusp:main result} becomes trivial since the last term is $ \Ord(1) $ and therefore dominates all the other terms, including $ \abs{v_x(\tau)} \leq \nnorm{m}_\R \sim 1 $.
Similarly, we may restrict ourselves to the setting where the quantity
\bels{min:def of chi}{
\chi \,:=\,  \alpha + \frac{\abs{\sigma}}{\alpha}
\,,
}
satisfies $ \chi \leq \chi_\ast $, for some sufficiently small threshold $ \chi_\ast \sim 1 $.
In particular, we assume that $ \chi_\ast  $ is so small that $ \chi \leq \chi_\ast$ implies $ \avg{v} \leq \eps_\ast $.

Let us denote by $ \gamma_k \in \C $,  $ k=0,1,2,... $,   generic $ \omega$-independent numbers, satisfying
\bels{generic gamma_k}{
\abs{\gamma_k} \,\lesssim\,\chi
\,.
}
Since $ \Pi \sim \abs{\sigma}/\alpha^2 $ and $ \Pi\leq \Pi_\ast $ we have $ \abs{\sigma} \leq \Pi_\ast\chi_\ast^2 $. 
From \eqref{stability of shape cubic} it hence follows that $ \psi \sim 1 $ for sufficiently small $ \chi_\ast \sim 1  $. 
Thus the cubic \eqref{cubic for d=omega} takes the form
\bels{cubic for min}{
\Theta(\omega)^{\13} \,+\,  \cI\13\1\alpha\1(1+ \gamma_3)\1\Theta(\omega)^{\12}
\,-\,
2\1\alpha^2(1+ \gamma_2)\1\Theta(\omega)& 
\,
\\
+\; 
(\11+\gamma_0+ (1+\gamma_1)\nu(\omega)\1)\,
\frac{\avg{f\1\abs{m}\1}}{\psi}\,
\omega\msp{7}&
\;=\;
0\,.
}
Using the following \emph{normal coordinates},
\bels{min: normal coordinates}{
\qquad\lambda \,&:=\, \Gamma\,\frac{\omega}{\2\alpha^{\13}\msp{-4}}
\\
\Omega(\lambda) \,&:=\,\sqrt{3\,}\,\biggl[
(1+ \gamma_4)\,\frac{1}{\alpha}\Theta\Bigl(\frac{\,\alpha^3\msp{-4}}{\Gamma}\lambda\Bigr)\,+\,\cI \,+\, \gamma_5
\biggr]\,,
}
where $ \Gamma := (\sqrt{27}/2)\avg{\1\abs{m}f}/\psi \sim 1 $,  \eqref{cubic for min} reduces to
\bels{min: normal form}{
\Omega(\lambda)^{\13} + 3\2\Omega(\lambda) \2+\2 2\1\Lambda(\lambda)
\,&=\, 0
\,.
} 
Here the constant term $ \Lambda :\R \to \C $ is given by
\bels{min:Lambda and mu}{
\Lambda(\lambda)\,&:=\,(\21+ \gamma_6+ (1+\gamma_7) \mu(\lambda))\2\lambda \,+\,  \gamma_8
\\
\mu(\lambda) \,&:=\, \nu\Bigl(\frac{\,\alpha^3\msp{-4}}{\Gamma}\lambda\Bigr)
\,.
}

The following lemma presents Cardano's solution for the reduced cubic \eqref{min: normal form} in a form that is convenient for our analysis. 
We omit the proof of this well know result.

\begin{lemma}[Roots of reduced cubic with positive linear coefficient]
\label{lmm:Roots of reduced cubic with positive linear coefficient}
The following holds:
\bels{}{
\Omega^{\13}+\,3\2\Omega \,+\, 2\1\zeta \;=\; (\1\Omega-\wht{\Omega}_+(\zeta))(\1\Omega-\wht{\Omega}_0(\zeta))(\1\Omega-\wht{\Omega}_-(\zeta))
\,,\qquad
\forall\1\zeta \in \C
\,,
}
where the three root functions $ \wht{\Omega}_a : \C \to \C $, $ a =0,\pm $, are given by
\begin{subequations}
\label{solutions to positive reduced cube}
\bels{defs of wth-Omega_a for min}{
\wht{\Omega}_0\,&:= 
-\22\2\Phi_{\mrm{odd}}
\\
\wht{\Omega}_\pm \,&:=\, 
\Phi_{\mrm{odd}}
\,\pm\,
\cI\2\sqrt{3\2}\,\Phi_{\mrm{even}}
\,,
}
with $ \Phi_{\mrm{even}} $ and $ \Phi_{\mrm{odd}} $ denoting the even and odd parts of the function $ \Phi : \C \to \C $,
\bels{def of Phi}{
\Phi(\zeta) \,:=\, \bigl(\sqrt{\21+\zeta^{\12}\,} \,+\,\zeta\,\bigr)^{\msp{-2}1/3}
\,,
}
\end{subequations}
respectively.
The roots \eqref{solutions to positive reduced cube} are analytic and distinct on the set,
\bels{def of wht-C}{
\wht{\C} \,:=\, \C \backslash \sett{\,\cI\2\xi: \xi \in \R, \abs{\xi} > 1}
\,.
} 
Indeed, if $ \wht{\Omega}_a(\zeta) = \wht{\Omega}_b(\zeta) $, for $ a\neq b$, then $ \zeta = \pm\1\cI \2$.
\end{lemma}

Since $ \Omega(\lambda) $, defined in \eqref{min: normal coordinates}, solves the cubic \eqref{min: normal form}, there exists $ A :\R \to \sett{0,\pm} $, such that  
\bels{min:Omega in terms of A and wht-Omega_a}{
\Omega(\lambda) = \wht{\Omega}_{A(\lambda)}(\Lambda(\lambda))
\,,\qquad\lambda \in \R
\,.
} 
In the normal coordinates the restriction $ \abs{\omega}\leq \delta $ becomes $ \abs{\lambda} \leq \lambda_\ast $, where
\bels{min:expansion range in normal coords}{
\abs{\lambda} \leq \lambda_\ast := \Gamma\frac{\delta}{\alpha^3\!}
\,.
}
Nevertheless, for sufficiently small $ \delta \sim 1 $ the function $ \Lambda $ in \eqref{min:Lambda and mu} is a small perturbation of the identity function. Indeed, from \eqref{min:Lambda and mu} and the bound \eqref{a priori for nu(omega)} on $ \nu $, we get
\bels{min:mu a priori bound}{
\abs{\mu(\lambda)} 
\,&\lesssim\, 
\absB{\Theta\Bigl(\frac{\1\alpha^3\!}{\Gamma}\lambda\1\Bigr)} + \alpha^3\abs{\lambda}
\\
&\lesssim\,
\alpha\1\abs{\lambda}^{\11/3}
\,\lesssim\,
\delta^{\11/3},
\qquad\text{when}\quad \abs{\lambda} \leq \lambda_\ast
\,.
}
Hence, if the thresholds $ \delta, \chi_\ast \lesssim 1 $  are sufficiently small, then
\bels{min:Lambda mapsto good set}{
\Lambda(\lambda) \in \mathbb{G}\,,
\quad\text{and}\quad
\abs{\Lambda(\lambda)} \,\sim\, \abs{\lambda} 
\,,\qquad
\abs{\lambda} \leq \lambda_\ast\,, 
}
where 
\bels{min: def of good set mathbb-G}{
\mathbb{G} \,:=\, \setB{\zeta \in \C:\dist\bigl(\1\zeta,\cI\1(-\infty,-1)\cup \cI\1(+1,+\infty)\1\bigr) \ge 1/2}
\,.
}
By Lemma~\ref{lmm:Roots of reduced cubic with positive linear coefficient} the root functions have uniformly bounded derivatives on this subset of $ \wht{\C} $.

The following lemma which is proven in Appendix \ref{sec:Cubic roots and associated auxiliary functions} is used for replacing $ \Lambda(\lambda) $ by $\lambda $ in \eqref{min:Omega in terms of A and wht-Omega_a}.

\begin{lemma}[Stability of roots]
\label{lmm:Stability of roots - pos} 
There exist positive constants $ c_1,C_1 $ such that if $ \zeta \in \mathbb{G} $ and $ \xi \in \C $ satisfy 
\bels{min stability: assumption}{
\abs{\1\xi\1} \,\leq\, c_1\2(\11+\abs{\zeta}\1)
\,,
}
then the roots \eqref{solutions to positive reduced cube} are stable in the sense that 
\bels{min:stability of roots}{
\absb{\2\wht{\Omega}_a(\zeta+\xi)-\2\wht{\Omega}_a(\zeta)} 
\;\leq\;
\frac{\!C_1\1\abs{\1\xi\1}}{\21\2+\abs{\zeta}^{2/3}\msp{-6}} 
\;\,,\qquad a=0\1,\2\pm
\,.
}
\end{lemma}

From \eqref{min:Lambda mapsto good set} we see that $ \Lambda(\lambda) \neq \pm\2\cI $, and hence the roots do not coincide. Moreover, we know from Lemma~\ref{lmm:Cubic for shape analysis} and \eqref{min: normal coordinates}:
\smallskip
\begin{enumerate}
\item[{\bf SP-1}] The function $ \lambda \mapsto \Omega(\lambda) $ is continuous.
\end{enumerate}
\smallskip
This simple fact will be the first of the four selection principles ({\bf SP}) used for determining the correct roots of the cubic \eqref{cubic for d=omega} in the following (cf. Lemma~\ref{lmm:Selection principles}).
Since the roots $ \wht{\Omega}_a|_{\mathbb{G}} $ are also continuous by Lemma~\ref{lmm:Roots of reduced cubic with positive linear coefficient}, we conclude that the labelling function $ A $ in \eqref{min:Omega in terms of A and wht-Omega_a} stays constant on the interval $ [-\lambda_\ast,\lambda_\ast] $.
In order to determine this constant,  $ a := A(\lambda) $, we use the second selection principle:
\smallskip
\begin{enumerate}
\item[{\bf SP-2}] The initial value $ \Omega(0) $ is consistent with $ \Theta(0) = 0\2$.
\end{enumerate}
\smallskip
Plugging $ \Theta(0) = 0 $ into \eqref{min: normal coordinates} yields
\bels{min:Omega(0)}{
\Omega(0) 
\,=\, 
\cI\sqrt{3\2}\2(\11+ \gamma_5) 
\,=\,  
\cI\sqrt{3\2} +  \Ord\Bigl(\alpha+\frac{\abs{\sigma}}{\alpha}\Bigr)\,.
}
On the other hand, using Lemma~\ref{lmm:Stability of roots - pos} and \eqref{min:Lambda and mu} we get 
\bels{min:wht-Omega(0)}{
\wht{\Omega}_a(\Lambda(0)) \,=\, \wht{\Omega}_a( \gamma_8) \,=\,  \wht{\Omega}_a(0) + \Ord\Bigl(\alpha+\frac{\abs{\sigma}}{\alpha}\Bigr)
\,,
}
where 
\[
\wht{\Omega}_0(0) \,=\, 0
\qquad
\text{and}\qquad
\wht{\Omega}_{\pm}(0) \,=\, \pm\1 \cI\sqrt{3}
\,.
\]
Comparing this with \eqref{min:Omega(0)} and \eqref{min:wht-Omega(0)}, we see that for sufficiently small $ \alpha + \abs{\sigma}/\alpha \lesssim \chi_\ast $, only the the choice $ A(0) = + $ satisfies {\bf SP-2}.  

As the last step we derive the expansion \eqref{smooth cusp:main result} using the formula 
\bels{v difference in terms of Theta}{
v_x(\tau_0+\omega) - v_x(\tau_0) 
\,&=\,
\abs{m_x}\2f_x 
\,\Im\,\Theta(\omega)
\,+\, \Ord\Bigl(\2\alpha\2\abs{\Theta(\omega)} + \abs{\Theta(\omega)}^2\!+\abs{\omega}\2\Bigr)
\,,
}
which follows by taking the imaginary part of \eqref{m bounded by Theta and b with Theta2-error}. We also used $ b_x = (1+\Ord(\alpha))\1f_x $ and $ f_x \sim 1 $ here.
Let us express $ \Theta $ in terms of the normal coordinates using \eqref{min: normal coordinates}
\bels{min:Theta solved}{
\Theta(\omega)
\;&=\, 
\frac{\alpha}{1+ \gamma_4}\Biggl[\,
 \frac{\wht{\Omega}_+\msp{-1}(\Lambda(\lambda))}{\!\sqrt{3\1}}\,-\,\cI
\2-\2
  \gamma_5
\Biggr]
\,.
}
Here, $ \omega $ and $ \lambda $ are related by \eqref{min: normal coordinates}. Since $ \Theta(0) = 0 $, and $ \Lambda(0) =  \gamma_8 $ (cf. \eqref{min:Lambda and mu}), we get
\[
\cI+ \gamma_5 = \frac{\wht{\Omega}_+\msp{-1}( \gamma_8)}{\!\sqrt{3\2}}
\,.
\]
Using this identity and
\[
\Lambda(\lambda) =  \gamma_8 + \Lambda_0(\lambda)
\qquad\text{with}\qquad
\Lambda_0(\lambda) := (\11+ \gamma_6+ (1+\gamma_7)\mu(\lambda)\1)\1\lambda
\,,
\]
we rewrite the formula \eqref{min:Theta solved} as
\bels{Theta as Omega-dif}{
\Theta(\omega) \;=\;
(\11+\Ord(\chi)\1)\,
\frac{\alpha}{\!\sqrt{3\2}}
\biggl[
\2\wht{\Omega}_+\msp{-1}(\1 \gamma_8+\Lambda_0(\lambda)\1) -\,\wht{\Omega}_+\msp{-1}( \gamma_8)
\biggr]
\,.
}
From \eqref{min:Lambda mapsto good set} we know that the arguments of $ \wht{\Omega}_+ $ in \eqref{Theta as Omega-dif} are in $ \mathbb{G} $. 
Using the uniform boundedness of the derivatives of $ \Omega|_{\mathbb{G}} $, and the bound $ \abs{\Phi(\zeta)} \lesssim 1+ \abs{\zeta}^{1/3} $, we get 
\bels{root difference}{
\absb{\1\wht{\Omega}_+\msp{-1}(\1 \gamma_8+\Lambda_0(\lambda)\1) -\,\wht{\Omega}_+\msp{-1}( \gamma_8)\1} \,\lesssim\, \min\setb{\abs{\lambda},\abs{\lambda}^{1/3}}
\,,
\qquad\abs{\lambda} \leq \lambda_\ast\,.
}
By using \eqref{root difference} in \eqref{min:mu a priori bound} and \eqref{Theta as Omega-dif} we estimate the sizes of both $ \mu(\lambda) $ and $ \Theta(\omega) $,
\bels{min:bound on mu and Theta}{
\abs{\1\mu(\lambda)} 
\,+\,
\absB{\Theta\Bigl(\frac{\1\alpha^3\!}{\Gamma}\lambda\1\Bigr)}
\;\lesssim\; 
\alpha \min\setb{\abs{\lambda},\abs{\lambda}^{1/3}}\,,
\qquad\abs{\lambda} \leq \lambda_\ast
\,.
}

In order to extract the exact leading order terms, we express the difference on the right hand side of \eqref{Theta as Omega-dif} using the mean value theorem
\bels{Omega-dif approximation}{
\wht{\Omega}_+\msp{-1}(\1 \gamma_8+\Lambda_0(\lambda)) -\,\wht{\Omega}_+\msp{-1}( \gamma_8) 
\;&=\;\,
\wht{\Omega}_+\msp{-1}(\Lambda_0(\lambda))-\,\wht{\Omega}_+\msp{-1}(0) 
\\
&\;+\, \gamma_8\2
\frac{\partial}{\partial\zeta}\Bigl[\2\wht{\Omega}_+\msp{-1}(\1\zeta+\Lambda_0(\lambda))-\,\wht{\Omega}_+\msp{-1}(\zeta)\Bigr]_{\zeta\,=\,\gamma}
\,,
}
where $ \gamma\in\mathbb{G}$ is some point on the line segment connecting $ 0 $ and $ \gamma_8 $. 
Using \eqref{min:bound on mu and Theta} and Lemma~\ref{lmm:Stability of roots - pos} on the first term on the right hand side of \eqref{Omega-dif approximation} shows 
\bels{Omega(Lambda0)-Omega(0)}{
\wht{\Omega}_+\msp{-1}(\Lambda_0(\lambda))-\wht{\Omega}_+\msp{-1}(0) 
\;=\; \wht{\Omega}_+\msp{-1}(\lambda)-\wht{\Omega}_+\msp{-1}(0) \,+\, \Ord\Bigl(\2\chi\2\min\setb{\abs{\lambda},\abs{\lambda}^{2/3}}\Bigr) 
\,.
}
From an explicit calculation we get $ \abs{\partial_\zeta\wht{\Omega}_+\msp{-1}(\zeta)} \lesssim 1 $, for $\zeta \in \mathbb{G} $. Thus
\[
\absbb{\frac{\partial}{\partial\zeta}\Bigl[\2\wht{\Omega}_+\msp{-1}(\1\zeta+\Lambda_0(\lambda))-\,\wht{\Omega}_+\msp{-1}(\zeta)\Bigr]_{\zeta\,=\,\gamma}
} 
\;\lesssim\; 
\min\setb{\,\abs{\lambda}\1,\11\1}
\,.
\]
Plugging this and \eqref{Omega(Lambda0)-Omega(0)} into \eqref{Omega-dif approximation} yields
\bels{Omega dif done}{
\wht{\Omega}_+\msp{-1}(\1 \gamma_8+\Lambda_0(\lambda)) -\,\wht{\Omega}_+\msp{-1}( \gamma_8) 
\,&=\;
\wht{\Omega}_+\msp{-1}(\lambda)
-\wht{\Omega}_+\msp{-1}(0)
\,+\, \Ord\Bigl(\,\chi\,\min\setb{\2\abs{\lambda},\abs{\lambda}^{2/3}}\Bigr)
\,.
}
Via \eqref{Theta as Omega-dif} we use this to represent the leading order term in \eqref{v difference in terms of Theta}.
By approximating all the other terms in  \eqref{v difference in terms of Theta} with \eqref{min:bound on mu and Theta} we obtain
\bels{min: accurate v difference}{
\msp{-10}
&v_x(\tau_0+\omega) - v_x(\tau_0) 
\\
&=\,
\abs{m}_x\2f_x\,\alpha\,
\frac{\Im\bigl[\2\wht{\Omega}_+\msp{-1}(\lambda)-\wht{\Omega}_+\msp{-1}(0)\1\bigr]}{\sqrt{3}}
\,+\,
\Ord\biggl(
\bigl(\alpha^2\msp{-2}+\abs{\sigma}\1\bigr)\min\setb{\abs{\lambda},\abs{\lambda}^{2/3}}
\biggr)\,.
}
Using the formulas \eqref{solutions to positive reduced cube} and \eqref{def of Phi}, we identify the universal shape function from \eqref{def of Psi_min},
\[ 
\Psi_{\mrm{min}}(\lambda) \,=\, \frac{\Im\bigl[\2\wht{\Omega}_+\msp{-1}(\lambda)-\wht{\Omega}_+\msp{-1}(0)\1\bigr]}{\sqrt{3}}
\,.
\]
Denoting $ h_x := (\alpha/\avg{v})\1f_x $ and writing $ \lambda $ in terms of $ \omega $ in \eqref{min: accurate v difference} the expansion \eqref{smooth cusp:main result} follows.
\end{Proof}

\section{Expansions around minima where generating density vanishes}

Together with Proposition~\ref{prp:Non-zero local minimum} the next result covers the behavior of $ v|_{\DDe} $ around its minima for sufficiently small $ \eps \sim 1 $.
For each $ \tau_0 \in \partial \supp v $, satisfying $ \sigma(\tau_0) \neq 0 $, we associate the {\bf gap length},
\bels{def of Delta(tau0)}{
\Delta(\tau_0) \,:=\, 
\inf \setb{\2\xi\in (\10\1,2\1\Sigma\2]:\avgb{\1v(\tau_0-\sign \sigma(\tau_0)\1\xi\1)}> 0\2}
\,,
} 
with the convention $ \Delta(\tau_0) := 2\1\Sigma $ in case the infimum does not exist.
We will see below that if $ \tau_0 \in \partial \supp v $, then $ \sigma(\tau_0) \neq 0 $ and $ \sign \sigma(\tau_0) $  is indeed the direction in which  the set $ \supp v $ continues from $ \tau_0 $. 
Because $ \supp v \subset [-\Sigma\1,\Sigma\2] $ the number $ \Delta(\tau_0) $ thus defines the length of the actual gap in $ \supp v $ starting at $ \tau_0 $, with the convention that the gap length is $ 2\1\Sigma $ for the extreme edges. 

Recall the definition \eqref{def of Psi_edge} of the universal edge shape function $ \Psi_{\!\mrm{edge}} : [\10\1,\infty) \to [\10\1,\infty) $.
\begin{proposition}[Vanishing local minimum]
\label{prp:Vanishing local minimum}
Suppose $ \tau_0 \in \supp v $ with $ v(\tau_0) = 0 $. Depending on the value of $\sigma = \sigma(\tau_0) $ either of the following holds:
\begin{itemize}
\titem{i} 
If $ \sigma(\tau_0) \neq 0 $, 
then $ \tau_0 \in \partial \supp v $ and $ \supp v $ continues in the direction $ \sign \sigma $, such that 
for $ (\1\sign \sigma)\,\omega \ge 0 $,
\bels{edge expansion}{
\msp{-10}
v_x(\tau_0+\omega) 
\,=\;\; &h_x\,
\Delta^{\!1/3}\2\Psi_{\!\mrm{edge}}\msp{-1}\biggl(\msp{-2}\frac{\2\abs{\1\omega}\2}{\Delta}\msp{-2}\biggr)
\,+\,
\Ord\biggl(
\min\setbb{\!\frac{\abs{\1\omega}}{\;\Delta^{\!1/3}\msp{-10}}\msp{10},\2\abs{\1\omega}^{2/3}\!}
\biggr)
\,,
}
where $ h_x = h_x(\tau_0) \sim 1 $, and $ \Delta = \Delta(\tau_0)  $ is the length of the gap in $ \supp v $ in the direction $ - \sign \sigma $ from $ \tau_0 $ (cf. \eqref{def of Delta(tau0)}). 
Furthermore, the gap length satisfies
\bels{}{
\Delta(\tau_0)  \,\sim\, \abs{\1\sigma(\tau_0)}^3
\,, 
}
while the shapes in the $x$-direction match at the opposite edges of the gap in the sense that $ h(\tau_1) = h(\tau_0) + \Ord_\BB(\2\Delta^{\!1/3}) $, for $ \tau_1 = \tau_0 -\sign \sigma(\tau_0)\,\Delta$.
\titem{ii} 
If $ \sigma(\tau_0) = 0 $, then $  \dist(\1\tau_0\1,\partial \supp v\1) \sim 1 $, and for some  $ h_x = h_x(\tau_0) \sim 1 $:
\bels{cusp expansion - 1st time}{
v_x(\tau_0+\omega\1) \,=\, 
h_x\,\abs{\1\omega}^{1/3} \!+\, \Ord\bigl(\1\abs{\omega}^{2/3}\bigr)
\,.
}
\end{itemize}
\end{proposition}

From the explicit formula \eqref{def of Psi_edge} one sees that the leading order term in \eqref{edge expansion} satisfies
\bels{scaling of scaled Psi_edge}{
\Delta^{\!1/3}\,\Psi_{\!\mrm{edge}}\msp{-2}\biggl(\msp{-2}\frac{\2\omega\2}{\Delta}\msp{-2}\biggr)
\;\sim\;
\begin{cases}
\displaystyle
\frac{\2\omega^{1/2}\!}{\Delta^{\!1/6}\msp{-5}}\quad &\text{when }0\leq \omega \lesssim \Delta\,;
\\
\,\omega^{1/3}
&\text{when }\omega \gtrsim \Delta\,.
\end{cases}
}
In particular, if an edge $ \tau_0 $ is separated by a gap of length $ \Delta(\tau_0) \sim 1 $ from the opposite edge of the gap, then $ v $ grows like a square root.

Proposition~\ref{prp:Vanishing local minimum} is proven at the end of Subsection~\ref{ssec:Two nearby edges} by combining various auxiliary results which we prove in the following two sections.
What is common with these intermediate results is that the underlying cubic \eqref{cubic for d=omega} is always of the form
\bels{cubic when v=0}{
\psi\2\Theta(\omega)^3 + \sigma\2\Theta(\omega)^2 + (1+\nu(\omega))\1
\avg{\1\abs{m}\1f}\,
\omega \,=\, 0
\,,\qquad
\psi+\abs{\sigma}^2 \sim 1
\,,
}
since $ \alpha(\tau_0) = v(\tau_0) = 0 $ at the base point  $ \tau_0 $. In order to analyze \eqref{cubic when v=0} we bring it to a {\bf normal form} by an affine transformation. 
This corresponds to expressing the variables $ \omega $ and $ \Theta $ in terms of {\bf normal variables} $ \Omega $ and $ \lambda $, such that
\bels{Omega as affine transformation of Theta}{
\Omega(\lambda) \,=\,
\kappa\,\Theta(\1\Gamma\lambda) + \Omega_0
\,,
}
with some $ \lambda$-independent parameters $ \kappa = \kappa(\tau_0),\Gamma=\Gamma(\tau_0) > 0 $, and $ \Omega_0  \in \C $. 
These parameters will be defined on a case by case basis. 
We remark, that in the proof of Proposition~\ref{prp:Non-zero local minimum} the coordinate transformations \eqref{min: normal coordinates} were of the form \eqref{Omega as affine transformation of Theta}.

In the following, the variable  $ \Omega(\lambda) $ will be identified with roots of various cubic polynomials that depend on the type of base points $ \tau_0 $, similarly to \eqref{min:Omega in terms of A and wht-Omega_a} above.
In order to choose the correct roots we use the following {\bf selection principles}.

\NLemma{Selection principles}{
If $ v(\tau_0) = 0 $ at the base point $ \tau_0 \in \supp v $ of the expansion \eqref{Omega as affine transformation of Theta}, then $ \Omega(\lambda) = \Omega(\lambda;\tau_0) $ defined in \eqref{Omega as affine transformation of Theta} has the properties:
\begin{enumerate}
\item[{\bf SP-1}] $ \lambda \mapsto \Omega(\lambda) $ is continuous;
\item[{\bf SP-2}] $ \Omega(0) = \Omega_0 $;
\item[{\bf SP-3}]$ \Im\bigl[\1\Omega(\lambda)-\Omega(0)\bigr] \ge 0 $, $\forall\1\lambda \in \R$;
\item[{\bf SP-4}] If the imaginary part of $ \Omega $ grows slower than a  square root in a direction $ \theta \in \sett{\pm 1\1}$, 
\[ 
\lim_{\xi \1\to\1 0_+\!} 
\xi^{-1/2}\2\Im\,\Omega(\1\theta\1\xi\1) = 0
\,, 
\]
then $ \Omega|_I $ is real and non-decreasing on an interval $ I := \sett{\1\theta\1\xi: 0<\xi < \Delta\2} $, with some $ \Delta > 0 $.
\end{enumerate}
}
For the proof, by combining \eqref{def of Theta(omega) at tau}, \eqref{b expanded} and \eqref{Omega as affine transformation of Theta} we see that
\bels{abstract Omega from m}{
\Omega(\lambda) \,=\, \kappa\,\avgB{\frac{f}{\abs{m}}\2,\2 m(\tau_0+\Gamma\1\lambda)-m(\tau_0)}+\2\Omega_0
\,,
}
where $ \kappa,\Gamma > 0 $ and $ \Omega_0 \in \C $ are from \eqref{Omega as affine transformation of Theta}.
Thus the first three selection principles follow trivially from the corresponding properties $ \Im\, m(\tau_0) = 0 $ and $ \Im\,m(\tau) \ge 0 $ of $ m $.
The property {\bf SP-4} follows from \eqref{abstract Omega from m} and the next result. 

\NLemma{Growth condition}{
Suppose $ v(\tau_0)=0 $ and that $ \avg{v} $ grows slower than any square-root in a direction $ \theta \in \sett{\pm} $, i.e.,
\bels{less than sqrt-growth in direction theta}{
\liminf_{\xi\1\to\1 0_+\!} \frac{\avg{v(\tau_0+\theta\1\xi)}}{\xi^{1/2}\!} \,=\, 0
\,.
}
Then $ \avg{v}$ actually vanishes, $ \Im\,\avg{m}|_I = 0 $, while $ \Re\,\avg{m} $ is non-decreasing on some interval $ I = \sett{\1 \tau_0+ \theta\2\xi:0\leq \xi\leq \Delta\1}$, for some $ \Delta > 0 $.

If the $ \liminf $ in \eqref{less than sqrt-growth in direction theta} is non-zero, then either $ \theta = \sign \sigma(\tau_0) $ or $ \sigma(\tau_0) = 0 $.
}

\begin{Proof}
We will prove below that if $ v(\tau_0) = 0 $, and 
\bels{tau0 is edge}{
\inf \setb{\1\xi>0:\avg{v(\tau_0+\theta\1\xi)}> 0} \;=\, 0
}
for some direction $ \theta \in \{\pm1\}$, then 
\bels{square root growth}{
\liminf_{\xi\1\to\1 0_+\!} \frac{\avg{v(\tau_0+\theta\1\xi)}}{\xi^{1/2}\!} \,>\,0
\,. 
}

Assuming this implication, the lemma follows easily: If \eqref{less than sqrt-growth in direction theta} holds, then \eqref{tau0 is edge} is not true, i.e., there is a non-trivial interval $ I = \sett{\theta\1\xi:0\leq \xi\leq \Delta\1}$, $ \Delta > 0 $, such that $ v|_I = 0 $. As the negative of a Hilbert-transform of $ v_x $ (cf. \eqref{m as stieltjes transform}), the function $ \tau \mapsto \Re\,m_x(\tau) $, is non-decreasing on  $ I $. This proves the first part of the lemma.  
We will now prove that \eqref{tau0 is edge} implies \eqref{square root growth}. The key idea is to use Lemma~\ref{lmm:Monotonicity} to prove that $ \avg{v} $ grows at least like a square root. However, first we use Proposition~\ref{prp:Non-zero local minimum} to argue that the indicator function on the right hand side of \eqref{minimum growth for v} is non-zero in a non-trivial neighborhood of $ \tau_0 $. 
To this end, assume $ 0 < \avg{v(\tau)} \leq \eps $ and $ \Pi(\tau) < \Pi_\ast $. If $ \eps,\delta > 0 $ are sufficiently small, then Proposition~\ref{prp:Non-zero local minimum} can be applied with $ \tau $ as the base point.
In particular, \eqref{smooth cusp:main result} and \eqref{min-expansion:scaling} imply 
\bels{min expansion scaling}{
\avg{v(\tau+\omega)} \,\sim\, \avg{v(\tau)} + \abs{\omega}^{1/3} 
\,>\,0
\,,
\qquad \abs{\omega}\leq \delta
\,.
} 
Suppose $ \tau_0 $ satisfies \eqref{tau0 is edge}. Since $ v(\tau_0) = 0 $ the lower bound in \eqref{min expansion scaling}, applied to $ \omega = \tau_0-\tau $, implies $ \abs{\tau-\tau_0} > \delta $.   
As $ \tau $ was arbitrary we conclude  $ \Pi(\tau) \ge \Pi_\ast $ for every $ \tau $ in the set
\[
I := \setb{\tau \in \R: \abs{\tau-\tau_0}\leq \delta\,,\;  0 < \avg{v(\tau)}\leq \eps}\,.
\]

Applying Lemma~\ref{lmm:Monotonicity} on $ I $, recalling the upper bound on $ \abs{\partial_z m} $ from Corollary~\ref{crl:Bound on derivative}, yields
\bels{derivative on I}{ 
\avg{v}^{-1} \lesssim\,
(\sign \sigma)\,\partial_\tau \avg{v} \,\lesssim\, \avg{v}^{-2}\,,\quad\text{on}\quad I
\,.
}
Since $ v $ is analytic when non-zero, and  $ \dist(\tau_0,I) = 0 $ by \eqref{tau0 is edge}, we conclude that $ I $ equals the interval with end points $ \tau_0 $ and $ \tau_1 := \tau_0 + \theta \delta $. Here we set $ \delta \lesssim \eps^3 $ so small that the $ 1/3$-H\"older continuity of $  m $ guarantees $ \avg{v} \leq \eps $ on $ I $.
Moreover, $ \sign \sigma(\tau) $ must equal the constant $ \theta $ for every $ \tau \in I $: If $ \sigma $ changed its sign at some point $ \tau_\ast \in I $ this would violate $ \Pi(\tau_\ast) \ge \Pi_\ast $ as $ \avg{v} $ is a continuous function.

Integrating \eqref{derivative on I} from $ \tau_0$ to $ \tau_1 $ we see that $ \avg{v(\tau_0+\theta\1\xi)}^2 \gtrsim \xi $ for any $ \xi \leq \abs{\tau_1-\tau_0} $. This proves the limit \eqref{square root growth}, and hence the first part of the lemma. 
The second part of the lemma follows from \eqref{derivative on I}.
\end{Proof}

\subsection{Simple edge and sharp cusp}

When $ \abs{\sigma} > 0 $ and $ \abs{\omega} $ is sufficiently small compared to $\abs{\sigma} $ the cubic term $ \psi\,\Theta(\omega)^3 $ in \eqref{cubic when v=0} can be ignored. In this regime the following simple expansion holds showing the square root behavior of $ v$ near an edge of its support.
 
\NLemma{Simple edge}{
If $ \tau_0 \in \supp v $ satisfies $ v(\tau_0) = 0 $ and $ \sigma = \sigma(\tau_0) \neq 0 $, then 
\bels{simple edge: expansion}{
v_x(\tau_0+\omega)
\,=\, 
\begin{cases}
\displaystyle
h'_x\,
\absB{\frac{\omega}{\,\sigma}}^{1/2} \msp{-10}+ 
\Ord\Bigl(\frac{\omega}{\sigma^2\msp{-2}}\Bigr) &\text{if }\quad 0 \leq (\sign \sigma)\2 \omega \leq c_\ast\abs{\sigma}^3\,;
\\
\displaystyle
\;0 &\text{if } -\1c_\ast\abs{\sigma}^3 \leq (\sign \sigma)\2 \omega \leq 0\,;
\end{cases}
}
for some sufficiently small $ c_\ast \sim 1 $. Here $ h' = h'(\tau_0) \in \BB $ satisfies $ h'_x \sim 1 $.
}

This result already shows that $ \supp v $ continues in the direction $ \sign \sigma(\tau_0) $ and in the opposite direction there is a gap of length $\Delta(\tau_0) \gtrsim \abs{\sigma(\tau_0)}^3 $ in the set $ \supp v $. We will see later (cf. Lemma~\ref{lmm:Size of small gap}) that for small $ \abs{\sigma(\tau_0)} $ there is an asymptotically sharp correspondence between $ \Delta(\tau_0) $ and $\abs{\sigma(\tau_0)}^3 $, as $ \Delta(\tau_0)  $ becomes very small.  
\begin{Proof}
Treating the cubic term $ \psi\2\Theta^3 $ in \eqref{cubic when v=0} as a perturbation, \eqref{cubic when v=0} takes the form
\bels{normal form for regular edge}{
\Omega(\lambda)^2 + 
\Lambda(\lambda)
\,=\, 0
\,,  
}
in the normal coordinates,
\bels{simple edge: normal coordinates}{
\lambda \,&:=\, 
\frac{\omega}{\sigma}
\\
\Omega(\lambda) \,&:=\, 
\frac{\Theta(\sigma\1\lambda\1)}{\msp{-6}\sqrt{\!\avg{\1\abs{m}\1f\1}\2}}
\,,
}
where $ \Lambda :\R \to \C $ is a multiplicative perturbation of $ \lambda $:
\bels{}{
\Lambda(\lambda) \,&:=\, (\11+\mu(\lambda))\2 \lambda 
\\
1 \,+\, \mu(\lambda)
\,&:=\, 
\frac{1\,+\,\nu(\sigma\1\lambda\1)}{1+(\psi/\sigma)\2\Theta(\sigma\1\lambda\1)}
\,.
} 
Let $ \lambda_\ast = c_\ast\abs{\sigma}^2 $, with some $ c_\ast \sim 1 $, so that the constraint $ \abs{\omega} \leq c_\ast \abs{\sigma}^3$ in \eqref{simple edge: expansion} translates into $ \abs{\lambda} \leq \lambda_\ast $.

Using the a priori bounds \eqref{a priori for Theta(omega)} and \eqref{a priori for nu(omega)} for $ \Theta $ and $ \nu $ yields
\bels{simple edge: a priori bound for mu}{
\abs{\1\mu(\lambda)} 
\,\lesssim\,
\Bigl(\21+\frac{\psi}{\abs{\sigma}}\Bigr)
\absb{\1\Theta(\sigma\1\lambda\1)} + \abs{\sigma}\2\abs{\lambda}
\,\lesssim\, 
c_\ast^{1/3}
\,.
} 
Hence, for sufficiently small $ c_\ast \sim 1 $ we get $ \abs{\mu(\lambda)} < 1 $, provided $ \abs{\lambda} \leq \lambda_\ast $.

Let us define two root functions $ \wht{\Omega}_a : \C \to \C $, $ a =\pm$, such that 
\bels{regular edge:normal form}{
\wht{\Omega}_a(\zeta)^2 +\zeta \,=\, 0\,,
}
by setting
\bels{reg edge: defs of wht-Omega_a}{
\wht{\Omega}_\pm(\zeta) \,:=\,
\,\pm
\begin{cases}
\cI\2\zeta^{\11/2} &\text{if } \Re\,\zeta \ge 0\,;
\\
-(-\zeta)^{\11/2} &\text{if } \Re\,\zeta < 0\,.
\end{cases}
}
Note that we use the same symbol $ \wht{\Omega}_a $ for the roots as in \eqref{solutions to positive reduced cube} for different functions. In each expansion $ \wht{\Omega}_a $ will denote the root function of the appropriate normal form of the cubic.

Comparing \eqref{normal form for regular edge} and \eqref{regular edge:normal form} we see that there exists a \emph{labelling function} $ A : \R \to \sett{\pm} $, such that 
\[
\Omega(\lambda) \1=\2 \wht{\Omega}_{A(\lambda)}(\Lambda(\lambda)) 
\,,
\]
for every $ \lambda \in \R $.
The function $ A|_{[-\lambda_\ast,\2\lambda_\ast]} $ will now be determined using the selection principles {\bf SP-1} and {\bf SP-3}.

The restrictions of the root functions onto the half spaces $ \Re\,\zeta > 0 $ and  $ \Re\,\zeta < 0 $ are continuous (analytic) and distinct, i.e., $ \wht{\Omega}_+(\zeta) \neq \wht{\Omega}_-(\zeta) $ for $ \zeta \neq 0 $. 
Since $ \Omega : \R \to \C $ is also continuous by {\bf SP-1},  $ A(\lambda) $ may change its value at some point $ \lambda = \lambda_0 $ only if $ \Lambda(\lambda_0) = 0 $.
Since $ \abs{\mu(\lambda)} < 1 $ for $ \abs{\lambda} \leq \lambda_\ast $ we conclude that $ \Lambda(\lambda) = 0 $ only for $ \lambda = 0 $. Thus, there exist two labels $ a_+,a_- \in \sett{\pm} $, such that
\bels{reg edge: A constant}{
A(\lambda) \,=\, a_\pm\qquad
\forall\,\lambda \in  \pm\2(\20,\lambda_\ast\1]
\,.
}
 
Let us first consider the case $ \lambda \ge 0 $, and show that $ a_+ = + $. Indeed, the choice $ a_+ = - $ is ruled out, since
\bels{reg edge:imaginary part}{
\Im\,\wht{\Omega}_-(\Lambda(\lambda)) 
\,&=\, 
\Im\Bigl[\,-\2\cI\2(\11+\mu(\lambda)\1)^{1/2}\2\lambda^{1/2}\Bigr]
\,=\,
-\1\lambda^{1/2} + \Ord\Bigl(\mu(\lambda)\2\lambda^{1/2}\Bigr)
}
is negative for sufficiently small $ c_\ast \sim 1 $ in \eqref{simple edge: a priori bound for mu}, and this violates the selection principle {\bf SP-3}. 

By definitions, 
\[
\abs{\Theta(\1\sigma\1\lambda)} 
\,\sim\, 
\abs{\2\wht{\Omega}_+(\lambda)} 
\,\lesssim\, 
\abs{\1\Lambda(\lambda)}^{1/2} \sim \abs{\lambda}^{1/2}
\,.
\]
Using $ \psi/\abs{\sigma} \lesssim \abs{\sigma}^{-1} $, with $ \abs{\sigma} \gtrsim 1 $, we write \eqref{simple edge: a priori bound for mu} in the form $ \abs{\mu(\lambda)} \lesssim  \abs{\sigma}^{-1}\abs{\lambda}^{1/2} $. Similarly, as \eqref{reg edge:imaginary part} we obtain
\[
\Omega(\lambda) 
\,=\, 
\wht{\Omega}_+(\lambda) + \Ord\Bigl(\mu(\lambda)\2\lambda^{1/2}\Bigr)
\,=\,
\cI\2\lambda^{1/2} +\, \Ord\Bigl(\frac{\lambda}{\sigma}\Bigr)
\,,
\qquad \lambda \in [\20\1, \lambda_\ast]
\,.
\]
Inverting \eqref{simple edge: normal coordinates} we obtain 
\bels{simple expansion for Im Theta}{
\Im\,\Theta(\omega) \,=\, \avg{\1\abs{m}\2f\1}^{\11/2}\absB{\frac{\omega}{\,\sigma}}^{1/2} \!+\,\Ord\Bigl(\frac{\omega}{\sigma^2\msp{-2}}\Bigr)\,,\qquad \sign \sigma = \sign \omega\,.
}
Taking the imaginary part of \eqref{m bounded by Theta and b with Theta2-error} and using \eqref{simple expansion for Im Theta} yields the first line of \eqref{simple edge: expansion}, with $ h'_x = \abs{m_x}\1f_x/\avg{\abs{m}\1f\1}^{1/2} $. Since $ \abs{m_x},f_x \sim 1 $, we also have $ h'_x \sim 1 $.

In order to prove the second line of \eqref{simple edge: expansion} we show that the gap length (cf. \eqref{def of Delta(tau0)}) satisfies
\bels{reg edge: gap length lower bound}{
\Delta(\tau_0)\,\gtrsim\, \abs{\1\sigma(\tau_0)}^3\,.
}
At the opposite edge of the gap $ \tau_1 := \tau_0-\sign \sigma(\tau_0)\2\Delta(\tau_0) $, the density $ \avg{v} $ increases, by definition, in the opposite direction than at $ \tau_0 $. 
By Lemma~\ref{lmm:Growth condition} the average generating density $\avg{v}$ increases at least like a square root function and either $ \sign \sigma(\tau_1) = - \sign \sigma(\tau_0) $ or $ \sigma(\tau_1) = 0 $.
Since $ \sigma $ is $1/3$-H\"older continuous, $ \sigma $ can not change arbitrarily fast. Namely, we have $ \Delta(\tau_0) \gtrsim \abs{\1\sigma(\tau_0)}^3 $, and this proves \eqref{reg edge: gap length lower bound}.
\end{Proof}

Although not necessary for the proof of the present lemma, it can be shown that  $ a_- := - $  using the selection principle {\bf SP-4}. The same reasoning will be used in the proofs of the next two lemmas (cf. \eqref{cusp:imaginary part} and discussion after that).

Next we consider the marginal case where the term $ \sigma\,\Theta(\omega)^2 $ is absent in the cubic \eqref{cubic when v=0}. In this case $ \avg{v} $ has a cubic root cusp shape around the base point. 

\NLemma{Vanishing quadratic term}{
If $\tau_0 \in \supp v $ is such that $ v(\tau_0) = \sigma(\tau_0) = 0 $, then
\bels{Vanishing quadratic term:expansion}{
v_x(\tau_0+\omega\1) \,=\, 
h_x\2\abs{\1\omega}^{1/3} \!+ \Ord\Bigl(\1\abs{\omega}^{2/3}\Bigr)
\,,
}
where $ h = h(\tau_0) \in \BB $ satisfies $ h_x \sim 1 $.
}

Contrasting this with Lemma~\ref{lmm:Simple edge} shows that $ \sigma(\tau_0) \neq 0 $ for $ \tau_0 \in \partial \supp v $. In particular, the gap length $ \Delta(\tau_0) $ is always well defined for  $ \tau_0 \in \partial \supp v $ (cf. \eqref{def of Delta(tau0)}).

\begin{Proof}
First we note that it suffices to prove \eqref{Vanishing quadratic term:expansion} only for $ \abs{\omega} \leq \delta $, where $ \delta \sim 1$ can be chosen to be sufficiently small.
When $ \abs{\omega} > \delta $ the last term may dominate the first term on the right hand side of \eqref{Vanishing quadratic term:expansion}, and thus we have nothing prove.
Since $ \sigma = 0 $, the quadratic term is missing in \eqref{cubic when v=0}, and thus the  cubic reduces to 
\bels{cusp: normal form}{
\Omega(\omega)^3 + \Lambda(\omega) = 0\,, 
}
using the normal coordinates 
\bels{normal coordinates:cusp}{
\lambda \,&:=\, \omega
\\
\Omega(\lambda) &:= \Bigl(\frac{\psi}{\avg{\1\abs{m}f}}\Bigr)^{\!1/3} \Theta(\lambda)
\,.
}
Here, $ \Lambda : \R \to \C $ is a perturbation of the identity function:
\bels{}{ 
\Lambda(\lambda) &:= (1+\nu(\lambda)\1)\2\lambda
\,.
}
Note that $ \psi \sim 1 $ because of \eqref{stability of shape cubic}. 

Let us define three root functions $ \wht{\Omega}_a : \C \to \C $, $ a=0,\pm $, satisfying 
\[
\wht{\Omega}_a(\zeta)^3 +\zeta = 0\,,
\]
by the explicit formulas
\bels{}{
\wht{\Omega}_0(\zeta) \,&:=\, -\, p_3(\zeta)
\\
\wht{\Omega}_\pm(\zeta) \,&:=\, \frac{-\11\pm \cI\2\sqrt{3}}{2}\,p_3(\zeta)
\,,
}
where $ p_3 : \C \to \C $ is a (non-standard) branch of the complex cubic root,
\bels{}{
p_3(\zeta)\,&:=\, 
\begin{cases}
\;\zeta^{\11/3}\quad&\text{when }\Re\,\zeta > 0\,;
\\
-\1(-\1\zeta\1)^{1/3} &\text{when }\Re\,\zeta < 0\,. 
\end{cases}
}

From \eqref{cusp: normal form} we see that there exists a labelling $ A : \R \to \sett{0,\pm} $, such that 
\bels{cusp: Omega in terms of wht-Omega and A}{
\Omega(\lambda) \,=\, \wht{\Omega}_{A(\omega)}(\Lambda(\lambda)) 
\,.
}

Similarly as before, we conclude that $ \Omega $ and the roots are continuous (cf. {\bf SP-1}) on $ \R $ and on the half-spaces $ \sett{\zeta \in\C:\pm\Re\2\zeta>0}$, respectively.
This implies that $ A(\lambda_0-0) \neq A(\lambda_0+0) $ if and only if $ \Lambda(\lambda_0)= 0 $. 
From the a priori estimate $ \abs{\nu(\lambda)} \lesssim \abs{\lambda}^{1/3}$ (cf. \eqref{a priori for nu(omega)}) we see that there exists $ \delta \sim 1 $ such that $ \Lambda(\lambda) \neq 0 $, for $ 0 < \abs{\lambda} \leq \delta $. Hence, we conclude 
\bels{}{
A(\lambda) = a_\pm\,,\qquad\forall\,\lambda \in \pm\1(\10\1,\delta\2]
\,.
}

The choices $ a_+ = - $ and $ a_- = + $ are excluded by the selection principle {\bf SP-3}:
Similarly as in \eqref{reg edge:imaginary part}, we get
\bels{cusp:imaginary part}{
\pm\,(\sign \lambda)\,\Im\,\wht{\Omega}_\pm(\Lambda(\lambda)) 
\,&=\,
\frac{\msp{-6}\sqrt{3\2}\2}{2}\abs{\lambda}^{1/3} + \Ord\Bigl(\mu(\lambda)\2\lambda^{1/3}\Bigr)
\,\ge\,
\abs{\lambda}^{1/3} - C\1\abs{\lambda}^{2/3}
\,.
}
From this it follows that $ \Im\,\wht{\Omega}_-(\Lambda(\lambda)) < 0 $ for small $\abs{\lambda} > 0 $. Thus {\bf SP-3} implies $ a_\pm \neq \mp $.

We will now exclude the choices $ a_\pm = 0 $. Similarly as \eqref{cusp:imaginary part} we use \eqref{a priori for nu(omega)} to get 
\bels{cusp: ready for SP-4}{
\Re\,\wht{\Omega}_0(\Lambda(\lambda)) \,&\leq\, - \lambda^{1/3} + C\lambda^{2/3}
\\ 
\Im\,\wht{\Omega}_0(\Lambda(\lambda)) 
\,&\lesssim\, 
\abs{\1\nu(\lambda)}\2\abs{\lambda}^{1/3} \lesssim \abs{\lambda}^{2/3} 
\,,
}
for $ \lambda \ge 0 $. 
If $ a_+ = 0 $, then these two bounds together would violate {\bf SP-4}.
The choice $ a_- =0 $ is excluded similarly. Thus we are left with the unique choices $ a_+ =+$ and $ a_- = - $. 

The expansion \eqref{Vanishing quadratic term:expansion} is obtained similarly as in the proof of Lemma~\ref{lmm:Simple edge}. 
First, we use \eqref{normal coordinates:cusp} and  \eqref{cusp:imaginary part} to solve for $ \Im\, \Theta(\omega) $. Then we take the imaginary part of \eqref{m bounded by Theta and f with X13-error} to express $ v_x(\tau_0+\omega) $ in terms of $ \Im\,\Theta(\omega) $. We identify 
\[
h_x \,:=\, \frac{\msp{-6}\sqrt{3\2}\2}{2}\Bigl(\frac{\avg{\1\abs{m}\1f\1}}{\psi}\Bigr)^{\!1/3}\2\abs{m_x}f_x\,, 
\]
in the expansion \eqref{Vanishing quadratic term:expansion}. From $ \psi,\abs{m},f \sim 1 $ it follows that $h_x \sim 1 $.
\end{Proof}

\subsectionl{Two nearby edges}

In this section we consider the generic case of the cubic \eqref{cubic when v=0} where neither the cubic nor the quadratic term can be neglected. 
First, we remark that Lemma~\ref{lmm:Simple edge} becomes ineffective as $ \abs{\sigma} $ approaches zero since the cubic term of 
\bels{generic edge cubic}{
\psi\2\Theta(\omega)^3 + \sigma\2\Theta(\omega)^2 + (1+\nu(\omega))\1
\avg{\abs{m}\1f\1}\2
\omega \,=\, 0
\,,\qquad
\psi,\sigma \neq 0
\,,
}
was treated as a perturbation of a quadratic equation along with $ \nu(\omega) $ in the proof. 
Thus we need to consider the case where $ \abs{\sigma} $ is small.
Indeed, we will assume that $ \abs{\sigma} \leq \sigma_\ast $, where  $\sigma_\ast \sim 1 $ is a threshold parameter that will be adjusted so that the analysis of the cubic \eqref{generic edge cubic} simplifies sufficiently.
In particular, we will choose $ \sigma_\ast $ 
so small that the number $ \wht{\Delta} = \wht{\Delta}(\tau_0) >  0 $ defined by 
\bels{def of wht-Delta}{
\wht{\Delta} 
\,:=\, 
\frac{4}{27\1\avg{\1\abs{m}\1f\1}}\,
\frac{\abs{\1\sigma}^3}{\psi^{\12}} 
\,,
}
satisfies
\bels{Delta upper bound in sigma}{
\wht{\Delta} \,\sim\, \abs{\sigma}^3
\,,\qquad
\text{provided}\quad\abs{\sigma} \leq \sigma_\ast
\,.
}
Note that the existence of $ \sigma_\ast \sim 1 $ such that \eqref{Delta upper bound in sigma} holds follows from  $ f_x,\abs{m_x} \sim 1 $ and the stability of the cubic \eqref{stability of shape cubic}. Indeed,  \eqref{stability of shape cubic} shows that $ \psi \sim 1 $ when $ \abs{\sigma} \leq \sigma_\ast $ for some small enough $ \sigma_\ast \sim 1 $.
We will see below (cf. Lemma~\ref{lmm:Size of small gap}) that $ \wht{\Delta}(\tau_0) $ approximates the gap length $ \Delta(\tau_0) $ when the latter is small.

Introducing the normal coordinates,
\bels{generic edge: def of lambda and Omega}{
\lambda \2&:=\, 
2\2\frac{\omega}{\wht{\Delta}}
\\
\Omega(\lambda) \;&:=\; 
3\,\frac{\psi}{\2\abs{\sigma}}\2
\Theta\Bigl(\frac{\wht{\Delta}}{2}\2\lambda\Bigr)
\,+\, \sign \sigma\,,
}
the generic cubic \eqref{generic edge cubic} reduces to
\bels{negative reduced cubic}{
\Omega(\lambda)^{\13} - \13\,\Omega(\lambda) \2+\1 2\2\Lambda(\lambda) \,=\, 0
\,,
}
with the constant term
\begin{align}
\label{small gap: def of Lambda}
\Lambda(\lambda) \,&:=\; 
\sign \sigma \,+\,(1+\mu(\lambda))\1\lambda\,,
\\
\label{small gap: def of mu}
\mu(\lambda)
\,&:=\,\nu\Bigl(\frac{\wht{\Delta}}{2}\lambda\Bigr) 
\,.
\end{align}
Here, $ \Lambda(\lambda) $ is considered as a perturbation of $ \sign \sigma + \lambda $. Indeed, from \eqref{a priori for nu(omega)} and \eqref{small gap: def of mu} we see that $ \abs{\mu(\lambda)} \lesssim \delta^{\11/3} $. 

The left hand side of equation \eqref{negative reduced cubic} is a cubic polynomial of $ \Omega(\lambda) $ with a constant term $ \Lambda(\lambda) $. 
It is very similar to \eqref{min: normal form} but with an opposite sign in the linear term. Cardano's formula in this case read as follows.

\NLemma{Roots of reduced cubic with negative linear coefficient}{
For any $ \zeta \in \C $, 
\bels{reduced negative cubic factorized}{
\Omega^{\13}-\,3\2\Omega \,+\, 2\2\zeta \;=\; (\1\Omega-\wht{\Omega}_+(\zeta))(\1\Omega-\wht{\Omega}_0(\zeta))(\1\Omega-\wht{\Omega}_-(\zeta))
\,,
}
where the three root functions $ \wht{\Omega}_\varpi : \C \to \C $, $ \varpi =0,\pm$, have the form
\begin{subequations}
\label{roots for small gap}
\bels{def of root functions for small gap}{
\wht{\Omega}_0 \,&:= -\1(\1\Phi_{+}+\1\Phi_{-})
\\
\wht{\Omega}_\pm \,&:=\, 
\frac{1}{2}
(\1\Phi_{+}+\1\Phi_{-})
\,\pm\,
\cI\2\frac{\!\sqrt{3\2}}{2\!}
(\1\Phi_{+}-\1\Phi_{-})
\,.
}
The auxiliary functions $ \Phi_{\pm} : \C \to \C $, are defined by (recall Definition~\ref{def:Complex powers})
\bels{def of Phi_pm}{
\Phi_{\pm}(\zeta) \,:=\, 
\begin{cases}
\bigl(\1\zeta \pm \sqrt{\zeta^{\12}-1\2}\,\bigr)^{\!1/3}\quad&\text{if}\quad\Re\,\zeta \ge 1\,,
\\
\bigl(\1\zeta \pm \cI\1\sqrt{1-\zeta^{\12}}\,\bigr)^{\!1/3}
&\text{if}\quad\!\abs{\1\Re\, \zeta \1} < 1\,,
\\
-\bigl(-\1\zeta \mp \sqrt{\zeta^{\12}-1\2}\,\bigr)^{\!1/3} 
&\text{if}\quad\Re\, \zeta \leq -1\,.	
\end{cases}
} 
\end{subequations}

On the simply connected complex domains
\bels{defs of C_a}{
\wht{\C}_0   \,:=\, \setb{\zeta \in \C : \abs{\Re\,\zeta} < 1}
\,,\qquad\text{and}\qquad
\wht{\C}_\pm \,:=\, \setb{\zeta \in \C : \pm\1\Re\,\zeta > 1}
\,,
}
the respective restrictions of $ \wht{\Omega}_a $ are analytic and distinct. Indeed, if  $ \wht{\Omega}_a(\zeta) = \wht{\Omega}_b(\zeta) $ holds for some $ a \neq b $ and $\zeta \in \C $, then $ \zeta = \pm\1 1 $.
}

This lemma is analogue of Lemma~\ref{lmm:Roots of reduced cubic with positive linear coefficient} but for \eqref{negative reduced cubic} instead of \eqref{min: normal form}. As before the meaning of the symbols $ \wht{\Omega}_a $, $ \lambda $, etc., is changed accordingly. 

Comparing \eqref{negative reduced cubic} and \eqref{reduced negative cubic factorized} we see that there exists a function $ A : \R \to \sett{\10,\pm} $ such that
\bels{Omega in terms of labelling function and the roots}{
\Omega(\lambda) = \wht{\Omega}_{A(\lambda)}(\Lambda(\lambda))
\,.
}
We will determine the values of $ A $ inside the following three intervals
\bels{def of I_k's}{
I_1 &:= -\1(\sign \sigma)\1[-\lambda_1,\20\2)
\\
I_2 &:= -\1(\sign \sigma)\1(\20\1,\lambda_2\1]
\\
I_3 &:= -\1(\sign\sigma)\1[\2\lambda_3,\lambda_1]
\,,
}
which are defined by their boundary points,
\bels{def of lambda_k's}{
\lambda_1 \,:=\, 2\2\frac{\delta}{\wht{\Delta}}\,,
\qquad
\lambda_2 \,:=\, 
2-\varrho\2\abs{\sigma}
\,,
\qquad
\lambda_3 \,:=\,
2+\varrho\2\abs{\sigma}  
\,,
}
for some $ \varrho \sim 1 $. The shape of the imaginary parts of the roots $\widehat \Omega_a$ on the intervals $I_1$, $I_2$ and $I_3$ is shown in Figure \ref{Fig:Cubic roots}. 
The number $ \lambda_1 $ is the expansion range $ \delta $ in the normal coordinates. From \eqref{Delta upper bound in sigma} it follows that
\bels{lambda_1 comparison to delta/abs-sigma3}{
\qquad 
c_1\frac{\delta}{\2\abs{\sigma}^3\!} 
\,\leq\, 
\lambda_1 
\,\leq\, 
C_1\frac{\delta}{\2\abs{\sigma}^3\!}
\,,\qquad\text{provided}\quad
\abs{\sigma} \leq \sigma_\ast\,.
}
\begin{figure}[h]
	\centering
	\hspace{-0.04\textwidth}
	\includegraphics[width=1.02\textwidth]{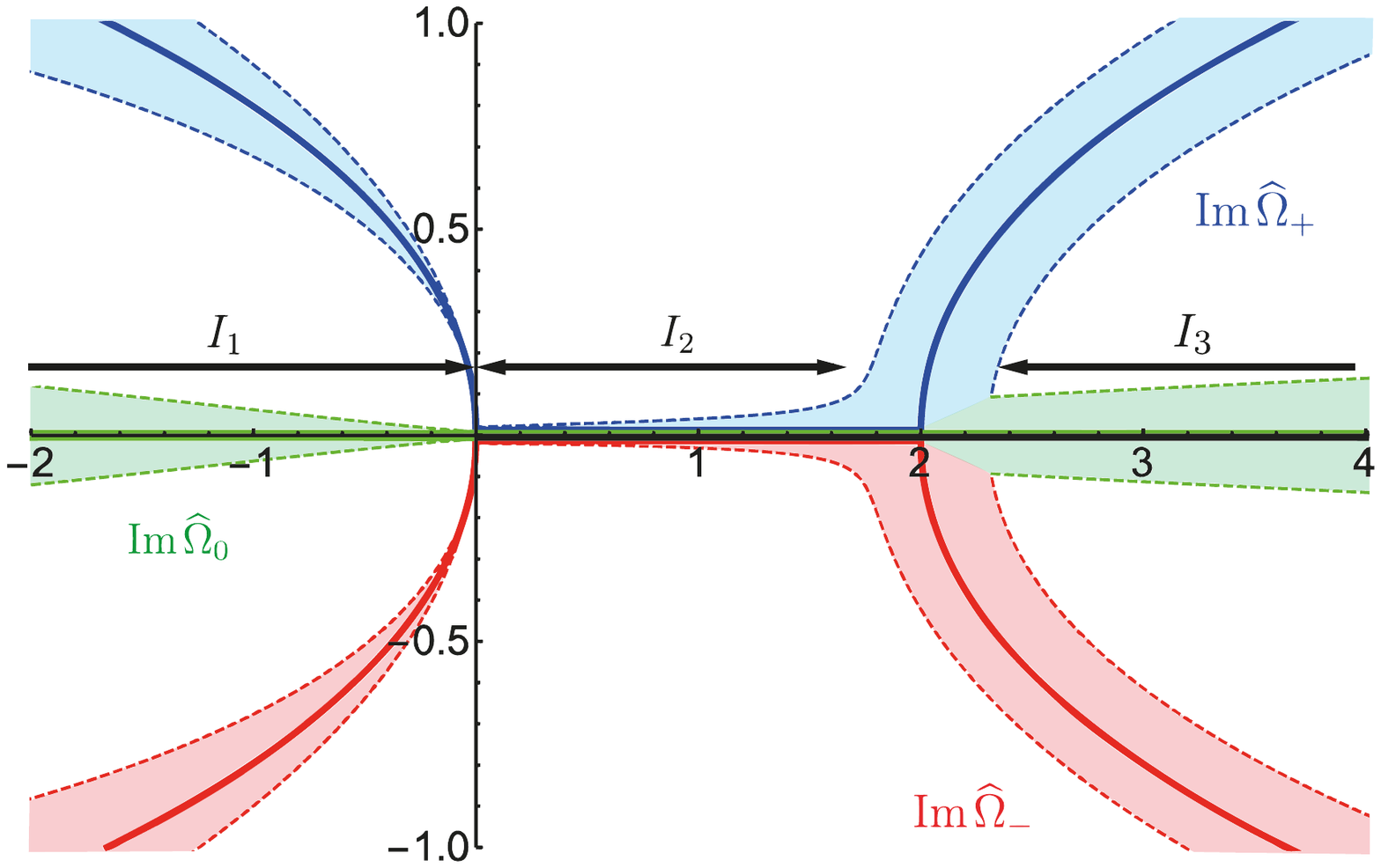}
	\caption{Imaginary parts of the three branches of the roots of the cubic equation. The true solution remains within the allowed error margin indicated by the dashed lines.}
\label{Fig:Cubic roots}
\end{figure}
The points $\lambda_2 $ and $ \lambda_3 $ will act as a lower and an upper bound for the size of the gap in $\supp v $ associated to the edge $ \tau_0 $, respectively. 
Given any $ \delta,\varrho \sim 1 $ 
we can choose $ \sigma_\ast \sim 1 $ so small that
\bels{good endpoints}{
\lambda_1 \,\ge\, 4\,,\quad\text{and}\quad 
1 \,\leq\, \lambda_2 \,<\, 2 \,<\, \lambda_3 \,\leq\, 3
\,,\qquad\text{provided}\quad
\abs{\sigma} \leq \sigma_\ast
\,.
}
In particular, the intervals \eqref{def of I_k's} are disjoint and non-trivial for a triple $ (\delta,\varrho,\sigma_\ast) $ chosen this way. 
The  value $ A(\lambda) $ can be uniquely determined using the selection principles if $ \lambda $ lies inside one of the intervals \eqref{def of I_k's}.

\begin{lemma}[Choice of roots]
\label{lmm:Choice of roots}
There exist $ \delta,\varrho,\sigma_\ast \sim 1 $, such that \eqref{good endpoints} holds, and if
\[ 
\abs{\sigma} 
\,\leq\, \sigma_\ast
\,,
\] 
then the restrictions of $ \Omega $ on the intervals $ I_k := I_k(\delta,\varrho,\sigma,\wht{\Delta}) $, defined in \eqref{def of I_k's}, satisfy:
\bels{Omega|_I_k's determined}{
\Omega|_{I_1} &= \wht{\Omega}_+ \msp{-2}\circ \Lambda|_{I_1}
\\
\Omega|_{I_2} &=  \wht{\Omega}_+  \msp{-2}\circ \Lambda|_{I_2}
\\
\Omega|_{I_3} &= \wht{\Omega}_+  \msp{-2}\circ \Lambda|_{I_3}
\,.
}
Moreover, we have
\bels{upper bound on gap}{ 
\Im\,\Omega(-\sign\sigma \,\lambda_3) \,>\, 0
\,.
}
\end{lemma}

The proof of the following simple result is given in Appendix~\ref{sec:Cubic roots and associated auxiliary functions}.

\begin{lemma}[Stability of roots]
\label{lmm:Stability of roots - neg}
On the connected components of $ \wht{\C} $ the roots \eqref{def of root functions for small gap} are stable, i.e., 
\bels{Perturbation of roots:weak bound}{
\absb{\2\wht{\Omega}_a(\zeta)-\wht{\Omega}_a(\xi)} 
\,\lesssim\,
\min\setB{\2\abs{\1\zeta\2-\2\xi\1}^{1/2}\!,\abs{\1\zeta\2-\2\xi\1}^{1/3}}
\,,\quad
(\zeta,\xi) \in \wht{\C}_-^2 \cup \wht{\C}_0^2 \cup \wht{\C}_+^2\,,\msp{-10}
}
holds for $a =-,0,+ $.

In particular, suppose $ \zeta$ and $ \xi $ are of the following special form
\bea{
\xi  \,&=\, -\theta+\lambda 
\\
\zeta \,&=\,  -\theta+(1+\mu'\1)\1\lambda
\,,
}
where $ \theta = \pm 1 $,  $ \lambda \in \R $ and $ \mu' \in \C $. 
Suppose also that $ \abs{\lambda-2\1\theta}\ge 6\2\kappa $, and $ \abs{\mu'} \leq \kappa$, for some $ \kappa \in (0,1/2)$. Then for each $a =-,0,+ $ the function $\wht{\Omega}_a$ satisfies
\bels{Perturbation of roots: away from opposite critical point}{
\absb{\2\wht{\Omega}_a(\zeta)-\wht{\Omega}_a(\xi)} \,\lesssim\, \frac{\min\setb{\abs{\lambda}^{1/2}\!,\2\abs{\lambda}^{1/3}}}{\kappa^{\11/2}\!}\,\abs{\1\mu'}
\,.
}
\end{lemma}

Using Lemma~\ref{lmm:Stability of roots - neg} we may treat $ \Lambda(\lambda) $ as a perturbation of $ \sign \sigma + \lambda $ by a small error term $ \lambda\2\mu(\lambda)$. 
By expressing the a priori bounds \eqref{a priori for nu(omega)} for $ \nu(\omega) $ in the normal coordinates \eqref{generic edge: def of lambda and Omega}, and recalling that $ \abs{\lambda} \leq \lambda_1 $ is equivalent to $ \abs{\omega} \leq \delta $, we obtain estimates for this error term,
\begin{subequations}
\label{mu bounds for label selection}
\begin{align}
\label{1st mu bound for label selection}
\qquad
\abs{\mu(\lambda)} 
\,&\leq\, 
C_2\abs{\sigma}\abs{\lambda}^{1/3} 
\\
\label{2nd mu bound for label selection}
&\leq\, 
C_3\2\delta^{\11/3}
\,,\qquad\text{provided}\quad
\abs{\sigma} \leq \sigma_\ast\,,\quad\abs{\lambda} \leq \lambda_1
\,.
\end{align}
In the following we will assume that $ \delta \leq (2\1C_3)^{-3} \sim 1 $, so that 
\bels{abs-mu leq 1/2}{
\sup_{\lambda \1:\1 \abs{\lambda}\leq \lambda_1}\abs{\1\mu(\lambda)} \,\leq\, \frac{1}{2}
\,,\qquad\text{provided}\quad
\abs{\sigma} \leq \sigma_\ast
\,.
} 
The a priori bound in the middle of \eqref{a priori for nu(omega)} also yields  the third estimate of $ \mu $ in terms of $ \Omega $ and $ \lambda $. Indeed, inverting \eqref{generic edge: def of lambda and Omega} and using $ \Omega(0) = \sign \sigma = 1 $ (also from \eqref{generic edge: def of lambda and Omega}), we get
\bels{mu bound for edge shape}{
\abs{\1\mu(\lambda)} \,&\lesssim\,
\abs{\sigma}
\abs{\1\Omega(\lambda)-\2\Omega(0)} 
\,+\, \abs{\sigma}^3\abs{\lambda}
\,,\qquad\text{provided}\quad
\abs{\sigma} \leq \sigma_0\,.
}
\end{subequations}
For the sake of convenience, we will restrict our analysis to the case $ \sign \sigma = -1 $. The opposite case is handled similarly.

We will use the notations $ \varphi(\tau+0) $ and $  \varphi(\tau-0) $, for the right and the left limits 
$ \lim_{\xi\downarrow \tau} \varphi(\xi) $ 
and 
$ \lim_{\xi \uparrow \tau} \varphi(\xi) $, 
respectively. 

\begin{Proof}[Proof of Lemma~\ref{lmm:Choice of roots}]
Let us assume  $ \sign \sigma = - 1 $. We will consider $ \delta \sim 1 $ and $ \varrho \sim 1 $ as free parameters which can be adjusted to be as small and large as we need, respectively. 
Given $ \delta \sim 1 $ and $ \varrho \sim 1 $ the threshold $ \sigma_\ast \sim 1 $ is then chosen so small that \eqref{good endpoints} holds.

First we show that $ A(\lambda) $ is constant on each $ I_k $, i.e., there are three labels $ a_k \in \sett{0,\pm}  $ such that 
\bels{A = a_k on I_k}{ 
A(\lambda) \,=\, a_k\,,\qquad\forall\,\lambda \in I_k
\,,\qquad k=1,2,3\,.
}
In order to prove this  we first recall that the root functions $ \zeta \mapsto \wht{\Omega}_a(\zeta)$, $ a,b=0,\pm $, are continuous on the domains $ {\wht{\C}_b} $, $ b=0,\pm$, and that they may coincide only at points $ \Re\,\zeta = \pm 1 $ (Indeed, the roots  coincide only at the two points $ \zeta = \pm 1 $.).
From Lemma~\ref{lmm:Cubic for shape analysis} and {\bf SP-1} we see that $  \Lambda,\Omega : \R \to \C $ are continuous. 
Hence, \eqref{A = a_k on I_k} will follow from
\bels{Lambda(I_k) subset wht-Cp_a}{
\Lambda(I_1) \subset \wht{\C}_-\,,
\qquad
\Lambda(I_2) \subset \wht{\C}_0\,,
\qquad 
\Lambda(I_3) \subset \wht{\C}_+
\,,
}
since $\abs{\1\Re\,\zeta\1} \neq 1 $ for $ \zeta \in \cup_a \wht{\C}_a $ (cf. \eqref{defs of C_a}). 

From \eqref{small gap: def of Lambda} and \eqref{abs-mu leq 1/2} we get
\bels{Re Lambda(I_1)}{
\Re\,\Lambda(\lambda) 
\,=\, 
-1 - (\11+ \Re\,\mu(\lambda))\abs{\lambda} 
\,\leq\,
-1 - \frac{1}{2}\abs{\lambda}
\,<\,-1
\,,\qquad
\lambda \in I_1\,,
}
and thus $ \Lambda(I_1) \subset \wht{\C}_- $.
Similarly, we get the first estimate below:
\bels{Re Lambda(I_2)}{
-1 + \frac{1}{2}\abs{\lambda} 
\,\leq\,
\Re\,\Lambda(\lambda) 
\,&\leq\,
-1 + (\11+C_2\abs{\sigma}\abs{\lambda}^{1/3})\2\abs{\lambda}
\\
&\leq\,
1 - (\1\varrho-2^{\14/3}C_2\1)\abs{\sigma}
\,,\qquad\lambda \in I_2
\,.
}
For the second inequality we have used \eqref{1st mu bound for label selection}, while for the last inequality we have estimated $ \lambda \leq \lambda_2 = 2 - \varrho\abs{\sigma} $. Taking $ \varrho $ sufficiently large yields $ \Lambda(I_2) \subset \wht{\C}_0 $.

In order to show $ \Lambda(I_3) \subset \wht{\C}_+ $ we split $ I_3 = [\lambda_3,\lambda_1] $ into two parts, $ [\lambda_3,4] $ and $ (4,\lambda_1] $ (note that $[\lambda_3,4] \subset I_3 $ by \eqref{good endpoints}).
In the first part we estimate similarly as in \eqref{Re Lambda(I_2)} to get
\bels{Re Lambda(I_3 small)}{
\Re\,\Lambda(\lambda) 
\,&\ge\, 
-1 + (1-C_2\abs{\sigma}\lambda^{1/3})\2\lambda
\\
&\ge\,
1 + (\1\varrho-4^{\14/3}C_2\1)\abs{\sigma}
\,,\qquad
\lambda_3 \leq \lambda \leq 4
\,.
}
Taking $ \varrho \sim 1 $ large enough the right most expression is larger than $ 1  $. If $ \lambda_1 > 4 $, we use the rough bound \eqref{abs-mu leq 1/2} similarly as in \eqref{Re Lambda(I_1)} to obtain
\[
\Re\,\Lambda(\lambda) 
\,=\,
-1 - (\11+ \Re\,\mu(\lambda))\2\lambda
\,\ge\,-1 +\frac{\lambda}{2} 
\,>\, 1
\,,\qquad
4 < \lambda \leq \lambda_1
\,.
\]
Together with \eqref{Re Lambda(I_3 small)} this shows that $ \Lambda(I_3) \subset \wht{\C}_+ $.

Next, we will determine the three values $ a_k $ using the four selection principles of Lemma~\ref{lmm:Selection principles}.

\medskip
\noindent{\scshape Choice of $ a_1$}:
The initial condition, i.e.,  {\bf SP-2}, must be satisfied,
\[ 
\wht{\Omega}_{a_1}(-1-0) \,=\,\wht{\Omega}_{a_1}(\Lambda(0-0))\,=\, \Omega(0) \,=\, -1\,. 
\]
This excludes the choice $ a_1 = 0 $ since $ \wht{\Omega}_0(-1-0) = 2 $.
The choice $ a_1 = - $ is excluded using $ 1/2$-H\"older continuity \eqref{Perturbation of roots:weak bound} of the roots \eqref{def of root functions for small gap} inside the domain $ \wht{\C}_- $, and \eqref{2nd mu bound for label selection}:
\bels{Im wht-Omega_- upper bound}{
\Im\,\wht{\Omega}_-(\Lambda(-\xi)) \,&=\,
\Im\Bigl[
\2\wht{\Omega}_-(-1-\xi) + \Ord\bigl(\1\abs{\1\mu(-\xi)\2\xi\1}^{1/2}\1\bigr)
\Bigr]
\\
&\leq\,-\1c\,\xi^{1/2} 
\,,
\qquad\qquad 0\leq \xi \leq 1
\,.
}
For the last bound we have used \eqref{1st mu bound for label selection} and the bound
\bels{PM Im wht-OmegaPM SQRT-growth}{
\pm\2\Im\,\wht{\Omega}_\pm(\21+\xi\1) 
\,=\, 
\pm\2\Im\,\wht{\Omega}_\pm(-1-\xi\1)
\,\ge\, 
c_3\2\xi^{\11/2}
\,,
\qquad 0\leq \xi \leq 1
\,,
}
which follows from the explicit formulas \eqref{def of root functions for small gap}.
Since \eqref{Im wht-Omega_- upper bound} violates {\bf SP-3} we are left with only one choice: $ a_1 = + $. 

\medskip
\noindent{\scshape Choice of $ a_2$}:
Since $  \wht{\Omega}_-(-1+0) = 2 $, while $ \Omega(0) = -1 $, we exclude the choice $  a_2 = - $ using {\bf SP-2}. 
Moreover, from the explicit formulas of the roots \eqref{def of root functions for small gap} it is easy to see that $ \Im\,\wht{\Omega}_a|_{(-1,1)} = 0 $ for each of the three roots $ a =\pm,0 $.
Similarly as in \eqref{Im wht-Omega_- upper bound} we estimate for small enough $ \lambda >0 $ the real and imaginary part of $\wht{\Omega}_0 \circ \Lambda$ by
\bels{small gap: ready to apply SP-4}{
\Re\,\wht{\Omega}_0(\Lambda(\lambda)) &\leq -1-c\2\lambda^{1/2} + 
C\abs{\sigma}^{1/2}\lambda^{2/3}
\\
\absb{\Im\,\wht{\Omega}_0(\Lambda(\lambda))} 
\,&=\, 
\absb{\,0 \,+\,  \Ord\bigl(\1\abs{\1\mu(\lambda)\1\lambda}^{1/2}\1\bigr)}  
\,\lesssim\, 
\abs{\sigma}^{1/2}\lambda^{2/3}
\,.
}
If $ a_2 = 0 $, then \eqref{small gap: ready to apply SP-4} would violate {\bf SP-4} for small $\lambda > 0 $. We are left with only one choice: $  a_2 = + $.

\medskip
\noindent{\scshape Choice of $ a_3$}:
Using the formulas \eqref{def of root functions for small gap} we get
\[ 
\setb{\2\wht{\Omega}_0(1\pm0)\2,\,\wht{\Omega}_+(1\pm0)\2,\,\wht{\Omega}_-(1\pm0)} \;=\, \sett{\11,-2}\,. 
\]
Thus, the $1/2$-H\"older regularity \eqref{Perturbation of roots:weak bound} of the roots (outside the branch cuts) implies
\bels{1/2-Holder for the variety}{
\dist\bigl(\2\wht{\Omega}_a(\zeta),\sett{1,-2}\bigr) \,\lesssim\, \abs{\1\zeta-1\2}^{1/2},
\qquad
\zeta \in \C\,,\;a =0,\pm
\,.
}
We will apply this estimate for
\[
\zeta = \Lambda(\lambda) = 1 + \Ord\bigl(\2\abs{\1\lambda-2}+\abs{\sigma}\2\bigr)
\,,\qquad
\lambda \in [\lambda_2,\lambda_3]\,.
\]
Using \eqref{1st mu bound for label selection} to estimate $ \mu(\lambda) $, and recalling that $ \abs{\lambda-2} \lesssim \abs{\sigma} $, for $ \lambda \in [\lambda_2,\lambda_3] $, \eqref{Omega in terms of labelling function and the roots} and  \eqref{1/2-Holder for the variety} yield
\bels{Omega between lambda_2 and lambda_3}{
\dist\bigl(\2\Omega(\lambda),\sett{\11,-2}\1\bigr) 
\,\leq\,
\max_a \dist\bigl(\,\wht{\Omega}_a(\Lambda(\lambda))\2,\sett{\11,-2}\1\bigr) 
\,\lesssim\;
\abs{\sigma}^{1/2}
\,,\qquad
\lambda \in [\lambda_2,\lambda_3] 
\,.
}
In particular, taking $ \sigma_\ast \sim 1 $ sufficiently small \eqref{Omega between lambda_2 and lambda_3} implies for every $\abs{\sigma} \leq \sigma_\ast $,
\[ 
\Omega([\lambda_2,\lambda_3]) \,\subset\, \mathbb{B}(1,1) \cup \mathbb{B}(-2,1) 
\,,
\]
where $ \mathbb{B}(\zeta,\rho) \subset \C $ is a complex ball of radius $\rho $ centered at $ \zeta $.      
Since $ a_2 = - $ and $\wht{\Omega}_-(1-0) = 1 $  we see that $ \Omega(\lambda_2-0) \in  \mathbb{B}(1,1) $. The continuity of $ \Omega $ (cf. {\bf SP-1}) thus implies
\[ 
\Omega([\lambda_2,\lambda_3]) \,\subset\,  \mathbb{B}(1,1) 
\,.
\] 
In particular, $ \abs{\2\Omega(\lambda_3)-1} \leq 1 $, while $ \abs{\2\wht{\Omega}_0(\Lambda(\lambda_3))-1\1} \ge 2 $, since 
$ \wht{\Omega}_0(1+0) =2 $ and $ \Lambda(\lambda_3) \in \wht{\C}_+ $ is close to $ 1 $.  
This shows that $ a_3 \neq 0 $. 

In order to choose $ a_3 $ among $ \pm $ we use \eqref{Perturbation of roots:weak bound} and the symmetry $ \Im\,\wht{\Omega}_- = -\1\Im\,\wht{\Omega}_+ $ to get
\bels{Im at the opposite edge}{
\pm\,\Im\,\wht{\Omega}_\pm(\Lambda(\lambda)) 
\,&\ge\, 
\Im\,\wht{\Omega}_+(-1+\lambda) - C\1\abs{\1\lambda\2\mu(\lambda)\1}^{1/2}
\,,\qquad\lambda \in I_3
\,.
}
Since $ \lambda_3 = 2+\varrho\1\abs{\sigma} \leq 4 $ combining \eqref{PM Im wht-OmegaPM SQRT-growth} and \eqref{1st mu bound for label selection} yields
\bels{pos Im-part at lambda_3}{
\pm\,\Im\,\wht{\Omega}_\pm(\Lambda(\lambda_3)) 
\,\ge\,
c\2(\1\lambda_3-2\1)^{1/2} \!- C\abs{\sigma}^{1/2}
\,=\,(c\2\varrho^{\11/2}\!-C\1)\abs{\sigma}^{1/2}
\,.
}
Taking $ \varrho \sim 1 $ sufficiently large, the last lower bound becomes positive. Thus, the choice: $ a_3 = - $ is excluded by {\bf SP-3}. We are left with only one choice $ a_3 = + $. The estimate \eqref{upper bound on gap} follows from \eqref{pos Im-part at lambda_3}.
\end{Proof}

For the rest of the analysis we always assume that the triple $ (\delta,\varrho,\sigma_\ast) $ is from Lemma~\ref{lmm:Choice of roots}. Next we determine the shape of the general edge when the associated gap in $ \supp v $ is small. 

\begin{lemma}[Edge shape]
\label{lmm:Edge shape}
Let $ \tau_0 \in \partial \supp v $ and suppose $\abs{\sigma(\tau_0)} \leq \sigma_\ast $, where $ \sigma_\ast \sim 1 $ is from Lemma~\ref{lmm:Choice of roots}.
Then $ \sigma = \sigma(\tau_0) \neq 0 $, and $ \supp v $ continues in the direction $ \sign \sigma $ such that
\bels{Omega at the close edge}{
\qquad \absb{\2\Omega(\lambda)\1-\,\wht{\Omega}_+(\11\msp{-2}+\msp{-2}\abs{\lambda}\1)}
\,\lesssim\,
\abs{\sigma} \min\setb{\abs{\lambda}\2,\abs{\lambda}^{2/3}}
\,,
\qquad \sign \lambda = \sign \sigma
\,.
}
In particular, 
\bels{Im-Omega at the close edge}{
\Im\,\Omega(\lambda) 
\,=\,  2\1 \sqrt{3}\,\Psi_{\!\mrm{edge}}\Bigl(\frac{\abs{\lambda}}{2}\Bigr) 
\,+\,
\Ord\Bigl(\abs{\sigma} \min\setb{\abs{\lambda}\2,\abs{\lambda}^{2/3}}\Bigr)
\,,
\qquad \sign \lambda = \sign \sigma
\,,
}
where the function $ \Psi_{\!\mrm{edge}} : [\10,\infty) \to [\10,\infty) $, defined in \eqref{def of Psi_edge}, satisfies
\bels{PsiEdge from wht-OmegaPlus}{
  2\1 \sqrt{3}\,\Psi_{\!\mrm{edge}}(\lambda) 
\,=\, \Im\,\wht{\Omega}_+\msp{-1}(\11+2\1\lambda\1) 
\,,
\qquad
\lambda \ge 0
\,.
}
\end{lemma}

We remark that from \eqref{def of Psi_edge} one obtains:
\bels{scaling of Psi_edge}{
\Psi_{\!\mrm{edge}}(\lambda) 
\,&\sim\, 
\min\setb{\lambda^{1/2}\!,\lambda^{1/3}} 
\,,
\qquad\lambda\ge 0\,.
}

\begin{Proof}[Proof of Lemma~\ref{lmm:Edge shape}]
The bound $ \sigma \neq 0 $ follows from Lemma~\ref{lmm:Vanishing quadratic term}. The statement concerning the direction of $ \supp v $ follows from Lemma~\ref{lmm:Simple edge}. 
Without loss of generality  we assume $ \sigma > 0 $. Let $ \delta\1,\sigma_\ast\sim 1  $ be from Lemma~\ref{lmm:Choice of roots}.  
The relation \eqref{Omega at the close edge} is trivial when $ \abs{\lambda} \gtrsim \delta/\abs{\sigma}^3 $ since $ \Omega(\lambda) $ and $ \wht{\Omega}_+(1+\lambda) $ are both $ \Ord(\lambda^{1/3}) $ by \eqref{a priori for Theta(omega)} and \eqref{roots for small gap}, respectively. Thus, we consider only the case $ \lambda \in I_1 = (\10\1,\lambda_1] $.
Using \eqref{Omega|_I_k's determined} and the stability estimate \eqref{Perturbation of roots: away from opposite critical point}, with $ \rho = 1 $, we get
\bels{accurate edge estimate}{
\Omega(\lambda) \,&=\; \wht{\Omega}_+(1+\lambda+\mu(\lambda)\1\lambda\1)
\\
&=\;\wht{\Omega}_+(1+\lambda) \,+\, \Ord\Bigl(\2\mu(\lambda) \min\setb{\lambda^{1/2}\!,\1 \lambda^{1/3}}\Bigr)
\,,
\qquad \lambda \in I_1 = (\20\1,\lambda_1\1]
\,.
}
From \eqref{mu bound for edge shape} we obtain 
\bels{accurate mu(lambda) estimate}{
\abs{\1\mu(\lambda)} 
\,&\lesssim\,
\abs{\1\sigma}\,\absb{\1\wht{\Omega}_+(\11+(1+\mu(\lambda))\1\lambda\1)-\2\wht{\Omega}_+(\11+0)} \,+\, \abs{\sigma}^3\2\lambda
\,.
}
The stability estimate \eqref{Perturbation of roots:weak bound} then yields
\bels{Omega-1 size estimate}{
&\absb{\,\wht{\Omega}_+(\11+(1+\mu(\lambda))\1\lambda\1)-\2\wht{\Omega}_+(\11+0)\2} 
\\
&\;\lesssim\,  
\min\setB{\absb{(1+\mu(\lambda))\1\lambda\1}^{1/2},\absb{(1+\mu(\lambda))\1\lambda\1}^{1/3}}
\\
&\;\lesssim\, \min\setb{\lambda^{1/2}\!,\1 \lambda^{1/3}}
\,,
}
where we have used  \eqref{abs-mu leq 1/2}   to 
obtain $ \abs{(1+\mu(\lambda))\1\lambda\1} \sim \lambda $. 
Plugging \eqref{Omega-1 size estimate}  into \eqref{accurate mu(lambda) estimate} and using the resulting bound in \eqref{accurate edge estimate} to estimate $ \mu(\lambda) $ yields \eqref{Omega at the close edge}.
The formula \eqref{Im-Omega at the close edge} follows by taking the imaginary part of \eqref{Omega at the close edge} and using \eqref{PsiEdge from wht-OmegaPlus}. 
In order to see that \eqref{PsiEdge from wht-OmegaPlus} is equivalent to our original definition \eqref{def of Psi_edge} of $ \Psi_{\!\mrm{edge}}(\lambda) $ we rewrite the right hand side of \eqref{PsiEdge from wht-OmegaPlus} using \eqref{def of root functions for small gap} and \eqref{def of Phi_pm}.
\end{Proof}

We know now already from Lemma~\ref{lmm:Choice of roots} that $ \Im\,\Omega $ is small in $ I_2 $ since $ a_2 = - $ and $ \Im\,\wht{\Omega}_-(-1+\lambda) = 0 $, $ \lambda \in I_2 $. The next result shows that actually $\Im\,\Omega|_{I_2} = 0 $ which bounds the size of the gap $\Delta(\tau_0) $ from below.

\begin{lemma}[Size of small gap]
\label{lmm:Size of small gap}
Suppose $ \tau_0 \in \partial \supp v $. Then the gap length $ \Delta(\tau_0) $ (cf. \eqref{def of Delta(tau0)}) is approximated by $ \wht{\Delta}(\tau_0) $ for small $ \abs{\sigma(\tau_0)} $, such that
\bels{ell/Delta ratio}{
\frac{\Delta(\tau_0)}{\wht{\Delta}(\tau_0)}
\;=\; 
1\,+\,\Ord\bigl(\sigma(\tau_0)\bigr)
\,.
}
In general $ \Delta(\tau_0) \sim \abs{\sigma(\tau_0)}^3 \lesssim \wht{\Delta}(\tau_0) $. 
\end{lemma}

\begin{Proof}
Let $ (\delta,\varrho,\sigma_\ast)  $ be from Lemma~\ref{lmm:Choice of roots}. 
If $\sigma = \sigma(\tau_0) $ satisfies $ \abs{\sigma} \ge \sigma_\ast $, then $ \Delta = \Delta(\tau_0) \gtrsim \abs{\sigma}^3 $ by the second line of \eqref{simple edge: expansion}. On the other hand, $ \Delta \leq 2 \Sigma$ and $ \abs{\sigma} \lesssim 1 $ by definitions \eqref{def of Delta(tau0)} and \eqref{defs of sigma and psi}, respectively. Thus, we find $ \Delta \sim \abs{\sigma}^3 $. Since $ \psi = \psi(\tau_0) \lesssim 1 $ we see from \eqref{def of wht-Delta} that $ \wht{\Delta} = \wht{\Delta}(\tau_0)\gtrsim \abs{\sigma}^3 $. Thus, the lemma  holds for $\abs{\sigma} \ge \sigma_\ast $.
Therefore from now on we will assume  $ 0< \abs{\sigma} \leq \sigma_\ast $ ($ \sigma \neq 0 $ by Lemma~\ref{lmm:Edge shape}). Moreover, it suffices to consider only the case $ \sigma < 0 $ without loss of generality. 

Let us define the gap length $ \lambda_0 = \lambda_0(\tau_0) $ in the normal coordinates as
\bels{def of lambda_0}{
\lambda_0 \,:=\, \inf\setb{\lambda>0:\Im\,\Omega(\lambda) > 0 }
\,.
} 
Comparing this with \eqref{def of Delta(tau0)} shows
\bels{lambda_0 in terms of ell}{
\lambda_0 \,=\, 2\2\frac{\Delta}{\wht{\Delta}} 
\,.
}
From \eqref{upper bound on gap} we already see that $ \lambda_0 \leq \lambda_3 $, which is equivalent to
\bels{upper bound for the gap}{
\Delta \,\leq\, (\11 \2+\2 \frac{\varrho}{2}\2\abs{\sigma})\1\wht{\Delta}
\,.
} 
Since $ \varrho \sim 1 $ the estimate \eqref{ell/Delta ratio} hence follows if we prove the lower bound,
\bels{ell bounded from below}{
\Delta \,\ge\, (\11-C\abs{\sigma})\2\wht{\Delta} 
\,.
}
Using the representation \eqref{Omega in terms of labelling function and the roots} and the perturbation bound \eqref{Perturbation of roots:weak bound} we get
\bels{ell lower bound}{
\Im\,\Omega(\lambda) 
\;&=\; 
\Im\,\wht{\Omega}_-(-1+\lambda) + \Ord\bigl(\2\abs{\1\lambda\1\mu(\lambda)}^{1/2}\bigr)
\\
&\leq\;
0 \,+\,  C_1\1\abs{\sigma}^{1/2}
\,,\qquad
\forall\,\lambda \in I_2
\,.
}
We will show that $ \lambda \mapsto \Im\,\Omega(\lambda) $, grows at least like a square root function on the domain $ \sett{\lambda: \Im\,\Omega(\lambda) \leq c\1\eps} $. 
More precisely,
we will show that if $ \lambda_0 \leq 2  $, then
\bels{Im Omega growth}{
\Im\,\Omega(\lambda_0+\xi\1) \,\gtrsim\; \xi^{\11/2}
\,,
\qquad
0\leq \xi \leq 1
\,.
} 
Assuming that \eqref{Im Omega growth} is known, the estimate \eqref{ell bounded from below} follows from \eqref{ell lower bound} and \eqref{Im Omega growth}. 
Indeed, if $ \lambda_0 \ge \lambda_2 = 2 -\varrho\abs{\sigma} $ then  \eqref{ell bounded from below}  is immediate as $ \varrho \sim 1 $. On the other hand, if $ \lambda_0 < \lambda_2 $, then 
\[
c_0\1(\lambda_2-\lambda_0)^{1/2} 
\,\leq\, 
\Im\,\Omega(\lambda_2) 
\,\leq\, 
C_1\abs{\sigma}^{1/2}
 \,.
\]
 by \eqref{Im Omega growth} and \eqref{ell lower bound}. 
Solving this for $ \lambda_0 $ yields 
\[
\lambda_0 \,\ge\, \lambda_2-(C_1/c_0)^2\abs{\sigma} \,\ge\,2-C\2\abs{\sigma}
\,,
\]
where $ \lambda_2 =2 -\varrho\abs{\sigma} $ with $ \varrho \sim 1 $ (cf. \eqref{def of lambda_k's}) has been used to get the last estimate.
Using \eqref{lambda_0 in terms of ell} we see that this equals \eqref{ell bounded from below}. Together with \eqref{upper bound for the gap} this proves \eqref{ell/Delta ratio}.

In order to prove the growth estimate \eqref{Im Omega growth}, we express it in the original coordinates $ (\omega,v(\tau_0+\omega)) $ using \eqref{generic edge: def of lambda and Omega}, \eqref{def of Theta(omega) at tau}, $ v(\tau_0+\Delta\1) = 0 $, and $ f,\abs{m} \sim 1 $ (note that $ b = f $ since $ v(\tau_0) = 0 $):
\bels{tau_0 basepoint}{
v(\tau_0+\Delta+\wti{\omega}\1) 
\,\gtrsim\,
\min\setB{\bigl(\21+\wht{\Delta}\msp{-1}(\tau_0)^{-1/6}\bigr)\2\wti{\omega}^{1/2},\2\wti{\omega}^{1/3}}
\,,\qquad
0\leq \wti{\omega} \leq \delta\,.
} 
Applying Lemma~\ref{lmm:Edge shape} with $  \tau_0+\Delta $ as the base point yields
\bels{tau_1 basepoint}{
v(\tau_0+\Delta+\wti{\omega}\1) 
\,\sim\, \min\setB{\bigl(\21+\wht{\Delta}\msp{-1}(\tau_0\msp{-2}+\Delta\2)^{-1/6}\bigr)\2\wti{\omega}^{1/2},\2\wti{\omega}^{1/3}}
\,,\quad
0\leq \wti{\omega} \leq \delta
\,.
} 
The relation \eqref{tau_1 basepoint} implies \eqref{tau_0 basepoint}, provided we show 
\bels{comparability of Delta's}{
\wht{\Delta}(\tau_0+\Delta) \,\lesssim\, \wht{\Delta}(\tau_0)
\,,
\quad\text{for}\quad \Delta \lesssim \wht{\Delta}(\tau_0)
\,.
}

From the definition \eqref{def of wht-Delta} we get
\bels{Delta at tau_0+ell}{
\wht{\Delta}(\tau_0+\Delta\1) 
\,&\sim\,  \frac{\abs{\1\sigma(\tau_0+\Delta\1)}^3}{\psi(\tau_0+\1\Delta\1)^2}\,.
}
Using the upper bound \eqref{upper bound for the gap} and \eqref{Delta upper bound in sigma} we see that 
\[
\Delta \,\lesssim\, \wht{\Delta}(\tau_0) \,\sim\, \abs{\sigma(\tau_0)}^3
\,, 
\] 
for sufficiently small $ \sigma_\ast \sim 1 $.
Since $ \sigma(\tau) $ is $1/3$-H\"older continuous in $ \tau $, we get
\bels{sigma at tau_0+ell}{
\abs{\1\sigma(\tau_0+\Delta\1)} 
\,\leq\, \abs{\1\sigma(\tau_0)} \2+\2 C\1\Delta^{\!1/3} 
\lesssim\, \abs{\1\sigma(\tau_0)}
\,.
}
From the stability of the cubic \eqref{stability of shape cubic} it follows that for small  enough $ \sigma_\ast \sim 1\, $ we have
\[
\psi(\tau_0\!+\Delta\1) \,\sim\, \psi(\tau_0)\,\sim\, 1
\,.
\]
Plugging this together with \eqref{sigma at tau_0+ell} into \eqref{Delta at tau_0+ell} yields \eqref{comparability of Delta's}.
\end{Proof}

We have now covered all the parameter regimes of $ \sigma $ and $ \psi $ satisfying \eqref{stability of shape cubic}. Combining the preceding lemmas yields the expansion around general base points $ \tau_0 $ where $ v(\tau_0) = 0 $. 
We will need the following representation of the edge shape function \eqref{def of Psi_edge} below:
\bels{Psi_Edge Puiseux-series}{
\Psi_{\!\mrm{edge}}(\lambda) \,=\, \frac{\2\lambda^{1/2}\msp{-10}}{ 3 }\,(\21+\wti{\Psi}(\lambda)) 
\,,
\qquad
\lambda \ge 0
\,,
}
where the smooth function $ \wti{\Psi} : [\10,\infty)\to \R $ has uniformly bounded derivatives, and $ \wti{\Psi}(0) = 0 $.

\begin{Proof}[Proof of Proposition~\ref{prp:Vanishing local minimum}]
Let $ \tau_0 \in \supp v $ satisfy $ v(\tau_0) = 0 $.
If $ \sigma(\tau_0) = 0 $, then the expansion \eqref{cusp expansion - 1st time} follows directly from Lemma~\ref{lmm:Vanishing quadratic term}. 

In the case $ 0 < \abs{\sigma(\tau_0)} \leq \sigma_\ast $  \eqref{Im-Omega at the close edge} in Lemma~\ref{lmm:Edge shape} yields \eqref{edge expansion} with $ \wht{\Delta} = \wht{\Delta}(\tau_0) $ in place of $ \Delta = \Delta(\tau_0) $.
Here, the threshold $ \sigma_\ast \sim 1 $ is fixed by Lemma~\ref{lmm:Choice of roots}.
We will show that replacing $ \wht{\Delta}  $ with $ \Delta $ in \eqref{edge expansion} yields an error that is so small that it can be absorbed into the sub-leading order correction of  \eqref{edge expansion}. 
Since the smooth auxiliary function $ \wti{\Psi} $ in the representation \eqref{Psi_Edge Puiseux-series} of $ \Psi_{\mrm{edge}} $ has uniformly bounded derivatives, we get for every $ 0 \leq \lambda \lesssim 1 $,
\bels{Psi_edge: replace Delta by ell}{
\Psi_{\!\mrm{edge}}(\1(1+\epsilon\1)\1\lambda\1) \,=\, (\11+\epsilon)^{1/2} 
\Psi_{\!\mrm{edge}}(\lambda)
\,+\,
\Ord\bigl(\,\epsilon\,\min\setb{\lambda^{3/2}\!,\2\lambda^{1/3}}\bigr)
\,,
\quad\lambda \ge 0\,,
} 
provided the size $ \abs{\epsilon} \lesssim 1 $ of $ \epsilon \in \R $ is sufficiently small.
On the other hand, if $ \abs{\lambda} \gtrsim 1 $ then \eqref{Psi_edge: replace Delta by ell} follows from  \eqref{Perturbation of roots: away from opposite critical point} of Lemma~\ref{lmm:Stability of roots - neg}.
Now by  Lemma~\ref{lmm:Size of small gap} we have $ \wht{\Delta} = (1+\abs{\sigma}\kappa\1)\2\Delta $, where $ \Delta = \Delta(\tau_0) $ and the constant $\kappa \in \R $ is independent of $ \lambda $, and can be assumed to satisfy $ \abs{\kappa} \leq 1/2 $ (otherwise we reduce $ \sigma_\ast \sim 1 $). 
Thus applying \eqref{Psi_edge: replace Delta by ell} with $  \epsilon = \abs{\sigma}\1\kappa = \Ord(\2\Delta^{\!1/3})$,  yields
\[
\abs{\sigma}\,\Psi_{\!\mrm{edge}}\biggl(\frac{\omega}{\wht{\Delta}}\biggr) 
\,=\,
\frac{(\21+\abs{\sigma}\1\kappa\1)^{1/2}\abs{\sigma}}{\Delta^{\!1/3}\!}
\;\Delta^{\!1/3}\2\Psi_{\!\mrm{edge}}\biggl(\frac{\omega}{\Delta}\biggr)
\,+\,
\Ord\biggl(\,\min\setbb{\frac{\2\abs{\omega}^{3/2}\msp{-20}}{\Delta^{\!5/6}\msp{-20}}\msp{15},\,\abs{\omega}^{\11/3}\!}\biggr)
\,,
\]
for $ \omega \ge 0 $. 
Here, the error on the right hand side is of smaller size than the subleading order term in the expansion \eqref{edge expansion}. 

From \eqref{Theta provides leading order change of m} we identify the formula for $ h_x $, in the case  $ 0< \abs{\sigma} \leq \sigma_\ast $:
\bels{identification of h_x}{
h_x \,:=\, 
\begin{cases}
\frac{ 2(\11+\abs{\sigma}\1\kappa\2)^{1/2}\!}{\sqrt{3}\2\psi}\frac{\abs{\sigma}}{\,\Delta^{\!1/3}\msp{-10}}\msp{8}\abs{m_x}\1f_x 
\quad
&\text{when}\quad0 < \abs{\sigma} \leq \sigma_\ast\,;
\\
 \frac{3\2\Delta^{\!1/6}\!}{\sqrt{\abs{\sigma}}}\,  h'_x
&\text{when}\quad\abs{\sigma} >\sigma_\ast\,.
\end{cases}
} 
For $ \abs{\sigma} \leq \sigma_\ast $ we used \eqref{Im-Omega at the close edge}.
In the case $ \abs{\sigma} > \sigma_\ast $, the function $ h'_x $ is from \eqref{simple edge: expansion}, and the function $ h $ is defined such that  
\bels{}{
h'_x\,
\absB{\frac{\omega}{\,\sigma}}^{1/2} 
\!=\, 
h_x\,
\Delta^{\!1/3}\2\Psi_{\!\mrm{edge}}\biggl(\frac{\omega}{\Delta}\biggr)
\,+\,
\Ord\biggl( \frac{\abs{\omega}^{3/2}}{\,\Delta^{\!7/6}}\biggr)
\,.
}
Here, the second term originates from the representation \eqref{Psi_Edge Puiseux-series} of $ \Psi_{\!\mrm{edge}} $. 
This proves \eqref{edge expansion}.

Finally, suppose $ \tau_0 $ and $ \tau_1 $ are the opposite edges of $ \supp v $, separated by a small gap of length $ \Delta \lesssim \sigma_\ast^3 $, between them. 
Now, $ f(\tau) $, $ \abs{m(\tau)} $ and $ \psi(\tau) $ are $1/3$-H\"older continuous in $ \tau $, and satisfy $ f,\1\abs{m},\1\psi \sim 1 $. 
Thus, the terms constituting $ h_x $ in the case $ \abs{\sigma} \leq \sigma_\ast $ in \eqref{identification of h_x}  satisfy 
\bels{f, abs-m, psi compared at opposite edges}{ 
\frac{f_x(\tau_1)}{f_x(\tau_0)} \,=\, 1+\Ord(\2\Delta^{1/3})
\,,
\quad
\frac{\abs{m_x(\tau_1)}}{\abs{m_x(\tau_0)}} \,=\, 1+\Ord(\2\Delta^{1/3})
\,,
\quad
\frac{\psi(\tau_1)}{\psi(\tau_0)} \,=\, 1+\Ord(\2\Delta^{1/3})\,.
}
Of course, $ \Delta = \Delta(\tau_0) = \Delta(\tau_1) $.   Moreover, by Lemma~\ref{lmm:Size of small gap},
\bels{Delta comparison at opposite edges}{
\frac{\wht{\Delta}(\tau_1)}{\wht{\Delta}(\tau_0)} \,=\, 1\1+\Ord(\1\Delta^{1/3})\,.
}
Using \eqref{def of wht-Delta} we express $ \abs{\sigma} $ in terms of $ \wht{\Delta}, f,\abs{m},\psi $, and hence \eqref{f, abs-m, psi compared at opposite edges} and \eqref{Delta comparison at opposite edges} imply
\bels{abs-sigma compared at opposite edges}{
\frac{\abs{\sigma(\tau_1)}}{\abs{\sigma(\tau_0)}} \,=\, 1+\Ord(\2\Delta^{1/3})
\,.
}
Thus, combining \eqref{f, abs-m, psi compared at opposite edges}, \eqref{Delta comparison at opposite edges}, and \eqref{abs-sigma compared at opposite edges}, we see from \eqref{identification of h_x}  that $ h(\tau_1) = h(\tau_0) + \Ord_\BB(\2\Delta^{1/3})$. This proves the last remaining claim of the proposition.
\end{Proof}

\section{Proofs of Theorems~2.6 and 2.11}

Pick $ \eps > 0 $, and recall the definitions \eqref{def of DDe} and \eqref{def of MMe} of $ \DDe $ and $ \MMe $, respectively.
In the following we split $ \MMe $ into two parts:
\bels{}{
\MM^{(1)} &:= \partial \supp v
\\
\MMe^{(2)} &:= \MMe\backslash \partial \supp v
\,.
}

\begin{Proof}[Proof of Theorem~\ref{thr:Shape of generating density near its small values}]
Combining Proposition~\ref{prp:Non-zero local minimum} and Proposition~\ref{prp:Vanishing local minimum} shows  that there are constants $ \eps_\ast ,\delta_1,\delta_2 \sim 1$ such that the following hold:
\begin{itemize}
\item[1.] 
If $ \tau_0 \in \MM^{(1)} $, then $ \sigma(\tau_0) \neq 0  $ and 
$v_x(\tau_0 + \omega\1) \,\ge\, c_1\abs{\omega}^{\11/2} $, for $ 0 \leq \sign \sigma(\tau_0)\,\omega \leq \delta_1  $.
\item[2.] 
If $ \tau_0 \in \MM_{\eps_\ast}^{(2)} $, then $ v_x(\tau_0 + \omega\1) \,\ge\, c_2\,\bigl(\,v_x(\tau_0) + \abs{\omega}^{\11/3}\,\bigr) $, for $ -\1\delta_2 \leq \omega\leq \delta_2$. 
\end{itemize}
In the case 1 each connected component of $ \supp v $ must be at least of length $2\1\delta_1 \sim 1 $. This implies \eqref{defining property of of alpha_i and beta_i}. 
In particular, by combining \eqref{bound on supp v} and \eqref{operator norms of S are comparable} 
we see that $ \supp v $ is contained in an interval of length $ 2\1\Sigma $, and therefore the number of the connected components $ K' $ satisfies $ K' \sim 1 $. 

In order to prove \eqref{avg-v < eps contained in nn} and \eqref{expansion of v around tau_0} we may assume that $ \eps \leq \eps_\ast $ and $ \abs{\omega} \leq \delta $ for some $ \eps_\ast,\delta \sim 1 $. 
Indeed, \eqref{avg-v < eps contained in nn} becomes trivial when $ C\eps^3 \ge 2\1\Sigma  $. Similarly, if $ \avg{v(\tau_0)} + \abs{\omega} \gtrsim 1 $, then $ \avg{v(\tau_0)} + \Psi(\omega) \sim 1 $ and thus the $ \Ord(\,\cdots) $-term in \eqref{expansion of v around tau_0} is $ \Ord(1) $. Since $ v \leq \nnorm{m}_\R \sim 1 $, the expansion \eqref{expansion of v around tau_0} is hence trivial.

Obviously the bounds in the cases 1. and 2. continue to hold if we reduce the parameters $ \eps_\ast,\delta_1,\delta_2 $.
We choose $ \eps_\ast \sim 1 $ so small that $ (\eps_\ast/c_1)^2 \leq \delta_1 $ and $  (\eps_\ast/c_2)^3 \leq \delta_2 $. 
Let us define the expansion radius around $ \tau_0 \in \MMe $ for every $ \eps \leq \eps_\ast $
\bels{}{
\delta_\eps(\tau_0) := 
\begin{cases}
(\eps/c_1)^2\quad&\text{if }\tau_0 \in \MM^{(1)}
\\
(\eps/c_2)^3\quad&\text{if }\tau_0 \in \MMe^{(2)}\,,
\end{cases}
}
and the corresponding expansion domains
\bels{}{
I_\eps(\tau_0) := 
\begin{cases}
\setb{\tau_0+\sign \sigma(\tau_0)\2\xi: 0\leq \xi \leq \delta_\eps(\tau_0)}
\quad&\text{if }\tau_0 \in \MM^{(1)}
\\
\bigl[\1\tau_0-\delta_\eps(\tau_0)\1,\2\tau_0 +\delta_\eps(\tau_0)\bigr]
&\text{if }\tau_0 \in \MMe^{(2)}\,.
\end{cases}
}
If $ \tau \in I_\eps(\tau_0) $ for some $ \tau_0 \in \MMe $ then either $ v_x(\tau) \ge c_1 \abs{\tau-\tau_0}^{1/2} $ or $ v_x(\tau) \ge c_2 \abs{\tau-\tau_0}^{1/3} $ depending on whether $\tau_0 $ is an edge or not. 
In particular, it follows that  
\bels{avg-v above eps on I-edge}{
\avg{v(\tau)} \,\ge\,\eps\,,
\qquad
\forall\,\tau \in  \partial I_\eps(\tau_0) \backslash \partial\supp v
\,.
}
This implies that each connected component of $ \DDe $ is contained in the expansion domain $ I_\eps(\tau_0) $ of some $ \tau_0 \in \MMe $, i.e., 
\bels{DD included into MM-environments}{
\DDe
\;\subset 
\bigcup_{\tau_0 \ins \MMe} I_\eps(\tau_0)
\,.
}
In order to see this formally let $ \tau \in \DDe\backslash \MMe $ be arbitrary, and define $ \tau_0 \in \MMe $ as the nearest point of $ \MMe $ from $ \tau $, in the direction,
\[
\theta := - \sign \partial_\tau \avg{v(\tau)}\,,
\] 
where $ \avg{v} $ decreases. In other words, we set
\bels{def of nearest point tau0}{
\tau_0 &:= \tau + \theta\2\xi_0\,,
\qquad\text{where}\qquad
\xi_0 := \inf \setb{\xi > 0 : \tau+\theta\2\xi \in \MMe}
\,.
}
From \eqref{def of nearest point tau0} it follows that if $ \tau_0 \in \partial \supp v $, then $ \supp v $ continues in the direction $ \sign (\tau-\tau_0) = -\theta $ from $ \tau_0 $.
We show that $ \abs{\tau-\tau_0} \leq \delta_\eps(\tau_0) $.
To this end, suppose $ \abs{\tau-\tau_0} > \delta_\eps(\tau_0)$, and define
\bels{tau_1 from tau and tau_0}{ 
\tau_1 \,:=\, \tau_0 \,+\, \sign (\tau-\tau_0)\,\delta_\eps(\tau_0)
\,,
}
as the point between $ \tau $ and $ \tau_0 $ exactly at the distance $ \delta_\eps(\tau_0) $ away from $ \tau_0 $. Now, $\tau_1 \notin \partial \supp v $ as otherwise $ \tau_0 $ would not be the nearest point of $ \MMe $ (cf. \eqref{def of nearest point tau0}).
On the other hand, by definition we have $ \tau_1 \in \partial I(\tau_0) $. Thus, the estimate \eqref{avg-v above eps on I-edge} with $ \tau_1 $ in place of $ \tau_0 $ yields
\[
\avg{v(\tau_1)} \,\ge\, \eps \,\ge\, \avg{v(\tau)}
\,.
\]
Since $ \avg{v} $ is continuously differentiable on the set where $ \avg{v} > 0 $ and $ (\tau_1-\tau)\,\partial_\tau \avg{v(\tau)} $ $ < 0 $ by \eqref{def of nearest point tau0} and \eqref{tau_1 from tau and tau_0}, we conclude that $ \avg{v} $  has a local minimum at some point $ \tau_2 \in \MMe $ lying between $ \tau $ and $ \tau_1 $. But this contradicts \eqref{def of nearest point tau0}. 
As $ \tau \in \DDe\backslash \MMe$  was arbitrary  \eqref{DD included into MM-environments} follows.

From Corollary~\ref{crl:Location of non-zero minima} we know that for every  $ \tau_1,\tau_2 \in \MMe^{(2)} $, either
\bels{dist between minima - 2}{
\abs{\tau_1-\tau_2} \,\ge\, c_3
\qquad\text{or}\qquad
\abs{\tau_1-\tau_2} \,\leq\, C_3\1\eps^4
\,,
}
holds.  
Let $ \sett{\gamma_k} $ be a maximal subset of $ \MMe^{(2)} $ such that its elements are separated at least by a distance $ c_3 $. Then the set $ \MM :=  \partial \supp v \cup \sett{\gamma_k} $ has the properties stated in the theorem. In particular, 
\[
\DDe \;\subset
\bigcup_{\tau_0\ins\partial \supp v} \msp{-15}I_\eps(\tau_0)
\;\cup\,
\bigcup_k\,
\bigl[\1\gamma_k-C\eps^3,\2\gamma_k+C\eps^3\1\bigr]
\,,
\]
since $ \MMe^{(2)} +[-C\eps^3,\2C\eps^3\1] \,\subset\, \cup_k [\2\gamma_k-2\1C\eps^3,\2\gamma_k+2\1C\eps^3\1] $ for sufficiently small $ \eps \sim1 $.
This completes the proof of Theorem~\ref{thr:Shape of generating density near its small values}.  
\end{Proof}

Next we show that the support of a bounded generating density is a single interval provided the rows of $ S $ can not be split into two well separated subsets. We measure this separation using the following quantity
\bels{def of xi_S(kappa)}{
\quad \xi_S(\kappa) \,:=\, 
\sup \Biggl\{\,
\inf_{\substack{x\1\in\1 A \\ y \1\notin\1 A}}\Bigl(\, \abs{\1a_x-a_y}+ \norm{S_x-S_y}_1\Bigr) \,:\,\kappa \leq \Px(A)\leq 1-\kappa
,\, A \subset \Sx\Biggr\}
\,
}
 for $\kappa \ge 0 $.

\begin{lemma}[Generating density supported on single interval]
\label{lmm:Generating density supported on single interval}
Assume $ S $ satisfies {\bf A1-3} and $ \nnorm{m}_\R \leq \Phi $ for some  $ \Phi < \infty $.
Considering $ \Phi $ as an additional model parameter, there exist $ \xi_\ast,\kappa_\ast \sim 1 $, such that under the assumption,
\bels{xi_S(kappa_ast) leq xi_ast}{ 
\xi_S(\kappa_\ast) \,\leq\, \xi_\ast
\,,
} 
the conclusions of Theorem~\ref{thr:Generating density supported on single interval} hold.
\end{lemma}

In Chapter \ref{chp:Examples} we present very simple examples of $ S $ which do not satisfy \eqref{xi_S(kappa_ast) leq xi_ast} and the associated generating density  $ v $ is shown to have a non-connected support. 

\begin{Proof}[Proof of Theorem~\ref{thr:Generating density supported on single interval}]
Let $ \xi_\ast,\kappa_\ast \sim 1 $ be from Lemma~\ref{lmm:Generating density supported on single interval}. Note that \eqref{xi_S(0) leq xi_ast} is equivalent to $ \xi_S(0) \leq \xi_\ast $, and $ \xi_S(\kappa') \leq \xi_S(\kappa) $, whenever $ \kappa' > \kappa $. Thus \eqref{xi_S(0) leq xi_ast} implies $ \xi_S(\kappa_\ast) \leq \xi_S(0) \leq \xi_\ast $, and hence the theorem follows from the lemma.
\end{Proof}

\begin{Proof}[Proof of Lemma~\ref{lmm:Generating density supported on single interval}] 
Since $ \nnorm{m}_\R \leq \Phi $  Theorem~\ref{thr:Shape of generating density near its small values} yields the expansion  \eqref{single interval: edges1} and  \eqref{single interval: edges2}  around the extreme edges  $ \alpha := \inf \supp\,v $ and $ \beta := \sup \supp\, v $, respectively.
In particular, there exists $ \delta_1 \sim 1 $ such that
\bels{v_x ext edge expansion}{
v_x(\alpha+\omega)\,\ge\,c_1\abs{\omega}^{1/2}
\quad\text{and}\quad v_x(\beta-\omega)  \,\ge\,c_1\abs{\omega}^{1/2}
\,,  
\quad\text{for}\quad
\omega \in[\10\1,\delta_1]
\,.
}
Let us write 
\[
m_x(\tau) \,=\, p_x(\tau)\1 u_x(\tau) + \cI\2v_x(\tau)
\]
where $ p_x = \sign \Re\,m_x \in \sett{-1,+1}$ and $ u_x := \abs{\1\Re\,m_x}\1, v_x = \Im\,m_x \ge 0  $.
By combining the uniform bound $ \nnorm{m}_\R \leq \Phi $ with \eqref{unif bound of m} we see that $ \abs{m_x} \sim 1 $. 
In particular, there exists $ \eps_\ast \sim 1 $ such that 
\bels{u or v is above eps}{
\max\sett{u_x,v_x} \,\ge\, 2\1\eps_\ast
\,.
} 
Since $ m_x(\tau)$ is continuous in $ \tau $, the constraint \eqref{u or v is above eps} means that  $ \Re\,m_x(\tau) $ can not be zero on the domain
\[
\mathbb{K} \,:=\,
\setB{\tau\in [-\Sigma,\2\Sigma\2]:\sup_x v_x(\tau) \leq \eps_\ast}
\,.
\]
If $ I $ is a connected component of $ \mathbb{K} $, then there is $ p_x^I \in \sett{-1,+1}$, $ x \in \Sx $, such that
\[
p(\tau) \,=\, p^I\,,\qquad\forall\2\tau \in I
\,.
\] 
Using \eqref{v_x ext edge expansion} we choose $ \eps_\ast \sim 1 $ to be so small that $ v_x(\alpha+\delta_1) $ and $ v_x(\alpha-\delta_1)$ are both larger than $ \eps_\ast $. It follows that $ \supp v $ is not contained in $ \KK $. Furthermore, we choose $ \eps_\ast $ so small that Lemma~\ref{lmm:Monotonicity} applies, i.e., $ v_x > 0 $ grows monotonically in $ \KK $ when $ \Pi \ge \Pi_\ast $.

We will prove the lemma by showing that if some connected component $ I $ of $ \KK $ satisfies,
\bels{def of cc I}{ 
I = [\tau_1,\tau_2] \subset  \mathbb{K}
\,, 
\qquad\text{where}\quad
\alpha+\delta_1 \leq \tau_1 < \tau_2 \leq \beta-\delta_1 
\,,
}
then the set
\bels{def of set A}{
A = A^I := \setb{x \in \Sx : p^I_x =+1}
}
satisfies
\begin{subequations}
\label{xi_s(kappa_ast) ge xi_ast}
\begin{align}
\label{Px(A) sim 1}
\Px(A) \,&\sim\, 1
\\
\label{L1-norm of S_x-S_y is sim 1} 
 \abs{a_x-a_y} +\norm{S_x-S_y}_1 \,&\sim\, 1\,,\quad 
x \in A\,,\; y \notin A
\,.
\end{align}
\end{subequations}
The estimates \eqref{xi_s(kappa_ast) ge xi_ast} imply $ \xi_S(\kappa_\ast) \ge \xi_\ast $, with $ \kappa_\ast = \Px(A) $ and $ \xi_\ast \sim 1 $. In other words, under the assumption \eqref{xi_S(kappa_ast) leq xi_ast} each connected component of $\mathbb{K} $ contains either $ \alpha $ or $ \beta $. Together with  \eqref{single interval: edges1} and \eqref{single interval: edges2}  this proves the remaining estimate \eqref{single interval: bulk} of the lemma, and the $ \supp v $ is a single interval.

In order to prove \eqref{Px(A) sim 1} we will show below that there is a point $ \tau_0 \in I $ such that 
\bels{sigma(tau_0) lesssim eps^2}{
\abs{\1\sigma(\tau_0)} \,\leq\,C_0\1\eps_\ast^2
\,,
}
where $ \sigma := \avg{\1p\1f^3} $ was defined in \eqref{defs of sigma and psi}. 
Let $ f_- := \inf_x f_x $ and $ f_+ := \sup_x f_x $. As $ m $ is uniformly bounded, Proposition~\ref{prp:Estimates when solution is bounded} shows that $ f_\pm \sim 1 $. Hence, \eqref{sigma(tau_0) lesssim eps^2} yields bounds on the size of  $ A $,
\bea{
\Px(A)\2f_+^3 - (1-\Px(A))\2f_-^3 \,&\ge\, \sigma(\tau_0) \,\ge\,-\2C_0\1\eps_\ast^2
\\
\Px(A)\2f_-^3 - (1-\Px(A))\2f_+^3 \,&\leq\, \sigma(\tau_0) \,\leq\,+\2C_0\1\eps_\ast^2
\,.
}
Solving for $ \Px(A) $, we obtain
\[
\frac{f_-^3-C_0\eps_\ast^2}{f_+^3+f_-^3} 
\,\leq\, 
\Px(A)
\,\leq\, 
\frac{f_+^3+C_0\eps_\ast^2}{f_+^3+f_-^3}
\,.
\]
By making $ \eps_\ast \sim 1 $ sufficiently small this yields \eqref{Px(A) sim 1}.

We now show that there exists $ \tau_0 \in I $ satisfying \eqref{sigma(tau_0) lesssim eps^2}. 
To this end we remark that at least one (actually exactly one) of the following three alternatives holds true:
\begin{enumerate}
\item[(a)] 
The interval $ I $ contains a non-zero local minimum $ \tau_0 $ of $ \avg{v}\, $.
\item[(b)] The interval $ I $ contains a left and right edge $ \tau_- \in  \partial \supp v $ and $ \tau_+\in \partial \supp v\, $.
\item[(c)] The average generating density $ \avg{v} $  has a cusp at $ \tau_0 \in I \cap (\supp v \backslash \partial \supp v) $ such that $ v(\tau_0) = \sigma(\tau_0) = 0\, $.
\end{enumerate} 
In the case (a), since $ m $ is smooth on the set where $ \avg{v} > 0 $, Lemma~\ref{lmm:Monotonicity} implies $ \Pi(\tau_0) < \Pi_\ast $, and thus \eqref{sigma(tau_0) lesssim eps^2} holds for $ C_0 \ge \Pi_\ast $.
In the case (b) we know that $ \pm\2\sigma(\tau_\pm) > 0 $ by Proposition~\ref{prp:Vanishing local minimum}. Since $ \sigma(\tau) $ is continuous (cf. Lemma~\ref{lmm:Cubic for shape analysis}), there hence exists $ \tau_0 \in (\tau_-,\tau_+) \subset I $ such that $ \sigma(\tau_0) = 0 $. Finally, in the case (c) we have $ \sigma(\tau_0) = 0 $ by Proposition~\ref{prp:Vanishing local minimum}.

Now we prove \eqref{L1-norm of S_x-S_y is sim 1}. Since $ v_x \leq u_x\leq \abs{m_x} \leq \Phi $ on $ I $, and  $m $ solves the QVE, we obtain for every $ x \in A $, $ y \notin A $ and $ \tau \in I $
\bels{S_x-S_y estimate for gap}{
\frac{1}{\Phi}
\,&\leq\, 
\frac{1}{u_x\!} + \frac{1}{u_y\!}
\,\leq\,
2\2\frac{u_x+u_y}{\abs{m_xm_y}}
\,\leq\,
2\frac{\abs{(u_x+u_y) +\cI\2(v_x-v_y)}}{\abs{m_x}\abs{m_y}} 
\,=\,
2\,\absB{\frac{1}{m_x\!}-\frac{1}{m_y\!}}
\\
&=\,
2\,\abs{\1 a_x-a_y +\avg{\1S_x-S_y,\1m\1}\1}
\,\leq\,
2\, (\abs{a_x-a_y}+\Phi\norm{S_x-S_y}_1)
\,.
} 
Here, the definition \eqref{def of set A} of $ A $ is used in the first  equality  while $ u_x \ge v_x $ was used in the second estimate. 
The bound \eqref{S_x-S_y estimate for gap}  implies  \eqref{L1-norm of S_x-S_y is sim 1}.

We have shown that $ \abs{\sigma} + \avg{v} \sim 1 $. By using this in Corollary~\ref{crl:Bound on derivative} we see that  $ v(\tau) $ is uniformly $ 1/2$-H\"older continuous everywhere.
\end{Proof}

\chapterl{Stability around small minima of generating density}

The next result will imply the  statement (ii) in Theorem~\ref{thr:Stability}. 
Since it plays a central role in the proof of local laws (cf. Chapter~\ref{chp:Local laws for large random matrices}) for random matrices in \cite{AEK2}, we state it here in the form that does not require any knowledge of the preceding expansions and the associated cubic analysis. 
In fact, together with our main results, Theorem~\ref{thr:Regularity of generating density} and Theorem~\ref{thr:Shape of generating density near its small values}, the next proposition is the only information we use in \cite{AEK2} concerning the stability of the QVE. 

\begin{proposition}[Cubic perturbation bound around critical points]
\label{prp:Cubic perturbation bound around critical points}
Assume $ S $ satisfies  {\bf A1-3}, $ \nnorm{m}_\R \leq \Phi $, for some $ \Phi < \infty $, 
and $ g,d \in \BB $ satisfy the perturbed QVE \eqref{perturbed QVE - 1st time} at some fixed $ z \in \eCp$.
There exists $ \eps_\ast \sim 1 $ such that if 
\bels{conditions for cubic perturbation analysis}{
\avg{\2\Im\,m(z)\1} \,\leq\, \eps_\ast
\,,\qquad\text{and }\qquad
\norm{\1g-m(z)} \,\leq\, \eps_\ast
\,,
}
then there is a function $ s:\eCp \to \BB $ depending only on $ S $ and $ a $, and satisfying
\bels{s-continuity}{ 
\norm{s(z_1)} \,\lesssim\, 1
\,,\qquad
\norm{s(z_1)-s(z_2)} \,\lesssim\, \abs{z_1-z_2}^{1/3}
\,,\qquad\forall\, z_1,z_2 \in \eCp 
\,,
}
such that the modulus of the complex variable %
\bels{def of Theta for d}{
\Theta \,=\, \avgb{s(z)\2,g-m(z)}
}
bounds the difference $ g-m(z) $, in the following senses:
\begin{subequations}
\label{critical perturbation-bounds}
\begin{align}
\label{critical perturbation-bound: sup}
\norm{g-m(z)} \;&\lesssim\; \abs{\Theta} \,+\, \norm{d}
\\
\label{critical perturbation-bound: w-average}
\abs{\avg{w,g-m(z)}} \;&\lesssim\; \norm{w}\2\abs{\Theta}\,+\,  \norm{w}\norm{d}^2\,+\,\abs{\avg{T(z)w,d\2}}
\,,\quad\forall\,w \in \BB\,.
\end{align}
\end{subequations}
Here the linear operator $ T(z) : \BB\to \BB $ depends only on $ S $ and $ a $, in addition to $ z $, and satisfies $ \norm{T(z)} \lesssim 1 $.
Moreover, $ \Theta $ satisfies a cubic inequality 
\bels{perturbations: cubic scaling relation}{
\absb{\2\abs{\Theta}^3 + \pi_2\1\abs{\Theta}^2 + \pi_1\1\abs{\Theta}\2} 
\;\lesssim\; 
\norm{d\1}^2 + \abs{\avg{\2t^{(1)}(z),d\2}}+ \abs{\avg{\2t^{(2)}(z),d\2}}
\,,
}
where $t^{(k)} : \overline{\Cp} \to \BB $, $ k=1,2 $, 
depend on $ S$, $ a $, and $ z $ only, and satisfy $ \norm{\1t^{(k)}(z)} \lesssim 1 $. 
The coefficients, $\pi_1$ and $\pi_2$, may depend on $S$, $z$, $ a $, as well as on $ g $. They satisfy the estimates,
\begin{subequations}
\label{scaling of pi's}
\begin{align}
\label{scaling of pi_1}
\abs{\1\pi_1} \;&\sim\; \avg{\2\Im\,m(z)}^2 \,+\, \abs{\sigma(z)}\2\avg{\2\Im\,m(z)} \,+ \frac{\Im \,z}{\avg{\2\Im\, m(z)}}  
\\
\label{scaling of pi_2}
\abs{\1\pi_2} \;&\sim\; \avg{\2\Im\,m(z)} \,+\, \abs{\sigma(z)}
\,,
\end{align}
\end{subequations}
where the $ 1/3$-H\"older continuous function $ \sigma:\eCp \to [\10,\infty) $ is determined by $ S $ and $ a $, and has the following properties: 
Let $ \MM = \sett{\alpha_i} \2\cup\2 \sett{\beta_j} \2\cup\2 \sett{\gamma_k} $  be the set \eqref{def of MM} of minima from Theorem~\ref{thr:Shape of generating density near its small values}, and suppose $ \tau_0 \in \mathbb{M} $ satisfies $\abs{z-\tau_0} = \dist(z,\mathbb{M}) $.
\begin{subequations}
\label{def of abs-sigma at MMe}
If $ \tau_0 \in \partial \supp v = \sett{\alpha_i} \cup \sett{\beta_j} $, then
\bels{abs-sigma: edge}{
\abs{\sigma(\alpha_i)} \,\sim\,\abs{\sigma(\beta_{i-1})}\,\sim\,(\alpha_i-\beta_{i-1})^{1/3}
}
with the convention $\beta_0 = \alpha_1 -1 $ and $ \alpha_{K'+1} = \beta_{K'}+1 $.
If $ \tau_0 \notin \partial\supp v = \sett{\gamma_k}$, then
\bels{abs-sigma: not edge}{
\abs{\sigma(\gamma_k)} \;\lesssim\;
\avg{\2\Im\,m(\gamma_k)}^{\12}
\,.
}
\end{subequations}
All the comparison relations depend only on the model parameters $\rho$, $L$, $ \norm{a}$, $\norm{S}_{\Lp{2}\to\BB}$ and $\Phi $.
\end{proposition}

We remark here that the coefficients $ \pi_k $ do depend on $ g $ in addition to $ S $ and $ a $, in contrast to the coefficients $ \mu_k $ in Proposition~\ref{prp:General cubic equation}. The important point is that the right hands sides of the comparison relations \eqref{scaling of pi_1} and \eqref{scaling of pi_2} are still independent of $ g $.   
This result is geared towards problems where $ d $ and $ g $ are \emph{random}. 
Such problems arise when the resolvent method, as described in Chapter \ref{chp:Local laws for large random matrices}, is used to 
study the local spectral statistics of \emph{Wigner-type} random matrices.
The continuity and size estimates \eqref{s-continuity}, \eqref{abs-sigma: edge} and \eqref{abs-sigma: not edge} will be used to extend high probability bounds for each individual $z$ to all $ z $ in a compact set of $ \Cp $ similarly as in the proof of Theorem~\ref{thr:Entrywise local law from AEK2}.
The various auxiliary quantities, such as $ s, \pi^{(k)} $, $ T $, etc., appearing in the proposition will be explicitly given in the proof, but their specific form is irrelevant for the applications, and hence we omitted them in the statement.

\begin{Proof}[Proof of Proposition~\ref{prp:Cubic perturbation bound around critical points}] 
Since $ z $ is fixed we write $ m = m(z) $, etc. 
By choosing $\eps_\ast \sim 1 $ small enough we ensure that both Lemma~\ref{lmm:Expansion of B in bad direction} and Proposition~\ref{prp:General cubic equation} are applicable. 
We choose $ s$  such that  $ \Theta $ becomes the component of $ u = (g-m)/\abs{m} $ in the direction $ b $ exactly as in Proposition~\ref{prp:General cubic equation}. Hence using the explicit formula \eqref{def of P} for the projector $ P $ we read off from $ \Theta\,b = Pu $, that
\bels{explicit form of s in the def of Theta}{
s := \frac{1}{\avg{\1b^{\12}}\!}\1\frac{\overline{b}}{\1\abs{m}}
\,.
} 
From Lemma~\ref{lmm:Expansion of B in bad direction} and Proposition~\ref{prp:Holder regularity in z and extension to real line} we see that this function has the properties \eqref{s-continuity}.

The first bound \eqref{critical perturbation-bound: sup} follows by using \eqref{r to leading order} and \eqref{bound on R and its adjoint on BB} in the definition \eqref{def of u} of $ u $. More precisely, we have
\[
\norm{g-m} 
\,\leq\, 
\norm{m}\norm{u} 
\,\leq\, 
\norm{m}\bigl(\2\abs{\Theta}\norm{b} + \norm{r}\bigr)
\,\lesssim\,
 \abs{\Theta} + \norm{d}
\,,
\] 
where  $ \norm{m} \sim 1 $, $ b = f + \Ord_\BB(\alpha) $, $ r = Rd + \Ord_\BB(\abs{\Theta}^2+\abs{d}^2) $, and $ \norm{R},\norm{f} \lesssim 1 $, have been used.

In order to derive \eqref{critical perturbation-bound: w-average} we first write
\bels{w-weighted g-m}{
\avg{\1w,g-m\1} \;=\; \avg{\1\abs{m}\1w,u\2} \;=\; \avg{\1\abs{m}\1w,b\1}\2\Theta + \avg{\1\abs{m}\1w,r\2}
\,.
}
Clearly, $ \abs{\avg{\1\abs{m}\1w,b\1}} \lesssim \norm{w} $. Moreover, using \eqref{r to leading order} we obtain
\bea{
\avg{\1\abs{m}\1w,r\2} \;&=\; \avgB{\2\abs{m}\1w,R\1d + \Ord_\BB(\2\abs{\Theta}^2\!+\norm{d}^2\1)\2} 
\\
&=\; 
\avg{R^\ast\msp{-2}(\abs{m}\1w),d\,} \,+\, \Ord\Bigl(\2\norm{m}\norm{w} \bigl(\2\abs{\Theta}^2+\norm{d}^2\bigr)\2\Bigr)
\,.
}
Plugging this into \eqref{w-weighted g-m}, and setting  $ T := R^\ast(\abs{m}\genarg) $, we recognize  \eqref{critical perturbation-bound: w-average}. The bound \eqref{bound on R and its adjoint on BB} yields $ \norm{T} \lesssim 1 $.

As a next step we show that \eqref{perturbations: cubic scaling relation} and \eqref{scaling of pi's} constitute just a simplified version of the cubic equation presented in Proposition~\ref{prp:General cubic equation}. 
Combining \eqref{general cubic} and \eqref{kappa bounds} we get
\bels{cubic bound with Theta^4-term absorbed into mu_3}{
\absb{\,\wti{\mu}_3\abs{\Theta}^3 \!+ \wti{\mu}_2\abs{\Theta}^2 \!+ \wti{\mu}_1\abs{\Theta}\,} \;\lesssim\; \abs{\avg{\abs{m}\1\overline{b},d\2}} +\norm{d}^2 + \abs{\avg{e,d\1}} 
\,,
}
where $ \wti{\mu}_1 := (\Theta/\abs{\Theta})\2\mu_1 $, $ \wti{\mu}_2 := (\Theta/\abs{\Theta})^2\mu_2 $ and $ \wti{\mu}_3 = (\Theta/\abs{\Theta})^3\mu_3 + \Ord(\abs{\Theta}) $. The last term in the definition of $ \wti{\mu}_3 $ accounts for the absorption of the $ \Ord(\abs{\Theta}^4) $-sized part of $ \kappa(u,d) $ in \eqref{general cubic}.  
Moreover, we have estimated the $ \Ord(\2\abs{\Theta}\2 \abs{\avg{e,d}}\2)$-sized part of $ \kappa $ by a larger $ \Ord(\abs{\avg{e,d}}) $ term.  
Recall that $ \abs{\Theta} \lesssim \eps_\ast $ from \eqref{conditions for cubic perturbation analysis}. Hence taking $ \eps_\ast \sim 1 $ small enough, the stability of the cubic (cf. \eqref{stability general cubic}) implies that there is $ c_0 \sim 1 $ so that  $ \abs{\wti{\mu}_2} + \abs{\wti{\mu}_3} = \abs{\mu_2} + \abs{\mu_3} + \Ord(\abs{\Theta}) \ge 2\1c_0 $ applies.
Hence the coefficients
\bels{}{
\pi_2 \,&:=\, \bigl(\,\wti{\mu}_2 +(\wti{\mu}_3-1)\1\abs{\Theta}\2\bigr)\,\Ind\setb{\abs{\mu_2} \ge c_0} \;+\; \frac{\wti{\mu}_2}{\wti{\mu}_3}\2 \Ind\setb{\abs{\mu_2} < c_0}
\\
\pi_1 \,&:=\, \wti{\mu}_1\Ind\setb{\abs{\mu_2} \ge c_0} \;+\; \frac{\wti{\mu}_1}{\wti{\mu}_3}\2 \Ind\setb{\abs{\mu_2} < c_0}
\,,
}
scale just like $ \mu_2 $ and $ \mu_1 $ in size, i.e., $ \abs{\pi_2}\sim \abs{\mu_2}$ and $ \abs{\pi_1}\sim\abs{\mu_1} $, provided $\eps_\ast$ and thus $|\Theta|$ is sufficiently small. Moreover, by construction the bound \eqref{cubic bound with Theta^4-term absorbed into mu_3} is equivalent to \eqref{perturbations: cubic scaling relation} once we set $ t^{(1)} :=  \abs{m}\1\bar{b} $ and $ t^{(2)} := e $. 

Let us first derive the scaling relation \eqref{scaling of pi_1} for $ \pi_1 $.
Using $ \sigma \in \R $, we obtain from \eqref{mu_1 expanded}:
\bels{abs-pi_1 scaling A}{
\abs{\1\pi_1} \;&\sim\;
\abs{\1\mu_1} \;=\; \absB{-\2\avg{f\1\abs{m}\1}\frac{\eta}{\alpha}\,+\,
\cI\12\1\sigma\2\alpha \,-\, 2\1(\1\psi\1-\1\sigma^2)\2\alpha^2 
\,+\,\Ord\bigl(\alpha^3 \!+ \eta\,\bigr)}
\\
&\sim\; 
\absB{\frac{\avg{f\1\abs{m}}}{2}\frac{\eta}{\alpha}+(\psi-\sigma^2)\2\alpha^2\,+\,\Ord\bigl(\alpha^3 \!+ \eta\,\bigr)} 
\,+\,
\absB{\2\sigma\1\alpha 
\,+\,
\Ord\bigl(\alpha^3 \!+ \eta\,\bigr)}
\,.
}
We will now use the stability of the cubic, $\psi + \sigma^2 \gtrsim 1$ (cf. \eqref{stability general cubic}). We treat two regimes separately. 

First let us assume that $2\1\sigma^2\leq \psi $. In that case $\psi \sim 1$, and we find
\bels{abs-pi_1 scaling B}{
\abs{\1\pi_1}
\,\sim\, 
\frac{\eta}{\alpha}+ \alpha^2 + \abs{\sigma}\1\alpha +\Ord\bigl(\alpha^3 \!+ \eta\,\bigr)
\,\sim\,
\frac{\eta}{\alpha}+ \alpha^2 + \abs{\sigma}\1\alpha
\,.
}
In order to get the first comparison relation we have used the fact that $ \psi - \sigma^2 \sim \psi \sim 1 $ and $ \avg{f\abs{m}} \sim 1 $ and hence the first two terms on the right hand side of the last line in \eqref{abs-pi_1 scaling A} can not cancel each other.
The second comparison in \eqref{abs-pi_1 scaling B} holds provided $ \eps_\ast \sim 1 $ is sufficiently small, recalling $ \alpha \sim \avg{v} \leq \eps_\ast $ (cf. \eqref{Phi:basic bounds}, so that the error can be absorbed into the term $\eta/\alpha+\alpha^2 $.

Now we treat the situation when $2\1\sigma^2> \psi $. In this case $\abs{\sigma}\sim 1$, and thus for small enough $ \eps_\ast $, we have
\bels{pi_1 scaling}{
\abs{\1\pi_1}
\;&\sim\; 
\absB{\1\frac{\eta}{\alpha}+\Ord(\alpha^2\!+\eta)\2}+\alpha
\;=\; 
\frac{\eta}{\alpha}+\alpha+\Ord(\alpha^2\!+\eta)
\;\sim\; 
\frac{\eta}{\alpha}+\alpha
\\
&\sim\; \frac{\eta}{\alpha}+|\sigma|\2\alpha+\alpha^2
\,.
}
Here, the first two terms in the last line of \eqref{abs-pi_1 scaling A} may cancel each other but in that case both of the terms are $ \Ord(\alpha^2) $ and hence the size of $ \abs{\1\pi_1} $ is given by the term $ \abs{\sigma}\alpha \sim \alpha $.

The scaling behavior \eqref{scaling of pi_2} of $ \pi_2 $ follows from \eqref{mu_2 expanded} using $ \norm{F}_{\Lp{2}\to\Lp{2}} = 1-\avg{f\1\abs{m}\1}\2\eta/\alpha \sim 1 $ (cf. \eqref{F and alpha} and \eqref{scaling of L2-norm of F in abs-z}) and the stability of the cubic, 
\bels{pi_2 scaling}{ 
\abs{\1\pi_2} 
\,\sim\, 
\abs{\1\mu_2} 
\,\sim\, 
\abs{\sigma} + \abs{\13\2\psi-\sigma^2}\2\alpha 
\,\sim\, \abs{\sigma} + \alpha
\,.
}

The formula \eqref{scaling of pi_1} now follows from \eqref{pi_1 scaling} and \eqref{pi_2 scaling} by using $ \alpha \sim \avg{\2\Im\,m\1} $.  The quantity $ \sigma = \sigma(z) $ was proven to be $ 1/3$-H\"older continuous already in Lemma~\ref{lmm:Expansion of B in bad direction}. 
In order to obtain the relation \eqref{abs-sigma: edge} we use \eqref{Delta upper bound in sigma} and Lemma~\ref{lmm:Size of small gap} to get
\[
\abs{\sigma(\tau_0)} 
\,\sim\, 
\wht{\Delta}(\tau_0)^{1/3} 
\,\sim\, 
\Delta(\tau_0)^{1/3}
,
\]
for $ \tau_0 \in \partial \supp v $ such that $ \abs{\sigma(\tau_0)} \leq \sigma_\ast $. On the other hand, if $ \abs{\sigma(\tau_0)} \ge \sigma_\ast $, where the threshold parameter $\sigma_\ast \sim 1 $ is from \eqref{Delta upper bound in sigma}, then also $ \Delta(\tau_0) \sim 1 $.
This proves \eqref{abs-sigma: edge}.

In order to obtain \eqref{abs-sigma: not edge} we consider the cases $ v(\gamma_k) = 0 $ and $ v(\gamma_k) > 0 $ separately. If $ v(\gamma_k) = 0 $ then Lemma~\ref{lmm:Vanishing quadratic term} shows that $ \sigma(\gamma_k) = 0 $. If $ v(\gamma_k) > 0 $ then $ \partial_\tau \avg{v(\gamma)}|_{\tau=\gamma_k} = 0 $. Lemma~\ref{lmm:Monotonicity} thus yields $ \abs{\sigma(\gamma_k)} \leq \Pi_\ast \avg{v(\gamma_k)}^2 $. Since $ \Pi_\ast \sim 1 $ this finishes the proof of \eqref{abs-sigma: not edge}.
\end{Proof}

Combining our two results concerning general perturbations, Lemma~\ref{lmm:Stability when m and inv-B bounded} and Proposition~\ref{prp:Cubic perturbation bound around critical points}, with scaling behavior of $ m(z) $ as described by Theorem~\ref{thr:Shape of generating density near its small values}, we now prove Theorem~\ref{thr:Stability}.

\begin{Proof}[Proof of Theorem~\ref{thr:Stability}] 
Recall the definition \eqref{def of B and a} of operator $ B $. 
We will show below that 
\bels{general inv-B bound}{
\norm{B(z)^{-1}} 
\,\lesssim\, \frac{1}{\2\varrho(z)^2+\varpi(z)^{\12/3}\msp{-3}}
\msp{8},\qquad
\abs{z} \leq 2\2\Sigma
\,,
}
where $ \varrho = \varrho(z) $ and $ \varpi=\varpi(z) $ are defined in \eqref{defs of varpi, rho, and delta}. 
Given \eqref{general inv-B bound} the assertion (i) of the theorem follows by applying Lemma~\ref{lmm:Stability when m and inv-B bounded} with $ \Phi $ introduced in the theorem and $ \Psi := (\varrho+\varpi^{\11/3})^{-2} \lesssim \eps^{-2} $, where the constant $ \eps \in (0,1)$ is from \eqref{condition for being away from critical points}.  
If $ \varrho \ge \eps_\ast $ or $ \varpi \ge \eps_\ast $ for some $ \eps_\ast \sim 1 $, then (ii) follows similarly from Lemma~\ref{lmm:Stability when m and inv-B bounded} with $ \Psi \sim 1 $.
Therefore, in order to prove (ii) it suffices to assume that $ \varrho,\varpi \leq \eps_\ast $ for some sufficiently small threshold $ \eps_\ast \sim 1 $. 

We will take $ \eps_\ast $ so small that Proposition~\ref{prp:Cubic perturbation bound around critical points} is applicable, and thus the cubic equation \eqref{perturbations: cubic scaling relation} can be written in the form 
\bels{cubic for delta}{
\absb{\,\abs{\Theta}^3 + \pi_2\abs{\Theta}^2 + \pi_1\abs{\Theta}\,} 
\;\lesssim\; 
\delta
\,,
}
with $ \delta = \delta(z,d) \leq \norm{d} $ given in \eqref{def of delta} of Theorem~\ref{thr:Stability}.
Combining the definition \eqref{def of Theta for d} of $ \Theta $ with the a priori bound \eqref{refined perturbation-bounds: a priori estimate for g-m} for the difference $ g-m $, we obtain
\bels{a priori bound for Theta and g-m}{
\abs{\Theta} \leq \norm{s} \norm{g-m} \,\lesssim\, \lambda\,(\1\varpi^{\22/3} + \rho\,)
\,.
}
For the last step we used also \eqref{s-continuity}.
We will now show that if \eqref{a priori bound for Theta and g-m} holds for sufficiently small $ \lambda \sim 1 $, then the linear term of the cubic \eqref{cubic for delta} dominates in the sense that
\bels{1-st order dominates the cubic}{
\abs{\1\pi_1} \,\ge\, 3\2\abs{\1\pi_2}\abs{\Theta}\,, 
\qquad\text{and}\qquad
\abs{\1\pi_1} \,\ge\, 3\2\abs{\Theta}^2 
\,.
}

Let us first establish \eqref{1-st order dominates the cubic} when $ \tau = \Re\,z \in \supp v $. 
From \eqref{a priori bound for Theta and g-m} and \eqref{scaling of pi's} we get
\begin{align}
\label{supp a priori for Theta}
\abs{\1\Theta\1} \;&\lesssim\, \lambda\,(\varrho  + \eta^{2/3})
\\
\label{supp bound for pi_1}
\abs{\1\pi_1} \2&\gtrsim\, (\1\abs{\sigma}+\alpha\1)\1\alpha
\\
\label{supp bound for pi_2}
\abs{\1\pi_2} \2&\sim\, \abs{\sigma} + \alpha
\,.
\end{align}
Here we have used the general property $ v_x \sim \avg{v} \sim \alpha $ that always holds when $ \nnorm{m}_\R \lesssim \Phi $. Since $ \tau \in \supp v $ we have $\varpi = \eta $ in \eqref{supp a priori for Theta}. Let us show that 
\bels{varrho +eta^2/3 lesssim alpha}{ 
\varrho \2+\2 \eta^{\12/3} \2\lesssim\, \alpha 
\,.
}
To this end, let $ \tau_0 = \tau_0(z) \in \MM_{\eps_\ast} $ be such that
\bels{ii: def of tau_0}{
\abs{\tau-\tau_0} \,=\, \dist(\tau,\MM_{\eps_\ast})
}
holds. 
If $ \tau_0 \notin \partial \supp v $, then (d) of Corollary~\ref{crl:Scaling relations} yields \eqref{varrho +eta^2/3 lesssim alpha} immediately (take $ \omega := \tau-\tau_0 $ in the corollary).
If on the other hand $ \tau_0 \in \partial \supp v $, then (a) of Corollary~\ref{crl:Scaling relations} yields
\[
\varrho \2+\2 \eta^{\12/3}
\,\lesssim\, 
\frac{\omega^{1/2}}{(\2\Delta+\omega)^{1/6}} + \eta^{\12/3}
\,\lesssim\, 
\frac{(\omega+\eta)^{1/2}}{(\2\Delta+\omega+\eta)^{1/6}}
\,\sim\,\alpha
\,,
\]
where $ \Delta = \Delta(\tau_0) $ is the gap length \eqref{def of Delta(tau0)} associated to the point $ \tau_0 \in \partial \supp v $ satisfying \eqref{ii: def of tau_0}.

Combining \eqref{varrho +eta^2/3 lesssim alpha} and \eqref{supp a priori for Theta} we get
$ \abs{\Theta} \lesssim \lambda\,\alpha $. 
Using this bound together with \eqref{supp bound for pi_1} and \eqref{supp bound for pi_2} we obtain  \eqref{1-st order dominates the cubic} for sufficiently small $ \lambda \sim 1 $.

Next we prove \eqref{1-st order dominates the cubic} when $ \tau \notin \supp v $, i.e., $ \varrho = 0 $. In this  case \eqref{a priori bound for Theta and g-m} and \eqref{scaling of pi's} yield
\begin{align}
\label{not supp a priori for Theta}
\abs{\Theta} \,&\lesssim\, \lambda\,\varpi^{2/3}
\\
\label{not supp bound for pi_1}
\abs{\1\pi_1} \,&\gtrsim\, \eta/\alpha
\\
\label{not supp bound for pi_2}
\abs{\1\pi_2} \,&\lesssim\, 1\,.
\end{align}
By combining the parts (b) and (c) of Corollary~\ref{crl:Scaling relations} we get
\bels{alpha scaling outside supp v}{
\alpha 
\,\sim\,
\frac{\eta}{(\Delta+\eta)^{1/6}\varpi^{1/2}} 
\,\lesssim\,
\eta\,\varpi^{-2/3}
\,,
}
where $ \Delta = \Delta(\tau_0) $ is the gap length \eqref{def of Delta(tau0)} associated to the point $ \tau_0 \in \partial \supp v $.
For the last bound in \eqref{alpha scaling outside supp v} we used $ \varpi \sim \omega + \eta \leq\, \Delta+\eta $.
Plugging \eqref{alpha scaling outside supp v} into \eqref{not supp bound for pi_1} we get 
\bels{not supp: scaling of pi_1}{
\abs{\1\pi_1} \,\gtrsim\, \varpi^{2/3}
\,.
}
Using this together with \eqref{not supp a priori for Theta} and \eqref{not supp bound for pi_2} we obtain \eqref{1-st order dominates the cubic} also when $ \tau \notin \supp v $.

The estimates \eqref{1-st order dominates the cubic} imply
\[
\abs{\Theta}^3 \,\lesssim\, \absb{\pi_1\Theta} \,\sim\, \absb{\,\abs{\Theta}^3 + \pi_2\abs{\Theta}^2 + \pi_1\abs{\Theta}\,}
\,.
\]
Using \eqref{cubic for delta} we hence get $ \abs{\Theta}^3 \lesssim \abs{\1\pi_1\Theta\1}  \lesssim \delta $,
from which it follows that
\bels{refined: Theta bound - A}{
\abs{\Theta} \,\lesssim\, \min\setbb{\msp{-2}\frac{\delta}{\abs{\1\pi_1}}\2,\1\delta^{\11/3}\msp{-2}}
\,.
}
If $ \tau \notin \supp v $ we have $ \varrho = 0 $ and thus \eqref{not supp: scaling of pi_1} can be written as 
\bels{pi_1 general lower bound}{
\abs{\1\pi_1} \,\gtrsim\, \varrho^{\12}+\varpi^{2/3}
\,.
}
This estimate holds also when $ \tau \in \supp v $. If the point $ \tau_0 = \tau_0(\tau) \in \MM_{\eps_\ast} $ satisfying \eqref{ii: def of tau_0} is not an edge of $ \supp v $, then \eqref{pi_1 general lower bound} follows immediately from (d) of Corollary~\ref{crl:Scaling relations} and from $ \abs{\pi_1} \gtrsim \alpha^2 $ from \eqref{supp bound for pi_1}. 
In order to get \eqref{pi_1 general lower bound} when  $ \tau \in \supp v $ and  $ \tau_0 \in \partial \supp v $ we set $\omega = \abs{\tau-\tau_0} $ and consider the cases $ \omega +\eta > c_0\2\Delta $ and $ \omega +\eta \leq c_0\2 \Delta $ for some small $ c_0 \sim 1  $ separately.
If $ \omega + \eta > c_0\2\Delta $, then we get
\bels{alpha^2 bounded from below}{
\alpha^2 
\,\sim\,
(\omega+\eta)^{2/3} 
\,\sim\, 
\omega^{\12/3} + \eta^{2/3} 
\,\sim\, 
\varrho^{\12} + \eta^{2/3}
\,,
}
using part (a) of Corollary~\ref{crl:Scaling relations} in both the first and the last estimate.
On the other hand, if $ \omega +\eta \leq c_0\2\Delta $ for sufficiently small $ c_0 \sim 1 $, then
\bels{abs-sigma like ell^1/3}{
\abs{\sigma} 
\,=\, 
\abs{\sigma(z)} 
\,\ge\, 
\abs{\sigma(\tau_0)} -C\abs{\tau_0-z}^{1/3} 
\,\gtrsim\, 
\Delta^{\!1/3} -C\1(\omega+\eta)^{1/3}
\,\ge \frac{1}{2}\2\Delta^{\!1/3}
\,,
}
where we have used $ 1/3$-H\"older continuity of $ \sigma $ and the relation \eqref{abs-sigma: edge} from Proposition~\ref{prp:Cubic perturbation bound around critical points}. For the last bound we have used $ \abs{\tau_0-z} \sim \omega +\eta $ as well.
Therefore, we have
\bels{abs-sigma alpha bounded from below}{
\abs{\sigma}\1\alpha 
\,\sim\, 
\Delta^{1/6}(\omega+\eta)^{1/2} 
\,\gtrsim\, 
\omega^{2/3} + \eta^{2/3} 
\,\gtrsim\, 
\varrho^2 +\eta^{2/3}
\,.
}
Here, we have used (a) of Corollary~\ref{crl:Scaling relations} twice.
Combining \eqref{alpha^2 bounded from below} and \eqref{abs-sigma alpha bounded from below} we get
\bels{abs-sigma alpha + alpha^2 bound from below on supp}{
\abs{\sigma}\1\alpha+\alpha^2
\,\gtrsim\, 
\varrho^{\12}+\varpi^{2/3}
\,,\qquad
\tau \in \supp v\,.
}
Using this in \eqref{supp bound for pi_1} yields \eqref{pi_1 general lower bound}  when $ \tau_0 \in \partial \supp v $.

By combining \eqref{refined: Theta bound - A} and \eqref{pi_1 general lower bound} we obtain 
\bels{Theta bounded by Upsilon}{ 
\abs{\Theta} \,\lesssim\, \Upsilon 
\,,
}
with $ \Upsilon = \Upsilon(z,d) $ defined in \eqref{refined stability: def of Upsilon}. 
The estimates \eqref{refined perturbation-bounds} now follow from \eqref{critical perturbation-bounds} using \eqref{Theta bounded by Upsilon}.

We still need to prove \eqref{general inv-B bound}. 
If $ \tau \in \supp v $, then \eqref{m BB-bounded: inv-B from L2 to BB} of Lemma~\ref{lmm:Bounds on B-inverse} shows
\[
\norm{B^{-1}} \,\lesssim\, \frac{1}{(\1\abs{\sigma}\1+\alpha)\2\alpha\2}
\,.
\]
Using \eqref{abs-sigma alpha + alpha^2 bound from below on supp} we get \eqref{general inv-B bound} when $ \tau \in \supp v $.
In the remaining case $ \tau\notin \supp v $ \eqref{general inv-B bound} reduces to
\bels{goal:inv-B norm outside supp v}{
\norm{B^{-1}} \,\lesssim\, \varpi^{\1-2/3}
\,.
}
In order to prove this we use \eqref{m BB-bounded: inv-B from L2 to BB} to get the first bound below:
\bels{inv-B norm outside supp v}{
\norm{B^{-1}} 
\,\leq\, 
1+\norm{B^{-1}}_{\Lp{2}\to\Lp{2}} 
\,\leq 
1\, + \frac{1}{1-\norm{F}_{\Lp{2}\to\Lp{2}}} 
\,\lesssim\, 
1 + \frac{\alpha}{\eta}
\,.
}
For the second estimate we have used the definition \eqref{def of B and a} of $ B $ and the identity \eqref{F and alpha}. Finally, for the third inequality we used $ \avg{f\1\abs{m}} \sim 1 $ to estimate $ 1-\norm{F}_{\Lp{2}\to\Lp{2}} \gtrsim \eta/\alpha $.
Using \eqref{alpha scaling outside supp v} in \eqref{inv-B norm outside supp v} yields \eqref{goal:inv-B norm outside supp v}.
This completes the proof of \eqref{general inv-B bound}. 
\end{Proof}

\chapterl{Examples}

In this chapter we present some simple examples that illustrate the need of various assumption made on $ a $ and $ S $.
Recall that the assumptions {\bf A1-3} where introduced in the beginning of Chapter~\ref{chp:Set-up and main results}, and they are used extensively throughout this paper.
Our main results are formulated under the additional assumption that $ m $ is bounded in $ \BB $.
Verifying this uniform boundedness
was treated as a separate problem in Chapter~\ref{chp:Uniform bounds}, and for this purpose the additional assumptions {\bf B1} and {\bf B2} along with the auxiliary function $ \Gamma $ were introduced.
In particular, the non-effective uniform bounds of Theorem~\ref{thr:Qualitative uniform bounds} were replaced by the corresponding quantitative results in the form of Theorem~\ref{thr:Quantitative uniform bounds when a = 0} and Theorem~\ref{thr:Quantitative uniform bound for general a}, which rely on {\bf B1-2} and assumptions on $ \Gamma $.

In the following sections we will demonstrate how the properties  {\bf A3} and {\bf B1} and the function $\Gamma $ are used to effectively rule out certain 'bad' behaviors of $ m $, by considering simple examples. 
Before going into details  let us shortly comment the remaining assumptions {\bf A1}, {\bf A2} and {\bf B2}, which we will not address any further.
The assumption {\bf A1} is structural in nature. 
It reflects the applications we have in mind, e.g., random matrix theory as explained in Chapter~\ref{chp:Local laws for large random matrices} and Section 3 of \cite{AEK1cpam}. 
On the other hand, for a full analysis of Laplace-like operator on rooted trees (cf. Chapter~\ref{chp:Introduction}) the assumption of symmetry of $ S $ should be lifted.
The smoothing assumption {\bf A2} was made for technical reasons. 
It is appropriate for the random matrix theory as it generalizes the upper bound \eqref{mean-field property} appearing in the definition of Wigner-type random matrices.
The property {\bf B2} on the other hand is a practical condition for easily obtaining an effective $ \Lp{2}$-bound on the solution $ m $ when $ a \neq 0 $ (cf. Remark~\ref{rmk:Positive diagonal and 1/2-Holder regularity}).

Besides demonstrating how the solution $ m $ can become unbounded, and how to exclude such blow-ups with the right assumptions, we also provide three other kinds of examples in this chapter.
First, in Section~\ref{sec:Effects of non-constant function a} we show that although generally playing a secondary role to $ S $ in our analysis, the non-constant function $ a $ can also affect the behavior of $ m $ significantly.
Second, in Section~\ref{sec:Discretization and reduction of the QVE} we explain how to switch between different representations of a given QVE, and possibly reduce the dimensionality of the problem.
Third, in Section~\ref{sec:Simple example that exhibits all universal shapes} we provide a very simple two parameter family of operators $ S $, for which the corresponding solution of the QVE with $ a=0 $, exhausts all the different local shapes of the generating density, described by our main result, Theorem~\ref{thr:Shape of generating density near its small values}.

Most of the examples here are represented in the special setting where $ \Sx $ is the unit interval and $ \Px $ is the restriction of the Lebesgue measure to this interval, i.e., 
\bels{std continuous QVE}{
(\1\Sx,\Borel,\Px\1) \,:=\, \bigl(\1[\10,1\1]\1,\1\Borel([\10,1\1])\1,\dif x\1\bigr)
\,,
}
with $ \Borel([\10,1\1])$ denoting the standard Borel $ \sigma $-algebra. 
Together, with the discrete case \eqref{RM setup} this is the most common setup for the QVE. 
An example, where a more complicated setup is natural is \cite{AZdep} (cf. also Subsection~3.4).

\section{The band operator, lack of self-averaging, and property {\bf A3}}
\label{sec:The band operator and lack of self-averaging and property A3}

The uniform primitivity assumption, {\bf A3}, was made to exclude choices of $S$ that lead to an essentially decoupled system. Without sufficient coupling of the components $m_x$ in the QVE the components of the imaginary part of the solution are not necessarily comparable in size, i.e., $v_x \sim v_y$, may not hold (cf. \eqref{v compared to avg-v} of Lemma \ref{lmm:Constraints on solution}).
No universal growth behavior at the edge of the support of the generating density, as described by Theorem~\ref{thr:Shape of generating density near its small values}, can be expected in this case, since the support of $v_x$ may not even be independent of $x$.

The simplest such situation is if the components may be partitioned into two subsets $ I $ and $ I^\cmpl = \Sx\backslash I$, that are \emph{completely} decoupled in the sense that $ S $ leaves invariant the families of functions which are supported either on $ I $ or $ I^\cmpl $. In this case the QVE decouples into two independent QVEs. These independent QVEs can then be  analyzed separately using the theory developed here.
Assumption {\bf A3} also excludes a situation, where the functions supported on $I$ are mapped to the function supported on the complement of $I$, and vice versa. This case has an instability at the origin $\tau=0$ (cf. Lemma~\ref{lmm:Scalable symmetric matrices} and Theorem~\ref{thr:Scalability and full indecomposability} in the discrete setup) and requires a special treatment of the lowest lying eigenvalue of $ S $ (cf. \cite{AltEK}).

Another example, illustrating why {\bf A3} is needed, is the case where $ a = 0 $ and the integral kernel of $ S $ is supported on a small band along the diagonal:
\[
S_{x y}\,=\, \eps^{-1}\1\xi(x+y)\2\Ind\setb{|x-y|\leq \eps/2}\,.
\]
Here, $\xi:\R\to (0,\infty)$ is some smooth function and $ \eps > 0 $ is a constant. 
For any fixed $ \eps $ the operator $ S $ satisfies {\bf A1-3} and {\bf B1}. Also, the conditions {\bf B2} and $ \Gamma(\infty) = \infty $ (cf. \eqref{def of Gamma(infty)}) hold for the corresponding QVE.
As $\eps$ approaches zero, however, the constant $ L $ from assumption {\bf A3} (among other model parameters such as $\norm{S}_{\Lp{2}\to \BB}$ from {\bf A2}) diverge. 
In the limit, $S$ becomes a multiplication operator and the QVE decouples completely, 
\[
-\2\frac{1}{m_x(z)}\,=\, z + \xi(x)\2m_x(z)
\,.
\]
The solution becomes trivial
\[
m_x(z) \,:=\, 
\xi(x)^{-1/2}m_{\mrm{sc}}\bigl(\2\xi(x)^{-1/2}z\2\bigr)
\,,
\]
where $ m_{\mrm{sc}} : \Cp \to \Cp  $ is the Stieltjes transform of Wigner's semi-circle law \eqref{SC}. 
In particular, the support of the component $v_x$ of the generating density depends on $ x $.

\section{Divergences in $ \BB $, outliers, and  function $ \Gamma $}
\label{sec:Divergences for special x-values: Outlier rows}

The purpose of this section is to illustrate the role of the auxiliary function $ \Gamma $, generated by the pair $ (a,S)$ through \eqref{def of Gamma}, in proving bounds for $ m $ in $ \BB $.
We present two simple families of QVEs for which the solutions $ m $ are uniformly bounded in $ \Lp{2} $, but for which the corresponding $ \Gamma $'s become increasingly ineffective in converting these bounds into $ \BB$-bounds for some members of these families.
In both cases a few exceptional row functions, $ S_x = (\1y \mapsto S_{xy})$, cause divergencies in the corresponding components, $ m_x $, of the solution. 
In the first example, the QVE can be solved explicitly and thus the divergence can be read off from the solution formula. The second example is a bit more involved. 
It illustrates a somewhat counterintuitive phenomenon of divergencies that may arise if one smoothens out discontinuities of the integral kernel of $ S $ on small scales.

\subsection{Simplest example of blow-up in $\BB$:}
Let $ a = 0 $. Consider the $2 \times 2$ - block constant integral operator $ S $ with the  kernel
\bels{simplest generic S}{
S_{xy}
\,=\,\lambda\2\Ind\sett{\1x\leq \delta,\, y>\delta\2}
+ \lambda\2\Ind\sett{\1 y\leq \delta,\, x>\delta\2} 
+ \Ind\sett{\1x> \delta,\, y>\delta\2}
\,,
}
parametrized by two positive constants $\lambda$ and $\delta $. 
For any fixed values of $\lambda>0 $ and $ \delta_\ast \in (\10\1,1/2) $, the properties {\bf A1-3} and {\bf B1} hold uniformly for every $ \delta \leq \delta_\ast $. 
In particular, the solutions are  uniformly bounded in $ \Lp{2} $ for $ \delta \leq \delta_\ast $, since the part (i) of Theorem~\ref{thr:Quantitative uniform bounds when a = 0} yields a uniform bound when $ \abs{z} \leq \eps $, for some $ \eps \sim 1 $, while \eqref{L2 bound on m when a=0} guarantees the $ \Lp{2}$-boundedness in the remaining domain $ \abs{z} > \eps $.
In fact, the solution for any parameter values has the structure
\bels{structure of 2 x 2 solution}{
m_x(z)\,=\,\mu(z)\2\Ind\sett{\1x\leq \delta\2} + \nu(z)\2\Ind\sett{\1x>\delta\2}
\,, 
}
where the two functions $\mu, \nu: \eCp  \to \eCp $ satisfy the coupled equations
\bels{equations for mu and nu}{
-\frac{1}{\1\mu(z)}\,=\, z + (1-\delta)\2\lambda\2\nu(z)\,,\qquad 
-\frac{1}{\1\nu(z)}\,=\,z+\lambda\2\delta\2\mu(z)+(1-\delta)\2\nu(z)\,.
}
\FloatBarrier
\begin{figure}[h]
	\centering
	\hspace{-0.025\textwidth}
	\includegraphics[width=1.01\textwidth]{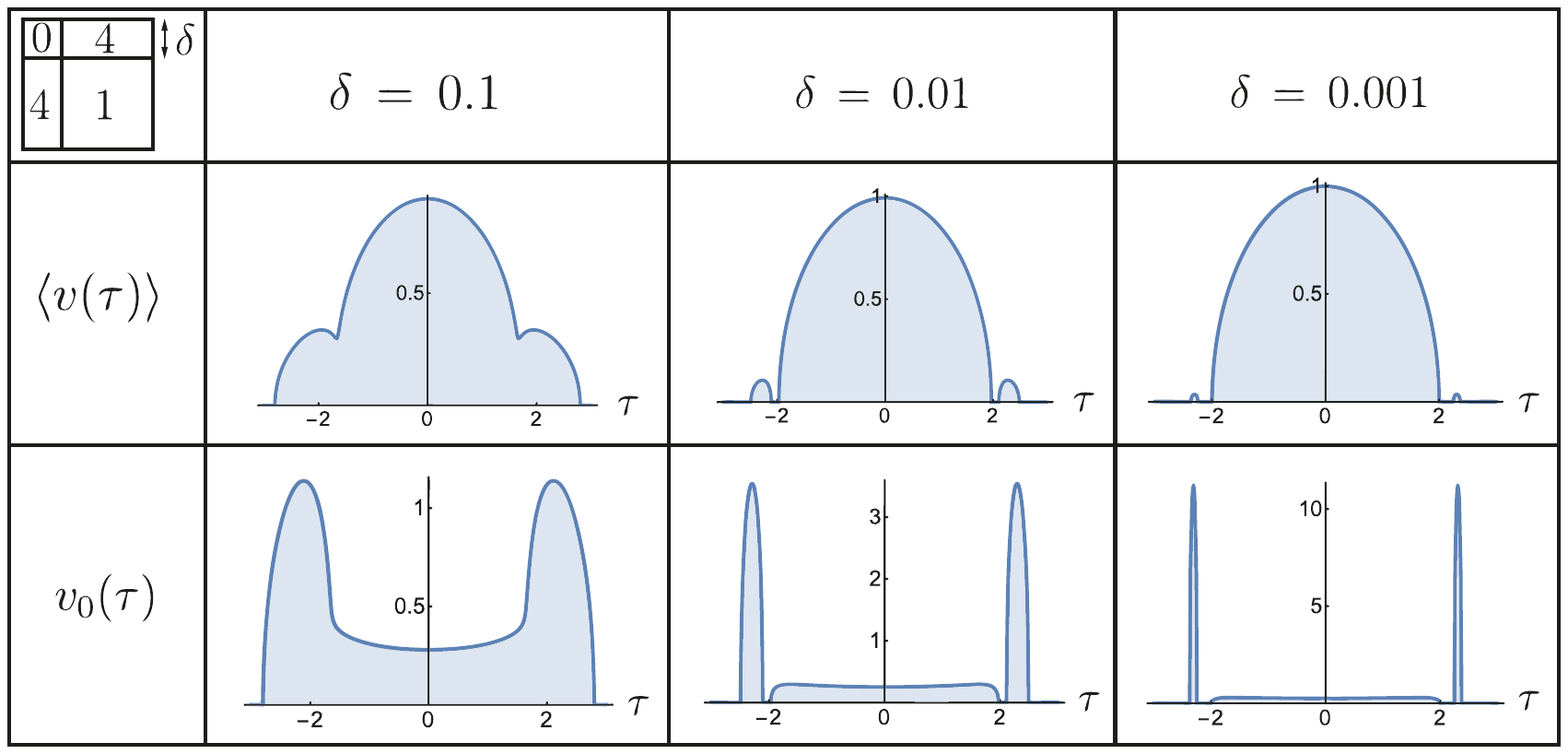}
	\caption{As $\delta$ decreases the average generating density remains bounded, but the $0$-th component of the generating density blows up at $\pm\tau_0$.}
	\label{Fig:SingleRowBlowUp}
\end{figure}
Let us consider a fixed $\lambda>2$. 
Then, as we take the limit $\delta \downarrow 0$ the strictly increasing function $ \Gamma $ generated by $ S $ through \eqref{def of Gamma}, satisfies
\bels{example: Gamma for model 1}{
\Gamma(\tau) \,\leq\, \sqrt{\11+\delta\1\tau^{\12}}
\,,\qquad
\tau \in (0,\infty)
\,.
}
This means that the uniform bound \eqref{uniform bound outside z=0} becomes ineffective as $ \Gamma^{-1}(\Lambda) \to \infty $ for any fixed $ \Lambda \in(1,\infty) $ as $ \delta \downarrow 0 $.  
Indeed, the row functions $ S_x $ indexed by a small set of rows $ x \in [0,\delta] $ differ from the row functions indexed by $ x \ge \delta $, and this leads to a blow-up in the components $m_x(z)$ with  $x \in [\10,\1\delta\1]$ at a specific value of $z$. 
More precisely, we find
\bels{the simplest BB-blow-up solved}{
\abs{\1\mu(\pm\tau_0\1)\1} \,\sim\, \frac{1\!}{\!\sqrt{\msp{-1}\delta\2}}
\,,
\qquad\text{at}\quad
\tau_0\,:=\, \frac{2\1\lambda}{\sqrt{\lambda^2-(\lambda-2)^2}}\,.
}
While the $\BB$-norm of $m$ diverges as $\delta$ approaches zero, the $\Lp{2}$-norm stays finite, because the divergent components contribute less and less. 
The situation is illustrated in Figure~\ref{Fig:SingleRowBlowUp}.

The integral kernel \eqref{simplest generic S} makes sense even for $\delta=0$. In this case we get for the generating measure the formulas,
\bea{
v_0(\dif\tau)
\,&=\, 
\frac{\lambda \2\sqrt{4-\tau^2}}{2\2\lambda^2-2\2 \tau^2(\lambda-1)}\,
\Ind\sett{\2\tau \in[-2,2\1]\2}\,\dif \tau
\,+\, 
\frac{\pi(\lambda-2)}{2(\lambda-1)}
\bigl(
\1\delta_{-\tau_0}(\dif \tau) + \delta_{\tau_0}(\dif \tau)\1
\bigr)
\,,
\\
v_x(\dif\tau)
\,&=\,
\frac{1}{2}\sqrt{4-\tau^2}\,
\Ind\sett{\2\tau \in[-2,2\2]\2}\,\dif \tau
\,,
\qquad\qquad x \in(\10\1,1\1]
\,.
}
The non-zero value that $v_0$ assigns to $\tau_0$ and $-\tau_0$ reflects the divergence of $m$ in the uniform norm at these points. 

In the context of random matrix theory the operator $S $ with small values of the parameter $\delta$ corresponds to the variance matrix (cf. Definition~\ref{def:Wigner-type random matrix}) of a perturbation of a Wigner matrix. The part of the generating density, which is supported around $\tau_0$ corresponds to a small collection of eigenvalues away from the bulk of the spectrum of the random matrix. These \emph{outliers} will induce a divergence in some elements of the resolvent \eqref{def of brm-G} of this matrix. This divergence is what we see as the divergence of $\mu$ in \eqref{the simplest BB-blow-up solved}. 

\subsection{Example of blow-up in $ \BB $ due to smoothing:}
We present a second example of a different nature, in which the bounds of Proposition~\ref{prp:Converting L2-estimates to uniform bounds} for converting $ \Lp{2}$-estimates of $ m(z) $ into uniform bounds become ineffective. 
The smoothing of discontinuities in $ S $ may cause blow-ups in the solution of the QVE (cf. Figure~\ref{Fig:QVEWithSlope}). This is somewhat surprising, since by conventional wisdom, smoother data implies smoother solutions. The key point here is that the smoothing procedure creates a few row functions that are far away from all the other row functions. The following choice of operator demonstrates this mechanism:
\[
S_{x y}^{(\eps)}\,=\, \frac{1}{2}(\1r_x\1s_y + r_y\1s_x)
\,.
\]
Here the two continuous functions $r,s: [\10,1\1] \to (0,1]$, are given by
\bea{
r_x \,&=\, \bigl(\21+\eps^{-1}(x-\delta)\1\bigr)\2\Ind\sett{\2\delta-\eps<x\leq \delta\2} 
\,+\, 
\Ind\sett{\1x>\delta\2}
\,,
\\
s_x \,&=\,
2\1\lambda\2\Ind\sett{\1x \leq \delta\2} 
\,+\, 
\bigl(\22\1\lambda-\eps^{-1}\bigl(\22\1\lambda-1)(x-\delta\1)\1\bigr)\2\Ind\sett{\2\delta<x\leq \delta+\eps\2}
\\
&\msp{28}+
\Ind\sett{\1x>\delta+\eps\2}
\,,
}
respectively. 
The parameters $\lambda>0$, $\delta \in (0,1)$ are considered fixed, while $ \eps \in (0,\delta) $ is varied.
The continuous kernel $S^{(\eps)}$ represents a smoothed out version of the $2 \times 2$-block operator $ S^{(0)} = S $  from \eqref{simplest generic S}.

In this case, $ \Gamma(\infty) = \lim_{\tau\to\infty}\Gamma(\tau) = \infty $ holds for each operator $S^{(\eps)}$,  $\eps>0$, as well as for the limiting operator $S^{(0)}$. However, the estimates \eqref{abs-m_x bounded using Gamma_x and L2-bound} and \eqref{simplified L2 to BB conversion}  become ineffective for proving uniform bounds, since for any fixed $ \tau < \infty $ the value $ \Gamma(\tau) $  becomes too small in the limit $ \eps \to 0 $. 
This is due to the distance that some row functions $ S_x^{(\eps)} $, with $ \abs{x-\delta}\leq \eps $, 
have from all the other row functions. 

Let $m=m^{(\eps)}$ denote the solution of the QVE corresponding to $S^{(\eps)}$. We will now show that, even though $m^{(0)}$ is uniformly bounded, the $\BB$-norm of $m^{(\eps)}$ diverges as $\eps$ approaches zero for certain parameters $\lambda$ and $\delta$. 

The solution $m=m^{(\eps)}$ has the form
\[
m_x(z)\,=\, -\2\frac{1}{z+ \varphi(z) \2 r_x + \psi(z)\2 s_x }\,.
\]
Here, the two functions $\varphi^{(\eps)}=\varphi=\avg{s,m},\psi^{(\eps)}=\psi=\avg{r,m}: \Cp \to \Cp$ satisfy the coupled equations
\bels{phi psi coupled equations}{
\varphi(z)\,&=\, -\int_{[\10,1]} \frac{s_x\2\dif x}{z+ \varphi(z) \2 r_x + \psi(z)\2 s_x}
\\
\psi(z) \,&=\, -\int_{[\10,1]}\frac{r_x\2\dif x}{z+ \varphi(z) \2 r_x + \psi(z)\2 s_x}\,.
}

In the parameter regime $\lambda\geq 10$ and $\delta \leq 1/10$ the support of the generating density of $m^{(0)}$ consists of three disjoint intervals,
\[
\supp v^{(0)}\,=\, 
\supp \varphi^{(0)}\,=\, \supp \psi^{(0)}\,=\, [-\beta_1,-\alpha_1]\cup[-\alpha_0,\alpha_0]\cup[\alpha_1, \beta_1]\,.
\]
\begin{wrapfigure}{r}{0.68\textwidth}
	\centering
	\vspace{-5pt}
	\includegraphics[width=0.68\textwidth]{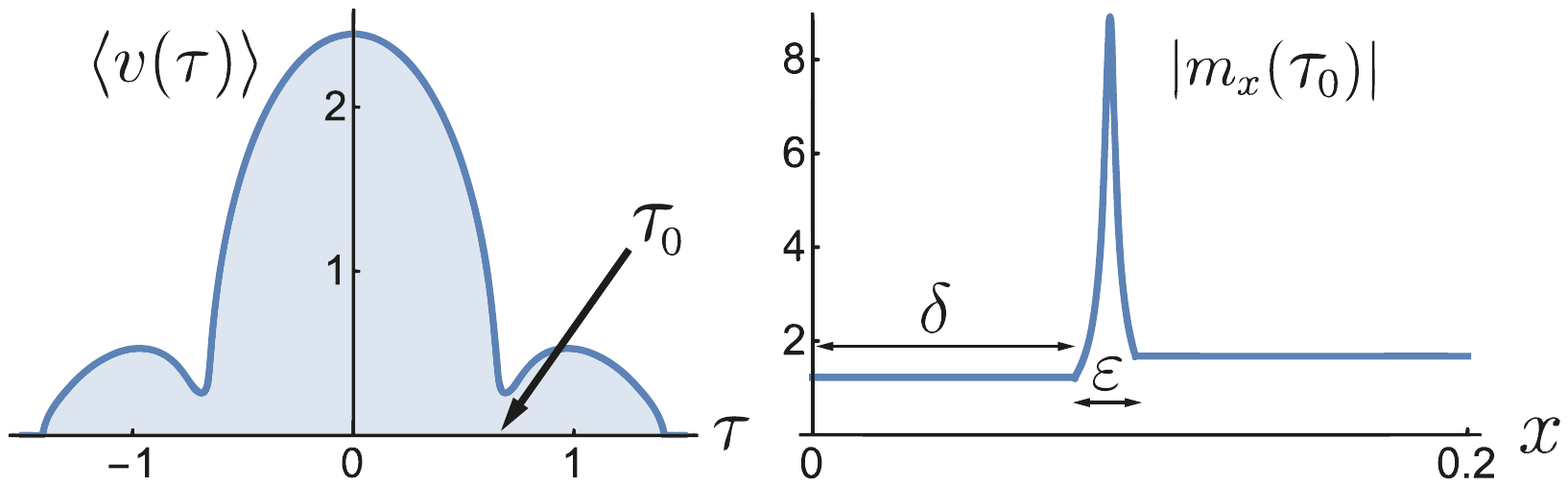}
	\caption{As $\eps$ decreases the average generating density remains bounded. The absolute value of the solution as a function of $x$ at a fixed value $\tau_0$ inside the gap of the limiting generating density has a blow up.}
	\label{Fig:QVEWithSlope}
	\vspace{-7pt}
\end{wrapfigure}
Inside the gap $(\alpha_0,\alpha_1)$ the norm $\norm{m^{(\eps)}}$ diverges as $\eps \downarrow 0$. This can be seen indirectly, by utilizing Theorem~\ref{thr:Shape of generating density near its small values}. We will now sketch an argument, which shows that assuming a uniform bound on $m$ leads to a contradiction. Suppose there were an $\eps$-independent bound on the uniform norm. Then a local version of Theorem~\ref{thr:Shape of generating density near its small values}  would be applicable and the generating density $v^{(\eps)}$ of $m^{(\eps)}$ could approach zero only in the specific ways described in that theorem. Instead, the average generating density $\avg{v^{(\eps)}}$ takes small non-zero values along the whole interval $(\alpha_0,\alpha_1)$, as we explain below. This contradicts the assertion of the theorem. 

In fact, a stability analysis of the two equations \eqref{phi psi coupled equations} for $\varphi^{(\eps)}$ and $\psi^{(\eps)}$ shows that they are uniformly Lipshitz-continuous in $\eps$. In particular,
for $\tau$ well inside the interval $(\alpha_0,\alpha_1)$ we have
\[
\Im\2 \varphi^{(\eps)}(\tau)+\Im\2\psi^{(\eps)}(\tau) \,\leq\, C\2\eps\,.
\]
Thus, the average generating density takes small values here as well, $\avg{\1v^{(\eps)}(\tau)\1}\leq C\1\eps$.
On the other hand, $\Im\2\varphi$ and $\Im \2\psi$ do not vanish on $(\alpha_0,\alpha_1)$. 
Their supports coincide with the support of the generating density, $v^{(\eps)}$. 
By Theorem~\ref{thr:Generating density supported on single interval} this support is a single interval for all $\eps>0$ and by the continuity of $\varphi$ and $\psi$ in $\eps$, every point $\tau \in(- \beta_1,-\alpha_1) \cup(\alpha_1, \beta_1)$ is contained in this interval in the limit $\eps\downarrow 0$.

This example demonstrates that certain features of the solution of the QVE cannot be expected to be stable under smoothing of the corresponding operator $S$. Among these features are gaps in the support of the generating density, as well as the universal shapes described by Theorem~\ref{thr:Shape of generating density near its small values}.

\section{Blow-up at $z=0$ when $a=0$ and assumption {\bf B1}}
\label{sec:Blow-up at z=0 when a=0 and assumption B1}

In the case $ a = 0 $, the point $ z = 0 $ plays a special role in the QVE. It is the only place where $ m(z) $ may become unbounded even in the $ \Lp{1}$-sense (cf. \eqref{BB-bound of v is bounded by avg-v}). 
In this section we give two simple examples which exhibit  different types of blow-ups at $ z = 0 $. 
Moreover, we motivate the assumption {\bf B1} by showing that it corresponds to a necessary condition for the solution to remain bounded in a stable way at $z = 0 $ when the dimension of $ \Sx $ is finite.

Suppose $ a = 0 $. The assumption {\bf B1} is designed to prevent divergencies in the solution at the origin of the complex plane. These divergencies are caused by the structure of small values of the kernel $S_{x y}$. 
In Section~\ref{sec:Uniform bound around z=0 when a=0} we saw that at $ z = 0 $ the QVE reduces to
\bels{DAD problem}{
v_x \int_\Sx S_{x y}v_y\1\Px(\dif y) \,=\, 1\,, \qquad x \in \Sx
\,,
}
where $v_x=\Im \2m_x(0)$. 
Thus the boundedness of $ m(z) $ for small $ \abs{z} $ is related to the solvability of \eqref{DAD problem}.
There is an extensive literature on \eqref{DAD problem} that dates back at least to \cite{S1964}.

In the discrete setup, with $ \Sx := \sett{1,\dots,N} $ and $ \Px(\sett{i}) := N^{-1}$, the solvability of \eqref{DAD problem} is equivalent to the scalability (cf. Definition~\ref{def:Square matrices with non-negative entries}) of the matrix $ \brm{S} = (s_{ij})_{i,j=1}^N $, with non-negative entries $ s_{ij} := N^{-1}S_{ij} $.
We refer to Appendix~\ref{sec:Scalability of matrices with non-negative entries} for a discussion of various issues related to scalability.
Theorem~\ref{thr:Scalability and full indecomposability} below shows that the discrete QVE has a unique bounded solution if and only if the matrix $ \brm{S} $ is fully indecomposable.
This bound may deteriorate in $ N $. 
However, if $ S $ is block fully indecomposable (the property {\bf B1}), then the bound on the solution depends only on the number of blocks (cf. \eqref{upper bound on wti-J on local v-avgs} and Lemma~\ref{lmm:Uniform bound on discrete minimizer}).

Let us go back to the continuum setting. If assumption {\bf B1} is violated, the generating measure may have a singularity at $z=0$. In fact, there are two types of divergencies that may occur. Either the generating density exists in a neighborhood of $\tau=0$ and has a singularity at the origin, or the generating measure has a delta-component at the origin. Both cases can be illustrated using the $2 \times 2$-block operator with the integral kernel \eqref{simplest generic S}.

The latter case occurs if the kernel $S_{xy} $ contains a rectangular zero-block whose circumference is larger than $2$. For $ S $  from \eqref{simplest generic S} this means that $\delta>1/2$. Expanding the corresponding QVE for small values of $z$ reveals
\[
v_x(\dif \tau)\,=\, \pi\frac{\2\delta-1}{\delta}\,\Ind\sett{x \leq \delta}\2\delta_0(\dif \tau)\,+\,\Ord(1)\1\dif \tau\,.
\]
The components of the generating measure with $x \in [\10,\delta\2]$ assign a non-zero value to the origin. 

The case of a singular, but existing generating density can be seen from the same example, \eqref{simplest generic S}, with the choice $\delta=1/2$. From an expansion of the QVE at small values of $z$ we find for the generating density:
\[
v_x(\tau)\,=\,(2\1\lambda)^{-2/3}\sqrt{3\2}\2|\tau|^{-1/3}\Ind\setb{2\1x \leq 1}\;+\;\Ord(\11\1)
\,.
\]

The blow-up at $z=0$ has a simple interpretation in the context of random matrix theory. It corresponds to an accumulation of eigenvalues at zero. If the generating density assigns a non-zero value to the origin, a random matrix with the corresponding $S$ as its variance matrix (cf. Definition~\ref{def:Wigner-type random matrix}) will have a kernel, whose dimension is a finite fraction of the size $N$ of the matrix.

Assumption {\bf B1} excludes the above examples. 
In general, it ensures that a discretized version, of dimension $ K $, of the original continuous problem \eqref{DAD problem} has a unique bounded and stable solution by the part (i) of Theorem~\ref{thr:Scalability and full indecomposability}.
The bounded discrete solution is then used in Section~\ref{sec:Uniform bound around z=0 when a=0} to argue that also the continuous problem has a bounded solution by using a variational formulation \eqref{DAD problem}.

\section{Effects of non-constant function $ a $}
\label{sec:Effects of non-constant function a}

For most of our analysis the function $ a \in \BB $ has played a secondary role. 
However, even for the simplest operator $ S $ the addition of a non-constant $ a $ to the QVE without $ a $ can alter the solution significantly.  
Indeed, let us consider the simplest case $ S_{xy} =1 $, so that {\bf A1-3} hold trivially. Since $ \avg{w,Sw} = \avg{w}^2 $, for any $ w \in \Lp{1} $, $ S $ satisfies also {\bf B2}, and thus Lemma~\ref{lmm:Quantitative L2-bound} yields a uniform $ \Lp{2}$-bound $ \sup_{z\ins \Cp} \norm{m(z)}_2 \lesssim 1 $. 
Since $ (Sm(z))_x = \avg{\1m(z)} $ for any $x $, we obtain a closed scalar integral equation for the average of $ m(z) $
\bels{averaged QVE for deformed Wigner}{
\avg{\1m(z)}
\,= \int_\Sx \frac{\Px(\dif x)}{\2z\1+\1a_x+\avg{\1m(z)}\2}
\,,
}
by integrating the QVE.
If $ a $ is piecewise $ 1/2$-H\"older regular in the sense of \eqref{def of PW-1/2-Holder}, then Theorem~\ref{thr:Quantitative uniform bound for general a} yields a uniform bound $ \nnorm{m}_\R \lesssim 1 $ (see Remark~\ref{rmk:Positive diagonal and 1/2-Holder regularity}). In particular, Theorem~\ref{thr:Shape of generating density near its small values} applies. 

In the random matrix context \eqref{averaged QVE for deformed Wigner} determines the asymptotic density of states of a \emph{deformed Wigner matrix}, 
\bels{deformed Wigner matrix}{
\brm{H} \,=\, \brm{A} + \brm{W}
\,,
}
where $ \brm{W} $ is an $ N $-dimensional Wigner matrix, and $ \brm{A} $ is a self-adjoint non-random matrix satisfying $ \Spec(\brm{A}) = \sett{a_i:1 \leq i \leq N} $, in the limit $ N \to \infty $ (cf. \cite{PasturDefWig}).

In the special case, that $ N $ is an even integer and $ \brm{A} $ has only two eigenvalues $ \pm\1\alpha $, both of degeneracy $ N/2 $, i.e., 
\bels{2-value a}{
a_k \,:= 
\begin{cases}
-\alpha \quad&\text{when } 1 \leq k \leq N/2
\\
+\alpha  &\text{when }N/2+1 \leq k \leq N\,,
\end{cases}
}
the equation \eqref{averaged QVE for deformed Wigner} can be reduced to a single cubic polynomial for $ \avg{m(z)} $. 
In \cite{BH} this matrix model \eqref{deformed Wigner matrix} was analyzed and the authors demonstrated that the asymptotic density of states may exhibit a cubic root cusp for some values of the parameter $ \alpha $.
The cubic root singularity seems natural in the special case \eqref{2-value a} as $ \avg{m(z)} $ satisfies a cubic polynomial.  
If the range of $ a $ contains $ p \in \N $ distinct values, then \eqref{averaged QVE for deformed Wigner} can be reduced to a polynomial of degree $ p+1 $. 
Our results, however, show that in spite of this arbitrary high degree, the worst possible singularity is cubic, and the possible shapes of the density of states are described by Theorem~\ref{thr:Shape of generating density near its small values}, as long as $ a $ is sufficiently regular.   


\sectionl{Discretization and reduction of the QVE}

By choosing $ \Sx := \sett{1,\dots,N} $ and $ \Px(\sett{i}) := N^{-1}$ for some $ N \in \N $ the QVE \eqref{QVE} takes the form 
\bels{discrete QVE}{
-\,\frac{1}{m_i\!}
\;=\,
z  +a_i +\frac{1}{N}\sum_{j=1}^N S_{ i j} m_j\, , \qquad  i =1,\dots, N
\,,
}
and hence this discrete vector equation is covered by our analysis. 
Alternatively, we may treat \eqref{discrete QVE} in the continuous setup \eqref{std continuous QVE} by defining a function $ a: [0,1] \to \R $ and the integral kernel of $ S $ on $[0,1]^2 $ by
\bels{discrete QVE as continuous one}{
a(x)  := \sum_{i=1}^N a_i\2 \chi_i(x)
\,,\qquad\text{and}\qquad
S(x,y):= \sum_{i,j=1}^NS_{ij}\2\chi_i(x)\2\chi_j(y)
\,,
}
respectively, with the auxiliary functions $ \chi_i : [0,1] \to \sett{0,1} $, $ i=1,\dots,N$, given by
\[
\chi_i(x) \,:= \Ind\setb{N\1x \in [\1i-1,i\1)}
\,.
\]
In order to distinguish between discrete and continuous quantities we have adapted in this section a special convention by writing the continuous variable $ x $ in the parenthesis and not as a subscript.
Since the continuous QVE conserves the block structure, and both the discrete and continuous QVEs have unique solutions $ \brm{m} = (m_i)_{i=1}^N$ and $ m = (x \mapsto m(x)) $, respectively, by Theorem~\ref{thr:Existence and uniqueness}, we conclude that these solutions are related by 
\bels{relation of cont and disc m}{
m(z;x) \,= \sum_{i=1}^N m_i(z)\2\chi_i(x)
\,.
}

This re-interpretation of a discrete QVE as a continuous one is convenient when comparing different discrete QVEs of non-matching dimensions $ N $. 
For example, the convergence of a sequence of QVEs generated by a smooth function $ \alpha : [\10\1,1\1] \to \R $ and a symmetric smooth function $ \sigma:[\10\1,1\1]^2\to [\10\1,\infty) $, through 
\[
\qquad
a_i \,:=\, \alpha\Bigl(\frac{i}{N}\Bigr)\,,
\quad\text{and}\quad
S_{ij} \,:=\, 
\sigma\Bigl(\frac{i}{N},\frac{j}{N}\Bigr)\,,\qquad i,j =1, \dots, N
\,,
\]
can be handled this way.
Indeed, if $ \brm{m} $ solves the discrete QVE then the functions $ m $ defined through the right hand side of \eqref{relation of cont and disc m} converge to the solution of the continuous QVE with $ a(x) = \alpha(x) $ and $ S(x,y):=\sigma(x,y) $ as $ N \to \infty $.

In particular, if the continuum operator satisfies {\bf A3} and {\bf B2}, or merely {\bf B1} in the case $ \alpha = 0 $ (all other assumptions are automatic in this case), then the convergence of the generating densities is uniform and the support of the generating density is a single interval for large enough $N$. 
This is a consequence of the stability result, Theorem~\ref{thr:Stability}, more precisely of Remark \ref{rmk:Perturbations on a and S} following it and of the fact that the limiting operator $ S $ is block fully indecomposable, and the knowledge about the shape of the generating density from Theorem~\ref{thr:Shape of generating density near its small values} and Theorem~\ref{thr:Generating density supported on single interval}.

We also have the following straightforward dimensional reduction. Suppose there exists a partition $ \mcl{I} $ of the first $ N $ integers, and numbers $ (\wht{S}_{IJ})_{I,J\in\mcl{I}} $ and $ (\wht{a}_I)_{I \in \mcl{I}} $, indexed by the parts, such that for every $ I,J \in \mcl{I} $ and $ i \in I $,
\[
\sum_{j \in J} S_{ij} = \abs{J}\wht{S}_{IJ}\,,
\quad\text{and}\quad
a_i = \wht{a}_I
\,.
\]   
Then $ m(z) $ is piecewise constant on the parts of $ \mcl{I} $, i.e., there exist numbers $ \wht{m}(z) = (\wht{m}_I(z))_{I\in\mcl{I}} $, such that $ m_i(z) = \wht{m}_I(z) $, for every $ i \in  I$.
The numbers $ \wht{m}(z) $ solve the $ \abs{\mcl{I}}$-dimensional reduced QVE,
\[
-\frac{1}{\wht{m}_I(z)} \,=\, z \1+\2 \wht{a}_I + \sum_{J\in\mcl{I}}
\frac{\abs{J}}{N}
\wht{S}_{\msp{-1}I\msp{-1}J}\2\wht{m}_J(z)%
\,.
\]
Here the right hand side can be written in the standard form \eqref{QVE} by identifying $ \Sx = \mcl{I} $ and $ \wht{\Px}(J) = \abs{J}/N $. 
In the special case where the matrix $ \brm{S} = (S_{ij})_{i,j=1}^N $ has constant row sums, $ N^{-1}\sum_j S_{ij} = 1 $, and $ \brm{a} = 0 $, the reduced QVE is one-dimensional, and is solved by the Stieltjes transform of the Wigner semicircle law \eqref{SC}

The dimension reduction argument generalizes trivially to more abstract setups. Indeed, we have already used such a reduction in Section~\ref{sec:Divergences for special x-values: Outlier rows}, where we reduced the analysis of the infinite dimensional QVE, with an integral kernel $ S_{xy} $ defined in \eqref{simplest generic S}, to  the study of the two-dimensional QVE \eqref{equations for mu and nu}.

\sectionl{Simple example that exhibits all universal shapes}

We will now discuss how all possible shapes of the generating density from Theorem~\ref{thr:Shape of generating density near its small values} can be seen in the simple example of the $2\times 2$-block operator $S $, defined in \eqref{simplest generic S},  by choosing the parameters $\lambda$ and $\delta$ appropriately. 
For the choice of parameters $\lambda>2$ and $\delta = \delta_c(\lambda)$ with
\[
\delta_c(\lambda)\,:=\, \frac{(\lambda-2)^3}{2 \2\lambda^3 - 3\2 \lambda^2+ 15 \2\lambda -7}\,,
\]
the generating density exists everywhere and its support is a single interval. 
\begin{figure}[h]
	\centering
	\includegraphics[width=\textwidth]{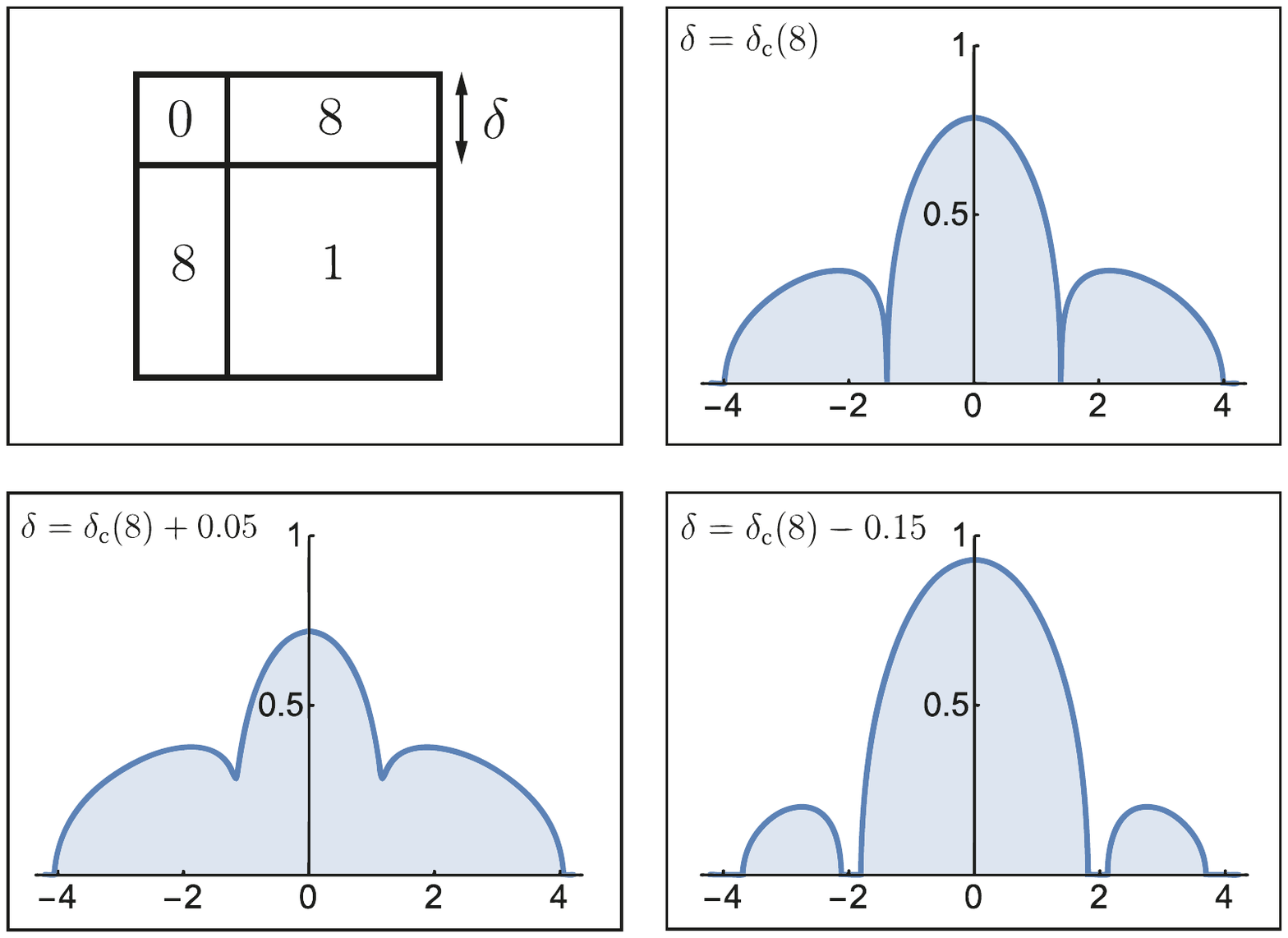}
	\caption{Decreasing $ \delta $ from its critical value $ \delta_{\rm c} $ opens a gap in the support of the average generating density. Increasing delta lifts the cubic cusp singularity.}
	\label{Fig:PerturbationAroundCriticalDelta}
\end{figure}
In the interior of this interval the generating density has exactly two zeros at some values $\tau_c$ and $-\tau_c$.
The shape of the generating density at these zeros in the interior of its own support is a cubic cusp, represented by the shape function $ \lim_{\rho \downarrow 0}\rho\2\Psi_{\rm min}(\omega/\rho^3) = 2^{2/3}\abs{\omega}^{1/3} $ (cf. Definition~\ref{def:Shape functions}).
If we increase $\delta$ above $\delta_c(\lambda)$, then the zeros of the generating density disappear. The support is a single interval with local minima close to $\tau_c$ and $-\tau_c$ and the shape around these minima is described by $ \rho\1\Psi_{\mathrm{min}}(\genarg/\rho^3)$ for some small positive $\rho$. Finally, if we decrease $\delta$ slightly below $\delta_c(\lambda)$ a gap opens in the support. Now the support of the generating density consists of three disjoint intervals and the shape of the generating density at the two neighboring edges is described by $\Delta^{\!1/3}\Psi_{\mathrm{edge}}(\genarg/\Delta)$, where $ \Delta \ll 1 $ is the size of the gap.
The different choices of $ \delta $ are illustrated in Figure \ref{Fig:PerturbationAroundCriticalDelta}.

\appendix
\chapter{Appendix}

The following simple comparison relations are used in the proof of Proposition~\ref{prp:Cubic perturbation bound around critical points} when $\Im\,z \neq 0 $ and $ \Re\,z $ is close to a local minimum of the generating density.

\begin{corollary}[Scaling relations]
\label{crl:Scaling relations}
Suppose the assumptions of Theorem~\ref{thr:Shape of generating density near its small values} are satisfied. There exists a positive threshold $\eps \sim 1$ such that for the set of local minima $ \mathbb{M}$, defined in \eqref{def of MM}, and any $\eta \in (\10\1,\eps\1]$, the average generating density has the following growth behavior close to the points in $\mathbb{M}$:
\begin{subequations}
\begin{enumerate}
\titem{a} {\bf Support around an edge:} 
At the edges $\alpha_i, \beta_{i-1}$ with $i =2, \dots, K'$, 
\bea{
\avgb{\2\Im\2m(\alpha_i+\omega + \cI\2\eta)}
\,&\sim\,
\avgb{\2\Im\2m(\beta_{i-1}-\omega + \cI\2\eta)}
\\
&\sim\, 
\frac{(\omega+\eta)^{1/2}}{(\alpha_i-\beta_{i-1}+\omega+\eta)^{1/6}}
\,,\qquad \omega \in [\10,\eps\1]\,.
}
\titem{b} 
{\bf Inside a gap:} 
Between two neighboring edges $\beta_{i-1} $ and $\alpha_i$ with $i =2, \dots, K'$,
\bea{
\avgb{\2\Im\2m(\tau + \cI\2\eta)}
\,\sim\;\, &\frac{\eta}{\2(\1\alpha_i-\beta_{i-1}+\eta)^{1/6}}
\\
&\times\,\biggl(
\frac{1}{(\tau-\beta_{i-1}+\eta)^{1/2}\!}\,
+
\frac{1}{(\alpha_i-\tau+\eta)^{1/2}\!}
\,\biggr)
\,,\quad
\tau \in [\1\beta_{i-1},\alpha_i]
\,.\msp{-30}
}
\titem{c} 
{\bf Support around an extreme edge:} Around the extreme points $ \alpha_1 $ and $ \beta_{K'} $ of $ \supp v $:
\bea{
\avgb{\2\Im\2m(\alpha_1+\omega + \cI\2\eta)}
\,&\sim\,
\avgb{\2\Im \2 m(\beta_{K'}-\omega + \cI\2\eta)}
\\
&\sim\,
\begin{cases}
\displaystyle
(\omega+\eta\1)^{1/2}\,,\quad	
&\omega \in [\10,\eps\1]\,;
\\
\displaystyle
\frac{\eta}{(\1\abs{\omega}+\eta)^{1/2}\!}\;,	
&\omega \in [-\eps,0\2]
\,.
\end{cases}
}
\titem{d} {\bf Close to a local minimum:} 
In a neighborhood of the local minima $\{\gamma_k\}$ in the interior of the support of the generating density,
\begin{equation*}
\avgb{\2\Im\2m(\gamma_k+\omega + \cI\2\eta)}
\,\sim\, 
\avg{\1v(\gamma_k)}+ (\2|\omega|+\eta\2)^{1/3}\,,\qquad \omega \in [-\eps,\2\eps\1]\,.
\end{equation*}
\end{enumerate}
\end{subequations}
All constants hidden behind the comparison relations depend on the parameters $\rho$, $L$, $\norm{S}_{\Lp{2}\to\BB}$ and $\Phi $. 
\end{corollary}

\begin{Proof}
The results follow by combining Theorem~\ref{thr:Shape of generating density near its small values} and the Stieltjes transform representation of the solution of QVE. 
We start with the claim about the growth behavior around the points $\{\gamma_k\}$. 
By the description of the shape of the generating density in Theorem~\ref{thr:Shape of generating density near its small values} and because of $\Psi_{\!\mrm{min}}(\lambda)\sim \min\{\lambda^2,\abs{\lambda}^{1/3}\}$ (cf. \eqref{def of Psi_min}), we have for small enough $\eps\sim 1$:
\[
\avg{v(\gamma_k +\omega)}
\,\sim\, 
\rho_k+\min\setb{\omega^2\!/\rho_k^5\1,|\tau|^{1/3}}
\,\sim\,
\rho_k+|\omega|^{1/3}
\,, 
\qquad 
\omega \in [-2\2\eps,2\2\eps]
\,.
\]
The constant $\rho_k$ is comparable to $\avg{v(\gamma_k)}$ by \eqref{minimum}. Thus, we find 
\[
\avgb{\2\Im\2m(\gamma_k +\omega + \cI\2\eta)}
=
\frac{1}{\pi}\int_{-\infty}^\infty \frac{\eta\2\avg{v(\tau)}\,\dif \tau}{\eta^2 + (\gamma_k +\omega-\tau)^2}
\sim \avg{v(\gamma_k)} +\int_{-2\2\eps}^{2\2\eps}\frac{\eta\2|\tau|^{1/3} \2\dif \tau}{\eta^2+ (\omega-\tau)^2}
\,, 
\]
for $ \omega \in [-\eps,\eps] $. The claim follows because the last integral is comparable to $(|\omega|+\eta)^{1/3}$ for any $\eps\sim 1$.

Let us now consider the case, in which an edge is close by. We treat only the case of a right edge, i.e., the vicinity of $\beta_i$ for $i=1 ,\dots, K'$. For the left edge the argument is the same. Here, Theorem~\ref{thr:Shape of generating density near its small values} and $\Psi_{\!\mrm{edge}}(\lambda) \sim \min\{\lambda^{1/2},\lambda^{1/3}\}$ (cf. \eqref{def of Psi_edge}) imply for small enough $\eps\sim 1$:
\[
\avg{v(\beta_i-\omega)}\,\sim\, \min\{\1\Delta^{-1/6}\omega^{1/2},\omega^{1/3}\}\,,
\qquad \omega \in [\10\1,2\2\eps\1]\,.
\]
The positive constant $\1\Delta$ is comparable to the gap size, $\1\Delta\sim\alpha_{i+1}-\beta_i$, if $\beta_i$ is not the rightmost edge, i.e., $i \neq K'$. In case $i=K'$, we have $\1\Delta\sim 1$. Let us set $\widetilde{\eps}:= \eps$ in case $i=K'$, and 
$ \widetilde{\eps}:= \min\{\eps,(\alpha_{i+1}-\beta_i)/2\}$ otherwise.
Then we find
\bea{
\avgb{\2\Im\2m(\beta_i +\omega +\cI\2\eta)}
\,&=\,
\frac{1}{\pi} 
\int_{-\infty}^\infty\frac{\eta\2\avg{v(\tau)}\,\dif \tau}{\eta^2 + (\beta_i +\omega-\tau)^2}
\\
&\sim\,
\eta\int_{0}^{2\2\eps}\frac{\min\{\1\Delta^{-1/6}\tau^{1/2},\tau^{1/3}\}}{\eta^2+ (\omega+\tau)^2}\,\dif \tau
\,,
\qquad \omega \in \bigl[-\eps,\wti{\eps}\,\bigr]
\,.
}
The contribution to the integral in the middle, coming from the other side $\alpha_{i+1}$ of the gap $(\beta_i, \alpha_{i+1})$, is not larger than the last expression, because the growth of the average generating density is the same on both sides of the gap. For the last integral we find
\begin{equation*}
\begin{split}
\eta\int_{0}^{2\2\eps}\frac{\min\{\1\Delta^{-1/6}\tau^{1/2},\tau^{1/3}\}}{\eta^2+ (\omega+\tau)^2}\,\dif \tau
\,&\sim\,
\begin{cases}
\displaystyle
\frac{\eta}{(\1\Delta+\eta)^{1/6}(\omega+\eta)^{1/2}\!}\;,\quad &\omega \in \bigl[\10,\wti{\eps}\2\bigr]\,;
\\
\displaystyle
\frac{(|\omega|+\eta)^{1/2}}{(\1\Delta+|\omega|+\eta)^{1/6}\!}\;
&\omega \in\bigl[-\1\eps,0\1\bigr]\,.
\end{cases}
\end{split}
\end{equation*}
This holds for any $\eps \sim 1$ and thus the claim of the lemma follows. 
\end{Proof}

\section{Proofs of auxiliary results in Chapter \ref{chp:Existence and uniqueness}}
\label{sec:Proofs of auxiliary results in Chapter:Existence and uniqueness}

\begin{Proof}[Proof of Lemma~\ref{lmm:Subcontraction}]
Recall that $ T $ is a generic bounded symmetric operator on $ \Lp{2} = \Lp{2}(\Sx;\C) $ that preserves non-negative functions. Moreover, the following is assumed:
\bels{}{
\exists\,h \in \Lp{2}
\quad\text{s.t.}\quad
\norm{h}_2 = 1
\,,\quad
Th \leq h 
\,,\quad\text{and}\quad
\eps := \inf_{x\ins\Sx} h_x > 0
\,.
}
We show that $ \norm{T}_{\Lp{2}\to\Lp{2}}\leq 1 $. Let us derive a contradiction by assuming  $ \norm{T}_{\Lp{2}\to\Lp{2}} > 1 $. 
We have
\bels{T^nh leq h}{
T^nh\leq h 
\,,\qquad \forall\,n \in \N\,.
}
Indeed, $ Th \leq h $ is true by definition, and \eqref{T^nh leq h} follows by induction.
 
Now, the property $ \norm{T}_{\Lp{2}\to\Lp{2}} > 1 $ would imply
\[
\exists\,u \in \BB\quad\text{s.t.}\quad
\norm{u}_2 = 1
\,,
\quad
u\geq 0\,,
\quad\text{and}\quad
\avg{u,Tu} > 1
\,.
\]
Since $ T $ is positive,  $ \avg{u,Tu} \leq \avg{\abs{u},T\abs{u}} $, so we may assume $ u \ge 0 $. Moreover, by standard density arguments we may assume $ \norm{u} < \infty $ as well.

Since $ \avg{u,Tu} >1 $, we obtain, by inserting $ u$-projections between the $ T$'s:
\bels{T^nu goes to infinite}{
\avg{u,T^nu} \,\ge\, \avg{u,Tu}\avg{u,T^{n-1}u}\,\ge\,\cdots\,\ge\, \avg{u,Tu}^n\to \infty \qquad\text{as } n\to \infty
\,.
}
The contradiction follows now by combining \eqref{T^nh leq h} and \eqref{T^nu goes to infinite}:
\bels{contradiction is ready}{
\avg{\1h,u} \,\ge\, \avg{T^nh,u} \,=\, \avg{\1h,T^nu} \,\ge\, \avg{\1h,u}\avg{u,T^nu}
\,.
}
The left hand side is less than $ \norm{h}_2\norm{u}_2 = 1 $. On the other hand, since $ h \geq \eps $, $ u \geq 0 $ and $\norm{u}_2 =1$ we have $ \avg{\1h,u} > 0 $. Thus \eqref{T^nu goes to infinite} implies that the right side of \eqref{contradiction is ready} approaches infinity as $ n $ grows.
\end{Proof}

\section{Proofs of auxiliary results in Chapter~5}
\label{sec:Proofs of auxiliary results in Chapter:Properties of solution} 

\begin{Proof}[Proof of Lemma~\ref{lmm:Spectral gap for positive bounded operators}]
\label{Proof of lmm:Spectral gap for positive bounded operators}
First we note that $ h $ is bounded away from zero by
\bels{}{
h \,=\, Th \,\geq\, \eps \int_\Sx \Px(\dif x)\, h_x \,.
}
Let $u$ be orthogonal to $h$ in $\Lp{2} $. Then we compute
\bea{
\avgb{u,(1\pm\1T)u}\;&=\;\frac{1}{2}\int \Px(\dif x)\int\Px(\dif y)\; T_{x y}
\left(
u_x\,\sqrt{\frac{h_y}{h_x}}\,\pm\, u_y\,\sqrt{\frac{h_x}{h_y}}
\;\right)^2
\\
&\geq\; 
\frac{\eps}{2\Phi^2} \int \Px(\dif x)\int\Px(\dif y)\; h_x\, h_y \Biggl(u_x^2\; \frac{h_y}{h_x}\,+\,u_y^2\;\frac{h_x}{h_y}\,\pm\, 2 \,u_x\, u_y\Biggr)
\\
&=\;
\frac{\eps}{\Phi^2} \int \Px(\dif x) \; u_x^2
\;,
}
where in the inequality we used $T_{x y}\geq \eps\geq \eps\2h_x h_y/\Phi^2$ for almost all $x,y \in \Sx$. Now we read off the following two estimates:
\[
\int_\Sx\Px(\dif x)\,u_x \, (Tu)_x
\,\leq\, \left(1-\frac{\eps}{\Phi^2\!}\,\right)\norm{u}_2^2,
\qquad
\int_\Sx\Px(\dif x)\,u_x \, (Tu)_x
\,\geq\,-\left(1-\frac{\eps}{\Phi^2\!}\,\right)\norm{u}_2^2
\,.
\]
This shows the gap in the spectrum of the operator $T$.
\end{Proof}

\begin{proof}[Proof of Lemma~\ref{lmm:Norm of B^-1-type operators on L2}]
In order to prove the claim \eqref{Norm of B^-1-type operators on L2} we will show
\bels{Norm of B^-1-type operators on L2 - with w}{
\norm{(U-T)w}_2
\,\ge\, 
c\,\theta\2\mathrm{Gap}(T)\norm{w}_2 
\,,\qquad 
\theta \2:=\2 
\abs{\21 - \norm{T}_2\avg{\1h,Uh\1}\1}
\,,
}
for all $w \in \Lp{2}$ and for some numerical constant $c>0$. 
To this end, let us fix $w$ with $\norm{w}_2=1$. 
We decompose $ w $ according to the spectral projections of $T$,
\bels{decomposition of w}{
w \,=\, \avg{\1h,w}\2 h + Pw
\,, 
}
where $P$ is the projection onto the orthogonal complement of $t$. During this proof we will omit the lower index $2$ of all norms, since every calculation is in $ \Lp{2} $.
We will show the claim in three separate regimes: 
\begin{itemize}
\item[(i)] 
$ 16\2\norm{Pw}^2 \2\geq \2 \theta$,
\item[(ii)] 
$ 16\2\norm{Pw}^2 \2< \2  \theta$ 
and $\theta\2\geq\2 \norm{PUh}^2$,
\item[(iii)] 
$ 16\2\norm{Pw}^2 \2< \2 \theta$ 
and $ \theta\2<\2 \norm{PUh}^2$.
\end{itemize}

In the regime (i) the triangle inequality yields
\[
\norm{(U-T)w}
\,\ge\, \norm{w}- \norm{Tw} 
\,=\, 1- 
\sqrt{
\abs{\avg{\1h,w}}^2\2\norm{T}^2 + \norm{\1TPw}^2}
.
\]
We use the simple inequality, $1-\sqrt{1-\tau\2}\geq \tau/2$, valid for every $\tau  \in [\10,1\1]$, and find
\begin{equation}
\begin{split}
2\2\norm{(U-T)w}\,
&\ge\, 
1 - \abs{\avg{\1h,w}}^2\1\norm{T}^2 - \norm{TP\1w}^2
\\
&\geq\, 
1 \,- \abs{\avg{\1h,w}}^2\1\norm{T}^2 \,-\, (\1\norm{T}-\mathrm{Gap}(T))^2\norm{Pw}^2
\\
&=\, \,1\,- \norm{T}^2 \,+\, \bigl(\12\1 \norm{T}  - \mathrm{Gap}(T)\1\bigr)\2\mathrm{Gap}(T)\2\norm{Pw}^2
\,.
\end{split}
\end{equation}
The definition of the first regime implies the desired bound \eqref{Norm of B^-1-type operators on L2 - with w}.

In the regime (ii) we project the left hand side of \eqref{Norm of B^-1-type operators on L2 - with w} onto the $h$-direction,
\begin{equation}
\label{regime ii step 1}
\norm{(U-T)w}\,=\, \norm{(\11-U^*T)w}\,\ge\, \abs{\avg{\1h,(\11-U^*T)w}}
\,.
\end{equation} 
Using the decomposition \eqref{decomposition of w} of $ w $ and the orthogonality of $ h $ and $ Pw $, we estimate further:
\bels{regime ii step 2}{
\abs{\avg{\1h,(1-U^*T)w}}
\,&\ge\,
\abs{\avg{\1h,w}}\2\abs{\21-\norm{T}\avg{\1h,U^*t\2}} - \abs{\avg{\1h,U^*TPw}}
\\
\,&\ge\,
\abs{\avg{\1h,w}}\1 \theta\2-\2\norm{PUh}\norm{Pw}
\,.
}
Since $ \theta \leq 2$ and by the definition of the regime (ii) we have $\abs{\avg{\1h,w}}^2=1-\norm{Pw}^2 \geq 1-\theta/16\geq 7/8$ and $\norm{PUh}\norm{Pw}\leq  \theta/4$. Thus, we can combine \eqref{regime ii step 1} and \eqref{regime ii step 2} to 
\[
\norm{(U-T)\1w}\,\geq\, \frac{\theta}{2}
\,.
\]

Finally, we treat the regime (iii). Here, we project the left hand side of \eqref{Norm of B^-1-type operators on L2 - with w} onto the orthogonal complement of $h$ and get
\begin{equation}
\label{regime iii step 1}
\norm{(U-T)w}
\,\ge\, 
\norm{P(U-T)w}
\,\ge\, 
\abs{\avg{\1h,w}}\1\norm{PUh} -\norm{P(U-T)Pw}
\,,
\end{equation}
where we inserted the decomposition \eqref{decomposition of w} again. 
In this regime we still have $\abs{\avg{\1h,w}}^2\geq 7/8$, and we continue with
\begin{equation}
\label{regime iii step 2}
\abs{\avg{\1h,w}}\norm{PUh} -\norm{P(U-T)Pw}
\,\ge\, 
\frac{3}{4}\norm{PUh}- 2\norm{Pw}
\,\ge\, 
\frac{\2\theta^{1/2}\msp{-10}}{2}.
\end{equation}
In the last inequality we used the definition of the regime (iii). 
Combining \eqref{regime iii step 1} with \eqref{regime iii step 2} yields
\[
\norm{(U-T)w}\,\ge\, \frac{\theta}{4}
\,,
\]
after using $ \norm{h}=1$ in \eqref{Norm of B^-1-type operators on L2 - with w} to estimate $ \theta \leq 2 $. 
\end{proof}

\sectionl{Scalability of matrices with non-negative entries}

In this appendix we provide some background material for Sections~\ref{sec:Uniform bound around z=0 when a=0} and \ref{sec:Blow-up at z=0 when a=0 and assumption B1}.
We start by introducing some standard terminology related to matrices with non-negative entries. First, let us denote $ [\1k,\1l\2] := \sett{\1k,k+1,\dots,\1l\2} $, for any integers $ k\leq l$. We use the shorthand $ [n] := [\11\1,n\1] $, and denote the $ \abs{I} \times \abs{J} $-submatrix 
\[ 
\brm{A}(I,J) \,:=\, (\1a_{ij})_{i\in I,j\in J}
\,,
\] for any non-empty sets $ I, J \subset [n] $. 
The set of all permutations of $ [n] $ is denoted by $ S(n) $, and we say that $ \brm{P} = (p_{ij})_{i,j=1}^n $ is a permutation matrix, if its entries are determined by some permutation $ \sigma \in S(n) $ through $ p_{ij} = \delta_{\sigma(i),j} $.

\begin{definition}
\label{def:Square matrices with non-negative entries}
Let $ \brm{A} = (a_{ij})_{i,j=1}^n $ be a square matrix with non-negative entries, $ a_{ij} \ge 0 $. Then:
\begin{itemize}
\item[(i)] $ \brm{A} $ is {\bf scalable} if there exist two diagonal matrices $ \brm{D} $ and $ \brm{D}' $ with  positive entries, such that the scaled matrix $ \brm{D}\brm{A}\brm{D}' $ is doubly stochastic. 
\item[(ii)] 
$ \brm{A} $ is {\bf uniquely scalable} if it is scalable and the pair of diagonal matrices $ (\brm{D},\brm{D}')$ is unique up to a scalar multiple.
\item[(iii)] 
$ \brm{A} $ has {\bf total support} if there exists a set of permutations $ T \subset S(n)$, such that 
\bels{total support defined}{
\qquad
a_{ij}  = 0 
\quad\text{if and only if}\quad
\sum_{\sigma \in T} \delta_{\sigma(i),j} = 0 
\,,\qquad
\forall\2i,j=1,\dots,n
\,.
}
\item[(iv)] 
$ \brm{A} $ is {\bf decomposable}
if it is not fully indecomposable, i.e., there exist two non-empty subsets $ I, J \subset [n] $ such that 
\bels{def of non-FID}{
\brm{A}(I,J) = \brm{0} 
\qquad\text{and}\qquad
\abs{I} +\abs{J} \ge n
\,.
}
\end{itemize}
\end{definition}
We remark that all these four properties of $ \brm{A} $ are invariant under the transformations $ \brm{A} \mapsto \brm{P}\brm{A}\brm{Q} $, where $\brm{P} $ and $ \brm{Q} $ are arbitrary permutation matrices. 
The defining condition \eqref{total support defined} for matrices $ \brm{A} $ with total support means that $ \brm{A}$ shares its zero entries with some doubly stochastic matrix. This fact follows from Birkhoff-von Neumann theorem which asserts that the doubly stochastic matrices are exactly the convex combinations of permutation matrices.

Besides the elementary properties stated in Proposition~\ref{prp:Properties of FID matrices} the fully indecomposable (FID) matrices are also building blocks for matrices with total support. Indeed, Theorem~4.2.8 of \cite{Brualdi91} asserts:
\begin{theorem}
\label{thr:PTSQ  = dsum of FID matrices}
If $ \brm{A} $ has  total support then there exist two permutation matrices $ \brm{P} $ and $ \brm{Q} $ such that $ \brm{P}\brm{A}\brm{Q} $ is a direct sum of FID matrices.
\end{theorem}

Consider the QVE with $ a = 0 $ at $ z = 0 $ in the discrete setup $ (\Sx,\Px) = (\1[n]\1,n^{-1}\abs{\genarg}) $.
From \eqref{DAD problem for S} we read off that this QVE has a unique solution of the form $ m(0) = \cI\2v $ provided the matrix $ \brm{S} $, with entries $ s_{ij} := n^{-1}S_{ij} $, is scalable such that  $ \brm{V}\brm{S}\brm{V} $ is doubly stochastic for the diagonal matrix $ \brm{V} = \diag(v_1,\dots,v_n) $. 
This observation together with the equivalence of (i) and (iii) in the following theorem shows that in the discrete setup the assumption {\bf B1} from Chapter~\ref{chp:Uniform bounds}, with the trivial blocks $ K = n $, is actually optimal in the part (i) of Theorem~\ref{thr:Quantitative uniform bounds when a = 0}.

\begin{theorem}[Scalability and full indecomposability]
\label{thr:Scalability and full indecomposability}
For a symmetric irreducible matrix $ \brm{A} $ with non-negative entries the following are equivalent:
\begin{itemize}
\titem{i} $ \brm{A} $ is uniquely scalable, with $ \brm{D}' = \brm{D} $ in Definition~\ref{def:Square matrices with non-negative entries};
\titem{ii} Every sufficiently small perturbation of $ \brm{A} $ is scalable, i.e., there exists a constant $ \eps > 0 $ such that any symmetric matrix $ \brm{A}' $, with non-negative entries, satisfying $ \max_{i,j}\abs{a_{ij}-a'_{ij}} \leq \eps $, is scalable; 
\titem{iii} $ \brm{A} $ is fully indecomposable. 
\end{itemize}
\end{theorem}

The proof of Theorem~\ref{thr:Scalability and full indecomposability} relies on the following fundamental result. 

\begin{theorem}[\cite{sinkhorn1967}]
\label{thr:General scalability}
A square matrix $ \brm{A} $ with non-negative entries is 
\begin{itemize}
\titem{i}
scalable if and only if it has a total support;
\titem{ii}
uniquely scalable if and only if it is fully indecomposable. 
\end{itemize}

Moreover, if  $ \brm{A} $ is scalable, then the doubly stochastic matrix $ \brm{D}\brm{A}\brm{D}' $, from Definition~\ref{def:Square matrices with non-negative entries}, is unique.
\end{theorem}
  
For the proof of Theorem~\ref{thr:Scalability and full indecomposability} we need also the following representation.

\begin{lemma}[Scalable symmetric matrices]
\label{lmm:Scalable symmetric matrices}
Suppose $ \brm{A} = (a_{ij})_{i,j=1}^n $ is an irreducible symmetric matrix with non-negative entries. 
If $ \brm{A} $ has a total support but is not fully indecomposable, then $ n $ is even, and there exists an $ n/2$-dimensional square matrix $ \brm{B} $, and a permutation matrix $ \brm{P} $, such that 
\bels{Representation for TS-FID matrix}{
\brm{A} \,=\,
\brm{P}
\mat{
\brm{0}\, & \brm{B}\, \\
\2\brm{B}^{\!\trans}\! & \brm{0}\,
}
\brm{P}^{-1}.
}  
\end{lemma}
\begin{Proof}[Proof of Lemma~\ref{lmm:Scalable symmetric matrices}]
Since $ \brm{A} $ is not FID there exists by Definition~\ref{def:Square matrices with non-negative entries} two non-empty subsets $ I, J \subset [n] $, such that \eqref{def of non-FID} holds.
Let us relabel the indices so that $ I = [1,n_2] $, $ J = [n_1,n_3] $, for some $ 1 \leq n_1 \leq n_2 \leq n_3 \leq n $. 
The relabelling corresponds to the conjugation by the permutation matrix $ \brm{P} $ in \eqref{Representation for TS-FID matrix}.  
By definition \eqref{def of non-FID} of $  I $ and $ J $ we have
\bels{}{
\brm{P}^{-1}\!\brm{A}\brm{P} \,=\, \mat{
\,\brm{A}_{11}\msp{-10} & \brm{0} & \brm{0} & \brm{A}_{14}
\\
\brm{0} & \brm{0} & \brm{0} & \brm{A}_{24}
\\
\brm{0} & \brm{0} & \brm{A}_{33}\msp{-10} & \brm{A}_{34}
\\
\,\brm{A}_{14}^{\!\trans}\msp{-10} & \brm{A}_{24}^{\!\trans}\msp{-10} & \brm{A}_{34}^{\!\trans}\msp{-10} & \brm{A}_{44}
}
\,,
}
where the blocks correspond to the four intervals $ I_1 = [1,n_1] $, $ I_2 = [n_1+1,n_2] $, $ I_3 = [n_2+1,n_3] $, and $ I_4 = [n_3+1,n_4] $, respectively. In the case, $ n_{k+1} = n_k $ the interval $ I_k $ is interpreted to be empty.

Now we show that $ \abs{I_2} \leq \abs{I_4} $.
Indeed, $ \brm{P}^{-1}\!\brm{A}\brm{P} $ has a zero block of size $ \abs{I_2} \times (n-\abs{I_4}) $.
No permutation matrix can have such a zero block if $ \abs{I_2} > \abs{I_4} $. As $ \brm{A} $, and thus also $ \brm{P}^{-1}\!\brm{A}\brm{P} $, has total support, the defining property \eqref{total support defined} could not hold for $ \brm{A} $ if $ \abs{I_2} > \abs{I_4} $ were true.

By definitions, $\abs{I} = \abs{I_1}+\abs{I_2} $ and $ \abs{J} = \abs{I_2}+\abs{I_3} $, and by assumption $ \abs{I} + \abs{J} \ge n $. 
Since $ n = \abs{I_1}+\abs{I_2}+\abs{I_3}+\abs{I_4} $, we conclude $ \abs{I_2} \ge \abs{I_4} $.
Since $ \abs{I_2} = \abs{I_4} $ the submatrix $ \brm{A}_{24} $ is square.
This implies that $ \sigma(I_4) = I_2 $ for the permutations $ \sigma \in T $ in the representation \eqref{total support defined}. 
This is equivalent to $ \sigma(I_1\cup I_2 \cup I_3) = I_1 \cup I_3 \cup I_4 $, and thus $ \brm{A}_{14} = \brm{0} $, $ \brm{A}_{34} = \brm{0} $, and $\brm{A}_{44} = \brm{0} $. 

But now we see that $ I_1 $ and $ I_3 $ must be empty intervals, 
otherwise $ \brm{A}_{11} $ would be an independent block of $ \brm{A} $, and thus  $ \brm{A} $ would not be irreducible.
Since $ I_1 = I_3 = \emptyset $,  we conclude $ I = J $. But this leaves us with the representation \eqref{Representation for TS-FID matrix} with $ \brm{B} := \brm{A}_{24} $.
\end{Proof}

\begin{Proof}[Proof of Theorem~\ref{thr:Scalability and full indecomposability}]
The equivalence of (i) and (iii) almost follows from the part (ii) of Theorem~\ref{thr:General scalability}. 
We are only left to exclude the possibility that $ \brm{A} $ is not FID since it is not uniquely scalable for general pairs $ (\brm{D},\brm{D}')$, but is actually uniquely scalable in the more restricted class of 'diagonal solutions' for which $\brm{D}'= \brm{D}$ holds.

To this end we show that if a symmetric and irreducible matrix $ \brm{A} $ with non-negative entries is scalable, then we may always choose $ \brm{D}' = \brm{D} $. 
First we recall that the doubly stochastic matrix $ \brm{B} := \brm{D}\brm{A}\brm{D}' $ is unique according to Theorem~\ref{thr:General scalability}.
Since $ \brm{A} $ is symmetric, $ \brm{D}'\!\brm{A}\brm{D} = \brm{B}^\trans $ is also doubly stochastic. 
By Theorem~\ref{thr:General scalability} we hence have $ \brm{D}'\!\brm{A}\brm{D} \,=\, \brm{D}\brm{A}\brm{D}' $.
We may write this in terms of the ratios $ \rho_i = d'_{ii}/d_{ii} $, as
\bels{cond for rho_i=rho_j}{
\rho_i = \rho_j\,,
\quad\text{whenever}\quad
a_{ij} > 0
\,. 
}
Pick any $ i\neq j $.
Since, $ \brm{A}$  is irreducible, there exists a sequence $ (k_s)_{s=0}^\ell $, $ \ell \leq n $, of indices such that $ k_0 = i $, $ k_\ell = j $, and $ a_{k_{s-1}k_s} > 0 $ for every $ s =1,\dots,\ell$, thus $ \rho_i = \rho_j $ by \eqref{cond for rho_i=rho_j}. 
We conclude $ \brm{D}' = \rho\2\brm{D} $, and thus we may choose $ \brm{D}' = \brm{D} $ by further scaling by a scalar. 

In order to prove the implication (iii) $ \implies $ (ii), choose $ 2\1\eps $ to be equal to the smallest non-zero entry of $ \brm{A} $. It follows that the $\eps $-perturbation $ \brm{A}' $ in (ii)  has a smaller set of entries equal to zero than $ \brm{A} $. 
Thus with this choice of $ \eps $ the zero set of the perturbation $ \brm{A}' $ may only decrease. 
By Definition~\ref{def:Full indecomposability}  $ \brm{A}' $ is thus also FID, and by Theorem~\ref{thr:General scalability} $ \brm{A}' $ is scalable. 

In order to prove the last implication (ii) $ \implies $ (iii), we assume that $ \brm{A} $ is not FID, and derive a contradiction by showing that the perturbed matrix,
\bels{A perturbed in i,j entry}{
\qquad 
\brm{A}' := \brm{A} + \eps\2\brm{\Delta}^{\!(ij)}
\,,
\qquad
(\brm{\Delta}^{\!(ij)})_{kl} := \Ind\sett{\2\sett{k,l}=\sett{i,j}\2}
\,,
}
does not have  total support for all choices of $ (i,j) $, regardless of how small $\eps > 0 $ is chosen.
We start by using Lemma~\ref{lmm:Scalable symmetric matrices} to write $ \brm{A} $ in the form 
\bels{TS representation}{
\brm{A} \,=\,
\mat{
\brm{0}\, & \brm{B}\, \\
\brm{B}^{\!\trans\!} & \brm{0}
}
\,.
}
Here we have also relabelled the indices such that $ \brm{P} = \brm{I} $ in \eqref{Representation for TS-FID matrix}. 
Suppose that we turn one of the zero entries in the first $ n/2 \times n/2 $ diagonal block non-zero, i.e., consider a perturbation \eqref{A perturbed in i,j entry}, for some $ i,j \leq n/2 $. 
We will show that there does not exist a subset $ T' $ of permutations $ S(n) $ such that the representation \eqref{total support defined}, with $ T $ replaced by $ T' $, holds for $ \brm{A}' $.
Indeed, suppose that there is such a set of permutations $ T' $. Since $ a_{ij} > 0 $ there must exist $ \sigma \in T' $ such that $ \sigma(i) = j $. 
This implies that 
\[
[\1n/2\1]\2\backslash\2 \sigma(\1[\1n/2\1]\1) \,=\, \sett{k}
\,,
\]
for some $ k \leq n/2 $. Since $ \sigma $ is a surjection on $ \sett{1,\dots,n} $ there must exist $ l \ge n/2+1$ such that $\sigma(l) = k $. In other words, there exists an entry $ (l,k) $ in the second diagonal block, $ l,k \ge n/2 $, such that $ a'_{lk} = a_{lk} > 0 $. Since this contradicts  \eqref{TS representation} and \eqref{A perturbed in i,j entry}, we conclude that $ \brm{A}' $ does not have total support.
\end{Proof}

\section{Variational bounds when $ \Re\,z = 0$}
\label{ssec:Variational bounds when Re z = 0}

\begin{Proof}[Proof of Lemma~\ref{lmm:Characterization as minimizer}]
Applying Jensen's inequality on the definition \eqref{def of functional J_eta} of $ J_\eta $ yields,
\[
J_\eta(w) 
\,\ge\, \avg{\1w, Sw} \2-\22\2\log\, \avg{w} +2\2\eta \2 \avg{w}
\,.
\]
The lower bound shows that the functional $J_\eta$ is indeed well defined and takes values in $(-\infty,+\infty] $.
Evaluating $J_\eta$ on a constant function shows that it is not identically $+\infty$.

Next we show that $J_{\eta}$  has a unique minimizer on the space $ \Lp{1}_+ $ (cf. definition \eqref{def of Lp_+}) of positive integrable functions. 
As the first step, we show that we can restrict our attention to functions, which satisfy the upper bound $w \leq 1/\eta$. 
To this end, pick $ w  \in \Lp{1}_+ $, such that the set $ \sett{x : w_x \ge \eta^{-1}}$ has positive $ \Px $-measure, and define the one parameter family of $ \Lp{1}_+$-functions
\[
w(\tau) \,:=\, 
w \,-\, \tau\,(w-\eta^{-1})_+\,,
\qquad 0 \leq \tau \leq 1\,,
\]
where $ \phi_+ := \max\sett{0,\phi} $, $ \phi \in \R $.
It follows that $ w(\tau) \leq w(0) = w $ and $ J_\eta(w(\tau))<\infty $ for every $\tau\in[0,1] $.
We will show that
\bels{}{
J_\eta\big(\min(w, \eta^{-1})\big)\,=\,J_\eta(w(1)) \,<\, J_\eta(w)\,.
}
For this we compute
\bels{derivative of J}{
\frac{\dif}{\dif \tau}
J_\eta(w(\tau))
\;=\,
-\12\,\avgbb{ \Big(Sw(\tau) + \eta \2- \frac{1}{w(\tau)\!}\,\Big)\big(w-\eta^{-1}\big)_{\msp{-3}+}}
\,.
}
Since $w\geq 0$ and therefore $Sw\geq 0$, the integrand is positive on the set of $x$ where $w_x>1/\eta$. Thus, the derivative 
\eqref{derivative of J} is strictly positive for $\tau \in [0,1)$. We conclude that the minimizer must be bounded from above by $\eta^{-1}$.

Now we use a similar argument to see that we may further restrict the search of the minimizer to functions which satisfy also the lower bound $w \geq \eta/(1+\eta^2)$. 
To this end, fix  $ w \in \Lp{1}_+ $ satisfying $ J_\eta(w)<\infty$ and $ \norm{w}_\infty \leq \eta^{-1} $. 
Suppose $ w <  \eta/(1+\eta^2) $, on some set of positive $ \Px$-measure, and set
\[
w(\tau) \,:=\,
w \,+\, \Bigl(\frac{\eta}{1\2+\2\eta^2}-w\Bigr)_+\,\tau
\,,
\]
so that $ w = w(0) \leq w(\tau) $, and $ J_\eta(w(\tau)) < \infty$, for every $\tau\in[0,1]$. 
Differentiation yields,
\[
\frac{\dif}{\dif \tau} \, J_\eta(w(\tau))
\,\leq \,
2\,\avgbb{\msp{-2}\Big(\,\frac{1}{\eta} \,+\, \eta \,- \frac{1}{w(\tau)}\,\Big)\Bigl(\frac{\eta}{\,1+\eta^2\!}- w\Bigr)_{\msp{-4}+}}
\,,
\]
where the term $ \eta^{-1} $ originates from  $ \norm{Sw(\tau)} \leq \norm{S}\norm{w(\tau)} \leq \eta^{-1} $.
Since $ \eta^{-1}+\eta = (\eta/(1+\eta^2))^{-1} $, and $ w < \eta/(1+\eta^2) $ on a positive set of positive measure, we again conclude that $J_\eta(w(1)) < J_\eta(w) $.

Consider now a sequence $(w^{(n)})_{n\in\N} $ in $ \Lp{1}_+ $ that satisfies
\[
\lim_{n \to \infty} J_\eta (w^{(n)}) \,=\, \inf_w J_\eta(w)
\qquad\text{and}\qquad
\frac{\eta}{1+\eta^{\12}} \,\leq\, w^{(n)} \,\leq\, \frac{1}{\eta}
\,.
\]
Obviously, $ w^{(n)} $ also constitutes  a bounded sequence of $\Lp{2}_+ $. Consequently, there is a  subsequence, denoted again by $ (w^{(n)})_{n\in\N} $, that converges weakly to an element $ w^\star $ of $ \Lp{2}_+ $. 
This weak limit also satisfies
\bels{w^star uniform bounds}{
\frac{\eta}{1+\eta^2} \,\leq \, w^\star_x \,\leq\, \frac{1}{\eta} 
\,,\qquad \forall\,x \in \Sx\,.
}
In order to conclude that $ w^\star $ is indeed a minimizer of $J_\eta$ we will show that $ J_\eta $ is weakly continuous in $ \Lp{2}_+ $ at all points $w^\star$ satisfying the bounds \eqref{w^star uniform bounds}. 
To this end, we consider the three term constituting $ J_\eta $ separately.
Evidently the averaging $ u \mapsto \avg{u}$ is weakly continuous. 
For the quadratic form we first compute for any sequence $ w^{(n)} $ converging weakly to $w^\star$:
\bels{quadratic S continuity bound}{
\absb{ \avg{w^{(n)}\!,\1Sw^{(n)}} - \avg{w^\star\,Sw^\star} }
\,\leq\, 
\bigl(\1\norm{w^{(n)}}_2+\norm{w^\star}_2\bigr)\,\norm{S(w^{(n)}-w^\star)}_2
\,.
}
Since the $\Lp{2} $-norm is lower-semicontinuous and $\norm{w^\star}_2\leq \norm{w^\star} \leq \eta^{-1}$, we infer
\[
\limsup_{n\to \infty}\absb{ \avg{w^{(n)}\!,\1Sw^{(n)}} - \avg{w^\star\,Sw^\star} }
\,\leq\, 
\frac{2}{\eta}\,
\limsup_{n\to \infty}\,\normb{S(w^{(n)}-w^\star)}_2
\,.
\]
Using the $ L^2$-function, $ S_x: \Sx \to [\10\1,\infty\1), y  \mapsto S_{xy} $, we obtain: 
\bea{
\norm{S(w^{(n)}-w^\star)}_2^2 \,&=\,\int_\Sx\! \Px(\dif x)\, \absB{\int_\Sx\!\Px(\dif y)\,S_{xy}\2(w^{(n)}-w^\star)_y}^2
\\
&=\; 
\int_\Sx \!\Px(\dif x)\, \absb{\avg{S_x(w^{(n)}-w^\star)}}^2
\,.
}
Here the weak convergence of $ w^{(n)} $ to  $ w^\star $ implies $ h^{(n)}_x := \abs{\avg{S_x(w^{(n)}-w^\star)}}^2 \to 0 $ for each  $ x $ separately. 
The uniform bound $ \abs{h^{(n)}_x} \leq \norm{S_x}_2^2\norm{w^{(n)}-w^\star}_2^2 \leq 2(\norm{w^{(n)}}_2^2-\norm{w^\star}_2^2)\norm{S}_{\Lp{2}\to\BB}^2\, $, and the dominated convergence then yield:
\[
\int_\Sx \!\Px(\dif x)\, \absb{\avg{S_x(w^{(n)}-w^\star)}}^2 \;=\; \int_\Sx\!\Px(\dif x)\,h^{(n)}_x \;\to\;0\,,\qquad
\text{as }n \to \infty\,.
\]
Hence the last term of \eqref{quadratic S continuity bound} converges to zero as $ n $ goes to infinity, and we have shown that the quadratic form is indeed weakly continuous at 
$w^\star$. 

Finally, we show that also the logarithmic term is weakly continuous at $w^\star$. Applying Jensen's inequality yields
\bea{
\absb{\la\2 \log w^{(n)} \ra -\la\2 \log w^\star \ra} 
\,=\, 
\absB{\avgB{\log\2\Bigl(\frac{w^{(n)}\!}{w^\star}\Bigr)}}
\,\leq\,
\absB{\,\log \,\avgB{\frac{w^{(n)}\!}{w^\star}}}
\,,
}
where the last average converges to $ 1 $ by the assumed weak convergence of $w^{(n)}$ to $w^\star$ and since $1/w^\star \in \Lp{2} $ by the lower bound in \eqref{w^star uniform bounds}. 

We have proven the existence of a positive minimizer $ w^\star \in \Lp{1} $ that satisfies \eqref{w^star uniform bounds}.
In order to see that $ w^\star_x = v_x(\cI \1\eta) $ for a.e. $ x \in \Sx$ we evaluate a  derivative of $ J_\eta(w^\star+\tau h)|_{\tau=0} $ for an arbitrary $ h \in \BB $. This derivative must vanish by the definition of $ w^\star $, and therefore
\bels{DAD problem for w^star}{
(Sw^\star)_x + \eta \2-\frac{1}{w_x^\star\!} \,=\, 0
\,,\qquad\text{for $ \Px $-a.e. }x \in \Sx
\,.
}
Since $ Sw $, with $ w \in \Lp{2}  $, is insensitive to changing the values of $ w_x $, for $ x \in I $, whenever $ I \subseteq \Sx $ is of measure zero, we may modify $ w^\star $ on the zero measure set where the equation of \eqref{DAD problem for w^star} is not satisfied, so that the equality holds everywhere. 
Since \eqref{DAD problem for w^star} equals QVE at $ z = \cI\1\eta $ Theorem~\ref{thr:Existence and uniqueness} implies that \eqref{DAD problem for w^star} has  $ v(\cI\1\eta) $ as the unique solution. 
We conclude that $ w^\star_x = v_x(\cI\1\eta)$ for a.e. $ x \in \Sx $.
\end{Proof}
\begin{Proof}[Proof of Lemma~\ref{lmm:Uniform bound on discrete minimizer}]
\label{Proof of Uniform bound on discrete minimizer}
Since $ \brm{Z} $ is FID, the exists by the part (ii) of Proposition~\ref{prp:Properties of FID matrices} a permutation $ \sigma $ of the first $ K $ integers, such that 
\[
\wti{\brm{Z}} = (\wti{Z}_{ij})_{i,j=1}^K\,,\qquad
\wti{Z}_{ij} :=Z_{i\sigma(j)}
\,,
\]
has a positive main diagonal, i.e., $ \wti{Z}_{ii} = 1 $ for every $ i $.
Let us define the convex function $ \Lambda : (0,\infty) \to \R $, by
\[
\Lambda(\tau) 
\,:=\, 
\frac{\varphi}{K}\,\tau + \log \frac{1}{\tau} 
\,,
\]
where $ \varphi > 0 $ and $ K \in \N $ are from {\bf B2}. Clearly, $ \lim_{\tau \to \infty} \Lambda(\tau) = \infty $ and $ \lim_{\tau \to 0} \Lambda(\tau) = \infty $. In particular, 
\bels{Lambda lower bound}{
\Lambda(\tau) \,\ge\, \Lambda_- 
\,,
}
where $ \abs{\Lambda_-} \lesssim 1 $, since $ \varphi $ and $ K $ are considered as model parameters, 

Using $ \wti{Z}_{ii} = 1 $ and $ w_i\wti{Z}_{ij}w_{\sigma(j)} \ge 0 $ in the definition \eqref{def of discretized minimizer} of $ \wti{J}(\brm{w}) $, we obtain
\bels{Lower bound wti-J}{
&\sum_i
\Lambda(w_iw_{\sigma(i)})
\\
&\leq\;
\sum_i
\Bigl(\,\frac{\varphi}{K}\, w_i\wti{Z}_{ii} w_{\sigma(i)}-\log\bigl[\,w_i \2 w_{\sigma(i)}\bigr]\Bigr)
\,+\,
\frac{\varphi}{K}\sum_{i\neq j}w_i\wti{Z}_{ij} w_{\sigma(j)}
\,=\,
\wti{J}(\brm{w}) 
\,.
}
Combining the assumption $ \wti{J}(\brm{w}) \leq \Psi $ with the lower bounds \eqref{Lambda lower bound} of $ \Lambda$ yields 
\bels{bounds for w_i-w_pi(i) pairs}{
w_k\1w_{\sigma(k)}
\;\sim\; 1
\,,\qquad 1 \leq k \leq K
\,.
}
Using \eqref{Lambda lower bound} together with \eqref{Lower bound wti-J} and the hypothesis of the lemma, $ \wti{J}(\brm{w}) \leq \Psi $, we obtain an estimate for the off-diagonal terms as well:
\bels{off-diagonal wti-Tw wti-w sum upper bound}{
\frac{\varphi}{K}\sum_{i\neq j}w_i\wti{Z}_{ij}w_{\sigma(j)} 
\;\leq\; 
\wti{J}(\brm{w})
-
\sum_i \Lambda(w_iw_{\sigma(i)})
\;\leq\; \Psi + K\abs{\Lambda_-}\,.
}
Since we consider $ (\varphi,K,\Psi) $ as model parameters, the bounds \eqref{bounds for w_i-w_pi(i) pairs} and \eqref{off-diagonal wti-Tw wti-w sum upper bound} together yield
\bels{upper bound on M_ij}{
M_{ij} := w_i\wti{Z}_{ij}w_{\sigma(j)} \,\lesssim\,1
\,.
}
This would imply the claim of the lemma, $ \max_i w_i \lesssim 1 $, provided we would have $ \wti{Z}_{ij} \gtrsim 1 $ for all $ i,j $. 
To overcome this limitation we compute the $ (K-1)$-th power of the matrix $ \brm{M} $ formed by the components \eqref{upper bound on M_ij}. This way we get to use the FID property of $ \brm{Z} $: 
\bels{lower bound for M_ij}{
(\1\brm{M}^{K-1})_{ij}
&=\;
\sum_{i_1,\dots,i_{K-2}} \msp{-10}
w_i \wti{Z}_{ii_1}w_{\sigma(i_1)}w_{i_1}\wti{Z}_{i_1i_2}w_{\sigma(i_2)}w_{i_2}
\\
&\msp{200}\times\;\wti{Z}_{i_2i_3}w_{\sigma(i_3)}\, \cdots\, w_{i_{K-2}}\wti{Z}_{i_{K-2}j}w_{\sigma(j)}
\\
&\ge\;
\Bigl(\min_k w_kw_{\sigma(k)}\Bigr)^{\!K-2} (\1\wti{\brm{Z}}^{K-1})_{ij} \;w_i\1w_{\sigma(j)}
\,.
}
Since  $ \brm{Z} $ is FID also $ \wti{\brm{Z}} $ is FID, and therefore $ \min_{i,j=1}^K(\2\wti{\brm{Z}}^{K-1})_{ij} \ge 1 $ (cf. the statements (i) and (iii) of Proposition~\ref{prp:Properties of FID matrices}). 
Moreover, by \eqref{bounds for w_i-w_pi(i) pairs} we have
$ \min_k w_kw_{\sigma(k)}
\sim\, 1$.
Thus choosing $ j = \sigma^{-1}(i) $, so that $ w_iw_{\sigma(j)} = w_i^2 $, \eqref{lower bound for M_ij} yields 
\[
w_i^2 \,\lesssim\, (\brm{M}^{K-1})_{\1i\2\sigma^{-1}(i)}
\,.
\]
This is $ \Ord(1) $ by \eqref{upper bound on M_ij}, and the proof is thus completed.
\end{Proof}


\section{H\"older continuity of Stieltjes transform}
\label{ssec:Holder continuity of Stieltjes transform}

In the proof of Proposition~\ref{prp:Holder regularity in z and extension to real line} we used the following quantitative bound which states that the H\"older regularity is preserved under Stieltjes transforms. 

\begin{lemma}[Stieltjes transform conserves H\"older regularity]
\label{lmm:Stieltjes transform inherits regularity}
Let $ \gamma \in (0,1) $. 
Consider an integrable, uniformly $ \gamma $-H\"older-continuous function $ \nu : \R \to \C $, 
\bels{assumption:Holder-continuity of nu}{
\abs{\1\nu(\tau_1)-\nu(\tau_2)}\,\leq\, C_0\2\abs{\1\tau_1-\tau_2}^\gamma
\,, 
\qquad \tau_1,\tau_2 \in \R
\,,
}
where $ C_0 < \infty $.
Then the Stieltjes transform $\Xi : \Cp \to \Cp $ of $ \nu $,
\[
\Xi(\zeta)\,:=\,\int_\R \frac{\,\nu(\tau)\1\dif \tau}{\1\tau\1-\1\zeta\1}
\,,\qquad
\zeta \in \Cp
\,,
\]
is also uniformly H\"older continuous with the same H\"older exponent, i.e., 
\bels{difference of S-transforms}{
\quad
\abs{\2\Xi(\zeta_1)-\Xi(\zeta_2)}
\,\leq\,
\frac{18\2C_0}{\2\gamma\2(1\msp{-1}-\gamma)}\,
\abs{\1\zeta_1-\zeta_2}^\gamma
\,, 
\qquad 
\zeta_1,\,\zeta_2 \in \Cp\,. 
}
\end{lemma}
A similar result can be read off from the estimates of Section 22 of \cite{Mu}. We provide the proof here for the convenience of the reader.

\begin{Proof}
The $ \Lp{1}(\R)$-integrability of $ \nu $ is only needed to guarantee that the Stieltjes transform is well defined on $ \Cp $.
We start by writing $\Xi$ in the form
\bels{StieltjesHoelderContinuityEqA}{
\Xi(\1\omega+\cI\1\eta\1)
\,=\,
\int_\R\frac{\2\nu(\tau)-\nu(\omega)}{\tau-\omega-\cI\1\eta}\,\dif \tau\,+\,\cI \pi\2 \nu(\omega)
\,,
\qquad \omega \in \R\,,\;\eta > 0\,.
}
We divide the proof into two steps. First we show that \eqref{difference of S-transforms} holds in the special case $ \Im\,\zeta_2=\Im\,\zeta_1 $. 
As the second step we show that \eqref{difference of S-transforms} also holds when $ \Re\,\zeta_2=\Re\,\zeta_1 $.
Together these steps imply \eqref{difference of S-transforms} for general $\zeta_1,\zeta_2 $.

Suppose that $ \zeta_k = \omega_k + \cI\1\eta $, for some $ \omega_1,\omega_2 \in \R $ and $ \eta > 0 $.
Using \eqref{StieltjesHoelderContinuityEqA} we write the difference of the Stieltjes transforms in the form
\bels{Re-part:xi-decomposition}{
\Xi(\omega_2+\cI\eta)-\Xi(\omega_1+\cI\1\eta)\,=\,\cI\1 \pi\1 \bigl[\2\nu(\omega_2)-\nu(\omega_1)\2\bigr]  
+\2I_1 +I_2 +I_3 +I_4
\,, 
}
where the integrals have been split into the following four parts:
\bea{
I_k\,&:=\, 
(-1)^k\!\int_\R
\frac{\nu(\tau)-\nu(\omega_k)}{\tau-\omega_k-\cI\1\eta}\,
\Ind\setB{\,\abs{\1\tau-\omega_1} \leq \abs{\1\omega_2-\omega_1}\,}
\,\dif\tau
\,,
\qquad k=1,2\,.
\\
I_3\,&:=\, (\1\nu(\omega_1)-\nu(\omega_2)\1)\int_\R
\frac{1}{\tau-\omega_1-\cI\1\eta}\,
\Ind\setB{\,\abs{\1\tau-\omega_1}> \abs{\1\omega_2-\omega_1}\,}\2\dif \tau
\,,
\\
I_4\,&:=\!
\int_\R (\1\nu(\tau)-\nu(\omega_2)\1)
\biggl(
\frac{1}{\2\tau-\omega_2-\cI\1\eta\1}-\frac{1}{\2\tau-\omega_1-\cI\1\eta\1}
\biggr)
\Ind\setB{\,\abs{\1\tau-\omega_1}> \abs{\1\omega_2-\omega_1}\,}\,\dif \tau
\,.
}
In the regime $ \abs{\tau-\omega_1} > \abs{\1\omega_2-\omega_1} $ we have added and subtracted an integral of $  \nu(\omega_2)\2(\tau-\omega_1-\cI\1\eta)^{-1} $ over $ \tau \in \R $.

The first term on the right hand side of \eqref{Re-part:xi-decomposition} is less than $ \pi\1C_0 \abs{\1\omega_2-\omega_1}^\gamma $ by the hypothesis \eqref{assumption:Holder-continuity of nu}. We will show that $ \abs{\1I_k} \leq C_k \abs{\1\omega_2-\omega_1}^\gamma $, where the constants $ C_k $ sum to something less than the corresponding constant on the right hand side of \eqref{difference of S-transforms}.

Using the $ \gamma$-H\"older continuity \eqref{assumption:Holder-continuity of nu} of $ \nu $, bringing absolute values inside the integrals, and ignoring $ \eta's $, it is easy to see that 
\bels{I_1 and I_2}{
\abs{\1I_1}
\,\leq\, \frac{2\1C_0}{\gamma}\abs{\1\omega_2-\omega_1}^\gamma\,,
\qquad\text{and}\qquad
\abs{\1I_2}
\,\leq\, 
\frac{4\1C_0}{\gamma}\abs{\1\omega_2-\omega_1}^\gamma
\,.
} 

Due to \eqref{assumption:Holder-continuity of nu}, for $ I_3 $ we only need to bound the size of the integral. The real part of the integral vanishes due to the symmetry. The imaginary part of the integral is bounded by $ \int_\R \eta\,(\1\eta^2+\lambda^2)^{-1}\dif \lambda = \pi $, and thus
\bels{I_3}{
\abs{\1I_3} \,\leq\,\pi\1C_0\abs{\1\omega_2-\omega_1}^\gamma
\,.
}

In order to estimate $ I_4 $ we bring absolute values inside the integral, ignore $ \eta $'s
\[
\absbb{\1\frac{1}{\2\tau-\omega_2-\cI\1\eta\1}\,-\,\frac{1}{\2\tau-\omega_1-\cI\1\eta\1}\1} 
\,\leq\,\frac{\abs{\1\omega_1-\omega_2}}{\abs{\tau-\omega_1}\abs{\tau-\omega_2}}
\,,
\]
and use the H\"older continuity \eqref{assumption:Holder-continuity of nu} of $ \nu $. This yields the first bound below:
\bels{I_4}{
\msp{-8}\abs{\1I_4}
\,\leq\, 
C_0\!\int_\R \!\frac{\,\abs{\2\omega_2-\omega_1}\,\Ind\setb{\2\abs{\1\tau-\omega_1}\msp{-1}>\msp{-1} \abs{\1\omega_2-\omega_1}\2}}{\1\abs{\1\tau-\1\omega_1}\1\abs{\2\tau-\omega_1-(\1\omega_2-\1\omega_1)\1}^{1\msp{-1}-\gamma}\msp{-8}}\,\dif \tau
\,\leq
\frac{2\1C_0}{\2\gamma\2(1-\gamma)}\2\abs{\1\omega_2-\omega_1}^\gamma
\!.\msp{-22}
}
Plugging this with \eqref{I_1 and I_2} and \eqref{I_3} into \eqref{Re-part:xi-decomposition} yields 
\bels{Stieltjes: Re parts}{
\abs{\2\Xi(\omega_2+\cI\eta)-\Xi(\omega_1+\cI\1\eta)} \,\leq\, \frac{15\1C_0}{\2\gamma\2(1-\gamma)}
\abs{\1\omega_1-\omega_2}^\gamma
\,.
}

Now it remains to prove \eqref{difference of S-transforms} in the special case, where $ \zeta_k = \omega + \cI\1\eta_k $, for some $ \omega \in \R $ and $ \eta_1,\eta_2 > 0 $. Using again the representation \eqref{StieltjesHoelderContinuityEqA} we obtain
\bea{
\Xi(\1\omega+\cI\1\eta_2)\1-\2\Xi(\omega+\cI\1\eta_1)
\,&=\,
\int_\R (\1\nu(\tau)-\nu(\omega)\1)
\biggl(
\frac{1}{\2\tau-\omega-\cI\1\eta_2}-\frac{1}{\2\tau-\omega-\cI\1\eta_1}
\biggr)\,\dif \tau
\\
&=\,
\cI
\int_\R \frac{\!(\1\eta_2-\eta_1)\2(\1\nu(\tau)-\nu(\omega)\1)\,\dif \tau}{(\1\tau-\omega-\cI\1\eta_2)\2(\1\tau-\omega-\cI\1\eta_1)\2}
\,.
}
Pulling the absolute values inside the integral yields
\bea{
\absb{\2\Xi(\1\omega+\cI\1\eta_2)\1-\2\Xi(\omega+\cI\1\eta_1)\1} 
\,&\leq\,
C_0\!\int_\R 
\frac{\abs{\1\eta_2-\eta_1}\,\dif\tau}{\abs{\1\tau-\omega\1}^{1-\gamma} \2\frac{1}{\!\sqrt{2\2}\1}\bigl(\2\abs{\tau-\omega}+\abs{\eta_2-\eta_1}\1\bigr)\,}
\\
&\leq\;
\frac{\!\sqrt{8\2}C_0}{\1\gamma\2(1-\gamma)}\,\abs{\1\eta_2-\eta_1}^\gamma
\,.
}
Adding this to \eqref{Stieltjes: Re parts} yields \eqref{difference of S-transforms}.
\end{Proof}

\sectionl{Cubic roots and associated auxiliary functions}

\begin{Proof}[Proof of Lemma~\ref{lmm:Stability of roots - pos} and Lemma~\ref{lmm:Stability of roots - neg}]
Let $ p_k :\C \to \C $, $ k \in \N $, denote any branch of the inverse of $ \zeta \mapsto \zeta^k $ so that $ p_k(\zeta)^k = \zeta $. We remark that if  $ p_k $ is the standard complex power function  (cf. Definition~\ref{def:Complex powers}) then the conventional notation $ \zeta^{1/k} $ is used instead of $ p_k(\zeta)$. 
  
The special functions $ \Phi $ and $ \Phi_\pm $ appearing in Lemma~\ref{lmm:Roots of reduced cubic with positive linear coefficient} and Lemma~\ref{lmm:Roots of reduced cubic with negative linear coefficient}, respectively, can be stated in terms of the single function
\bels{def of generic Phi}{
\Phi(\zeta) \,:=\, p_3(\,p_2(\11\msp{-1}+\msp{-1}\zeta^{\12}\1)+\zeta\,)
\,,
}
by rotating $ \zeta $ and $ \Phi $ and choosing the functions $ p_2 $ and $p_3 $ appropriately.  
For example, if $ \abs{\Re\,\zeta} < 1$, i.e., $ \zeta \in \wht{\C}_0 $ (cf. \eqref{defs of C_a}), then $ \Phi(\pm\1\cI\1\zeta\1)^3 = \pm\1\cI\1\Phi_{\mp}(\zeta) $, with the standard definition of the complex powers.
In order to treat both the lemmas in the unified way, we hence consider the generic function \eqref{def of generic Phi} that is analytic on a simple connected open set $ D $  of $ \C $ such that $ \pm\1\cI \notin D $.

Straightforward estimates show that 
\bels{1/2-Holder for generic Phi}{
\abs{\1\Phi(\zeta)-\Phi(\xi)} \,\leq\,C_1\abs{\zeta-\xi}^{1/2}
}
and
\bels{bound on derivative of Phi}{
\abs{\1\partial_\zeta\Phi(\zeta)} 
\,\leq\, 
C_3
\begin{cases}
\abs{\zeta-\cI\1}^{\1-1/2} \!+ \abs{\zeta+\cI\1}^{\1-1/2}
\quad&\text{when }\abs{\zeta}\leq 2
\\
\,\abs{\1\zeta\1}^{\1-2/3}
&\text{when }\abs{\zeta}>2\,.
\end{cases}
} 
The roots $ \wht{\Omega}_a(\zeta) $ defined in both \eqref{solutions to positive reduced cube} and \eqref{roots for small gap} are of the form:
\bels{roots as linear combinations of Phi_k's}{
\Omega(\zeta) \,=\, \alpha_1\Phi^{(1)}\msp{-1}(\omega_1\zeta\1) + \alpha_2\Phi^{(2)}\msp{-1}(\omega_2\zeta\1)
\,.
}
Here $ \Phi^{(1)} $ and $ \Phi^{(2)} $ satisfy \eqref{def of generic Phi} but with different choices of branches and branch cuts for the square and the cubic roots. The coefficients $ \alpha_1,\alpha_2,\omega_1,\omega_2 \in \C $ satisfy $ \abs{\alpha_k} \leq 2 $ and $ \abs{\omega_k} = 1 $ for $ k=1,2$.

The perturbation results of Lemma~\ref{lmm:Stability of roots - pos} and Lemma~\ref{lmm:Stability of roots - neg} now follow from \eqref{bound on derivative of Phi} and the mean value theorem:
\bels{error bound for generic Phi}{
\abs{\Phi(\zeta+\gamma)-\Phi(\zeta)} \,\leq\,\abs{\gamma\1}\sup_{0\1\leq\1 \rho\1\leq\1 1} \abs{\1\partial_\zeta\Phi(\zeta+\rho\1\gamma\1)}
\,.
}
Indeed, Lemma  \ref{lmm:Stability of roots - pos} follows directly by choosing $ D = \setb{\zeta \in \C : \dist(\zeta,\mathbb{G}) \leq 1/4}$ with $ \mathbb{G} $ defined in \eqref{min: def of good set mathbb-G}, and $ \gamma := \xi $. 
Since $ \zeta \in \mathbb{G} \subset D $ the condition \eqref{min stability: assumption} for $ c_1 = 1/12  $ guarantees that $ \zeta+\xi \in D $.
As $ \dist(\1\pm\1\cI\1,D) = 1/4 $ the estimate \eqref{min:stability of roots} follows using \eqref{bound on derivative of Phi} in \eqref{error bound for generic Phi}.

In order to prove \eqref{Perturbation of roots: away from opposite critical point} we consider the case $ \zeta = \cI\2(-\1\theta+\lambda\1) $ and $ \gamma = \cI\2\mu'\1\lambda $, where $ \theta =\pm\11$,  $ \abs{\lambda-2\1\theta}\ge 6\2\kappa $, and $ \abs{\mu'} \leq \kappa$, for some $ \kappa \in (0,1/2) $.
We need to bound the distance between the argument $ \zeta + \rho\1\gamma $, of the derivative in \eqref{error bound for generic Phi} to the singular points $ \pm \1\cI $ from below. 
Assume $ \theta = 1 $ w.l.o.g. Then the distance of $ \zeta + \rho\1\gamma $ from $ -\cI $ is bounded from below by 
\[
\absb{\zeta + \rho\1\gamma +\1\cI\2} \,\ge\,\abs{\lambda}/2\,,
\]  
since $ \abs{\rho\1\mu'} \leq \kappa \leq 1/2 $. Similarly, we bound the distance between $ \zeta + \rho\1\gamma $ and $+\cI $ from below
\bea{
\absb{\zeta + \rho\1\gamma -\1\cI\2} 
\,&=\, 
\absb{2\rho\1\mu' + (1+\rho\1\mu')(\lambda-2)} 
\,\ge\, 
\absb{(1+\rho\1\mu')(\lambda-2)} - 2\1\rho\1\abs{\mu'} 
\\
&\ge\,
\kappa + \abs{\lambda-2}/2
\,,
}
where for the last estimate we have used the assumption  $ \abs{\lambda-2\1\theta} = \abs{\lambda-2} \ge 6\1\kappa $.
These bounds apply for arbitrary $ 0\leq \rho \leq 1 $. 
Hence they can be applied to estimate the derivative in \eqref{error bound for generic Phi} using \eqref{bound on derivative of Phi}. This way we get
\[
\absb{\1\Phi^{(k)}(\zeta+\gamma)-\Phi^{(k)}(\zeta)} \,\leq\, C_4\kappa^{-1/2}\min\setb{\abs{\lambda}^{1/2}\!,\2\abs{\lambda}^{1/3}}\,\abs{\1\mu'}
\,.
\]
Applying this in \eqref{roots as linear combinations of Phi_k's} yields  \eqref{Perturbation of roots: away from opposite critical point}.
\end{Proof}


\providecommand{\bysame}{\leavevmode\hbox to3em{\hrulefill}\thinspace}
\providecommand{\MR}{\relax\ifhmode\unskip\space\fi MR }
\providecommand{\MRhref}[2]{%
  \href{http://www.ams.org/mathscinet-getitem?mr=#1}{#2}
}
\providecommand{\href}[2]{#2}

\end{document}